\documentclass{memo-l}

\usepackage{graphicx,amsmath,amsthm,verbatim,amssymb,enumerate,stmaryrd,fancyhdr}
\usepackage{xcolor}
\definecolor{darkgreen}{rgb}{0,0.75,0}
\definecolor{darkred}{rgb}{0.75,0,0}
\definecolor{darkmagenta}{rgb}{0.5,0,0.5}

\usepackage[colorlinks,citecolor=darkgreen,linkcolor=darkred,urlcolor=darkmagenta]{hyperref}

\newtheorem{theorem}{Theorem}[chapter]
\newtheorem{lemma}[theorem]{Lemma}
\newtheorem{prop}[theorem]{Proposition}
\newtheorem{corollary}[theorem]{Corollary}
\newtheorem{problem}[theorem]{Problem}

\theoremstyle{definition}
\newtheorem{definition}[theorem]{Definition}
\newtheorem{example}[theorem]{Example}

\theoremstyle{remark}
\newtheorem{remark}[theorem]{Remark}

\numberwithin{section}{chapter}
\numberwithin{equation}{chapter}

\newcommand{\R}{\mathbb{R}}
\newcommand{\mr}[1]{{\tt \href{http://www.ams.org/mathscinet-getitem?mr=#1}{MR#1}}}

\newcommand\norm[1]{\left\lVert#1\right\rVert} 

\newcommand{\set}[1]{\left\{ #1 \right\}}
\newcommand{\Set}[2]{\left\{ #1 \, \left| \; #2 \right. \right\}}
\newcommand{\Sett}[2]{\left\{ #1 \; : \; #2 \right\}}
\newcommand{\abs}[1]{{\left\vert\kern-0.25ex #1     \kern-0.25ex\right\vert}}
\newcommand{\EE}{\mathbb{E}}
\newcommand{\N}{\mathbb{N}}
\newcommand{\E}{\mathcal{E}}
\newcommand{\Z}{\mathbb{Z}}
\newcommand{\one}{\mathbf{1}}
\newcommand{\nint}[2]{\llbracket #1, #2 \rrbracket}
\newcommand{\floor}[1]{\left\lfloor {#1} \right\rfloor}

\newcommand{\gr}{\operatorname{grad}}
\def\Xint#1{\mathchoice
{\XXint\displaystyle\textstyle{#1}}%
{\XXint\textstyle\scriptstyle{#1}}%
{\XXint\scriptstyle\scriptscriptstyle{#1}}%
{\XXint\scriptscriptstyle\scriptscriptstyle{#1}}%
\!\int}
\def\XXint#1#2#3{{\setbox0=\hbox{$#1{#2#3}{\int}$ }
\vcenter{\hbox{$#2#3$ }}\kern-.6\wd0}}
\def\dashint{\Xint-}
\newcommand{\BMO}{\operatorname{BMO}}
\newcommand{\supp}{\operatorname{supp}}
\newcommand{\tphi}{\tilde{\phi}}
\newcommand{\tmu}{\tilde{\mu}}
\newcommand{\PP}{\mathbb{P}}
\newcommand{\spec}{\operatorname{Spectrum}}
\makeindex

\begin{document}

\frontmatter

\title[Random walks on metric  measure spaces]{Harnack inequalities and Gaussian estimates for random walks on metric  measure spaces}

\author[M. Murugan]{Mathav Murugan}
\address{Center for Applied Mathematics, Cornell University, Ithaca, NY 14853, USA.}
\email{mkm233@cornell.edu}

\author[L. Saloff-Coste]{Laurent Saloff-Coste}
\address{Department of Mathematics, Cornell University, Ithaca, NY 14853, USA.}
\email{lsc@math.cornell.edu} \thanks{Both the authors were partially supported by NSF grants DMS 1004771 and DMS 1404435.}
\date{\today}
\date{\today}




\begin{abstract}
We characterize Gaussian estimates for transition probability of a discrete time Markov chain in terms of geometric properties of the underlying state space.
In particular, we show that the following are equivalent:
\begin{enumerate}
\item Two sided Gaussian bounds on heat kernel
 \item A scale invariant Parabolic Harnack inequality
 \item Volume doubling property  and a scale invariant Poincar\'{e} inequality.
\end{enumerate}
The underlying state space is a metric measure space, a setting that includes both manifolds and graphs as special cases. 
An important feature of our work is that our techniques are robust to small perturbations of the underlying space and the Markov kernel.
In particular, we show the stability of the above properties under quasi-isometries.
We discuss various applications and examples. 
\end{abstract}

\maketitle

\tableofcontents


\mainmatter

\chapter{Introduction} \label{ch-intro}
The goal of this work is to characterize Gaussian estimates for Markov chains and parabolic Harnack inequality for a corresponding discrete time version of heat equation by two geometric properties
on the state space
\begin{enumerate}[1.]
 \item Large scale volume doubling property
 \item Poincar\'{e} inequality.
\end{enumerate}
A precise statement of this characterization is given in Theorem \ref{t-main0}.
The Gaussian estimates mentioned are upper and lower bounds for the iterated transition probability kernel.
The  parabolic Harnack inequality is a regularity estimate for non-negative solutions of the discrete time heat equation given by $u(k+1,x) = [P u(k,\cdot) ](x)$, where $P$ is the Markov operator corresponding to the given Markov chain.

The hardest and most useful implication in the characterization is that the conjunction of the volume doubling property and Poincar\'{e} inequality
implies the two sided Gaussian estimates and parabolic Harnack inequality.
The volume doubling property  and Poincar\'{e} inequality are concrete properties the validity of which can be verified given the geometric data on the space.
Also, an important consequence of this characterization is the stability of Gaussian estimates and parabolic Harnack inequality under quasi-isometric transformation of the underlying space.

An analogous characterization is well-known for diffusions on Riemannian manifolds \cite{Gri91,Sal92}(or more generally local Dirichlet spaces \cite{Stu96}) and for discrete time Markov chains on graphs \cite{Del99}.
We extend the characterization of Gaussian estimates for Markov chain to a large family of state spaces that includes both graphs and Riemannian manifolds.
Various applications of Gaussian estimates and Harnack inequalities are discussed.

Another motivation comes from the work of Hebisch and Saloff-Coste \cite{HS93} on random walks on groups. By the main results of \cite{HS93}, we know that many natural translation-invariant
Markov chains on groups (discrete and continuous groups) of polynomial
volume growth satisfy two-sided Gaussian estimates. However the arguments in \cite{HS93} for proving Gaussian lower bounds are specific to the case of translation-invariant Markov chains as the authors of \cite{HS93} note
``We want to emphasize that a number of key points of the argument presented below are specific to the case of translation invariant Markov chains''.
To this end they conjecture   ``We have no doubt that, if $G$ has polynomial volume growth a corresponding Gaussian lower bound holds for (non transition-invariant) Markov chains as well.
However, we have not been able to prove this result. We hope to come back to this question in the future.'' \cite[Remark 2]{HS93}. Our work validates their conjecture.

A remarkable feature of our work is that the arguments we develop are robust under small perturbations of the Markov kernel and the geometry of the underlying state space.
In particular, we show that parabolic Harnack inequality and Gaussian estimates for heat kernel for symmetric Markov chains is stable under quasi-isometric change of the state space
and small changes in the Markov kernel. We do not rely on symmetries of the space (like group structure or transitivity) or on algebraic properties of the kernel (like translation invariance).
As a consequence, the main results are new even when the state space is  $\R^n$.

Heat kernel estimates and Harnack inequalities have been subjects of extensive research for more than fifty years.
To place our results in a historical context, we will describe precisely the characterization of Gaussian estimates of heat kernel and parabolic Harnack inequality in the context
of diffusions over manifolds developed in \cite{Gri91,Sal92}. We will also  mention several related works, applications and other historical remarks.

\section{Diffusions on Riemannian manifolds}
For the purpose of the introduction, we describe our results in the restricted setting of weighted Riemannian manifolds.
Let $(M,g)$ be a complete  Riemannian manifold equipped with the Riemannian measure $\nu(dy)$.
A weighted Riemannian manifold $(M,g,\mu)$ is a Riemannian manifold $(M,g)$ equipped with a measure $\mu(dy) = \sigma(y) \nu(dy)$, where $0<\sigma \in \mathcal{C}^\infty(M)$ is the \emph{weight},
and the associated \emph{weighted Laplacian} is given by  $\Delta = - \sigma^{-1} \operatorname{div}\left( \sigma \operatorname{grad} \right)$\footnote{The negative sign is to ensure that $\Delta$ has non-negative spectrum. Note that $\Delta$
depends on the Riemannian metric $g$ and the weight $\sigma$.}.
We might sometimes consider a Riemannian manifold $(M,g)$ as a weighted Riemannian manifold equipped with Riemannian measure and Laplace-Beltrami operator.
We denote the   balls centered at $x$ and radius $r$ by $B(x,r):= \Sett{y}{d(x,y) \le r}$ and the volume of the balls by
$V(x,r) := \mu \left( B(x,r)\right)$. We denote the open balls by  $B(x,r)^\circ:= \Sett{y}{d(x,y) < r}$

The \emph{heat kernel} of the weighted Riemannian manifold $(M,g,\mu)$ is the fundamental solution of a parabolic partial differential equation, the heat diffusion equation
\begin{equation}\label{e-heat}
 \left(\frac{\partial }{\partial t} + \Delta \right) u =0.
\end{equation}
That is heat kernel is a  function $(t,x,y) \mapsto p(t,x,y)$ defined on $(0,\infty)\times M \times M$ such that for each $y \in M$, $(t,x) \mapsto p(t,x,y)$  is a solution of \eqref{e-heat}
and for any $\phi \in \mathcal{C}_c^\infty(M)$, $u(t,x) = \int_M p(t,x,y) \phi(y) \mu(dy)$ tends to $\phi(x)$ as $t$ tends to $0$. In other words, the heat kernel allows
us to solve the Cauchy initial value problem for \eqref{e-heat}.
Equivalently, we may view $p(t,x,y) \mu(dy)$ as the distribution at time $t$ of a stochastic  process $\left( X_t \right)_{t >0}$
started at $x$ (the diffusion driven by $\Delta$ on M). These two viewpoints are related by the formula
\begin{equation} \label{e-view}
 u(t,x)= \int_M p(t,x,y) \, \mu(dy)=  \EE_x (u_0(X_t))
\end{equation}
where $u$ is the solution of  Cauchy initial value problem for \eqref{e-heat} with initial value condition $u_0$.

The most classical example of heat kernel is the Gauss-Weierstrass kernel on $\R^n$  equipped with the Lebesgue measure, the  Laplacian $\Delta= -\sum_{i=1}^n \frac{\partial^2 }{\partial x_i^2}$ and the heat kernel is
given by
\[
p(t,x,y) = \left(\frac{1}{4 \pi t} \right)^{n/2} \exp \left(- \frac{d(x,y)^2}{4t} \right),\hspace{2mm} t >0 ,\hspace{2mm} x,y \in \R^n.
\]
We will present a well-known geometric characterization of those weighted Riemannian manifolds on which the heat kernel satisfies \emph{two-sided Gaussian bounds}, 
that is having the property that there exists positive reals $c_1,c_2,C_1,C_2$ such that
\[
 \frac{c_1}{V(x,\sqrt{t})} \exp \left( - \frac{d(x,y)^2}{c_2 t} \right) \le p(t,x,y) \le  \frac{c_1}{V(x,\sqrt{t})} \exp \left( - \frac{d(x,y)^2}{c_2 t} \right)
\]
for all $t> 0$ and for all $x,y \in M$.

Next, we describe Harnack inequalities for $(M,g,\mu)$ equipped with the weighted Laplacian $\Delta$.
We say that $(M,g,\mu)$ satisfies \emph{elliptic Harnack inequality} if there exists $C>0$ such that any non-negative harmonic function $u$  in a ball $B(x,r)$ (that is $u$ satisfies $\Delta u \equiv 0$ in $B(x,r)^\circ$)
satisfies the inequality
\begin{equation} \label{e-ehi-i}
 \sup_{B(x,r/2)} u \le C  \inf_{B(x,r/2)} u.
\end{equation}
The constant $C \in (0,\infty)$ is independent $x$, $r$ and $u$. An important consequence of the elliptic Harnack inequality in $\R^n$ for the Laplacian $\Delta=-\sum_{i=1}^n \frac{\partial^2}{\partial x_i^2}$ is that
global positive harmonic functions must be constant (Liouville property).

J. Moser \cite{Mos61} proved elliptic Harnack inequality \eqref{e-ehi-i} for divergence form operators of the type
\begin{equation} \label{e-uei}
 \mathcal{L} = -\sum_{i,j=1}^n \frac{\partial}{\partial x_i} a_{i,j}\frac{\partial}{\partial x_j}
\end{equation}
where $a_{i,j}$ are bounded measurable real functions on $\R^n$ satisfying $a_{i,j}=a_{j,i}$ and the uniform ellipticity condition:
\[
 \forall x \in \R^n, \hspace{2mm} \forall \xi \in \R^n,\hspace{3mm} \lambda \norm{\xi}^2 \le \sum_{i,j} a_{i,j}(x) \xi_i \xi_j \le \Lambda \norm{\xi}^2
\]
for two constants  $0<\lambda \le \Lambda < \infty$. This elliptic Harnack inequality implies the crucial H\"{o}lder continuity for solutions\footnote{by solutions we mean weak solutions.} of the associated elliptic equation $\mathcal{L}u \equiv 0$, a result
proved earlier by E. de Giorgi \cite{DeG57} and J. Nash \cite{Nas58} (See also \cite{Mos60}).

An important motivation behind the H\"{o}lder continuity of solutions obtained by de Giorgi, Nash and Moser \cite{DeG57, Nas58, Mos60}
was to solve one of the famous Hilbert problems. Hilbert's nineteenth problem asks whether the minimizers of Dirichlet integrals
\[
 E(u) = \int_\Omega F(\nabla u(x) ) \, dx
\]
are always smooth, if $F$ is smooth and strictly convex, where $\Omega \subset \R^n$ is bounded.
E. de Giorgi and J. Nash independently answered Hilbert's question in the affirmative.
We refer the interested reader to \cite[Theorem 14.4.1]{Jos13} for a detailed exposition
of the smoothness of the minimizers of Dirichlet integrals using H\"{o}lder regularity estimates of de Giorgi and Nash.

It is a long standing open problem to characterize (in geometric terms) those weighted manifolds that satisfy elliptic Harnack inequality.
A related open question is to determine whether or not elliptic Harnack inequality is preserved under quasi-isometries.
However several examples of Riemannian manifolds that satisfy elliptic Harnack inequality are known. For instance, Cheng and Yau \cite{CY75}
proved that there exists a constant depending only $n=\operatorname{dim}(M)$ such that for any positive solution $u$ of $\Delta u = 0$ in $B(x,r)^\circ$ on a
Riemannian manifold $(M,g)$ with Ricci curvature bounded from below by $-K$ for some $K \ge 0$ satisfies
\begin{equation}\label{e-elgrad}
 \abs{\nabla \ln(u)} \le C (r^{-1} + K) \hspace{2mm} \mbox{ in $B(x,r/2)$.}
\end{equation}
When $K=0$, integrating the  gradient estimate along minimal paths we  immediately obtain the elliptic Harnack inequality \eqref{e-ehi-i} for Riemannian
manifolds with non-negative Ricci curvature.

We now describe the parabolic version of \eqref{e-ehi-i}. For any $x \in M$, $s \in \R$, $r >0$, let $Q=Q(x,s,r)$ be the cylinder
\[
 Q(x,s,r)= (s-r^2,s) \times B(x,r)^\circ.
\]
Let $Q_+$ and $Q_-$ be respectively the upper and lower sub-cylinders
\[
 Q_+ = (s- (1/4)r^2,s)\times B(x,r/2)^\circ, \hspace{3mm} Q_-= (s-(3/4)r^2,s-(1/2)r^2)\times B(x,r/2)^\circ.
\]
We say that $(M,g,\mu)$ satisfies \emph{parabolic Harnack inequality} if there exists a constant $C$ such that
for all $x \in M$, $s \in \R$, $r>0$ and for all positive solutions of  $\left(\frac{\partial}{\partial t} - \Delta \right)u =0$ in $Q=Q(x,s,r)$, we have
\begin{equation}\label{e-phi-i}
 \sup_{Q_-} u \le C \inf_{Q_+} u.
\end{equation}
The constants $1/4,3/4,1/2$ appearing in the definition of $Q_+,Q_-$ are essentially arbitrary choices.
The main difference between elliptic and parabolic Harnack inequalities is that the cylinders $Q_+$ and $Q_-$ are disjoint in \eqref{e-phi-i} whereas
in the elliptic case \eqref{e-ehi-i} the infimum and supremum are taken over the same ball.

J. Moser attributes the first parabolic Harnack inequality to Hadamard and Pini for operators with constant coefficients on $\R^n$.
In \cite{Mos64}, J. Moser proved the parabolic Harnack inequality for uniformly elliptic operators in divergence form as given by \eqref{e-uei}.
As in the elliptic case, the parabolic Harnack inequality \eqref{e-phi-i} implies H\"{o}lder continuity of the corresponding solutions.
This H\"{o}lder continuity was first obtained by J. Nash \cite{Nas58} in the parabolic setting and Moser's parabolic Harnack inequality gives an alternative proof of H\"{o}lder continuity.
For a proof of Harnack inequality using the ideas of Nash, we refer the reader to the work of Fabes and Stroock \cite{FS86}

The gradient estimates \eqref{e-elgrad} of Cheng and Yau was generalized to the parabolic case by P. Li and S.T. Yau in \cite{LY86}.
The parabolic gradient estimates in \cite{LY86} implies that complete Riemannian manifolds with non-negative Ricci curvature satisfy parabolic Harnack inequality \eqref{e-phi-i}.

In contrast to elliptic Harnack inequality, there is a satisfactory description of weighted Riemannian manifolds that satisfy the parabolic Harnack inequality as described below.
\begin{theorem}\label{t-grisal} (\cite{Gri91,Sal92})
Let $(M,g,\mu)$ be a weighted, non-compact, complete Riemannian manifold equipped with the weighted Laplacian $\Delta$ and let  $p(t,x,y)$ denote the corresponding heat kernel. The following three properties are equivalent:
 \begin{enumerate}[(a)]
  \item The parabolic Harnack inequality: there exists a constant $C_H>0$ such that, for any ball $B=B(x,r)$, $x \in M$, $r >0$ and for any smooth positive solution
  $u$ of $\left(\frac{\partial}{\partial t} + \Delta \right)u =0$ in the cylinder $(s-r^2,s)\times B(x,r)^\circ$, we have
  \[
   \sup_{Q_-} u \le C_H \inf_{Q_+} u
  \]
with
\[
  Q_+ = (s- (1/4)r^2,s)\times B(x,r/2)^\circ, \hspace{3mm} Q_-= (s-(3/4)r^2,s-(1/2)r^2)\times B(x,r/2)^\circ.
\]
\item Two sided Gaussian estimates of the heat kernel: there exists positive reals $c_1,c_2,C_1,C_2$ such that
\[
 \frac{c_1}{V(x,\sqrt{t})} \exp \left( - \frac{d(x,y)^2}{c_2 t} \right) \le p(t,x,y) \le  \frac{c_1}{V(x,\sqrt{t})} \exp \left( - \frac{d(x,y)^2}{c_2 t} \right)
\]
for all $t> 0$ and for all $x,y \in M$.
\item The conjunction of
\begin{itemize}
 \item The volume doubling property: there exists $C_D>0$ such that for all $x \in M$, for all $r>0$ we have
 \[
 V(x,2r) \le C_D V(x,r).
 \]
\item The Poincar\'{e} inequality: there exists $C_P>0$,$\kappa \ge 1$ such that for any ball $B=B(x,r)$, $x \in M$, $r >0$ and for all $f \in C^\infty(M)$, we have
\begin{equation}\label{e-opoin}
 \int_B \abs{f-f_B}^2 \,d\mu \le C_P r^2 \int_{\kappa B} \abs{\operatorname{grad} f}^2 \, d\mu,
\end{equation}
where $\kappa B= B(x,\kappa r)$ and $f_B= \frac{1}{\mu(B)} \int_B f \, d\mu$.
\end{itemize}
 \end{enumerate}
\end{theorem}
\begin{example} We present examples of complete, non-compact, weighted Riemannian manifolds satisfying parabolic Harnack inequality and Gaussian bounds on the heat kernel.
We refer the reader to \cite[Section 3.3]{Sal10} for a more extensive list of examples.
\begin{itemize}
 \item Complete Riemannian manifolds with non-negative Ricci curvature.
  The parabolic Harnack inequality was first obtained in this case by Li and Yau using a gradient estimate \cite{LY86}.
  The volume doubling property follows from Bishop-Gromov inequality \cite[Theorem III.4.5.]{Cha06} and the Poincar\'{e} inequality follows from the work of P. Buser \cite{Bus82} (See \cite[Theorem 5.6.5]{Sal02} for a different proof).
  \item Convex domains and complement of convex domains in Euclidean space. We refer the reader to the monograph \cite{GS12} for this and other examples in this spirit.
  \item Connected Lie groups with polynomial volume growth. By a theorem of Y. Guiv'arch, we know that Lie groups with polynomial volume growth satisfies volume doubling property.
   Moreover, Lie groups with polynomial volume growth satisfy Poincar\'{e} inequality\cite[Theorem 5.6.1]{Sal02}. Examples include nilpotent Lie groups like Euclidean spaces and Heisenberg group.
   See also \cite[Theorem VIII.2.9]{VSC92}.
   Moreover volume doubling property and Poincar\'{e} inequality holds for subelliptic `sum of squares' operators satisfying the H\"{o}rmander condition \cite[Chapter V and VIII]{VSC92} under the
     Carnot-Carath\'{e}odory  metric. See also \cite[Section 5.6.1]{Sal02}.
   \item The Euclidean space $\R^n$, with $n \ge 2$ and weight $(1 + \abs{x}^2)^{\alpha/2}$, $\alpha \in (-\infty,\infty)$ satisfies parabolic Harnack inequality if and only if
   $\alpha > -n$. It satisfies the elliptic Harnack inequality for all $\alpha \in \R$.  These examples are from \cite{GS05}.
 \item Any  complete, weighted Riemannian manifold with bounded geometry that is quasi-isometric to a complete, weighted Riemannian manifold satisfying parabolic Harnack inequality.
 We say a weighted Riemannian manifold $(M,g,\mu)$ with weight $\sigma$ has \emph{bounded geometry} if (a) There exists $K \ge 0$ such that $\operatorname{Ric} \ge - Kg$ (b) There exists $C_1 >1$ such that
 $\sigma(x)/\sigma(y) \in (C_1^{-1} ,C_1)$ for all $x,y \in M$ with $d(x,y) \le 1$ (c) There exists $C_2>1$ such that $ C_2^{-1} \le V(x,1) \le C_2$ for all $x \in M$.
 This illustrates the stability of parabolic Harnack inequality and  two-sided Gaussian estimates under quasi-isometry \cite{CS95,Kan85}.
\end{itemize}
\end{example}
The primary goal of our work is to extend Theorem \ref{t-grisal} in the context of discrete time Markov chains on a large class of spaces that include both weighted Riemannian manifolds and graphs.
As mentioned before the hardest and most useful part of the Theorem \ref{t-grisal} is (c) implies (a) and (b).
The implication (c) implies (a) was proved independently by Grigor'yan \cite{Gri91} and Saloff-Coste \cite{Sal92}.
Both \cite{Gri91} and \cite{Sal92} observed that volume doubling is necessary to obtain (a).
In \cite{Sal92}, Saloff-Coste proved that Poincar\'{e} inequality is also a necessary condition to prove (a) using an argument due to Kusuoka and Stroock \cite{KS87}.

The proof of (c) implies (a) in \cite{Sal92} is an adaptation of Moser's iteration method.
Moser's iteration method relies on Poincar\'{e} inequality and a Sobolev inequality.
The main contribution of \cite{Sal92}
is to obtain a Sobolev inequality using volume doubling and Poincar\'{e} inequality (See also \cite[Chapter 5]{Sal02}, \cite{Sal95}).
A. Grigor'yan \cite{Gri91} carried out a different  iteration argument that relied on an equivalent Faber-Krahn inequality instead of a Sobolev inequality to prove (c) implies (a).
Using the methods of \cite{Sal92}, K.T. Sturm \cite{Stu96}  generalized the above equivalence to diffusions on strongly local Dirichlet spaces.
More recently in \cite{HS01}, Hebsich and Saloff-Coste developed an alternate approach to prove Gaussian bounds and parabolic Harnack inequality using (a).
This approach relies on an elliptic H\"{o}lder regularity estimate and Gaussian upper bounds to prove parabolic Harnack inequality. We will use the approach  outlined in \cite{HS01} in our work.

Aronson \cite{Aro67} was the first to use parabolic Harnack inequality to obtain Gaussian bounds on the heat kernel in the context of divergence form uniformly elliptic  operators in $\R^n$ as given in \eqref{e-uei}.
Although in Aronson's work, the parabolic Harnack inequality was used only to obtain Gaussian lower bounds, both Gaussian upper and lower bounds can be easily obtained using
parabolic Harnack inequality. Conversely Nash's approach aimed at deriving Harnack inequality from two-sided Gaussian bounds on the heat kernel was further developed by
Krylov and Safonov \cite{KS80} and by  Fabes and Stroock \cite{FS86}.

\section{Random walks on graphs}
T. Delmotte extended Theorem \ref{t-grisal} for discrete time Markov chains on graphs, which we now describe. To precisely describe the result, we will introduce some notions concerning
symmetric Markov chains.
Let $(M,d,\mu)$ be a \emph{metric measure space} by which we mean a metric space $(M,d)$ equipped with a Borel measure $\mu$. Recall that we denote closed ball by $B(x,r)$ and their measure by $V(x,r)=\mu(B(x,r))$.
We require $V(x,r) \in (0,\infty)$ for all $x \in M$ and for all $r \in (0,\infty)$.

Let $(X_n)_{n \in \N}$ be a Markov chain with state space $M$ and let $P$ be the corresponding Markov operator.
Further, we assume that $P$ has a kernel $p_1:M \times M \to \R$ with respect to the measure $\mu$, that is for each $x \in M$, we have  $p_1(x,\cdot) \in L^1(M,\mu)$ satisfying
\begin{equation} \label{e-fund}
 P f(x) = \EE_x f(X_1) = \int_M p_1(x,y) f(y) \, \mu(dy)
\end{equation}
for all $f \in L^\infty(M)$. Here $\EE_x$ denotes that the Markov chain starts at $X_0=x$.
The equation \eqref{e-fund} represents the basic relation between the Markov chain $(X_n)_{n \in \N}$, corresponding Markov operator $P$ and its kernel $p_1$ with respect to $\mu$.
We will assume that our Markov chain is stochastically complete that is $P \one = \one$ or equivalently $\int_M p_1(x,y) \mu (dy) = 1$ for all $x \in M$.

We further assume that the kernel $p_1$ satisfies $p_1(x,\cdot) \in L^\infty(M,\mu)$ for all $x \in M$ and that $p_1$ is symmetric
\begin{equation} \label{e-sym}
 p_1(x,y) = p_1(y,x)
\end{equation}
for $\mu \times \mu$-almost every $(x,y) \in M\times M$. Under the symmetry assumption \eqref{e-sym} and the assumption $p_1(x,\cdot) \in L^\infty(M,\mu)$ for all $x \in M$, we define the iterated Markov kernel as
for the Markov chain as
\[
 p_{k+1}(x,y):=  \int_M p_k(x,z) p_1(y,z) \, \mu(dz)
\]
for all $x,y \in M$ and for all $k \in \N^*$.
It is easy to check  that $p_k(x,y)\mu(dy)$ is the distribution of $X_k$ given that the random walk starts at $X_0=x$ (See Lemma \ref{l-kernel}).
The function $(k,x,y) \mapsto p_k(x,y)$ is called the `heat kernel' for the symmetric Markov chain $(X_n)_{n \in \N}$ driven by $P$ on $(M,d,\mu)$.

Next, we introduce the Laplacian and heat equation for discrete time Markov chains.
The Laplace operator $\Delta_P$ corresponding to the random walk driven by $P$ is
\[
 \Delta_P = I - P.
\]
The corresponding discrete time heat equation is
\begin{equation} \label{e-disheat}
 \partial_k u + \Delta_P u_k  \equiv 0
\end{equation}
for all $k \in \N$, where $\partial_k u ( \cdot) = u(k+1,\cdot) - u(k,\cdot)$ denotes the  difference operator and $u_k(\cdot) = u(k,\cdot)$.
Note that \eqref{e-disheat} can be rewritten as $u_{k+1}= P u_k$. Therefore the `solution' to the heat equation \eqref{e-disheat} can be
written as
\[
 u(k,x)= P^k u_0(x) = \int_M p_k(x,y) u_0(y) \, \mu(dy) = \EE_x u_0(X_k)
\]
for all $x \in M$ and for all $k \in \N^*$ where $u_0$ is the initial value.
Note that the above equation is analogous to its continuous time counterpart \eqref{e-view}.


To describe the work of T. Delmotte, we consider a given graph as a metric measure space $(M,d,\mu)$ where $M$ is the vertex set of the graph, $d$ is the graph distance metric and $\mu$ is a measure on the set of vertices such that each vertex has positive measure.
In this context $p_1(x,y)=p_1(y,x)$ for all $x,y \in M$ is sometimes called the \emph{conductance}. We denote integer intervals by $\nint{a}{b} = \Sett{k \in \Z}{a \le k \le b}$ for any $a,b \in \Z$.
The following theorem of T. Delmotte is  the analogue of Theorem \ref{t-grisal} for Markov chains on graphs.
\begin{theorem}\label{t-Del} (\cite{Del99}) Let $(M,d,\mu)$ be an infinite graph equipped with a measure $\mu$ on the set of vertices $M$.
Consider a Markov chain on $M$ with a symmetric kernel $p_k$ with respect to  $\mu$.
Further we assume that there exists $\alpha>0$ such that \footnote{ The upper bound in \eqref{e-delas} was not explicitly stated in \cite{Del99}. However the upper bound must hold due to the volume doubling property. Moreover the statement
of Poincar\'{e} inequality and parabolic Harnack inequality is slightly different but equivalent to the ones presented in \cite{Del99}.}
\begin{equation} \label{e-delas}
\alpha \frac{\one_{B(x,1)}(y)}{V(x,1)} \le p_1(x,y) \le \alpha^{-1}  \frac{\one_{B(x,1)}(y)}{V(x,1)}
\end{equation}
for all $x,y \in M$. Then the following properties are equivalent:
\begin{enumerate}[(a)]
 \item The parabolic Harnack inequality: there exists $\eta \in (0,1)$,$C_H>1,R_H>0$ such that for all balls $B(x,r)$, $x \in M$, $r > R_H$ and
 for all non-negative functions $u: \N \times M \to \R$ that satisfies
 $\partial_k u + \Delta u_k \equiv 0$ in $\nint{0}{\lfloor 4\eta^2 r^2 \rfloor} \times B(x,r)$, we have
 \[
  \sup_{Q_\ominus} u \le C_H \inf_{Q_\oplus} u
 \]
where
\begin{align*}
 Q_\ominus &= \nint{\lceil (\eta^2/2) r^2 \rceil}{\lfloor \eta^2 r^2 \rfloor} \times B(x,(\eta/2) r), \\
 Q_\oplus &= \nint{ \lceil2 \eta^2 r^2 \rceil}{\lfloor 4 \eta^2 r^2 \rfloor} \times B(x,(\eta/2) r)
 \end{align*}
 \item Two sided Gaussian bounds on the heat kernel: there exists $C_1,C_2>0$ such that
for all $x,y \in M$ and for all $n \in \mathbb{N}^*$ satisfying $n \ge 2$, we have
\begin{equation}\label{e-guei} p_n(x,y) \le \frac{C_1}{V(x,\sqrt{n})} \exp \left( - \frac{d(x,y)^2}{C_2n} \right)
\end{equation}
Further there exists $c_1,c_2,c_3>0$ such that for all $x,y \in M$ satisfying $d(x,y)\le c_3 n$ and for all $n \in \mathbb{N}^*$ satisfying $n \ge 2$
\begin{equation}\label{e-glei} p_n(x,y) \ge \frac{c_1}{V(x,\sqrt{n})} \exp \left( -\frac{ d(x,y)^2}{c_2n }\right)
\end{equation}
\item The conjunction of
\begin{itemize}
 \item The volume doubling property: there exists $C_D>0$ such that for all $x \in M$, for all $r>0$ we have
 \[
 V(x,2r) \le C_D V(x,r)
 \]
\item The Poincar\'{e} inequality: there exists $C_P>0$,$\kappa \ge 1$ such that for any ball $B=B(x,r)$ that satisfies $x \in M$, $r >1$ and for all $f \in L^2(M,\mu)$, we have
\[
 \int_B \abs{f-f_B}^2 \,d\mu \le C_P r^2 \int_{\kappa B} \left( \frac{1}{V(y,1)} \int_{B(y,1)} \abs{f(z) -f(y)}^2 \, \mu(dz)\right) \, \mu(dy),
\]
where $\kappa B= B(x,\kappa r)$, $f_B= \frac{1}{\mu(B)} \int_B f \, d\mu$.
\end{itemize}
\end{enumerate}
\end{theorem}
Delmotte's strategy to prove Theorem \ref{t-Del} is to use Moser's  iteration  method as developed in \cite{Sal92,Sal95} to prove a continuous time parabolic Harnack inequality.
The next step is to prove
estimates on the corresponding continuous time kernel obtained using parabolic Harnack inequality. Then a comparison between
discrete and continuous time kernels provides Gaussian bounds on $p_k$ which in turn yields parabolic Harnack inequality for the discrete time heat equation \eqref{e-disheat}.
The comparison argument is tricky because the continuous time heat kernel has non-Gaussian behavior as discovered by Pang \cite{Pan93} and E.B. Davies \cite{Dav93}.
The discrete nature of space and time causes numerous other difficulties during Moser iteration that were overcome successfully by Delmotte.

\section{Main results}
Next, we state a version of our main result in a restricted setting. Recall that a weighted Riemannian manifold $(M,g,\mu)$ is a Riemannian manifold $(M,g)$ equipped
with a measure $\mu$ such that $\mu(dy)= \sigma(y) \nu(dy)$, where $\nu$ is the Riemannian measure and  $\sigma \in \mathcal{C}^\infty(M)$ is the  \emph{weight function}.
\begin{theorem} \label{t-main0}
 Let $(M,g,\mu)$ be a complete non-compact, weighted Riemannian manifold such that there exists $K \ge 0$ such that $\operatorname{Ric} \ge - Kg$.
  Furthermore there exists $C_1 \ge 1$ such that the weight function $\sigma$ satisfies
  $C_1^{-1} \le \sigma(x)/\sigma(y) \le C_1$ for all $x,y \in M$ with $d(x,y) \le 1$.
  Consider a Markov chain on $M$ with a symmetric kernel $p_k$ with respect to  $\mu$.
Further we assume that there exists $C_0>1,h>0,h'\ge h$ such that
\begin{equation} \label{e-repi}
C_0^{-1}\frac{\one_{B(x,h)}(y)}{V(x,h)} \le p_1(x,y) \le C_0 \frac{\one_{B(x,h')}(y)}{V(x,h')}
\end{equation}
for all $x \in M$ and for $\mu$-almost every $y \in M$. Then the following properties are equivalent:
\begin{enumerate}[(a)]
 \item The parabolic Harnack inequality: there exists $\eta \in (0,1)$,$C_H>1,R_H>0$ such that for all balls $B(x,r)$, $x \in M$, $r > R_H$ and
 for all non-negative functions $u: \N \times M \to \R$ that satisfies
 $\partial_k u + \Delta u_k \equiv 0$ in $\nint{0}{\lfloor 4\eta^2 r^2 \rfloor} \times B(x,r)$, we have
 \[
  \sup_{Q_\ominus} u \le C_H \inf_{Q_\oplus} u
 \]
where
\begin{align*}
 Q_\ominus &= \nint{\lceil (\eta^2/2) r^2 \rceil}{\lfloor \eta^2 r^2 \rfloor} \times B(x,(\eta/2) r), \\
 Q_\oplus &= \nint{ \lceil2 \eta^2 r^2 \rceil}{\lfloor 4 \eta^2 r^2 \rfloor} \times B(x,(\eta/2) r)
 \end{align*}
 \item Two sided Gaussian bounds on the heat kernel: there exists $C_1,C_2>0$ such that
for all $x,y \in M$ and for all $n \in \mathbb{N}^*$ satisfying $n \ge 2$, we have
\begin{equation}\label{e-guei2} p_n(x,y) \le \frac{C_1}{V(x,\sqrt{n})} \exp \left( - \frac{d(x,y)^2}{C_2n} \right)
\end{equation}
Further there exists $c_1,c_2,c_3>0$ such that for all $x,y \in M$ satisfying $d(x,y)\le c_3 n$ and for all $n \in \mathbb{N}^*$ satisfying $n \ge 2$
\begin{equation}\label{e-glei2} p_n(x,y) \ge \frac{c_1}{V(x,\sqrt{n})} \exp \left( -\frac{ d(x,y)^2}{c_2n }\right)
\end{equation}
\item The conjunction of
\begin{itemize}
 \item The volume doubling property: there exists $C_D>0$ such that for all $x \in M$, for all $r>0$ we have
 \[
 V(x,2r) \le C_D V(x,r)
 \]
\item The Poincar\'{e} inequality: there exists $C_P>0$,$\kappa \ge 1$ such that for any ball $B=B(x,r)$, $x \in M$, $r >1$ and for all $f \in L^2(M,\mu)$, we have
\begin{equation} \label{e-npoin}
 \int_B \abs{f-f_B}^2 \,d\mu \le C_P r^2 \int_{\kappa B} \left( \frac{1}{V(y,1)} \int_{B(y,1)} \abs{f(z) -f(y)}^2 \, \mu(dz)\right) \, \mu(dy),
\end{equation}
where $\kappa B= B(x,\kappa r)$, $f_B= \frac{1}{\mu(B)} \int_B f \, d\mu$.
\end{itemize}
\end{enumerate}
\end{theorem}
The Poincar\'{e} inequalities presented in Theorem \ref{t-grisal} and Theorem \ref{t-main0} are related.
We will show that the Poincar\'{e} inequality \eqref{e-opoin} implies \eqref{e-npoin} (See Proposition \ref{p-rmtfae}). A partial converse of the previous statement hold as well.
\begin{example}
 Consider a  complete, non-compact Riemannian manifold $(M,g)$ with non-negative Ricci curvature whose unit balls have a uniform positive volume lower bound.
 Define a Markov kernel $p(x,y) = \frac{\one_{B(x,1)} (y)}{V(x,1)}$ for all $x,y \in M$. Although $p$ is a Markov kernel with respect to the Riemannian measure $\nu$, $p(x,y) \neq p(y,x)$ in general.
 However $q_1(x,y)=\frac{p(x,y)}{V(y,1)}$ is a symmetric Markov kernel with respect to $\mu(dx) = V(x,1) \nu(dx)$ where $V$ denotes the volume with respect to $\nu$.
 By the remark preceding this example, $(M,g,\mu)$ satisfy the Poincar\'{e} inequalities \eqref{e-npoin} and \eqref{e-opoin}.
 Moreover $(M,g,\mu)$ satisfies volume doubling property. Hence the iterated kernel $q_n$ satisfies two-sided Gaussian bounds and the corresponding Laplacian satisfies the parabolic Harnack inequality.
  Similarly many other examples known in the diffusion case can be extended to the discrete-Markov chain case due to Proposition \ref{p-rmtfae}.
\end{example}
The role of Theorem \ref{t-main0} is only to illustrate  our main result  without introducing additional definitions.
We provide an unified approach to study random walks on both discrete and continuous spaces.
 We prove Theorem \ref{t-main0} as a corollary of a general result  that also gives an alternate proof of Theorem \ref{t-Del}.

Given the previous works on characterization of parabolic Harnack inequality and Gaussian bounds \cite{Gri91, Sal92,  Stu96, Del99, HS01} our results should not be surprising.
However we encounter new difficulties  that had to be resolved here and which were not present in previous works.
Recall that Moser's iteration method for  Harnack inequalities relies on repeated application of a Sobolev inequality\cite{Sal92,  Stu96, Del99}.
Grigor'yan's iteration method in \cite{Gri91} uses an equivalent Faber-Krahn inequality that is equivalent to the Sobolev inequality \cite{BCLS95}.

The Sobolev inequalities in the previous settings are of the form
\begin{equation} \label{e-sobi}
 \norm{f}_2^{2 \delta/ (\delta-2)} \le \frac{C r^2}{V_\mu(x,r)^{2/\delta}} \left( \mathcal{E}(f,f) + r^{-2} \norm{f}_2^2 \right)
\end{equation}
for all `nice' functions $f$ supported in $B(x,r)$. However for discrete time Markov chains, the Dirichlet form satisfies the inequality
 $\E(f,f) = \langle (I-P)f,f \rangle \le 2 \norm{f}_2^2$.
 This along with \eqref{e-sobi} implies that $L^2(B(x,r)) \subseteq L^{2 \delta/ (\delta-2)} (B(x,r))$ for all balls $B(x,r)$ which can happen only
 if the space is discrete. Hence for discrete time Markov chains on Riemannian manifolds the Sobolev inequality \eqref{e-sobi} cannot possibly be true.
 We prove and rely on a weaker form of the Sobolev inequality \eqref{e-sobi} which seems to be too weak to run Moser's iteration directly to prove parabolic Harnack inequality (See Theorem \ref{t-Sob}).
 Instead we use Moser's iteration to prove a version of the mean value inequality which in turn gives Gaussian upper bounds.
We adapt a method of \cite{HS01} which uses elliptic Harnack inequality and Gaussian upper bounds to prove Gaussian lower bounds (See Chapter \ref{ch-glb}).
Another difficulty that is new to our setting is explained in the beginning of Section \ref{s-dimp}.

In the context of diffusions on complete Riemannian manifolds the  Sobolev inequality \eqref{e-sobi} is equivalent to the conjunction of volume doubling property
and Gaussian upper bounds on the heat kernel \cite[Theorem 5.5.6]{Sal02}. In the previous statement, we may replace Sobolev inequalities with a similar but equivalent set of functional inequalities
called Faber-Krahn inequalities both in the context of diffusions on Riemannian manifolds \cite{Gri94} and for random walks on graphs \cite[Theorem 1.1]{CG98}.
We extend the above equivalences for random walks on a large class of metric measure spaces (Theorem \ref{t-main2}).

\section{Guide for the monograph}
This monograph is organized as follows.
In Chapter \ref{ch-mg}, we present the setting of quasi-geodesic spaces satisfying certain doubling hypotheses, study its basic properties and develop techniques that would let us compare discrete and continuous spaces.

In Chapter \ref{ch-pi}, we introduce Poincar\'{e} inequalities and discuss various examples and non-examples of spaces satisfying Poincar\'{e} inequality.
We study how these new Poincar\'{e} inequalities on metric measure spaces compare with the previously studied Poincar\'{e} inequalities on graphs and Riemannian manifolds.
Then we show that Poincar\'{e} inequality is stable under quasi-isometric transformation of quasi-geodesic spaces.

In Chapter \ref{ch-markov}, we introduce various hypotheses on the Markov chain, Dirichlet forms and study their basic properties.
In Chapter \ref{ch-sob}, we introduce and prove a Sobolev inequality under the assumptions of large scale volume doubling and Poincar\'{e} inequality.
In Chapter \ref{ch-ehi}, we use Sobolev inequality and Poincar\'{e} inequality to run the Moser iteration argument to prove elliptic Harnack inequality.

Chapter \ref{ch-gub} is devoted to the  proof of Gaussian upper bounds using Sobolev inequality.
In addition, we show that Sobolev inequality is equivalent to the conjunction of Gaussian upper bounds on the heat kernel  and large scale volume doubling property.
In Chapter \ref{ch-glb} we prove  Gaussian lower bounds using elliptic Harnack inequality and Gaussian upper bounds. This completes the proof that large scale volume doubling property
and Poincar\'{e} inequality implies two sided Gaussian bound on the heat kernel.

In Chapter \ref{ch-phi}, we prove parabolic Harnack inequality using Gaussian bounds.
Moreover, we prove large scale volume doubling property and Poincar\'{e} inequality using parabolic Harnack inequality, and thereby completing the proof of the characterization parabolic Harnack inequality and Gaussian bounds.
In Chapter \ref{ch-apply}, we mention various applications of Gaussian estimates and Harnack inequalities.
In Appendix  \ref{a-example}, we collect various examples and supplement them with figures and discussions.

\chapter{Metric Geometry} \label{ch-mg}

Let $(M,d,\mu)$ be a locally compact metric measure space where $\mu$ is a Radon measure with full support.
Let $\mathcal{B}(M)$ denote the Borel $\sigma$-algebra on $(M,d)$.
Let $B(x,r):= \{ y \in M : d(x,y) \le r \}$ denote the closed ball in $M$ for metric $d$ with center $x$ and radius $r>0$.
Let $V(x,r):= \mu(B(x,r))$ denote the volume of the closed ball centered at $x$ of radius $r$. Since $M$ is a Radon measure with full support, we have that
$V(x,r)$ is finite and positive for all $x \in M$ and for all $r >0$.
In this section, we introduce some assumptions on the metric $d$ and measure $\mu$ and study some consequences of those assumptions.

\section{Quasi-geodesic spaces}
The main assumption  on the metric $d$ of the metric measure space $(M,d,\mu)$ is that of quasi-geodesicity.
In Riemannian geometry, the distance between two points of a manifold is defined as the infimum of lengths of curves joining them.
Such a relation between distance and length of curves is observed more generally in length spaces.
\begin{definition}
 Let $(M,d)$ be a metric space. For $x,y \in M$, a
\emph{path} from $x$ to $y$ is a continuous map $\gamma:[0,1] \to M$ such that $\gamma(0)=x$ and $\gamma(1)=y$.
We define the \emph{length} $L(\gamma) \in [0,\infty]$ of a path $\gamma$ is the supremum
\[
 L(\gamma)= \sup_{P[0,1]} \sum_{i} d( \gamma(t_{i-1}),\gamma(t_i)).
\]
taken over all partition $0=t_0 < t_1 < \ldots < t_k = 1$ of $[0,1]$.
\end{definition}
The length of a path is a non-negative real number or $+\infty$.
\begin{definition}
 The \emph{inner metric}
or \emph{length metric} associated with metric space $(M,d)$ is the function $d_i:M \times M \to [0,\infty]$ defined by
\[
 d_i(x,y) = \inf_\gamma L(\gamma)
\]
where the infimum is taken over all paths $\gamma$ from $x$ to $y$. $(M,d)$ is called a \emph{length space}
if $d_i=d$. A metric for which $d=d_i$ is called an intrinsic metric.
\end{definition}
\begin{remark} \label{r-nograph}All Riemannian manifolds equipped with Riemannian distance are length spaces.
Since infimum of an empty set is $+\infty$, for points $x,y$ in different connected components of a metric space $(M,d)$, we have $d_i(x,y) = +\infty$.
Hence graphs with natural combinatorial metric are not length spaces because distinct vertices belong to different connected components under the metric topology.
See  \cite[Chapter 1]{Gro07} or \cite[Chapter 2]{BBI01}  for a comprehensive introduction of length spaces.

\end{remark}

One of the goals of this work is to provide an unified approach to the study of random walks in continuous spaces like Riemannian manifolds and discrete spaces like graphs.
In view of  Remark \ref{r-nograph}, we would like to consider spaces more general than length spaces to include disconnected metric spaces like graphs.
Quasi-geodesic spaces provides a natural setting to include both length spaces and graphs as special cases. Quasi-geodesic spaces are equipped with a weak notion of geodesics called chains.
We recall the following definition of chain and various notions of geodesicity as presented by Tessera in \cite{Tes08}.
\begin{definition}
Consider a metric space $(M,d)$ and $b>0$. For $x,y \in M$, a \emph{$b$-chain} between from $x$ to $y$,  is a sequence
$\gamma:x=x_0,x_2,\ldots,x_m=y$ in $M$ such that for every $0 \le i < m$, $d(x_i,x_{i+1}) \le b$.
We define the \emph{length} $l(\gamma)$ of a $b$-chain $\gamma: x_0,x_1,\ldots,x_m$ by setting
\[
 l(\gamma)=\sum_{i=0}^{m-1} d(x_i,x_{i+1}).
\]
Define a new distance function $d_b:M \times M \to [0,\infty]$ as
\begin{equation}\label{e-db}
 d_b(x,y) = \inf_{\gamma} l(\gamma)
\end{equation}
where $\gamma$ runs over every $b$-chain from $x$ to $y$. We say a metric space $(M,d)$ is
\begin{itemize}
 \item $b$-geodesic  if $d(x,y)=d_b(x,y)$ for all $x,y \in M$.
 \item quasi-$b$-geodesic  if there exists $C>0$ such that $d_b(x,y) \le C d(x,y)$ for all $x,y \in M$.
 \item quasi-geodesic if there exists $b>0$ such that $(M,d)$ is quasi-$b$-geodesic.
\end{itemize}
\end{definition}
\begin{remark}  We collect some simple consequences of the definitions.
 \begin{itemize}
  \item Any $b$-geodesic space is quasi-$b$-geodesic. Moreover $b$-quasi-geodesic space is $b_1$-quasi-geodesic for all $b_1 \ge b$.
  \item Any length space is $b$-geodesic for all $b>0$.
  \item  Graphs with natural combinatorial metric are $b$-geodesic if and only if $b \ge 1$. If $b<1$, then $d_b(x,y)=+\infty$ for distinct vertices $x$ and $y$.
 \end{itemize}
\end{remark}
The following lemma guarantees that quasi-geodesic spaces are endowed with sufficiently short chains at many length scales.
\begin{lemma}[Chain lemma] \label{l-chain}
Let $(M,d)$ be a quasi-$b$-geodesic space for some $b>0$. Then there exists $C_1\ge1$ such that for all $b_1 \ge b$ and for all $x,y \in M$,
there exists a $b_1$-chain $x=x_0,x_1,\ldots,x_m=y$  with $m \le \left\lceil\frac{C_1 d(x,y)}{b_1} \right\rceil$.
\end{lemma}
\begin{proof}
 Since $(M,d)$ is quasi-$b$-geodesic,
 there exists $C>0$ such that for all $x,y \in M$, there exists a $b$-chain $x=y_0,y_1,\ldots,y_n=y$ satisfying $\sum_{i=0}^{n-1} d(y_i,y_{i+1}) \le  C d(x,y)$.
 We define a smaller $b_1$-chain $x_0,x_1,\ldots,x_m$ where  $x_k=y_{i_k}$.
 We choose $i_0=0$ and define $i_k$ successively by
 \[
 i_k= \max \{ 1 \le j \le n : d(y_{i_{k-1}},y_{j}) \le b_1 \}
\]
for $k \ge 1$. Define $m=\min \{j : y_{i_j}=y \}$. By the definition of $i_k$ we have that
\begin{equation*}
 d(x_i,x_{i+1}) + d(x_{i+1},x_{i+2}) \ge d(x_i,x_{i+2}) > b_1
\end{equation*}
for all $i=0,1,\ldots,m-2$. Therefore we have
\begin{equation*}
\sum_{i=0}^{m-1} d(x_i,x_{i+1})  > \frac{b_1}{2} (m-1).
\end{equation*}
By triangle inequality, we have
$\sum_{i=0}^{m-1} d(x_i,x_{i+1}) \le \sum_{i=0}^{n-1} d(y_i,y_{i+1}) \le C d(x,y)$.
Therefore
\[
 m \le 1+ \frac{2 C}{b_1} d(x,y).
\]
Hence the choice $C_1= 2C+1$ satisfies the desired conclusion.
\end{proof}
\section{Doubling hypothesis}
The main assumption that we recall below on the Radon measure $\mu$ is the doubling property.
For a metric measure space $(M,d,\mu)$, we denote volume of balls by $V(x,r)=\mu(B(x,r))$.
\begin{definition} We define the following doubling hypothesis:
\begin{itemize}
\item[$(VD)_{\operatorname{loc}}$] We say a space $(M,d,\mu)$ satisfies the local volume doubling property $(VD)_{\operatorname{loc}}$, if for all $r >0$, there exists $C_r$ such that
\begin{equation*} \label{doub-loc}
V(x,2 r)  \le C_{r} V(x,r)\tag*{$(VD)_{\operatorname{loc}}$}
\end{equation*}
for all $x \in M$.
\item[$(VD)_{\infty}$] We say a space $(M,d,\mu)$ satisfies the large scale doubling property $(VD)_{\infty}$, if  there exists positive reals $C_{r_0},r_0$ such that
\begin{equation*} \label{doub-inf}
V(x,2 r)  \le C_{r_0} V(x,r)\tag*{$(VD)_{\infty}$}
\end{equation*}
for all $x \in M$ and $r \ge r_0$.
\item[$(VD)$] We say a space $(M,d,\mu)$ satisfies the global volume doubling property $(VD)$, if  there exists a constant $C_{D}>0$ such that
\begin{equation*} \label{doub-glob}
V(x,2 r)  \le C_{D} V(x,r)\tag*{$(VD)$}
\end{equation*}
for all $x \in M$ and $r >0$.
\end{itemize}
\end{definition}
\begin{remark}
 The property \ref{doub-glob} implies \ref{doub-inf} and \ref{doub-loc}. The property \ref{doub-loc} is a condition local in $r$ but uniform in $x \in M$
 while \ref{doub-inf} and \ref{doub-glob} are uniform in both $x$ and $r$.
The property \ref{doub-loc} is a very weak property of bounded geometry introduced in \cite{CS95}.
Since $C_r$ depends on $r$, the local volume doubling property does impose too much constraint on volume growth as $r \to \infty$.
However, we will see in Lemma \ref{l-doub-prop} that large scale doubling can be used to control volume of large balls.
\end{remark}
\begin{example}
We describe some examples satisfying the above hypothesis on volume growth. Every connected graph with bounded degree and equipped with the counting measure satisfies \ref{doub-loc}.
By Bishop-Gromov inequality \cite[Theorem III.4.5.]{Cha06},  Riemannian manifolds with Ricci curvature bounded from below satisfy \ref{doub-loc} and  Riemannian manifolds with non-negative Ricci curvature satisfy \ref{doub-glob}.
\end{example}
We collect some basic properties of spaces
satisfying the above doubling hypothesis \ref{doub-loc} and \ref{doub-inf}.
\begin{lemma}(\cite[Lemma 2.1]{CS95}) \label{l-vloc}
 If $(M,d,\mu)$ satisfies \ref{doub-loc}, then for all $r_1,r_2 >0$, there exists $C_{r_1,r_2}$ such that
 \begin{equation}\label{e-volc} \leavevmode
  V(x,r_2) \le C_{r_1,r_2} V(x,r_1)
 \end{equation}
for all $x \in M$. In particular, for all $x,y \in M$, such that $d(x,y) \le R$, we have
\[
 V(x,r) \le C_{r,R+r} V(y,r)
\]
\end{lemma}

\begin{proof}
Let $k$ be the smallest integer such that $2^k r_1 \ge r_2$. By repeated application of \ref{doub-loc}, the choice
\[
 C_{r_1,r_2}= \prod_{i=0}^{k-1} C_{2^i r_1}
\]
satisfies, \eqref{e-volc} where the constant $C_{2^i r_1}$ is  from \ref{doub-loc}. The second part follows from $B(x,r) \subseteq B(y,R+r)$ and \eqref{e-volc}.
\end{proof}
The large scale doubling property \ref{doub-inf} along with \ref{doub-loc} implies a polynomial volume growth upper bound.
\begin{lemma} \label{l-doub-prop}
 Let $(M,d,\mu)$ be a metric measure space satisfying \ref{doub-loc} and \ref{doub-inf}. Then for all $b >0$, there exists $C_b > 0$ such that
\begin{equation} \label{e-vd}
V(x,2 r)  \le C_{b} V(x,r)
\end{equation}
for all $x \in M$ and $r \ge b$. Moreover this $C_b$ satisfies
\begin{equation}
 \label{e-vd1} \frac{V(x,r)}{V(x,s)} \le C_b \left( \frac{r}{s} \right)^\delta
\end{equation}
for all  $x \in M$, for all $b \le s < r$ and for all $\delta \ge \log_2 C_b$.
Furthermore
\begin{equation}
 \label{e-vd2}  \frac{V(x,r)}{V(y,s)} \le C_b^2 \left( \frac{r}{s} \right)^\delta
\end{equation}
holds for all $b \le s \le r$, for all $x \in M$, for all $y \in B(x,r)$ and for all $\delta \ge \log_2 C_b$ .
\end{lemma}
\begin{proof}
 Let $r_0,C_{r_0}$ be constants from \ref{doub-inf}. There is nothing to prove if $r_0 \le b$. Assume $r_0 > b$ and  let $r$ be such that $b \le r < r_0$.
Then by Lemma \ref{l-vloc} and \ref{doub-inf}
\[
 V(x,2r) \le V(x,2r_0) \le C_{r_0} V(x,r_0) \le C_{r_0} C_{b,r_0} V(x,b) \le C_{r_0 } C_{b,r_0} V(x,r).
\]
 The case $r \ge r_0$ follows from \ref{doub-inf} which concludes the proof of \eqref{e-vd}. \\
 Let $b \le s < r$ , $k= \log_2 (r/s)$ and $\delta \ge \log_2 C_D$.
 Then from \eqref{e-vd}, we get \eqref{e-vd1},
 \[
  \frac{V(x,r)}{V(x,s)} \le   \frac{V(x,2^{\lceil k \rceil} s)}{V(x,s)} \le C_b^{k+1} \le C_b \left( \frac{r}{s} \right)^\delta.
 \]
To obtain \eqref{e-vd2} from \eqref{e-vd1}, note that
\[
 \frac{V(x,r)}{V(y,s)} \le \frac{V(y,2r)}{V(y,s)} \le C_b  \frac{V(y,r)}{V(y,s)} \le  C_b^2 \left( \frac{r}{s} \right)^\delta.
\]
\end{proof}
The equation \eqref{e-vd1} implies a polynomial upper bound on the volume growth.
In quasi-geodesic spaces, we can reverse the inequality \eqref{e-vd1} and obtain a polynomial lower bound for all radii small enough compared to the diameter.
The property stated in following lemma is often called the reverse volume doubling property.
It was known for graphs and Riemannian manifolds and our proof follows similar ideas.
\begin{lemma}\label{l-rvd}
 Let $(M,d,\mu)$ be a quasi-$b$-geodesic space with the measure $\mu$ satisfying \ref{doub-loc} and \ref{doub-inf}. Then there exists $c ,\gamma>0$ such that
 \begin{equation} \label{e-rvd}
  \frac{V(x,r)}{V(x,s)} \ge c \left( \frac{r}{s} \right)^\gamma
 \end{equation}
for all $x \in M$ and for all $b \le s \le r \le \operatorname{diam}(M)$, where $\operatorname{diam}(M)= \sup\{ d(x,y): x,y \in M \}$ denotes the diameter of $(M,d,\mu)$.
\end{lemma}

\begin{proof}
We first consider the case $b \le s \le r \le \frac{\operatorname{diam}(M)}{5}$.
 Let $x \in M$ and let $z \in M$ be chosen such that $d(x,z) \ge (3/7) \operatorname{diam}(M)$. Let $x=x_0,x_1,\ldots,x_m=z$ be a $s$-chain with minimal number of points $m$.
 Therefore there exists $3 \le k \le m$ such that
 $2s <d(x_k,x)\le 3s$. \\
 Since $d(x_k,s)>2s$, we have $B(x_k , s) \cap B(x,s) = \emptyset$. By Lemma \ref{l-doub-prop}, there exists $\epsilon>0$ such that
 \[
  V(x,3s) \le V(x_k,6 s) \le \epsilon^{-1} V(x_k,s)
 \]
Therefore we obtain
\begin{equation} \label{e-rvd1}
  V(x,4s) \ge V(x,s) +V(x_k,s) \ge (1+\epsilon)V(x,s)
\end{equation}
for all $x \in M$ and $b \le s \le \operatorname{diam}(M)/5$. Define $k= \log_4(r/s)$ and $\gamma= \log_4(1+\epsilon)$. Then by \eqref{e-rvd1}
\[
 \frac{V(x,r) }{V(x,s)} \ge \frac{V(x, 4^{\lfloor k \rfloor} s )}{V(x,s)} \ge (1+ \epsilon)^{k-1}= (1+\epsilon)^{-1}  \left( \frac{r}{s}\right)^\gamma
\]
for all $x \in M$ and $b \le s < r \le \operatorname{diam}(M)/5$.

The other cases follow from Lemma \ref{l-vloc} and by choosing \[c= \min( (1+\epsilon)^{-1} 5^{-\gamma}, C_{\operatorname{diam}(M)/5,\operatorname{diam}(M)}.\]
\end{proof}

\section{Quasi-isometry}
One of the goals of this work is to develop arguments which are robust to small perturbations in the geometry of the underlying space;
for example addition of few edges in a graph or small changes in the metric of a Riemannian manifold. We study properties that depends mainly
on the large scale geometry of the underlying space.
In this spirit, the concept of quasi-isometry was introduced by Kanai in \cite{Kan85} and in the more restricted setting of groups by Gromov in \cite{Gro81}.
Informally, two metric spaces are quasi-isometric if they have the same large scale geometry. Here is a precise definition:
 \begin{definition} \label{d-qi-ms}
  A map $\phi:(M_1,d_1) \to (M_2,d_2)$, between metric spaces is called a quasi-isometry if the following conditions are satisfied:
  \begin{itemize}
   \item[(i)] There exists $a \ge 1$ and $b \ge 0$ such that
   \[
    a^{-1} d_1(x_1,x_2) - b \le d_2( \phi(x_1), \phi(x_2) ) \le a d_1(x_1,x_2) + b
   \]
   for all $x_1,x_2 \in M_1$.
   \item[(ii)] There exists  $\epsilon>0$, such that for all $y \in M_2$ there exists $x \in M_1$ such that $d_2(\phi(x),y) < \epsilon$.
    \end{itemize}
We say metric spaces $(M_1,d_1)$ and $(M_2,d_2)$ are quasi-isometric if there exists a quasi-isometry $\phi:(M_1,d_1) \to (M_2,d_2)$.
 \end{definition}
 \begin{remark} Quasi-isometry is an equivalence relation among metric spaces.
  Quasi-isometry is also called as \emph{rough-isometry} or \emph{coarse quasi-isometry}. Property (i) of Definition \ref{d-qi-ms} above is called \emph{roughly bi-Lipschitz} and (ii)
  is called \emph{roughly surjective}.
 \end{remark}
 We remark that a quasi-isometry is not necessarily a continuous map. Moreover,  quasi-isometry is not necessarily injective and not necessarily surjective.
However, we can construct a  \emph{quasi-inverse} $\phi^{-}: (M_2,d_2) \to (M_1,d_1)$ as $\phi^{-}(y) =x$
where $x \in M_1$ is chosen so that $d_2( \phi(x),y) < \epsilon$ where $\epsilon$ is given by the above definition.

 We now describe some well-known examples of quasi-isometry. The space $\mathbb{R}^d$ with Euclidean metric and $\mathbb{Z}^d$ with standard graph metric (same as $L^1$ metric)
 are quasi-isometric. Consider a finitely generated group $\Gamma$ with a finite system of generator $A$.
 For an element $\gamma \neq 1$, let $\abs{\gamma}_A$ denote the smallest positive integer
 $k$ such that a product of $k$ elements of $A \cup A^{-1}$, and put $\abs{1}_A = 0$. This $\abs{\cdot}_A$ is called the \emph{word norm} of $\Gamma$ and defines a \emph{word metric}
 $d_A(\gamma_1,\gamma_2) = \abs{\gamma_1^{-1} \gamma_2}_A$. In other words, $d_A$ is the graph metric in the Cayley graph of $\Gamma$ corresponding to the symmetric generating set
 $A \cup A^{-1}$. Assume two finite generating sets $A$ and $B$ of a group $\Gamma$ which induces metric $d_A$ and $d_B$ respectively. Then $(\Gamma,d_A)$ and $(\Gamma,d_B)$ are quasi-isometric
 (See \cite[Proposition 1.15]{Roe03}).
 Therefore every finitely generated group defines a unique word metric space up to quasi-isometry and
 we may often view a finitely generated group up as a metric space without explicitly specifying the generating set.
  A large class of examples of quasi-isometry is given by the \v{S}varc-Milnor theorem.
 We refer the reader to \cite[Theorem 1.18]{Roe03} for a proof and original references.
 \begin{theorem} (\v{S}varc-Milnor theorem) \label{t-sm}
  Suppose that $(M,d)$ is a length space and $\Gamma$ is a finitely generated group equipped with a word metric acting properly and cocompactly by isometries on $M$.
  Then $\Gamma$ is quasi-isometric to $(M,d)$. The map $\gamma \mapsto \gamma.x_0$ is a quasi-isometry for each fixed base point $x_0 \in M$.
 \end{theorem}

 Note that the quasi-isometry between $\mathbb{Z}^d$ and $\mathbb{R}^d$ is a special case of Theorem \ref{t-sm}.
 We will give a general construction of \emph{net} which approximates a quasi-geodesic space using a graph with combinatorial metric in next subsection.

The notion of quasi-isometry was extended to metric measure spaces by Couhlon and Saloff-Coste in \cite{CS95} which they called ``isometry at infinity''. Let
 $(M_i,d_i,\mu_i)$, $i=1,2$ be two metric measure spaces. Define
 \[
  V_i(y,r) = \mu_i \left( \{ z \in M_i: d_i(y,z) \le r \} \right).
 \]
 \begin{definition} \label{d-qi-mms}
  A map $\phi:(M_1,d_1,\mu_1) \to (M_2,d_2,\mu_2)$, between metric measure  spaces is called a \emph{quasi-isometry} if the following conditions are satisfied:
  \begin{itemize}
   \item[(i)] $\phi:(M_1,d_1) \to (M_2,d_2)$ is a quasi-isometry between metric spaces;
   \item[(ii)] There exists a constant $C>0$  such that
   \[
    C^{-1} V_1(x,1) \le V_2(\phi(x),1) \le C V_1(x,1)
   \]
   for all $x \in M_1$.
  \end{itemize}
  We say metric measure spaces $(M_1,d_1,\mu_1)$ and $(M_2,d_2,\mu_2)$ are quasi-isometric if there exists a quasi-isometry $\phi:(M_1,d_1,\mu_1) \to (M_2,d_2,\mu_2)$.

 \end{definition}
 \begin{remark}
  Quasi-isometry is an equivalence relation for metric measure spaces satisfying  local volume doubling property \ref{doub-loc}.
  The notion of large scale equivalence as defined in Definition 5.5 of
  \cite{Tes08} is more general.
  That is every quasi-isometry is a large scale equivalence.
  However a map between quasi-geodesic metric measure spaces satisfying \ref{doub-loc} is a quasi-isometry if and
  only if it is large scale equivalence. See \cite[Remark 5.7]{Tes08}.
 \end{remark}

The arguments in this work implies that the long term behavior of natural random walks depends mainly on the large scale geometry of the quasi-geodesic space.
Other important examples of properties invariant under quasi-isometries are large scale doubling and Poincar\'{e} inequality.
(See Proposition \ref{p-qivd} and Proposition \ref{p-robustmms}). We conclude this subsection by proving that
the large scale doubling property is preserved by quasi-isometries for metric measure spaces satisfying \ref{doub-loc}.
It is due to Couhlon and
Saloff-Coste in \cite{CS95}. We need the following definition:

\begin{definition}\label{d-maxsep}
 Let $(M,d)$ be a metric space with $X \subseteq M$ and let $R>0$. Then a subset $Y$ of $X$ is \emph{$R$-separated} if $d(y_1,y_2) > R$ whenever $y_1$ and $y_2$ are distinct points of $Y$,
and a $R$-separated subset  $Y$ of $X$ is called maximal if it is maximal among all $R$-separated subsets of $X$ with respect to the partial order of inclusion.
\end{definition}
The existence of maximal $R$-separated subsets follows from Zorn's lemma.

The following lemma compares volume of balls between quasi-isometric metric measure spaces.
\begin{lemma} \label{l-qi-vc} (\cite[Proposition 2.2]{CS95})
 Let $\Phi:(M_1,d_1,\mu_1)$ and $(M_2,d_2,\mu_2)$ be a  quasi-isometry between metric measure spaces satisfying \ref{doub-loc}.
Then for all $h >0$,  there exists $C_h>0$ such that
\[
 C_h^{-1} V_1(x,C_h^{-1}r) \le V_2(\Phi(x),r) \le C_h V_1(x,C_h r)
\]
for all $x \in M_1$ and for all $r \ge h$.
 \end{lemma}
\begin{proof} We denote balls and volumes by $B_i$ and $V_i$ respectively for $i=1,2$.
 Let $R\ge h$ such that $aR-b = R' >0$ where $a,b$ is from Definition \ref{d-qi-ms}. Let $Y$ be a
 maximal $R$-separated subset of $B(x,r)$.
 Thus
 $B(x,r) \subseteq \cup_{y \in Y} B(y,R)$. Hence
 \begin{equation} \label{qiv0}
   V_1(x,r) \le \sum_{y \in Y} V_1(y,R)
 \end{equation}
By Lemma \ref{l-vloc} and Definition \ref{d-qi-mms}, we have
\begin{equation} \label{qiv1}
  V_1(y,R) \le C_{1,R} V_1(y,1) \le C_{1,R} C V_2(\Phi(y),1).
\end{equation}
for all $y \in Y$.
The balls $\{B(y,R/2)\}_{y \in Y}$ are pairwise disjoint and hence the balls
$B(\Phi(x_i),R'/2)$ are pairwise disjoint. By Lemma \ref{l-vloc}
\begin{equation} \label{qiv2}
 V_2(\Phi(x_i),h) \le C_{h,R'} V_2(\Phi(x_i),R'/2)
\end{equation}
Combining \ref{qiv0},\ref{qiv1} and \ref{qiv2}
\begin{align}
 \nonumber V_1(x,r) &\le \sum_{y \in Y} C_{1,R} C C_{1,R'} V_2 (\Phi(y), R'/2) \\
 & \le   C_{1,R} C C_{1,R'} V_2 (\Phi(x), ar + b +R'/2) \label{qiv3}
\end{align}
The last step follows from the definition of quasi-isometry, triangle inequality and that $B(\Phi(x_i),R'/2)$ are pairwise disjoint.
We can choose $C_2$ large enough so that, $ ar+b +R'/2 \le C_2r$ for all $r \in [h,\infty)$.
Hence by Lemma \ref{l-doub-prop}, we have the desired lower bound on $V_2$ for all $r \ge R$ and by Lemma \ref{l-vloc} for all $r \ge h$.
Similar argument applied to quasi-inverse $\Phi^{-1}$ yields
\[
 V_2(\Phi(x),r) \le C V_1 (\Phi^{-1} \circ \Phi (x) , Cr).
\]
The conclusion follows from the fact that $d_1(\Phi^{-1} \circ \Phi (x) ,x)$ is bounded uniformly for all $x \in M_1$.
\end{proof}
For metric measure spaces satisfying \ref{doub-loc}, the condition \ref{doub-inf} is preserved by quasi-isometries.
This is the content of the following lemma.
\begin{prop} \label{p-qivd} (\cite[Proposition 2.3]{CS95})
Let $(M_1,d_1, \mu_1)$ and $(M_2,d_2,\mu_2)$ be quasi-isometric spaces satisfying \ref{doub-loc}. Then
$(M_1,d_1,\mu_1)$ satisfies \ref{doub-inf} if and only if  $(M_2,d_2,\mu_2)$ satisfies \ref{doub-inf}.
\end{prop}
\begin{proof}
Let $\Phi:M_2 \to M_1$ be a quasi-isometry.
Using Lemma \ref{l-qi-vc}, there exists $C>0$ such that
 \[
  C^{-1} V_2(x,C^{-1}r) \le V_1(\Phi(x),r) \le C V_2(x,Cr)
\]
for all $x \in M_2$ and $ r \ge 1$. Hence by \eqref{e-vd1}, we have
\[
\frac{ V_2(x,2r)}{V_2(x,r)} \le C^2 \frac{ V_1( \Phi(x), 2C r) }{V_1(\Phi(x),C^{-1} r)} \le C^2 C_D (2 C^2)^\delta
\]
for all $r \ge \max(C,1)$.
\end{proof}
\section{Approximating quasi-geodesic spaces by graphs}
One might think of $\mathbb{Z}^d$ as a graph approximation or discretization of $\mathbb{R}^d$.
More generally,  we  can approximate quasi-geodesic spaces by graphs.
By  \cite[Proposition 6.2]{Tes08}, a metric space is quasi-isometric to a graph if and only if it is quasi-geodesic.
Therefore quasi-geodesic spaces form a natural class of metric spaces that can be roughly approximated by graphs.  \\
We begin by recalling some standard definitions and notation from graph theory.
We restrict ourselves to simple graphs.
A \emph{graph} $G$ is a pair $G=(V,E)$ where $V$ is a set (finite or infinite) called the \emph{vertices} of $G$ and $E$ is a subset of
$\mathcal{P}_2(V)$ (\emph{i.e.},two-element subsets of $V$) called the \emph{edges} of $G$.
A graph $(V,E)$ is \emph{countable} (resp. \emph{infinite}) if $V$ is a countable (resp. infinite) set.
We say that $p$ is a neighbor of $q$ (or in short $p \sim q$),  if $\{x,y\} \in E$.
The \emph{degree} of $p$ is the number of neighbors of $p$, that is $\operatorname{deg}(p)= \abs{ \{ q \in V: p \sim q\}}$.
A graph $(V,E)$ is said to have \emph{bounded degree} if $\sup_{v \in V} \deg(v) < \infty$.

A finite sequence $(p_0,p_1,\ldots,p_l)$ of points in $V$ is called a \emph{path} from $p_0$ to $p_l$ of \emph{length} $l$, if each $p_k$ is a neighbor of $p_{k-1}$.
A graph $G=(V,E)$ is said to be \emph{connected} if for all $p,q \in V$, there exists a path from $p$ to $q$.
For points $p,q \in V$ of a graph $G=(V,E)$,
let $d_G(p,q)$ denote the minimum of the lengths of paths from $p$ to $q$ with $d_G(p,q)= + \infty$ if there exists no path from $p$ to $q$.
This makes $(V,d_G)$ an extended metric space. The graph $(V,E)$ is connected if and only if $(V,d_G)$ is a metric space.
The extended metric $d_G$ is called \emph{graph metric} or \emph{combinatorial metric}  of $G$.
Notice that we can recover a  graph $(V,E)$ from its (extended) metric space structure $(V,d_G)$ and vice-versa.

Using the above identification, we view a connected graph as a metric space. We would like to view a connected graph as a metric measure space.
This motivates the definition of weighted graph. A \emph{weight} $m: V \to (0,\infty)$ on a graph $(V,E)$ is a positive function on $V$.
With a slight abuse of notation, $m$ induces a measure on $V$ (also denoted by $m$) as
\[
 m(A)= \sum_{v \in A} m(v)
\]
for each $A \subseteq V$. A \emph{weighted graph} is a graph $(V,E)$ endowed with a weight $m$.
By the above, we will identify a weighted graph $G=(V,E)$ with weight $m$ as a (possibly extended) metric measure space $(V,d_G,m)$.

The definition of $\epsilon$-net is due to Kanai in the setting of Riemannian manifolds (See \cite{Kan85}) and was extended in \cite{CS95} for weighted Riemannian manifolds.
\begin{definition}\label{d-net}
 A $\epsilon$-net of a metric measure space $(M,d,\mu)$ is a weighted graph $G=(V,E)$ with weight $m$ described as follows: The vertices $V$ is a maximal $\epsilon$-separated subset of $M$.
 The edges $E$ are defined by $\{x,y\} \in E$  if and only if $0< d(x,y) \le 3 \epsilon$.
 The weight $m$ is defined as $m(x)= \mu\left(B(x,\epsilon)\right)$.
 Let $d_G$ denote the graph metric of $G$.
 We often alternatively view the $\epsilon$-net as (extended) metric measure space $(V,d_G,m)$ defined by
  the corresponding weighted graph.
\end{definition}
The above definition does not guarantee $\epsilon$-net to be a connected graph. However it is connected and countable in many situations as described in the lemma below.
We collect the basic properties of nets in Proposition \ref{p-net} which builds on the ideas of \cite{Kan85}, \cite{Kan86b} and \cite{CS95}.
\begin{prop} \label{p-net}
 Let $(M,d,\mu)$ be a quasi-$b$-geodesic metric measure space satisfying \ref{doub-loc} and let $\epsilon \ge b$.
 Let $G=(X,E)$ be an $\epsilon$-net of $(M,d,\mu)$ with weight $m$ and let $(X,d_G,m)$ denote the corresponding extended metric measure space.
 Then we have the following:
 \begin{itemize}
\item [(a)] The collection of balls $I=\{ B(x, \epsilon/2) : x \in X\}$ is pairwise disjoint
and the collection $J=\{ B(x, \epsilon) : x \in X \}$ covers $M$ where $B(.,.)$ denotes closed metric ball in $(M,d)$.

\item[(b)] Bounded degree property: The graph $(X,E)$ is of bounded degree, that is  $\sup_{p \in X} \operatorname{deg}(p) < \infty$.

\item[(c)]  $(X,d_G,m)$ satisfies \ref{doub-loc}.

\item[(d)]  There exists $A >0$  such that
\begin{equation} \label{e-kan-est}
 \frac{1}{3 \epsilon} d(x,y) \le d_G(x,y) \le A d(x,y) + A
\end{equation}
for all $x,y \in X$. Therefore $G$ is a connected graph and $(X,d_G,m)$ is a metric measure space.

\item[(e)] The metric measure spaces $(M,d,\mu)$ and $(X,d_G,m)$ are quasi-isometric.

\item[(f)] $X$ is a countable set. Moreover if $\operatorname{diameter}(M,d)=\infty$, then $X$ is an infinite set.

\item[(g)] If $(M,d,\mu)$ satisfies \ref{doub-inf}, then so does $(X,d_G,m)$.

\item[(h)] Finite overlap property: Define
\[
 N_p(\delta)= \abs{ \{ x \in X: d(x,p) \le \delta \}}.
\]
for each $\delta>0$ and  $p \in M$.
Then $\sup_{p \in M} N_p (\delta) < \infty$.

 \end{itemize}
 \end{prop}
 \begin{proof} We denote the volume of balls in $(M,d,\mu)$ and $(X,d_G,m)$ by $V_M$ and $V_G$ respectively.

(a) The collection $I$ is pairwise disjoint because $X$ is $\epsilon$-separated. The collection $J$ covers $M$ due to the maximality of $X$. \\
(b) Let $d(p)$ denote the degree of a vertex $p$. Since $I$ is a disjoint collection, using Lemma \ref{l-vloc}
\begin{align*}
 V_M(p,4\epsilon) &\ge \sum_{q \in V, q \sim p} V_M (q , \epsilon/2) \\
 & \ge C_{\epsilon/2, 7 \epsilon}^{-1} \sum_{q \in V, q \sim p} V_M(q , 7 \epsilon) \ge d(p) V_M (p,4 \epsilon) C_{\epsilon/2, 7 \epsilon}^{-1}.
\end{align*}
This yields $d(p) \le  C_{\epsilon/2, 7\epsilon}$ for all $p \in X$. \\
(c) Let $x,y \in X$ with $x \sim y$. By Lemma \ref{l-vloc}, we obtain
\[
 \frac{m(y)}{m(x)} \le \frac{V(x,4\epsilon)}{V(x,\epsilon)} \le C_{\epsilon,4\epsilon}.
\]
Hence we have the uniform estimate
\begin{equation} \label{e-Cm}
 C_m=\sup_{x,y \in X, x \sim y}  \frac{m(y)}{m(x)} < \infty.
\end{equation}

By the above inequality and (b), we have
\begin{equation} \label{e-VG}
 m(x) \le V_G(x,r) \le m(x) C_m^r \left(\sup_{x \in X} \operatorname{deg}(x) \right)^r
\end{equation}
for all $x \in X$ and $r >0$. This along with (b) yields \ref{doub-loc}. \\
(d) Let $x,y \in X$. 
By triangle inequality we
have $d(x,y) \le 3 \epsilon d_G (x,y)$. By Lemma \ref{l-chain}, there exists $C_1 \ge 1$ and  an $\epsilon$-chain $x=x_0,x_1,\ldots,x_k=y$ in $(M,d)$ such that
$k \le  \lceil (C_1 d(x,y))/\epsilon \rceil$.
Since $J$ covers $M$, for each $x_i \in M$ we can choose $y_i \in X$ such that $d(x_i,y_i) \le \epsilon$ for $i=0,\ldots,k$.
Note that $x_0=y_0$ and $x_k=y_k$. By triangle inequality $d(y_i,y_{i+1}) \le 3 \epsilon$ or equivalently $y_i \sim y_{i+1}$ or $y_i=y_{i+1}$ for all $i=0,1,\ldots,k-1$.
Therefore
\[
d_G(x,y)=d_G(y_0,y_k) \le k \le C_1\left( \frac{d(x,y)}{\epsilon} + 1\right).
\]
This implies \eqref{e-kan-est} which implies the remaining conclusions.

(e) It follows from (d) that the inclusion map $\Phi:(X,d_G) \to (M,d)$ is a quasi-isometry of metric spaces.
Substituting $m(x)=V_M(x,\epsilon)$ and $r=1$ in \eqref{e-VG} and  using (b), \eqref{e-Cm} and Lemma \ref{l-vloc}, there exists $C>0$ such that
\[
 C^{-1}V_G(x,1)\le V_M(x,1) \le  C V_G(x,1).
\]
This proves that $\Phi$ is a quasi-isometry between the metric measure spaces $(M,d,\mu)$ and $(X,d_G,m)$. \\
(f) It follows from (b) and (d) that $G$ is a connected graph with bounded degree. Hence $X$ is countable.
 By \eqref{e-kan-est}, we have
$\operatorname{diameter}(X,d_G) \ge \operatorname{diameter}(M,d)/3 \epsilon$. Therefore if $\operatorname{diameter}(M,d)= \infty$, we have that $G=(X,E)$ is a connected graph with infinite diameter. Hence $X$ is infinite.
\\
(g) It follows from (c),(e) and Proposition \ref{p-qivd}. \\
(h) The proof is similar to (b). Using (a) and Lemma \ref{l-vloc}, we have
\begin{align*}
 V(p,\delta + \epsilon) &\ge \sum_{x \in X: d(x,p) < \delta} V(x, \epsilon/2) \\
 &\ge C_{\epsilon/2, 2 \delta + \epsilon}^{-1} \sum_{x \in X: d(x,p) < \delta} V(x,2\delta +\epsilon) \\
 &\ge N_p(\delta)  C_{\epsilon/2, 2 \delta + \epsilon}^{-1} V(p,\delta+ \epsilon).
\end{align*}
This yields the uniform bound $N_p(\delta)  \le C_{\epsilon/2, 2 \delta + \epsilon}$.
\end{proof}

The bounded degree property and the estimate \eqref{e-Cm} are true for all weighted graphs $(X,d,m)$ satisfying \ref{doub-loc} as shown below.
\begin{lemma}
 \label{l-g-vloc}
 Let $(X,d,m)$ be a metric measure space  satisfying \ref{doub-loc} that corresponds to a weighted graph $G=(X,E)$ with weight $m$.
 Then $G$ is of bounded degree
and
\begin{equation}
 \label{e-Cm1} \sup_{x,y \in X: x \sim y } \frac{m(y)}{m(x)} = C_m < \infty
\end{equation}
\end{lemma}
\begin{proof}
 By \ref{doub-loc}, there exists $C>0$ such that
 \[
   m(y) \le V(x,1) \le C V(x,1/2) = C m(x)
 \]
 for all $x,y \in X$ with $x \sim y$.
The above inequality shows that  $C_m \le C$ and $\sup_{x \in X} \operatorname{deg}(x) \le C^2$
\end{proof}

\chapter{Poincar\'{e} inequalities} \label{ch-pi}
Poincar\'{e} inequalities and its many variants are functional inequalities that have been extensively studied.
Many results in classical theory of Sobolev spaces, H\"{o}lder regularity estimates for solutions of elliptic and parabolic partial differential equations,
properties of harmonic functions, Harnack inequalities can be generalized to spaces satisfying volume doubling and a Poincar\'{e} inequality. See the introduction in \cite{HK00} for a survey and references.

Roughly speaking Poincar\'{e} inequalities control the variance of a function on a smaller ball by its Dirichlet energy (integral of the square of gradient) on a larger ball.
We start by reviewing Poincar\'{e} inequalities on weighted Riemannian manifolds.
Recall that a \emph{weighted  Riemannian manifold} $(M,g,\mu)$ is a Riemannian manifold $(M,g)$ equipped with a measure $\mu$ having a
smooth positive density $w$ with respect to the Riemannian measure induced by the metric $g$. The above function $w$ with $0<w \in \mathcal{C}^\infty(M)$ is called a \emph{weight}.
Recall that the \emph{gradient} $\gr f$ of a function $f \in C^\infty(M)$ is defined as the vector field satisfying $g( \gr f, Y)= Yf$ for all vector fields $Y$.
The \emph{length of the gradient} is denoted by $\abs{\gr f} = \sqrt{ g(\gr f, \gr f)}$.
We denote the Riemannian distance function by $d$, which makes $(M,d)$ a length space.
In a context when distance function is important, we will denote the weighted Riemannian manifold $(M,g,\mu)$ as a metric measure space $(M,d,\mu)$.
As before for $(M,d,\mu)$,  we denote closed ball and their volumes by $B(x,r)$ and $V(x,r)$ respectively.
\begin{definition}
 We say that a complete weighted Riemannian manifold $(M,g)$ with measure $\mu$ satisfies
 a Poincar\'{e} inequality \ref{poin-rm} if there exists $C_1>0$, $C_2 \ge 1$ such that for all $f \in C^\infty (M)$, for all $x \in M$ and for all $r>0$,
\begin{equation*}
 \label{poin-rm} \tag*{$(P)_{Rm}$} \int_{B(x,r)} \abs{f(y)-f_{B(x,r)}}^2 \mu(dy) \le C_1 r^2 \int_{B(x,C_2 r)} \abs{\gr f(y)}^2 \mu(dy)
\end{equation*}
where $f_{B(x,r)}$ denote the $\mu$-average of $f$ in $B(x,r)$
\[
 f_B = \frac{1}{V(x,r)} \int_{B(x,r)} f(y) \mu(dy).
\]
\end{definition}

The above inequality is sometimes called a weak, local, scale-invariant or $L^2$ Poincar\'{e} inequality but we will refrain from using such adjectives.
The word \emph{local} means that we are interested in average and integrals \emph{around some point $x$}.
This is in contrast with \emph{global} Poincar\'{e} inequality in which average and integrals are over the whole space $M$.
The Poincar\'{e} inequality is \emph{scale-invariant} or \emph{uniform} to emphasize the fact the the constants $C_1$ and $C_2$ is independent of $x$ or $r$.
For $1\le p < \infty$, we might replace \ref{poin-rm} with the  \emph{$L^p$  Poincar\'{e} inequality}
\[
\int_{B(x,r)} \abs{f(y)-f_{B(x,r)}}^p \mu(dy) \le C_1 r^p \int_{B(x,C_2 r)} \abs{\gr f(y)}^p \mu(dy).
\]
instead of $L^2$ version presented above.
For spaces satisfying global doubling property, one can always take $C_2=1$ in \ref{poin-rm}.
This is due to D. Jerison by a Whitney decomposition argument \cite{Jer86} (see also \cite[ Section 5.3.2]{Sal02}).
The Poincar\'{e} inequality with $C_2=1$ is called \emph{strong} Poincar\'{e} inequality as opposed to the \emph{weak} inequality \ref{poin-rm}.

\section{Gradient and Poincar\'{e} inequality at a given scale}
To generalize the Poincar\'{e} inequality \ref{poin-rm} to metric measure spaces, we must find a suitable definition for ``length of gradient''.
We will consider a class of random walks that spreads over different distances. Therefore we define length of gradient over different scales for a metric measure space.
We use the following definition due to \cite{Tes08} for length of gradient at a scale $h$ for a function $f:M \to \mathbb{R}$ with $f \in L^\infty(M,\mu)$ (denoted by $\abs{\nabla f}_h$).
\begin{definition} \label{d-grad}
Let $(M,d,\mu)$ be  a metric measure space. For any function $f \in L^\infty_{\operatorname{loc}}(M,\mu)$, the length of gradient at a scale $h$ for  $f$ is defined as the function
\begin{equation}\label{e-gradh}
 \abs{\nabla f}_h(x) = \left( \frac{1}{V(x,h)} \int_{B(x,h)} \abs{f(y) -f(x)}^2  \mu(dy)\right)^{1/2}.
\end{equation}
for all $x \in M$.
\end{definition}

\begin{remark}
 Our definition of  $\abs{\nabla f}_h$ coincides with $\abs{\nabla f}_{h,2}$ in the notation of Tessera \cite{Tes08}.
\end{remark}
Now that we are armed with length of gradient, we define the corresponding Poincar\'{e} inequality.
\begin{definition}
 We say that a metric measure space $(M,d,\mu)$ satisfies
 a Poincar\'{e} inequality at scale $h$,  if there exists $C_1>0$, $C_2 \ge 1,r_0 > 0$ such that for all $f \in L_{\operatorname{loc}}^\infty (M,\mu)$, for all $x \in M$ and for all $r\ge r_0$.
\begin{equation*}
 \label{poin-mms} \tag*{$(P)_{h}$} \int_{B(x,r)} \abs{f(y)-f_{B(x,r)}}^2 \mu(dy) \le C_1 r^2 \int_{B(x,C_2 r)} \left(\abs{\nabla f}_h(y)\right)^2 \mu(dy) \hypertarget{poin-mms}{}
\end{equation*}
where $f_{B(x,r)}$ denote the $\mu$-average of $f$ in $B(x,r)$
\[
 f_B = \frac{1}{V(x,r)} \int_{B(x,r)} f(y) \mu(dy).
\]
We will denote the above inequality by $P_h(r_0,C_1,C_2)$ or simply $(P_h)$.
\end{definition}
The rest of the chapter is devoted to the study of various properties and examples of the above Poincar\'{e} inequality \ref{poin-mms}.
In particular, we will show that for a weighted Riemannian manifold the Poincar\'{e} inequality
at scale $h$ \ref{poin-mms}, generalizes the Poincar\'{e} inequality \ref{poin-rm} under some mild hypothesis.
One of the main results that we will see in this chapter is that Poincar\'{e} inequality \ref{poin-mms} is preserved under quasi-isometries.

The following simple fact will be frequently used in rest of this chapter. Let $(M,d,\mu)$ be a metric measure space and let $A \subset M$ with $0<\mu(A)< \infty$. Then for every function $f \in L^\infty(A)$
\begin{equation} \label{e-mean}
 \inf_{\alpha \in \mathbb{R}} \int_A \abs{f(y)-\alpha}^2 \mu(dy) = \int_A \abs{f(y)-f_A}^2 \mu(dy)
\end{equation}
where $f_A$ is the $\mu$-average of $f$ in A,
\[
 f_A= \frac{1}{\mu(A)} \int_A f \,d\mu.
\]
In other words, mean minimizes squared error.

A Poincar\'{e} inequality at scale $h$ implies a Poincar\'{e} inequality at all larger scales $h'$ with $h'\ge h$.
\begin{lemma} \label{l-largesc}
 Let $(M,d,\mu)$ be a metric measure space satisfying \ref{doub-loc} and Poincar\'{e} inequality \ref{poin-mms} at scale $h$. Then for all $h' \ge h$, $(M,d,\mu)$ satisfies \hyperlink{poin-mms}{$(P)_{h'}$}.
\end{lemma}
\begin{proof}
 Assume $P_h(r_0,C_1,C_2)$. Then for all functions $f \in L^\infty_{\operatorname{loc}}$ and for all balls $B(x,r)$ with $r \ge r_0$ and $x \in M$, we have
 \begin{align*}
 \lefteqn{ \int_{B(x,r)} \abs{f - f_{B(x,r}}^2 \, d \mu } \\
 &\le C_1 r^2 \int_{B(x,C_2 r)} \abs{ \nabla_h f}^2 \, d\mu \\
  & =  C_1 r^2 \int_{B(x,C_2 r)}  \int_{B(x,C_2 r + h')}  \abs{f(y)- f(z)}^2 \frac{\one_{d(x,y) \le h}}{V(y,h)} \,dz \,dy \\
    & \le  C_{h,h'} C_1 r^2 \int_{B(x,C_2 r)}  \int_{B(x,C_2 r + h')}  \abs{f(y)- f(z)}^2 \frac{\one_{d(x,y) \le h'}}{V(y,h')} \,dz \,dy
 \end{align*}
which is $P_{h'}(r_0,C_1 C_{h,h'}, C_2)$. In the last line above, we used Lemma \ref{l-vloc}.
\end{proof}
\begin{remark} \label{r-pscale} A question now arises: At what scales $h$ does a Poincar\'{e} inequality $(P)_h$ hold ?
We have a satisfactory answer for length spaces and graphs.
If a graph satisfies Poincar\'{e} inequality at some scale, it satisfies Poincar\'{e} inequality at all scales $h \ge 1$ (See Corollary \ref{c-scalebig1}).
Moreover, a graph does not satisfy Poincar\'{e} inequality for scales smaller than 1 because the gradient at scales smaller than $1$ is identically zero.
If a length space  satisfies Poincar\'{e} inequality at some scale, then it satisfies Poincar\'{e} inequality at all positive scales (See Corollary \ref{c-allscales}).
We will see in Proposition \ref{p-pscale} that if $(P)_h$ is satisfied at  for some $h>0$ then $(P)_h$ is true for all $h > b$.
We analyze an example which is neither a graph nor a length space (See Example \ref{x-bl}) to show that $h>b$ is the best possible bound.
\end{remark}

We now show that the constant $r_0$ in $P_h(r_0,C_1,C_2)$ is flexible.
\begin{lemma} \label{l-pflex}
 Assume the Poincar\'{e} inequality $P_h(r_0,C_1,C_2)$ holds for a metric measure space $(M,d,\mu)$. Then for every $r_1>0$ and there exists constants
 $C_1',C_2'$ such that the Poincar\'{e} inequality $P_h(r_1,C_1',C_2')$ holds.
\end{lemma}
\begin{proof}
 The non-trivial case to check is $r_1 < r_0$. Assume $B(x,r)$ with $ r_1 \le r < r_0$. Then for all functions $f \in L^\infty_{\operatorname{loc}}(M)$,  by \eqref{e-mean} we have
 \begin{align*}
  \int_{B(x,r)} \abs{f - f_{B(x,r)}}^2 \,d\mu  &\le  \int_{B(x,r)} \abs{f - f_{B(x,r_0)}}^2 \,d\mu  \\
  & \le   \int_{B(x,r_0)} \abs{f - f_{B(x,r_0)}}^2 \,d\mu.
 \end{align*}
Combining the above inequality with $P_h(r_0,C_1,C_2)$ yields
\begin{equation*}
\int_{B(x,r)} \abs{f(y) - f_{B(x,r)}}^2 \,dy \le  C_1 r_0^2 \int_{B(x,C_2 r_0)} \abs{\nabla f}_h^2 \, d\mu.
\end{equation*}
Hence we can choose $C_1'=C_1 (r_0/r_1)^2$ and $C_2'= C_2 (r_0/r_1)$.
\end{proof}

\section{Robustness under quasi-isometry}
Since quasi-isometry between metric measure spaces satisfying \ref{doub-loc} is an equivalence relation, we may expect that a quasi-isometry preserves
certain invariants of such spaces. For instance, we saw in Proposition \ref{p-qivd} that quasi-isometry preserves the large scale doubling property.
In this section, we shall see that quasi-isometry preserves Poincar\'{e} inequality \ref{poin-mms}.
The approach for proving robustness of functional inequalities can traced back to the seminal works of Kanai \cite{Kan85, Kan86a, Kan86b} and further developments by Couhlon and Saloff-Coste \cite{CS95}.

The idea is to show that a functional inequality on the metric measure space is equivalent to a similar functional inequality on its net.
Since quasi-isometry is an equivalence relation, it suffices to show that the functional inequality on graphs is preserved under quasi-isometries.
To compare functional inequalities back and forth between a metric measure space and its net, we need to be able to transfer functions on metric measure space to functions on its net and vice-versa.
We start by developing those tools.

As before, let $(M,d,\mu)$ be a quasi-$b$-geodesic metric measure space satisfying \ref{doub-loc} and let  $(X,d_G,m)$ be its $\epsilon$-net for some fixed $\epsilon \ge b$.
By Proposition \ref{p-net}, we have that $(X,d_G,m)$ is a metric measure space satisfying \ref{doub-loc}.
Moreover the graph corresponding to $(X,d_G,m)$ is connected, countable with bounded degree. Let $D_X = \sup_{x \in X} \operatorname{deg}(x) < \infty$ be the maximum degree.
We will denote closed balls in  $(M,d,\mu)$ and  $(X,d_G,\mu)$ by $B$ and $B_G$ respectively. Similarly, we denote their corresponding volumes by $V$ and $V_G$ respectively.

Given a function $g \in L^\infty_{\operatorname{loc}}(M,\mu)$, we a define a function $\tilde{g} : X \to \mathbb{R}$ on its net as
\begin{equation} \label{e-mms2net}
 \tilde{g}(x) = \frac{1}{V(x,\epsilon)} \int_{B(x,\epsilon)} g \,d\mu.
\end{equation}
for all $x \in M$.
Conversely, given a function $f:X \to \mathbb{R}$ on the net, we define $\hat{f}: M \to \mathbb{R}$ as
\begin{equation} \label{e-net2mms}
\hat{f}= \sum_{x \in X} f(x) \theta_x
\end{equation}
where $\theta_x : M \to \mathbb{R}$ is defined by
\begin{equation} \label{e-punity}
  \theta_x(p)= \frac{\one_{B(x,\epsilon)}(p)}{\sum_{y \in X} \one_{B(y, \epsilon)} (p)}.
\end{equation}
The sum in \eqref{e-net2mms} and denominator of \eqref{e-punity} is a finite sum due to the finite overlap property of Proposition \ref{p-net}(h).
Moreover, there exists a constant $c_X>0$ such that $\{ \theta_x \}_{x\in X}$ is a partition of unity ($\sum_{x \in X} \theta_x \equiv 1$) satisfying
\begin{equation} \label{e-pu1}
 c_X \one_{B(x,\epsilon)} \le \theta_x \le \one_{B(x,\epsilon)}
\end{equation}
for all $x \in X$.The above properties of the partition of unity $\theta_x$ can be verified using Proposition \ref{p-net}.

We will now compare norms and gradients for the transfer of functions  between metric measure space and its net. For a metric measure space $(M,d,\mu)$ and  for all $ f \in L^\infty_{\operatorname{loc}}(A)$
where $A \subset M$, we denote by
\[
 \norm{f}_{p,A}= \left( \int_A \abs{f}^p \,d\mu \right)^{1/p}.
\]
We adapt the same notation for its net by considering it as a metric measure space.
\begin{definition}\label{d-disg}
For a function $f:X \to \mathbb{R}$ on a graph $(X,E)$,
we define the \emph{discrete gradient} of $f$ at $x$ as
\[
 \delta f(x) = \left( \sum_{y \sim x} \abs{f(y)-f(x)}^2 \right)^{1/2}.
\]
\end{definition}
This definition of discrete gradient was used to define Poincar\'{e} inequality for graphs in \cite{CS95}.
We now show that our definition of $\abs{\nabla f}_1$ is comparable to $\delta f$.
\begin{lemma} \label{l-discgrad}
 Let $(X,d_G,m)$ be a weighted graph satisfying \ref{doub-loc}. Then there exists $C>0$ such that
 \[
  C^{-1} \abs{\nabla f}_1(x) \le \delta f(x) \le C  \abs{\nabla f}_1(x)
 \]
for all functions $f:X \to \mathbb{R}$ and for all $x \in X$.
\end{lemma}
\begin{proof}
We write the gradient as
\[
  \left( \abs{\nabla f}_1(x) \right)^2 = \frac{1}{ m(x) + \sum_{y \in X: y \sim x} m(y) } \sum_{y \in X: y \sim x} \abs{f(y) - f(x)}^2 m(y).
 \]
The conclusion immediately follows from Lemma \ref{l-g-vloc}.
\end{proof}
\begin{remark}
 Therefore our Poincar\'{e} inequality $(P)_1$ generalizes the Poincar\'{e} inequality for graphs considered by Delmotte \cite{Del97,Del99}.
 Using the above lemma, our definition of $(P)_1$ for graphs is equivalent to the $L^2$ version of $(P)$ for graphs in \cite{CS95}.
\end{remark}

The next lemma compares gradient of a function on net and with its  metric measure space version.
\begin{lemma}  \label{l-grad-comp-mms}
Let $(M,d,\mu)$ be a  quasi-$b$-geodesic metric measure space  satisfying \ref{doub-loc} and let $(X,d,m)$ be its $\epsilon$-net for some $\epsilon \ge b$.
For all $h>0$, there exists positive reals $C,C'$ such that for all $x \in M$, for all $r \ge 1$,  and for all
functions $f:X \to \mathbb{R}$, we have
\[
  \norm{ \abs{\nabla \hat{f}}_h }^2_{2, B(x,r)}  \le C \norm{\delta f}_{2, B_G(\bar{x}, C' r )}^2
\]
where $\bar{x} \in X$ is such that $d(x,\bar{x}) \le \epsilon$ and  $\hat{f}:M\to \mathbb{R}$ is defined as in \eqref{e-net2mms}.
\end{lemma}
\begin{proof}
 Using Lemma \ref{l-vloc}, Proposition \ref{p-net} (a) and \eqref{e-kan-est}, there exists $C_1 >0$ such that
 \begin{align}
 \lefteqn{ \int_{B(x,r)}\int_M \abs{\hat{f}(y)- \hat{f}(z) }^2 \frac{1_{d(y,z)\le h}}{V(y,h)}\,dz\,dy} \label{e-gcm1} \\
\nonumber & \le   \sum_{s \in B_G(\bar{x},C_1 r)} \int_{B(s,\epsilon)} \int_M \abs{\hat{f}(y)- \hat{f}(z) }^2 \frac{1_{d(y,z)\le h}}{V(y,h)}\,dz\,dy
 \end{align}
 for all $x \in M$ and $r \ge 1$.
For all $s \in X$, $y \in B(s,\epsilon)$ and $z \in B(y,h)$, we have
\begin{align*}
 \hat{f}(y) - \hat{f}(z) &=\sum_{t \in X} f(t) (\theta_t(y) - \theta_t(z))  = \sum_{t \in X} (f(t)-f(s)) ( \theta_t (y) - \theta_t(z) ) \\
 &= \sum_{t \in X, d(s,t) \le 2 \epsilon+h} (f(t)-f(s)) ( \theta_t (y) - \theta_t(z) )
\end{align*}
 For the last line, if $d(s,t)> 2 \epsilon+h$, then by triangle inequality $d(t,y) >h+ \epsilon$, $d(t,z) > \epsilon$ and therefore
$\theta_t(y)=\theta_t(z)=0$ whenever $d(s,t)> 2 \epsilon+h$.
Let $D_X < \infty$ be the maximum degree of the net from Proposition \ref{p-net}(b) and $n_0 = A(h+ 2 \epsilon)+ A+h$ where $A$ is from \eqref{e-kan-est}.
Since $\abs{B_G(s,n_0)} \le 2 N^{n_0}$,  we have
\begin{equation} \label{e-gcm2}
 \abs{\hat{f}(p_1)- \hat{f}(p_2)} \le 2 \sum_{t \in B_G(s,n_0)} \abs{f(t)-f(s)} \le 4 D_X^{n_0} \sup_{t \in B_G(s,n_0)} \abs{f(t)-f(s)}
\end{equation}
Let $p_0,p_1,\ldots,p_{d_G(s,t)}$ be a path from $s$ to $t$.
For all $t \in B_G(s,n_0)$,  by Cauchy-Schwarz inequality we have
\begin{equation} \label{e-gcm3}
 \abs{f(t)-f(s)}^2\le \left( \sum_{i=0}^{d_G(t,s)-1}( f(p_i) - f(p_{i+1}) ) \right)^2 \le n_0 \sum_{p \in B_G(s,n_0)} \abs{\delta f(p)}^2.
\end{equation}
Combining \eqref{e-gcm1},\eqref{e-gcm2} and \eqref{e-gcm3}
 \begin{align}
\norm{ \abs{\nabla \hat{f}}_h }^2_{2, B(x,r)}  \nonumber & \le   \sum_{s \in B_G(\bar{x},C_1 r)}   4 N^{2n_0} n_0 \sum_{t \in B_G(s,n_0)} \abs{\delta f(t)}^2  m(s) \\
\nonumber & \le   \sum_{s \in B_G(\bar{x},C_1 r)}  C_m^{n_0} 4 N^{2n_0} n_0 \sum_{t \in B_G(s,n_0)} \abs{\delta f(t)}^2  m(t) \\
\nonumber & \le  8 C_m^{n_0}  D_X^{3n_0} n_0  \sum_{s \in B_G(\bar{x},C_3 r)} \abs{\delta f(t)}^2  m(t)
 \end{align}
 for all $x \in M$ and all $r \ge 1$.
The second line follows from  \eqref{e-Cm} and the last line from $\abs{B(t,n_0)} \le 2 D_X^{n_0}$.
\end{proof}

The following proposition shows that Poincar\'{e} inequalities can be transferred between a metric measure space and its net.
\begin{prop}\label{p-poin-netmms}
 Let $(M,d,\mu)$ be a $b$-quasi-geodesic space satisfying \ref{doub-loc} and let $(X,d,m)$ be its $\epsilon$-net for some $\epsilon \ge b$.
 Then for all $h \ge 5 \epsilon$,  $(X,d_G,m)$ satisfies $(P)_1$ if and only if
 $(M,d,\mu)$ satisfies $(P)_h$.
\end{prop}
\begin{remark}
In general, we do not know if the inequality $h \ge 5\epsilon$ in the above statement is required. We believe that $h > b$ is sufficient but we are unable to prove this.
\end{remark}

\begin{proof}[Proof of Proposition \ref{p-poin-netmms}]
 Suppose $(X,d_G,m)$ satisfies $P_1(r_0,C_1',C_2')$. 
 
 Let $g \in L^\infty_{\operatorname{loc}}$  and let $\tilde{g}:X \to \mathbb{R}$ be defined as \eqref{e-mms2net}.
Let $x \in M$ and $r \ge r_0$ be arbitrary. Let $\bar{x} \in X$ be such that $d(x,\bar{x}) \le \epsilon$. There exists $C_1>0$ such that, we have
\begin{align}
\label{e-nms1}  \lefteqn{\int_{B(x,r)} \abs{g(y) - g_{B(x,r)}}^2 dy}\\
\nonumber  & \le    \int_{B(x,r)} \abs{g(y) - \alpha}^2 dy \le \sum_{p \in B_G(\bar{x},C_1r)} \int_{B(p,\epsilon)}  \abs{g(y) - \alpha}^2 dy \\
\nonumber    &\le 2 \sum_{p \in B_G(\bar{x}, C_1r )} \left( \int_{B(p,\epsilon)}  \abs{g(y) - \tilde{g}(p)}^2 dy + m(p) \abs{ \tilde{g}(p) - \alpha}^2 \right)
 \end{align}
for all $\alpha \in \mathbb{R}$ and all functions $g$.
The second line above follows from \eqref{e-mean},  Proposition \ref{p-net} (a) and \eqref{e-kan-est}. The last line follows from $(a+b)^2 \le 2(a^2 + b^2)$.
The first term above is bounded using Jensen's inequality as
\begin{equation*}
 \int_{B(p,\epsilon)}  \abs{g(y) - \tilde{g}(p)}^2 dy \le \frac{1}{V(p,\epsilon)}  \int_{B(p,\epsilon)}  \int_{B(p,\epsilon)}  \abs{g(y) - g(z)}^2 \,dz\,dy
\end{equation*}
Hence by Lemma \ref{l-vloc}, we have
\begin{align}
\nonumber  I_1 &=  \sum_{p \in B_G(\bar{x}, C_1r )}  \int_{B(p,\epsilon)}  \abs{g(y) - \tilde{g}(p)}^2 dy \\
\nonumber &\le  C_{\epsilon, 6\epsilon} \sum_{p \in B_G(\bar{x}, C_1r )}   \int_{B(p,\epsilon)}  \int_{B(p,\epsilon)}  \abs{g(y) - g(z)}^2 \frac{1_{d(y,z)\le 2\epsilon}}{V(y,5 \epsilon)} \,dz\,dy \\
 \label{e-nms2}& \le  C_2 \norm{\abs{\nabla g}_{5 \epsilon}}^2_{2,B(x,C_{3} r)}
\end{align}
for some $C_2,C_3$ large enough.
We used Lemma \ref{l-vloc} and triangle inequality in second line above and Proposition \ref{p-net}(h) and \eqref{e-kan-est} in the last line.
Choose $\alpha = \tilde{g}_{B_G(\bar{x}, C_1r )}$ in \eqref{e-nms1},
 so as to apply $P_1(r_0,C'_1,C_2')$ on $(X,d,m)$ to  bound the second term in \eqref{e-nms1} as
\begin{equation}\label{e-nms3}
 I_2 =  \sum_{p \in B_G(\bar{x}, C_1r )}  m(p) \abs{ \tilde{g}(p) - \alpha}^2    \le  C_4 r^2 \norm{\delta \tilde{g}}^2_{2,B_G(\bar{x},C_5 r)}
\end{equation}
For all $p,q \in X$  satisfying $p \sim q$,  by Jensen's inequality and triangle inequality we have
\begin{align*}
 \abs{\tilde{g}(p)-\tilde{g}(q)}^2 &\le \frac{1}{m(p)m(q) } \int_{B(p,\epsilon)} \int_{B(q,\epsilon)} \abs{g(y)-g(z)}^2\,dz\,dy \\
 & \le  \frac{1}{m(p)m(q) } \int_{B(p,\epsilon)} \int_{B(q,\epsilon)} \abs{g(y)-g(z)}^2 1_{d(y,z)\le 5 \epsilon}\,dz\,dy
\end{align*}
Hence for all $p \in X$,
\begin{align}
\nonumber \lefteqn{\abs{\delta \tilde{g}(p)}^2 m(p) } \\
\nonumber &\le  \sum_{q \in X, q\sim p} \frac{1}{V(q,\epsilon)}  \int_{B(p,4\epsilon)} \int_{B(p,4\epsilon)} \abs{g(y)-g(z)}^2 1_{d(y,z)\le 5 \epsilon}\,dz\,dy \\
\nonumber &\le  C_{ \epsilon, 9 \epsilon } \sum_{q \in X, q\sim p}  \int_{B(p,4\epsilon)} \int_{B(y,4\epsilon)} \abs{g(y)-g(z)}^2 \frac{1_{d(y,z)\le 5 \epsilon}}{V(y,5 \epsilon)} \,dz\,dy \\
 \label{e-nms4} &\le   C_{ \epsilon , 9\epsilon} D_X \int_{B(p,4\epsilon)} \int_{B(p,4\epsilon)} \abs{g(y)-g(z)}^2 \frac{1_{d(y,z)\le 5 \epsilon}}{V(y,5 \epsilon)} \,dz\,dy
\end{align}
The third line follows from Lemma \ref{l-vloc} and the last line from  bounded degree property of Proposition \ref{p-net}(b).
Combining \eqref{e-nms3}, \eqref{e-nms4} along with   \eqref{e-kan-est} and Proposition \ref{p-net}(h) gives
\begin{equation}
 \label{e-nms5} I_2 \le C_6 r^2  \norm{\abs{\nabla g}_{5 \epsilon}}^2_{2,B(x,C_{7} r)}.
\end{equation}
Combining \eqref{e-nms1},\eqref{e-nms2} and \eqref{e-nms5} yields Poincar\'{e} inequality $(P)_{5\epsilon}$ for  $(M,d,\mu)$.
By Lemma \ref{l-largesc}, we get $(P)_h$ for all $h \ge 5 \epsilon$.

Conversely, suppose that $(M,d,\mu)$ satisfies $P_h(r_1,C_3',C_4')$ for some $h \ge 5 \epsilon$.
Let $f: X \to \mathbb{R}$ be an arbitrary function and define $\hat{f}:M \to \mathbb{R}$ as in \eqref{e-net2mms}.
Denote $B_G(p,r)$ be an arbitrary ball in $(X,d,m)$ where $r \ge r_1$.
Then using \eqref{e-mean}, \ref{doub-loc} and the inequality $(a+b)^2 \le 2(a^2 +b^2)$ we have
\begin{align}
\label{e-nms6} \lefteqn{\sum_{q \in B_G(p,r)} \abs{f(q)-f_{B_G(x,r)}}^2 m(q)}\\
\nonumber &\le \sum_{q \in B_G(p,n)} \abs{f(q)-\alpha}^2 m(q) \le C_{\epsilon/2} \sum_{q \in B_G(x,n)} \int_{B(q,\epsilon/2)}  \abs{f(q)-\alpha}^2 \,d\mu \\
\nonumber & \le 2 C_{\epsilon/2} \sum_{q \in B_G(p,n)} \int_{B(q,\epsilon/2)}  \left(\abs{f(q)-\hat{f}(y)}^2  +  \abs{\hat{f}(y)-\alpha}^2 \right) \,dy
\end{align}
for all $\alpha \in \mathbb{R}$.
 Using Proposition \ref{p-net}(a) and \eqref{e-kan-est},there exists positive reals $C_8,C_{11},C_{12}$ such that for all  $r\ge \min(1,r_1/C_8)$ and all functions $f$,  we have
\begin{align}
\nonumber J_2 &=\sum_{q \in B_G(p,r)}  \int_{B(q,\epsilon/2)} \abs{\hat{f}(y)-\alpha}^2\,dy
\nonumber  \le  \int_{B(p,C_8r)} \abs{\hat{f}(y)-\alpha}^2\,dy  \\
& \le  C_9 r^2 \norm{\abs{\nabla \hat{f}}_h}^2_{2,B(p,C_{10} r)}
\label{e-nms7}  \le  C_{11} r^2 \norm{\delta f}^2_{2, B_G(p, C_{12} r)}.
\end{align}
In the second step above, we fix  $\alpha=\hat{f}_{B(p,C_2 r)}$ and apply Poincar\'{e} inequality \ref{poin-mms} and in the last line we apply Lemma \ref{l-grad-comp-mms}.
Let $q \in X$ and $y \in B(q,\epsilon/2)$.
Since $\hat{f}(y) = \sum_{t \in X: d_G(t,q) \le 1} \theta_t(y) f(t) $, we have
\begin{align*}
 \abs{f(q)- \hat{f}(y)} & =  \abs{\sum_{t \in X: d(t,q) \le 1} \theta_t(y) (f(q)- f(t)) }
   \le  \sum_{t \in X: d(t,q) \le 1} \abs{(f(q)- f(t)) } \\
 & \le   \delta f(q) \sqrt{D_X}.
\end{align*}
The last line follows from Cauchy-Schwarz inequality and maximum degree $D_X$ from Proposition \ref{p-net}(b). Using
this estimate, we have
\begin{align}
\nonumber J_1 &= \sum_{q \in B_G(p,r)} \int_{B(q,\epsilon/2)} \abs{f(y)- \hat{f}(y)}^2 \,dy
\nonumber \le  D_X C_{\epsilon/2} \sum_{y \in B_G(p,r)} \abs{\delta f(q)}^2 m(q) \\
\label{e-nms8} & \le  D_X C_{\epsilon/2} \norm{\delta f}^2_{2, B_G(p,r)}.
\end{align}
Thus $(P)_1$ for $(X,d,m)$ follows from \eqref{e-nms6}, \eqref{e-nms7} and \eqref{e-nms8} along with Lemma \ref{l-discgrad}.
\end{proof}
 We now show that Poincar\'{e} inequality $(P)_1$ is preserved under quasi-isometry for graphs.

Let $(X,d,m)$ be a weighted graph. Then for the closed balls in the graph, we have $B(x,r)= B(x, \lfloor r \rfloor)$.
Hence by Lemmas \ref{l-pflex}  and \ref{l-discgrad}, we have  the following
equivalent definition of $(P)_1$: A weighted graph $(X,d,m)$ satisfies $(P)_1$,
if there exists $C_1>0$, $C_2 \ge 1$ such that for all $f :X \to \mathbb{R}$, for all $x \in X$ and for all $n \in \mathbb{N}^*$.
\begin{equation} \label{e-altgrad}
\sum_{y \in B(x,n)} \abs{f(y)-f_{B(x,n)}}^2 \mu(dy) \le C_1 n^2 \sum_{B(x,C_2 n)} \abs{\delta f(y)}^2 m(y)
\end{equation}
where $f_{B(x,n)}$ is the average of $f$ in $B(x,n)$ with respect to measure $m$. We will use the alternate definition for the proposition below.
\begin{prop} [\cite{CS95}, Proposition 4.2] \label{p-graphrobust}
Let $(X_1,d_1,m_1)$ and $(X_2,d_2,m_2)$ be quasi-isometric weighted graphs that satisfy \ref{doub-loc}.
 Then $(X_1,d_1,m_1)$ satisfies $(P)_1$ if and only if  $(X_2,d_2,m_2)$ satisfies $(P)_1$.
\end{prop}

\begin{proof}
 We denote the balls, volume of balls and gradient of $(X_i,d_i,m_i)$  by $B_i,V_i,\delta_i$ respectively for $i=1,2$.

 Assume that $(X_1,d_1,m_1)$ satisfies $(P)_1$.
 Let $\Phi:X_1 \to X_2$ be a quasi-isometry with $ \cup_{x \in X_1} B_2( \Phi(x), k) = X_2$ for some $k \in \mathbb{N}^*$.
 Let $f:X_2 \to \mathbb{R}$ and let $f_k(x)$ denote the average of $f$ in $B_2(x,k)$ with respect to measure $m_2$.
 Applying $(P)_1$ to the function $f_k \circ \Phi :X_1 \to \mathbb{R}$, we have
 \begin{equation}
  \label{e-rpg1}  \norm{f_k \circ \Phi - (f_k \circ \Phi)_{B_1(x,n)}}^2_{2, B_1(x,n)} \le C_1 n^2  \norm{\delta_1 (f_k \circ \Phi)}^2_{2, B_1 (x, C_1'n)}
 \end{equation}
For all $y \in X_1$, we have
\begin{align}
\label{e-rpg2} \abs{\delta_1( f_k \circ \Phi)(y)}^2 m_1(y) &\le C_2  \abs{\delta_1( f_k \circ \Phi)(y)}^2 m_2 (\Phi(y)) \\
\nonumber & \le  C_2 D_{X_1} \sup_{w_1 \in X_1 : w_1 \sim y} \abs{f_k(\Phi(w_1))- f_k(\Phi(y))}^2 m_2(\Phi(y))
\end{align}
The first line follows from the quasi-isometry condition $ m_1(y) \le C' m_2 (\Phi(y))$ and the second line from bounded degree property of Lemma \ref{l-g-vloc}.
Since $\Phi$ is a quasi-isometry, there exists $l >0$ such that $B_2(\Phi(y),l) \subseteq \Phi(B_1(y,1))$ for all $y \in X_1$.
An application of Cauchy-Schwarz inequality along the minimal path $\Phi(w_1)=p_0,p_1, \ldots, p_s= \Phi(y)$ gives
\begin{align}
  \abs{f_k(\Phi(w_1))- f_k(\Phi(y))}^2 &\le l \sum_{i=0}^{s-1} \abs{f_k(\Phi(p_i))- f_k(\Phi(p_{i+1}))}^2 \nonumber\\
 \label{e-rpg3} &\le l \sum_{z \in B_2(\Phi(y),l)} \abs{\delta_2 f_k (z)}^2
\end{align}
for all $y,w_1 \in X_1$ such that $y \sim w_1$.
Combining \eqref{e-rpg2}, \eqref{e-rpg3} and \eqref{e-Cm1} of Lemma \ref{l-g-vloc}, we obtain
\begin{equation}
 \label{e-rpg4} \abs{\delta_1( f_k \circ \Phi)(y)}^2 m_1(y) \le C_2 D_{X_1} l C_m^l \sum_{z \in B_2(\Phi(y),l)} \abs{\delta_2 f_k (z)}^2 m_2(z)
\end{equation}
Since $\Phi$ is a quasi-isometry, there exists $C_2'>0$ such that
\[
\cup_{z \in \Phi(B_1(x,C_1'n))} B_2(z,l) \subseteq B_2(\Phi(x),C_2'n)
\]
 for all $x \in X_1$ and $n \in \mathbb{N}^*$.
Combining this with \eqref{e-rpg4} and Lemma \ref{l-g-vloc} gives
\begin{equation} \label{e-rpg5}
 \norm{\delta_1 (f_k \circ \Phi)}^2_{2,B_1(x,C_1'n)} \le C_3  \norm{\delta_2 f_k}^2_{2, B_2(\Phi(x),C_2'n)}
\end{equation}
for all $n \in \mathbb{N}^*$, for all $x \in M$ and for all functions $f$.
We write,
\begin{align*}
 \abs{\delta_2f_k(z)}^2 &= \sum_{y \in X_2:y \sim z} \abs{f_k(z)-f_k(y)}^2 \\
 &\le 2 \sum_{y \sim z} \left( \frac{1}{V_2(z,k)} \sum_{t \in B_2(z,k)} \abs{f(t)-f(z)}^2 m_2(t) \right. \\
 &  \hspace{35pt} \left. +\frac{1}{V_2(y,k)} \sum_{s\in B_2(y,k)} \abs{f(s)-f(z)}^2 m_2(s) \right) \\
 & \le  \frac{2D_{X_2}}{V_2(z,k)} \sum_{t \in B_2(z,k)} \abs{f(t)-f(z)}^2 m_2(t) \\
 & \hspace{4mm} + \frac{C_4}{V_2(z,k)} \sum_{s \in B_2(z,k+1)} \abs{f(s)-f(z)}^2 m_2(s).
\end{align*}
The second and third lines above follow from $(a+b)^2\le2(a^2+b^2)$ along with Jensen's inequality. The last two lines follow from
Lemmas \ref{l-g-vloc} and \ref{l-vloc} to compare $V_2(z,k)\le V_2(y,k+1) \le C_{4} V_2(y,k)$. By Lemma \ref{l-g-vloc}, we have $m_2(t) \le C_3' V_2(z,k)$ for all
$z \in X_2$ and for all $t \in B_2(z,k+1)$. It follows that
\[
\abs{ \delta_2 f_k(z)}^2  \le C_5 \sum_{t \in B_2(z,k+1)} \abs{f(t)-f(z)}^2
\]
An application of Cauchy-Schwarz inequality similar to \eqref{e-rpg3} yields
\[
\abs{ \delta_2 f_k(z)}^2 \le C_5 (k+1) D{X_2}^{k+1} \sum_{y \in B_2(z,k+1)} \abs{\delta_2 f(y)}^2
\]
Finally by Lemma \ref{l-g-vloc},
\begin{equation}
 \label{e-rpg6} \norm{\delta_2 f_k(z)}^2_{2,B_2(\Phi(x),C_1'n)} \le C_6  \norm{\delta_2 f}^2_{2,B_2(\Phi(z),C_4'n)}
\end{equation}
Combining \eqref{e-rpg1}, \eqref{e-rpg5} and \eqref{e-rpg6}, we have
\begin{equation}
 \label{e-rpg7} \norm{f_k \circ \Phi - (f_k \circ \Phi)_{B_1(x,n)}}^2_{2,B_1(x,n)} \le C_6 n^2  \norm{\delta_2 f}^2_{2,B_2(\Phi(z),C_4'n)}
\end{equation}
Suppose we prove that
\begin{align}
 \label{e-rpg8} \norm{f -f_{B_2(\Phi(x),n)}}^2_{2, B_2(\Phi(x),n)} &\le  C_8 \norm{ \delta_2 f}^2_{2,B_2(\Phi(x),C_5'n)}\\
 \nonumber &  + C_9 \norm{f_k\circ \Phi - (f_k\circ \Phi)_{B_1(x,C_5'n)}}^2_{2, B_1(x,C_5' n)}.
\end{align}
Then \eqref{e-rpg7} and \eqref{e-rpg8} gives
\begin{equation}
\label{e-rpg9}    \norm{f -f_{B_2(\Phi(x),n)}}^2_{2,B_2(\Phi(x),n)} \le C_{10} n^2  \norm{ \delta_2 f(z)}^2_{2,B_2(\Phi(x), C_6'n)}.
\end{equation}
for all $x \in M_1$ and for all $n \in \mathbb{N}^*$.
Thus we obtain Poincar\'{e} inequality for all balls centered in the image of $\Phi$.
Let $y \in M_2$. Then there exists $\bar{y} \in M_1$ such that $y \in B_2(\Phi(\bar{y}),k)$. It follows from \eqref{e-mean} that
\begin{align*}
  \norm{f-f_{B_2(y,n)}}^2_{2,B_2(y,n)}&\le  \norm{f-f_{B_2(\Phi(\bar{y}),n+k)}}^2_{2,B_2(y,n)} \\
 & \le \norm{f-f_{B_2(\Phi(\bar{y}),n+k)}}^2_{2,B_2(\Phi(\bar{y}), n+k)}
\end{align*}
Hence by \eqref{e-rpg9}, we have $(P)_1$ for $(X_2,d_2,m_2)$. \\
It remains to show \eqref{e-rpg8}. Let $\Phi^{-1}:M_2 \to M_1$ denote the quasi-inverse such that
$\Phi^{-1}(m_2) \in M_1$ is such that $d_2(m_2, (\Phi \circ \Phi^{-1})(m_2)) \le  k$.
We have by \eqref{e-mean} and $(a+b)^2 \le 2(a^2 +b^2)$  that
\[
\norm{f -f_{B_2(\Phi(x),n)}}^2_{2, B_2(\Phi(x),n)} \le \norm{f -\alpha}^2_{2, B_2(\Phi(x),n)} \le 2S_1 +2S_2
\]
where
\[
 S_1= \norm{f - f_{k} \circ \Phi \circ \Phi^{-1} }^2_{2,  B_2(\Phi(x),n)}
\]
and
\[
 S_2 =  \norm{ f_{k} \circ \Phi \circ \Phi^{-1} -\alpha}^2_{2,  B_2(\Phi(x),n)}
\]
for all $\alpha \in \mathbb{R}$.
Let $\bar{z}=\Phi \circ \Phi^{-1} (z)$, then $d_2(z,\bar{z}) \le k$. We bound $S_1$ as
\begin{align*}
S_1 &=  \sum_{z \in B_2(\Phi(x),n)} \abs{f(z)- f_{k}  (\bar{z})}^2 m_2(z) \\
& \le   \sum_{z \in B_2(\Phi(x),n)} \left( \frac{1}{V_2(\bar{z},k)} \sum_{t \in B_2(\bar{z},k)} \abs{f(z)- f  (t)}^2 m_2(t) \right) m_2(z) \\
& \le  C_{11} \sum_{z \in B_2(\Phi(x),n)} \sum_{t \in B_2(z,2k)} \abs{\delta_2 f(t)}^2 m_2(z)\\
& \le  C_{12} \norm{\delta_2f}^2_{2,B_2(\Phi(x),C_7'n)}.
\end{align*}
The second line follows from Jensen's inequality. The third line follows from  $d_2(z,\bar{z}) \le k$ and an application of Cauchy-Schwarz inequality similar to \eqref{e-rpg3}.
The last two lines follows from bounded degree property and \eqref{e-Cm1} of Lemma \ref{l-g-vloc}.
For the second term $S_2$, we have
\begin{align*}
 S_2 & \le  C_{13} \sum_{z \in B_2(\Phi(x),n)} \abs{f_{k} \circ \Phi \circ \Phi^{-1} (z) -  \alpha }^2 m_1(\Phi^{-1}(z)) \\
 & \le  C_{14}   \norm{f_{k} \circ \Phi  - \alpha}^2_{2,B_1(x,C_8'n)}
\end{align*}
We use the fact that $\Phi$ and $\Phi^{-1}$ are quasi-isometries.
Indeed, for $C_8'$ big enough $\Phi^{-1}(B_2(\Phi(x),n)) \subset B_1(x,C_8'n)$, since $\Phi^{-1}$ is a
quasi-isometry with \[d_2(x, \Phi\circ \Phi^{-1}(x)) \le k.\] Moreover $\abs{ \{ z \in X_2: \Phi^{-1}(z) = w \}}$ is uniformly bounded over all $w \in X_1$.
Choose $C_5' = \max(C_7',C_8')$. The bounds on $S_1$ and $S_2$ along with the choice $\alpha=(f_k\circ \Phi)_{B_1(x,C_5'n)}$  concludes the proof of \eqref{e-rpg8}.
\end{proof}
\begin{corollary} \label{c-scalebig1}
 Let $(X,d,m)$ be a weighted graph satisfying \ref{doub-loc} and let $h \ge 1$. Then  $(X,d,m)$ satisfies $(P)_1$ if and only if $(X,d,m)$ satisfies $(P)_h$.
\end{corollary}
\begin{proof}
 By Lemma \ref{l-largesc}, $(P)_1$ implies $(P)_h$.\\
 Conversely, assume $(X,d,m)$ satisfies $(P)_h$.  Fix $k=\lfloor h \rfloor$.
Since $\abs{\nabla f}_h= \abs{\nabla f}_k$ for all functions $f:X \to \mathbb{R}$, $(X,d,m)$ satisfies $(P)_k$.
\emph{$k$-fuzz} of a weighted graph is defined as the weighted graph $(X,d_k,m)$ where the edges are defined by
 $d_k(x,y)=1$ if and only if $1 \le d(x,y) \le k$ for $x,y \in X$. It is straightforward to verify that the $k$-fuzz $(X,d_k,m)$ satisfies \ref{doub-loc} and is quasi-isometric to $(X,d,m)$.
Since $(X,d,m)$ satisfies $(P)_k$, the $k$-fuzz $(X,d_k,m)$ satisfies $(P)_1$.
 Hence by Proposition \ref{p-graphrobust}, $(X,d,m)$ satisfies $(P)_1$.
\end{proof}
As outlined at the start, the robustness of Poincar\'{e} inequality on graphs in Proposition \ref{p-graphrobust} can be transferred to arbitrary quasi-geodesic spaces using Proposition \ref{p-poin-netmms}.
\begin{prop} \label{p-robustmms}
For $i=1,2$, let $(M_i,d_i,\mu_i)$ be quasi-$b_i$-geodesic spaces satisfying \ref{doub-loc}. Assume $(M_1,d_1,\mu_1)$ and $(M_2,d_2,\mu_2)$ are quasi-isometric.
Let $h_1 \ge 5 b_1$ and  for all $h_2 \ge 5 b_2$.
Then $(M_1,d_1,\mu_1)$ satisfies $(P)_{h_1}$ if and only if $(M_2,d_2,\mu_2)$ satisfies $(P)_{h_2}$.
\end{prop}
\begin{proof}
  It is a direct consequence of Propositions \ref{p-poin-netmms} and \ref{p-graphrobust}.
\end{proof}
The above Proposition along with the fact that length space is $b$-geodesic for all $b>0$ gives the following.
\begin{corollary} \label{c-allscales}
 Let $(M,d,\mu)$ be a length space satisfying \ref{doub-loc}. Then for every $h_1,h_2>0$, $(M,d,\mu)$ satisfies $(P)_{h_1}$ if and only if $(M,d,\mu)$ satisfies $(P)_{h_2}$.
\end{corollary}
\section{Poincar\'{e} inequalities in Riemannian manifolds}
In this section, we see the relationship between various Poincar\'{e} inequalities on a weighted Riemannian manifold. We start by introducing some Poincar\'{e} inequalities from \cite{CS95}.
\begin{definition}
 We say that a complete weighted Riemannian manifold $(M,g)$ with measure $\mu$ satisfies
\ref{poin-inf} if there exists $r_0>0$, $C_1>0$, $C_2 \ge 1$ such that for all $f \in C^\infty (M)$, for all $x \in M$ and for all $r\ge r_0$, we have
\begin{equation*}
 \label{poin-inf} \tag*{$(P)_{\infty}$} \int_{B(x,r)} \abs{f(y)-f_{B(x,r)}}^2 \mu(dy) \le C_{r_0} r^2 \int_{B(x,C_2 r)} \abs{\gr f(y)}^2 \mu(dy)
\end{equation*}
where $f_{B(x,r)}$ denote the average of $f$ in $B(x,r)$ with respect to $\mu$.
 We say that a complete weighted Riemannian manifold $(M,g)$ with measure $\mu$ satisfies
 \ref{poin-loc} if there exists $C_1>0$, $C_2 \ge 1$ such that for all $f \in C^\infty (M)$, for all $x \in M$ and for all $r\ge0$, we have
\begin{equation*}
 \label{poin-loc} \tag*{$(P)_{\operatorname{loc}}$} \int_{B(x,r)} \abs{f(y)-f_{B(x,r)}}^2 \mu(dy) \le C_{r} \int_{B(x,C_2 r)} \abs{\gr f(y)}^2 \mu(dy)
\end{equation*}
where $f_{B(x,r)}$ denote the average of $f$ in $B(x,r)$ with respect to $\mu$.
\end{definition}
It is clear that \ref{poin-rm} implies \ref{poin-inf} and \ref{poin-loc}.
The inequality \ref{poin-loc} is a weak assumption. For instance, manifolds with a lower bound on Ricci curvature satisfy \ref{poin-loc}.
Inequality \ref{poin-inf} is a large scale version of \ref{poin-rm}.
\begin{prop} (\cite[Proposition 6.10]{CS95}) \label{p-robrm} Let $(M,g,\mu)$ be a weighted Riemannian manifold satisfying \ref{doub-loc} and \ref{poin-loc} and let $(X,d,m)$ be its
$\epsilon$-net for some $\epsilon>0$. Then $(M,g)$ with measure $\mu$ satisfies \ref{poin-inf} if and only if $(X,d,m)$ satisfies $(P)_1$.
\end{prop}
Propositions \ref{p-robrm} and \ref{p-poin-netmms} along with Corollary \ref{c-allscales} gives the following
\begin{prop} \label{p-rmtfae}
 Let $(M,g,\mu)$ be a weighted Riemannian manifold with Riemannian distance $d$. Denote by $(X,d_G,m)$ be an $\epsilon$-net of $(M,d,\mu)$ for some $\epsilon>0$.
 Assume $(M,d,\mu)$ satisfies \ref{doub-loc} and \ref{poin-loc}.
 Then the following are equivalent:
 \begin{itemize}
  \item[(a)] $(M,d,\mu)$ satisfies \ref{poin-inf}.
  \item[(b)] $(M,d,\mu)$ satisfies \ref{poin-mms} for some $h>0$.
  \item[(c)] $(M,d,\mu)$ satisfies \ref{poin-mms}  for all $h>0$.
  \item[(d)] $(X,d_G,m)$ satisfies \hyperlink{poin-mms}{$(P)_1$}.
  \item[(e)] $(X,d_G,m)$ satisfies \ref{poin-mms}  for some $h \ge 1$.
 \end{itemize}
\end{prop}

\section{Poincar\'{e} inequality: Examples and Non-examples}
A large class of examples for $(P)_h$ can be obtained from Proposition \ref{p-robustmms} and \ref{p-rmtfae}.
For instance, Buser proved \ref{poin-rm} for Riemannian manifolds with non-negative Ricci curvature.
Therefore by Proposition \ref{p-rmtfae}, Riemannian manifolds with non-negative Ricci curvature satisfy $(P)_h$ for all positive scales $h$. The following example is from \cite{GS05}.
\begin{example}\label{x-grisal}[Euclidean space with radial weights]
 Consider $\mathbb{R}^n$ with standard Riemannian metric $g$, Euclidean distance $d$ and measure $d\mu_\alpha(x)=(1+\abs{x}^2)^{\alpha/2} \,dx$.
 It is easy to verify that $(\mathbb{R}^n,d,\mu_\alpha)$ satisfies \ref{doub-loc} and \ref{poin-loc}.
 Moreover $(\mathbb{R}^n,d,\mu_\alpha)$ satisfies \ref{doub-inf} if and only if $\alpha > -n$.
 If $n \ge 2$, then $(\mathbb{R}^n,d,\mu_\alpha)$ satisfies \ref{poin-inf} and therefore $(P)_h$ for all values of $\alpha \in \mathbb{R}$ and $h>0$ (See Remark 3.13 in \cite{GS05}).
However, $(\mathbb{R},d,\mu_\alpha)$ does not satisfy \ref{poin-inf} for $\alpha \ge 1$. It can be easily seen using the test function $f_\alpha(x)=  \int_0^x (1+t^2)^{-\alpha/2} \, dt$.
By \cite[Theorem 7.1(i)]{GS05},  $(\mathbb{R},d,\mu_\alpha)$ satisfies \ref{poin-inf} if $-1<\alpha <1$.
Due to an unpublished result of Grigor'yan and Saloff-Coste, $(\mathbb{R},d,\mu_\alpha)$ satisfies \ref{poin-inf} if and only if $\alpha < 1$.
\end{example}
\begin{example}\label{x-bl}
We describe an example of quasi-geodesic space which is neither a graph, nor a length space. Consider the `Broken line'  $BL\subset \mathbb{R}$
\[
BL= \bigcup_{n \in \mathbb{Z}} [n-1/4,n+1/4]
\]
It is quasi-$b$-geodesic if and only if $b \ge 1/2$. We equip it with the Euclidean distance $d$ and restriction of Lebesgue measure $\mu$ on $BL$.
 $(P)_h$ is not true for $(BL,d,\mu)$ if $h \le 1/2$. It can be seen using the test function $f:BL \to \mathbb{R}$ defined by $f(x)= (-1)^{\lfloor x+1/4\rfloor}$.
 However, it can be shown that for $(BL,d,\mu)$ satisfies $(P)_h$ for all $h>1/2$.
\end{example}
 \begin{example}[Hyperbolic space]\label{x-hyperb}
  Consider the Hyperbolic $n$-space $\mathbb{H}^n$  with Riemannian distance $d_H$ and Riemannian measure $\mu$.  $(\mathbb{H}^n,d_H,\mu)$ satisfies \ref{doub-loc} and \ref{poin-loc}.
  However $(\mathbb{H}^n,d,\mu)$ does not satisfy \ref{doub-loc} because the volume of balls grows exponentially. Further $(\mathbb{H}^n,d_H,\mu)$ does not satisfy
  the Poincar\'{e} inequality \ref{poin-inf}.

  Another example in  a similar spirit is the infinite $d$-regular tree $\mathbb{T}_d$ equipped with graph distance metric and counting measure. It is easy to very that if $d \ge 3$,  $\mathbb{T}_d$ does not satisfy
  \ref{doub-inf} and does not satisfy \ref{poin-mms} for all $h >0$.
 \end{example}
 Examples \ref{x-grisal} and \ref{x-hyperb} illustrates all the four possibilities that can occur with properties \ref{doub-inf} and \ref{poin-inf}.  It is summarized in the table below.
\begin{table}
    \begin{tabular}[c]{ | l | l | p{7cm} |}
    \hline
    \ref{doub-inf} & \ref{poin-inf} & Examples \\ \hline
    True & True &  $(\mathbb{R}^n, d,\mu_\alpha)$ with $n\ge2$ and $\alpha > -n$ or $n=1$ and $\alpha \in (-1,1)$  \\ \hline
    True & False & $(\mathbb{R},d,\mu_\alpha)$ with $\alpha \ge 1$ \\ \hline
    False & True & $(\mathbb{R}^n,d,\mu_\alpha)$ with $\alpha \le -n$  \\  \hline
    False & False & $(\mathbb{H}^n,d_H,\mu)$ \\   \hline
    \end{tabular}
     \caption{Examples of spaces in relation to the properties \ref{doub-inf} and \ref{poin-inf}}
     \label{table:example}
\end{table}


\chapter{Markov kernel, Semigroup and Dirichlet forms} \label{ch-markov}
In this chapter, we consider Markov chains on metric measure space $(M,d,\mu)$. Let $\mathcal{B}$ denote the Borel $\sigma$-algebra on $(M,d)$.
Our work concerns long term behavior of a natural family of Markov chains on the state space $M$.
We will recall some standard definitions and facts about discrete time Markov chains.

 A \emph{Markov transition  function} is a map $\mathcal{P}:M \times \mathcal{B}:[0,\infty)$ such that $x \mapsto \mathcal(x,A)$ is $\mathcal{B}$-measurable function on $M$ for all $A \in \mathcal{B}$
 and $A \mapsto \mathcal{P}(x,A)$ is a probability measure on $(M,\mathcal{B})$ for all $x \in M$. A Markov transition  function $\mathcal{P}$ on $(M,\mathcal{B})$ is  $\mu$-\emph{symmetric} if
 \begin{equation} \label{e-symmetry} \int_M \int_M u_1(x) u_2(y) \mathcal{P}(x,dy) \mu(dx) = \int_M \int_M u_1(x) u_2(y) \mathcal{P}(x,dy) \mu(dx) \end{equation} for all measurable functions $u_1,u_2:M \to [0,\infty)$.
\begin{remark}
For the rest of this work, we assume that the our Markov transition function is $\mu$-symmetric with respect to some measure $\mu$.
\end{remark}
 Associated with a $\mu$-symmetric Markov transition  function $\mathcal{P}$ is a \emph{Markov operator} $P$, which is a linear operator defined by
 \begin{equation}  \label{e-markop}
  Pf(x) = \int_M f(y) \mathcal{P}(x,dy)
 \end{equation}
on the set of bounded measurable functions. The operator $P$ extends as a contraction operator on $L^p(M)=L^p(M,\mu)$ for all $p \in [1,\infty]$.
With a slight abuse of notation, we denote this extension again by $P:L^p(M) \to L^p(M)$ for each $1 \le p \le \infty$.
Moreover $P$ is positivity preserving, \emph{i.e.} if $f \ge 0$ then $Pf \ge 0$.

The $n$-th iteration $P^n$ of the operator $P$ is just the operator associated with kernel $\mathcal{P}^n$ defined inductively by
\[ \mathcal{P}^n(x,A) := \int_M \mathcal{P}^{n-1}(z,A)\mathcal{P}(x,dz) \] for all $x \in M$, for all measurable sets $A \in \mathcal{B}$ and $\mathcal{P}^1:= \mathcal{P}$.
We now have the \emph{Markov semigroup} of linear operators $\left( P^n \right)_{n \in \mathbb{N}_0}$ where $P^0$ is the identity operator on $L^2(M)$.
The Chapman-Kolmogorov equation is given by
\begin{equation} \label{e-ck}
 \mathcal{P}^{m+n}(x,A)= \int_M \mathcal{P}^n(z,A) \mathcal{P}^m(x, dz)
\end{equation}
for all $A \in \mathcal{B}$ and for all $m,n \in \N^*$.
By Fubini's theorem, \eqref{e-ck}  implies the semigroup property $P^{m+n} f= P^m P^n f$ for all $m,n\in \mathbb{N}_0$ and $f \in L^1(M)$.

The operator $\Delta:=I-P$ is the \emph{Laplacian} which generates the \emph{Dirichlet form}
\[
 \mathcal{E}(f,f) = \langle f, \Delta f\rangle_{L^2(M)} = \frac{1}{2} \int_M \int_M \left(f(x) - f(y) \right)^2 \mathcal{P}(x,dy)\, \mu(dx).
\]
on $L^2(M)$ with full domain $\mathcal{D}(\mathcal{E}) = L^2(M)$.

For every Markov transition function $\mathcal{P}$ on $(M,\mathcal{B})$ there exists a \emph{Markov chain} $(X_n, \mathbb{P}_x)_{n \in \mathbb{N}_0, x \in M}$ on
some path space $(\Omega, \mathcal{F})$ such that
\[
 \mathcal{P}(x,A) = \mathbb{P}_x (X_1 \in A).
\]
(one can always choose the canonical path space $\Omega= M^{\otimes \mathbb{N}_0}, \mathcal{F}= \mathcal{B}^{\otimes \mathbb{N}_0}$ and
$X_n(\omega)= \omega_n$ for $\omega=(\omega_0,\omega_1,\ldots)$.)
The transition function $\mathcal{P}^n$ is then given by $\mathcal{P}^n(x,A)= \mathbb{P}_x (X_n \in A)$ and the operator $P^n$ by $P^nf(x) = \mathbb{E}_x f(X_n)$.
The $\mu$-symmetry of $\mathcal{P}$ is equivalent to the \emph{$\mu$-reversibility} of the Markov chain:
\[
 \mathbb{P}_\mu (X_0 \in A , X_1 \in B) = \mathbb{P}_\mu(X_1 \in A , X_0 \in B)
\]
where $\mathbb{P}_\mu$ is a measure (not necessarily a probability measure) defined by $\mathbb{P}_\mu(\cdot):= \int_M \mathbb{P}_x(\cdot) \mu(dx)$.

 If $\mathcal{P}(x,\cdot) \ll \mu$ for all $x \in M$, we denote its kernel by $p:M \times M \to [0,\infty)$, that is \[
                                                                                                                      \mathcal{P}(x,A) = \int_A p(x,y) \mu(dy)
                                                                                                                     \]
for all $x \in M$ and for all $A \in \mathcal{B}$.
 The kernel $p$ is called a \emph{Markov kernel} with respect to $\mu$. The kernel $p(x,\cdot)$ is the Radon-Nikodym derivative of $\mathcal{P}(x,\cdot)$ with respect to $\mu$, that is
 $\mathcal{P}(x,A) = \int_A p(x,y) \mu(dy)$ for all $x \in M$ and all $A \in \mathcal{B}$. The $\mu$-symmetry of $\mathcal{P}$ implies
 symmetry of kernel, that is $p(x,y)=p(y,x)$ for all $\mu \times \mu$  almost every $(x,y) \in M \times M$. By definition, we have
 $p(x,\cdot) \in L^1(M,\mu)$ for all $x \in M$. However, we  assume that $p(x, \cdot) \in L^\infty(M,\mu)$ for all $x \in M$. Under the assumption $p(x,\cdot) \in L^1 \cap L^\infty$, we define iteratively
 \begin{equation} \label{e-kernel}
  p_{k+1}(x,y) :=  \left[ P p_k(x,\cdot)\right](y) = \int_M p_{k}(x,z) p_1(y,z ) \mu(dz)
 \end{equation}
where $p_1 :=p$ and $k \in \N^*$.   The function $p_k$ for $k \in \N^*$ is called the \emph{heat kernel}.
We will show some basic properties of heat kernel defined in \eqref{e-kernel}.
\begin{lemma}\label{l-kernel}
 Let $(M,d,\mu)$ be a metric measure space and let $\mathcal{P}$ be a $\mu$-symmetric Markov transition function satisfying
 $\mathcal P(x,\cdot) \ll \mu$ for all $x \in M$. Let $p_1(x,\cdot) = \frac{d \mathcal P(x ,\cdot)}{d \mu}$ denote the corresponding Markov kernel.
 Assume further that $p_1(x,\cdot) \in L^\infty(M,\mu)$ for all $x \in M$. The the kernel $p_k$ defined in \eqref{e-kernel} satisfies
 \begin{enumerate}
  \item[(a)] $p_k(x,\cdot) =   \frac{d \mathcal P^k(x ,\cdot)}{d \mu}$ for all $k \in \N^*$. That is $\mathcal{P}^k(x,A) = \int_A p_k(x,z) \mu(dz)$ for all $x \in M$, for all $k \in \N^*$ and for all $A \in \mathcal{B}$.
  \item[(b)] $p_k(x,y) = p_k(y,x) \in [0,\infty)$ for all $x,y \in M$ and for all $k \ge 2$.
  \item[(c)] $p_{k+l}(x,y)= P^k\left( p_l(x,\cdot)\right)(y)$ for all $x,y \in M$ and for all $k,l \in \N^*$.
 \end{enumerate}
\end{lemma}
\begin{proof}
 Since $p_1(x,\cdot) \ge 0$ $\mu$-almost everywhere for all $x \in M$, by induction we have that $p_k(x,y) \in [0, +\infty]$ for all $x,y \in M$ and for all $k \ge 2$.
Therefore by induction on $k$, we have
 \begin{align*}
  \int_M p_{k+1} (x,y) \, dy &= \int_M \int_M p_k(x,z) p_1(y,z) \, dz \, dy = \int_M p_k(x,z) \int_M p_1(z,y) \, dy \, dz \\
  &= \int_M p_k(x,z)  \, dz = 1.
 \end{align*}
In the first line above we used Fubini's theorem and that $p_1(y,z)=p_1(z,y)$ $\mu\times \mu$-almost everywhere.
Since $\norm{p_k(x,\cdot)}_1=1$ for all $x \in M$ and for all $k \in \N^*$, we have
\[
 p_{k+1}(x,y) = \norm{ p_{k}(x,\cdot) p_1(y,\cdot )}_1 \le \norm{ p_{k}(x,\cdot) }_1 \norm{p_1(y,\cdot )}_\infty \le \norm{p_1(y,\cdot )}_\infty < \infty
\]
for all $k \in \N^*$ and for all $x,y \in M$.

First we show (b) by induction. The result is obvious for $k=2$. If $k \ge 2$, we have
\begin{align*}
 p_{k+1}(x,y) &= \int_{M} p_k(x,z) p_1(y,z) \, dz =\int_{M} p_k(z,x) p_1(y,z) \, dz \\
 & = \int_M \int_M p_{k-1}(z,w) p_1(x,w) p _1(y,z) \, dw \, dz.
\end{align*}
In the first line above, we used the induction hypothesis. By the above formula for $p_{k+1}(x,y)$ it is clear that
$p_{k+1}(x,y)= p_{k+1}(y,x)$ for all $x,y \in M$.

Now we verify (a) by induction. For $k \in \N^*$, we have
\begin{align*}
 \mathcal{P}^{k+1}(x,A) &= \int_M \mathcal{P}^k(z,A) \mathcal P(x,dz) = \int_M \left( \int_A p_k(z,w) \,dw \right) p_1(x,z) \,dz\\
 &=  \int_A  \int_M p_k(z,w) p_1(x,z) \,dz  \,dw = \int_A  \int_M p_k(w,z) p_1(x,z) \,dz  \,dw \\
 &=  \int_A   p_{k+1}(w,x)   \,dw =  \int_A   p_{k+1}(x,w)   \,dw
\end{align*}
for all $A \in \mathcal{B}$.
In the first line above, we used induction hypothesis, the second line follows from Fubini's theorem, (b) and the $\mu\times \mu$-a.e.\ symmetry of $p_1$. The last line again follows from (b).

By definition of $p_k$ \eqref{e-kernel}, we have
\[
 p_{k+1}(x,y) = P \left( p_{k}(x,\cdot) \right) (y)
\]
for all $x ,y \in M$. Therefore (c) follows from repeated application of the above equality.
\end{proof}

\begin{remark}
In light of (iii) above, we may alternatively define $p_k(x, \cdot)$ as the Radon-Nikodym derivative $\frac{d \mathcal{P}^k}{d \mu}$.
However this alternate definition for $p_k(x,y)$ makes sense only for $\mu$-almost every $y \in M$ (for a fixed value of $x$).
Nevertheless, since $p_1(y,\cdot) \in L^\infty$ and $p_{k-1}(x,\cdot) \in L^1$, it is clear that for all $k \ge 2$, the function $(x,y) \mapsto p_k(x,y)$ defined in \eqref{e-kernel}
is well-defined for all $x \in M$ and for all $y \in M$.
Hence for $k \ge 2$, $p_k:M \times M \to \R_{\ge 0}$ is a genuine function (as opposed to $p_k(x,\cdot)$ just being in $L^1$).
For $k\ge 2$, $p_k$ is a genuine function on $M \times M$ but $p_1(x,\cdot) \in L^1 \cap L^\infty$ for all $x \in M$.
\end{remark}
Many questions concerning the long term behavior of the Markov chain can be answered if we know $p_k$. Therefore estimates on $p_k(x,y)$ for all $x,y \in M$ and for all $k \in \N^*$
is of importance. Based on the remarks above on $p_k$, any bound on $p_k(x,\cdot)$ must be understood in the $\mu$-almost everywhere sense for $k =1$ and in a point-wise sense for $k\ge 2$.
The estimates on heat kernel gives both qualitative (e.g. recurrence/transience, Liouville property) and
quantitative (e.g. estimates on Green's function, H\"{o}lder regularity) information on the long term behavior of the Markov chain. See Chapter \ref{ch-apply} for  applications of Gaussian estimates on the heat kernel.
 \begin{example} \label{x-ballwalk}
 Let $(M,d,\mu)$ satisfy \ref{doub-loc} and let $h>0$. Consider the natural \emph{ball walk} with Markov kernel $k$ with respect to $\mu$ defined as $k(x,y)= \frac{\one_{B(x,h)}(y)}{V(x,h)}$.
 The corresponding Markov transition function $\mathcal{K}$ is not necessarily $\mu$-symmetric because $k(x,y) \neq k(y,x)$ in general. Consider the measure $\mu' \ll \mu$ with
 $\frac{d \mu'}{d\mu}(x)= V(x,h)$. The Markov kernel of $\mathcal{K}$ with respect to $\mu'$ is $p(x,y)= \frac{\one_{B(x,h)}(y)}{V(x,h)V(y,h)}$. Hence $\mathcal{K}$ is $\mu'$-symmetric.
 Such ball walks on compact Riemannian manifolds were studied in \cite{LM10}.
 \end{example}
 A Markov chain $(X_n,\mathbb{P}_x)_{n \in \mathbb{N}_0,x \in M}$ is said to be \emph{lazy} if $\inf_{x \in M} \mathbb{P}_x(X_1=x) >0$.
 \begin{example}\label{x-lazy}
 Consider a metric measure space $(M,d,\mu)$ with a $\mu$-symmetric Markov transition function $\mathcal{P}$.
Define the Markov transition function  \[\mathcal{P}_L(x,A):= \frac{1}{2}(\mathcal{P}(x,A)+ \delta_x(A))\] where $\delta_x(A) = \one_A(x)$ denotes the Dirac measure at $x$.
Note that $\mathcal{P}_L$ $\mu$-symmetric and corresponds to a lazy Markov chain. Assume $\mathcal{P}$ has a kernel $p$ with respect to $\mu$.
Then $\mathcal{P}_L$ has a kernel with respect to $\mu$ if and only if $\delta_x \ll \mu$ for all $x \in M$.
If $P$ is the Markov operator corresponding to $\mathcal{P}$, then $P_L=(I+P)/2$ is the Markov operator corresponding to $\mathcal{P}_L$, where $I$ is the identity operator on $L^p(M)$.
Hence the corresponding Laplacian operators
$\Delta$ and $\Delta_L$ are related by $\Delta_L= \Delta/2$.
 \end{example}

Some basic properties of a symmetric Markov kernel are listed without proof in the lemma below.
\begin{lemma}[Folklore] \label{l-mop}
 Let $\mathcal{P}$ denote a $\mu$-symmetric Markov transition function over a metric measure space $(M,d,\mu)$ and let $P$ be the corresponding Markov operator.
 Then $P$ is a contraction on all $L^p(M,\mu)$, that is
 \begin{equation}
\norm{P f}_{p} \le \norm{f}_p  \label{e-contr}
 \end{equation}
 for all $p \in [1,\infty]$ and for all $ f \in L^p(M)$.
 A consequence of \eqref{e-contr} is the inequality
  \begin{equation}
\E(f,f) = \langle(I-P)f,f\rangle \le2 \norm{f}_2^2  \label{e-2contr}
 \end{equation}
 for all $ f \in L^2(M)$.
 Moreover $P$ is self-adjoint on $L^2(M)$, that is
 \begin{equation}
\langle f, Pg \rangle =  \langle Pf, g \rangle \label{e-selfadj}
 \end{equation}
for all $f ,g \in L^2(M,\mu)$ where $\langle f_1,f_2 \rangle = \int_M f_1 f_2 \,d\mu$ denotes the inner product on $L^2(M,\mu)$.
\end{lemma}
We list some elementary properties of a symmetric Markov kernel below.
\begin{lemma}[Folklore]\label{l-mker}
 Let $\mathcal{P}$ denote a $\mu$-symmetric Markov transition function over a metric measure space $(M,d,\mu)$ and let $p$ be the corresponding Markov kernel.
 Then for all $x \in M$, the function
 \begin{equation}
n \mapsto p_{2n}(x,x)  \label{e-noninc}
 \end{equation}
 is non-increasing. Moreover we have
 \begin{equation}
  p_{2n}(x,y) \le p_{2n}(x,x)^{1/2} p_{2n}(y,y)^{1/2} \label{e-kcauchy}
 \end{equation}
 for all $x,y \in M$ and for all $n \in \mathbb{N}^*$.
\end{lemma}
\begin{proof}
 Note that the first claim follows from \eqref{e-contr} by
 \[
p_{2n+2}(x,x) =  \norm{ p_{n+1}(x,.)}_2^2 = \norm{P p_n(x,.)}_2^2 \le  \norm{ p_{n}(x,.)}_2^2 = p_{2n}(x,x)  .
 \]
 For \eqref{e-kcauchy}, we simply use Cauchy-Schwarz inequality to obtain
 \[
  p_{2n}(x,y)= \langle p_n(x,.), p_n(y,.) \rangle \le \norm{p_n(x,.)}_2 \norm{p_n(y,.)}_2 .
 \]

\end{proof}

 \section{Assumptions on the Markov chain}
 We introduce the main assumptions on the Markov chain in the following definition.
\begin{definition} \label{d-compat}
 For $h>0$, a Markov transition function $\mathcal{P}$ on $(M,\mathcal{B})$ is said to be \emph{$(h,h')$-compatible} with $(M,d,\mu)$ if
 \begin{itemize}
  \item[(a)] $\mathcal{P}$ is $\mu$-symmetric.
  \item[(b)]  There exists a  kernel $p_1$ such that $\mathcal{P}(x,A)= \int_A p_1(x,y) \mu(dy)$ for all $x \in M$ and for all $A \in \mathcal{B}$.
  By (a), we have $p_1(x,y)=p_1(y,x)$ for all $\mu \times \mu$-almost every $(x,y) \in M \times M$.
  \item[(c)] There exists reals $c_1,C_1 >0$ and $h' \ge h $ such that
  \begin{equation}\label{e-compat}
  \frac{c_1}{V(x,h)} \one_{B(x,h)}(y) \le p_1(x,y) \le \frac{C_1}{V(x,h')} \one_{B(x,h')}(y)
  \end{equation}
for all $x \in M$ and for $\mu$-almost every $y \in M$.
 \item[(d)] There exists $\alpha>0$ such that
 \begin{equation} \label{e-da}
  p_2(x,y)\ge \alpha p_1(x,y)
 \end{equation}
for all $x \in M$ and for $\mu$-almost every $y \in M$, where $p_2$ is defined by \eqref{e-kernel}.
 \end{itemize}
The corresponding Markov kernel $p_1$ is said to be \emph{$(h,h')$-compatible}  with  $(M,d,\mu)$.
If a Markov transition function $\mathcal{P}$ satisfies (a),(b),(c) above we say that $\mathcal{P}$ (respectively $p_1$) is \emph{weakly $(h,h')$-compatible} with $(M,d,\mu)$.

Similarly, we  say the corresponding Markov operator $P$ is (weakly) $(h,h')$-compatible with $(M,d,\mu)$ if the Markov transition function $\mathcal{P}$ is (weakly) $(h,h')$-compatible with $(M,d,\mu)$.
 \end{definition}

\begin{remark}
\begin{itemize}
\item[(i)]Let $(M,d,\mu)$ satisfy \ref{doub-loc} and $h_1 \ge h_2 >0$. If a Markov kernel $p_1$ is $(h_1,h')$-compatible with $(M,d,\mu)$ then $p_1$ is $(h_2,h')$-compatible with $(M,d,\mu)$.
\item[(ii)]The condition (d) in Definition \ref{d-compat} may seem unnatural, but is important for certain technical reasons.
The proofs on Caccioppoli inequality (Lemma \ref{l-caccio}) and discrete time integrated maximum principle (Proposition \ref{p-imp}) and
relies crucially on laziness of walks. Condition (d) enables us to compare the behavior of a given random walk with its lazy version as presented in Example \ref{x-lazy}.
\item[(iii)] There are several examples for which (d) is satisfied.
For instance, a Markov kernel on weighted graphs satisfying \ref{doub-loc} is weakly $(h,h')$-compatible if and only if it is $(h,h')$-compatible.
Consider a Markov kernel $p$ weakly $(h,h)$-compatible with a length space $(M,d,\mu)$ satisfying \ref{doub-loc}, then $p$ is $(h,h)$ compatible.
\item[(iv)] Lemmas \ref{l-wscompat} and \ref{l-gausscompare} show that the assumption (d) is not restrictive for obtaining Gaussian estimates.
\item[(v)] The condition \eqref{e-compat} is an analog of the uniform ellipticity condition \eqref{e-uei}.
\end{itemize}
\end{remark}
We record some important consequences of Condition (d) in Definition \ref{d-compat}.
\begin{lemma}\label{l-con-d}
 Let $(M,d,\mu)$ be a metric measure space and let $P$ be Markov operator that is $(h,h')$-compatible with $(M,d,\mu)$.
Then the corresponding kernel $p_k$ satisfies
\begin{equation}\label{e-pcomp}
 p_{k+1}(x,y) \ge \alpha p_{k}(x,y)
\end{equation}
for all $x,y \in M$ and for all $k \ge 2$ where $\alpha$ is same as in \eqref{e-da}.
Moreover the operator $(P - (\alpha/2)I)^2$ is positivity preserving, that is if $f:M \to \R$ satisfies $f \ge 0$, then $(P - (\alpha/2)I)^2f \ge 0$.
\end{lemma}
\begin{proof}
 Since $P^k$ is a Markov operator, by \eqref{e-da} and Lemma \ref{l-kernel}(c) we have
 \[
 p_{k+2}(x,y)- \alpha p_{k+1}(x,y)= P^k\left[p_2(x,\cdot)- \alpha p_1(x,\cdot)\right](y)   \ge 0
 \]
for all $k \in \N^*$ and for all $x,y \in M$. This proves \eqref{e-pcomp}.

By \eqref{e-da} and $f \ge 0$, we have
\begin{align*}
 (P - (\alpha/2)I)^2f (x) &= (P^2-\alpha P)f(x) + (\alpha/2)^2 f(x) \\ &\ge  (P^2-\alpha P)f(x) = \int_M f(y) (p_2(x,y)- \alpha p_1(x,y)) \, dy \ge 0
\end{align*}
for all $x \in M$.
\end{proof}

The following lemma shows that a large enough convolution power of a weakly compatible kernel is compatible under some mild conditions.
\begin{lemma}\label{l-wscompat}
Let $(M,d,\mu)$ be a quasi-$b$-geodesic space satisfying \ref{doub-loc} and let $p_1$ be a Markov kernel weakly $(h,h')$-compatible with $(M,d,\mu)$ for some $h >b$.
Then there exists  $k \in \mathbb{N}^*$ for which $p_l$ is $(h,lh')$-compatible with $(M,d,\mu)$ for all $l \in \N^*$ such that $l \ge k$.
\end{lemma}
\begin{proof}
Properties (a) and (b) of Definition \ref{d-compat} follows directly from the weak compatibility of $p_1$.
It only remains to check properties (c) and (d).
Assume that $p_1$ satisfies \eqref{e-compat}. Let $x,y \in M$ with $d(x,y) \le h'$.
By Lemma \ref{l-chain}, there exists even number $k \in \mathbb{N}^*$ such that for all $l \ge k \ge 2$, there exists a $b$-chain $x_0,x_1,\ldots,x_{l}$ with $x_0=x$, $x_{l}=y$. Define $h_1=(h-b)/2$.
By Chapman-Kolmogorov equation
\begin{align}
\nonumber \lefteqn{p_{l}(x,y)} \\
\nonumber & \ge  \int_{B(x_{l-1},h_1)} \ldots  \int_{B(x_1,h_1)}  p(x,y_1) p(y_1,y_2)\ldots p(y_{l-1},y) \,dy_1 dy_2 \ldots dy_{l-1}\\
\nonumber & \ge  \int_{B(x_{l-1},h_1)} \ldots \int_{B(x_1, h_1)} \frac{c_1^{l-1}}{V(x,h) V(y_1,h) \ldots V(y_{l-1},h)} \,dy_1 dy_2 \ldots dy_{l-1} \\
\nonumber & \ge  \int_{B(x_{l-1},h_1)} \ldots \int_{B(x_1, h_1)} \frac{c_1^{l-1} C_{h,2h}^{2-l} }{V(x,h) V(x_1,h) \ldots V(x_{l-1},h)} \,dy_1 dy_2 \ldots dy_{l-1} \\
\label{e-ws0}& \ge  \frac{c_1^{l-1} C_{h,2h}^{2-l} }{V(x,h) }
\end{align}
The third line above follows weakly $(h,h')$-compatible condition \eqref{e-compat} and the fourth line follows from Lemma \ref{l-vloc}.
Combining with the fact that $p$ is weakly $(h,h')$-compatible along with Lemma \ref{l-vloc} gives the following lower bound: For all $l \ge k$ and $l \in \N^*$, there exists $c_{1,l}>0$ such that
\begin{equation} \label{e-ws1}
 \min(p_l(x,y), p_{l+1}(x,y)) \ge \frac{c_{1,l}}{V(x,'h)} \one_{B(x,h')}(y)
\end{equation}
for all $x,y \in M$.
Hence by \eqref{e-ws1} and \eqref{e-compat} we get $p_{l+1} \ge \alpha_l p_1$ for some $\alpha_l >0$.
Since $P$ is positivity preserving, we have
\[
 p_{2l }(x,y) = \left( P^{l-1} p_{l+1}(x,.) \right)(y) \ge \alpha_l \left( P^{l-1} p_1(x,.) \right)(y)=\alpha_l p_l(x,y)
\]
which is condition (d) of Definition \ref{d-compat}.
Note that \eqref{e-ws1} implies that $p_l$ satisfies the lower bound in condition (c) of Definition \ref{d-compat}.

Now we turn to the corresponding upper bound for $p_l$.
Since $P$ is a contraction on $L^\infty$, there exists $C_1>0$ such that $p_m(x,y) \le C_1/V(x,h)$ for all $x,y \in M$ and all $m \in \N^*$.
By triangle inequality  $p_m(x,y)=0$ if $d(x,y) > m h'$ for all $m \in \N^*$ and for all $x,y \in M$. Hence by Lemma \ref{l-vloc} we have the desired conclusion.
\end{proof}
\begin{remark} We now justify the condition $h>b$ in the above lemma. It is to avoid pathological examples of the following kind: Consider a ball walk of Example \ref{x-ballwalk} with $h \le 1/2$ on
 Broken line space $(BL,d,\mu)$ from Example \ref{x-bl}. It is easy to check that such a random walk never leaves a connected component.
 Similarly, the ball walk of Example \ref{x-ballwalk} with $h<1$ on a graph always stays at one point.
\end{remark}

\section{Gaussian estimates}
The main property of a Markov kernel that we are interested in are Gaussian estimates for its iterated kernel $p_n$.

\begin{definition} \label{d-ge}
A $\mu$-symmetric Markov kernel  $p$ on $(M,d,\mu)$ is said to satisfy Gaussian upper bound \ref{gue} if
there exists $C_1,C_2>0$ such that
\begin{equation*}
 \label{gue} \tag*{$(GUE)$} p_n(x,y) \le \frac{C_1}{V(x,\sqrt{n})} \exp \left( - d(x,y)^2/C_2n \right)
\end{equation*}
for all $x,y \in M$ and for all $n \in \mathbb{N}^*$ satisfying $n \ge 2$.

Similarly, a $\mu$-symmetric Markov kernel  $p$ on a metric measure space $(M,d,\mu)$ is said to satisfy Gaussian lower bound \ref{gle} if
there exists $c_1,c_2,c_3>0$ such that
\begin{equation*}
 \label{gle} \tag*{$(GLE)$} p_n(x,y) \ge \frac{c_1}{V(x,\sqrt{n})} \exp \left( - d(x,y)^2/c_2n \right)
\end{equation*}
for all $x,y \in M$ satisfying $d(x,y)\le c_3 n$ and for all $n \in \mathbb{N}^*$ satisfying $n \ge 2$.

A $\mu$-symmetric Markov kernel  $p$ on a metric measure space $(M,d,\mu)$ is said to satisfy two sided Gaussian bound \hypertarget{ge}{$(GE)$}  if it satisfies \ref{gue} and \ref{gle}.
\end{definition}

The condition $d(x,y) \le c_3 n$ in \ref{gle} is needed because $p_n(x,y)$ vanishes for compatible kernels if $d(x,y) \ge c n$ for some constant $c >0$.
In many situations, the above Gaussian estimates are equivalent to the following (a priori weaker) estimates which are easier to prove.
We require the estimates in Definition \ref{d-ge} to hold only for large enough $n$ in the definition below.
\begin{definition}\label{d-ge-inf}
A $\mu$-symmetric Markov kernel  $p$ on $(M,d,\mu)$ is said to satisfy Gaussian upper bound \ref{gue-inf} if
there exists $C_1,C_2,n_0>0$ such that
\begin{equation*}
 \label{gue-inf} \tag*{$(GUE)_\infty$} p_n(x,y) \le \frac{C_1}{V(x,\sqrt{n})} \exp \left( - d(x,y)^2/C_2n \right)
\end{equation*}
for all $x,y \in M$ and for all $n \in \mathbb{N}^*$ such that $n \ge n_0$.

The conditions \hypertarget{gle-inf}{$(GLE)_\infty$} and \hypertarget{ge-inf}{$(GE)_\infty$} are defined analogously.
\end{definition}
Under mild conditions, we show that \hyperlink{ge-inf}{$(GE)_\infty$} implies \hyperlink{ge}{$(GE)$}.
\begin{lemma}\label{l-gaussinf}
 Let $(M,d,\mu)$ be a quasi-$b$-geodesic space satisfying \ref{doub-loc} and
 let $p_1$ be a Markov kernel weakly $(h,h')$-compatible with $(M,d,\mu)$ for some $h >b$. The following hold:
 \begin{itemize}
 \item[(a)] If $p_1$ satisfies \ref{gue-inf}, then $p_1$ satisfies \ref{gue}.
 \item[(b)]  If $p_1$ satisfies \hyperlink{gle-inf}{$(GLE)_\infty$} , then $p_1$ satisfies \ref{gle}.
 \item[(c)]  If $p_1$ satisfies \hyperlink{ge-inf}{$(GE)_\infty$} , then $p_1$ satisfies \hyperlink{ge}{$(GE)$} .
 \end{itemize}
\end{lemma}
\begin{proof}
Note that $p_1$ satisfies \eqref{e-compat}.
\begin{itemize}
\item[(a)] The Gaussian upper estimate for $p_n$ where $n \ge n_0$ follows from \ref{gue-inf}. If $n < n_0$, we simply use that $P$ is a contraction in $L^\infty$ along with \eqref{e-compat} to obtain
    \begin{align*}
    p_n(x,y) &\le \frac{C_1 \one_{B(x,n_0h')}(y)}{V(x,h')} \\
    &\le \frac{C_2}{V(x,\sqrt{n})} \exp \left( - \frac{d(x,y)^2}{C_2 n } \right)
    \end{align*}
    for all $x,y \in M$ and for all $n < n_0$. The first line above follows from triangle inequality, $\norm{P}_{L^\infty \to L^\infty}=1$ and \eqref{e-compat}. The second line follows from Lemma \ref{l-vloc}.
\item[(b)]
The Gaussian lower bounds for $p_n$ where $n \ge n_0$ follows from $(GLE)_\infty$.
Let $h_1= \min(h/2,h-b)$. Using ideas similar to the proof of Lemma \ref{l-wscompat} (see \eqref{e-ws0}), there exists $c_2,c_3,c_4>0$ such that
\begin{align}
\nonumber \lefteqn{p_n(x,y)} \\
 \nonumber & \ge \int_{B(x,h_1)} \ldots  \int_{B(x,h_1)}  p(x,y_1) p(y_1,y_2)\ldots p(y_{n-1},y) \,dy_1 dy_2 \ldots dy_{n-1}\\
\nonumber & \ge  \frac{c_2 c_3^n \one_{B(x,b)}(y)}{V(x,h) } \ge \frac{c_4}{V(x,\sqrt{n})} \exp \left( - \frac{d(x,y)^2}{c_4n} \right)
\end{align}
for all $n < n_0$ and for all $x,y \in M$ such that $d(x,y) \le (b/n_0) n$.
\item[(c)] It is a direct consequence of (a) and (b).
\end{itemize}
\end{proof}

\begin{lemma}\label{l-gausscompare}
 Let $(M,d,\mu)$ be a quasi-$b$-geodesic space satisfying \ref{doub-loc} and let $p$ be a Markov kernel weakly $(h,h')$-compatible with $(M,d,\mu)$ for some $h >b$.
For some $k \in \mathbb{N}^*$, if $p_k$ satisfies $(GE)_\infty$ then  $p$ satisfies $(GE)$.
\end{lemma}
\begin{proof}
By Lemma \ref{l-gaussinf} it suffices to show that $p$ satisfies $(GUE)_\infty$ and $(GLE)_\infty$.

Suppose $p=p_1$ satisfies \eqref{e-compat}.
 For $n \ge k$, there exists $A= kh'>0$ such that
 \begin{equation}\label{e-gc1}
 p_n(x,y) \le \sup_{z \in B(y,A)} p_{k \lfloor n/k \rfloor}(x,z)
  \end{equation} for all $x,y \in M$ and for all $n \in \N^*$ with $n \ge k$. This follows from Chapman-Kolmogorov equation along with the fact that  the support of $p_l(\cdot,y)$ is contained in $B(y,kh')$ for all $l \le k$.
  Since $p_k$ satisfies $(GUE)_\infty$, there exists $C_1,C_2>0$ and $n_0>0$ such that
  \begin{equation}
   p_{ mk} (x,y) \le \frac{C_1}{V(x,\sqrt{m})} \exp \left( -\frac{d(x,y)^2}{C_2 m} \right) \label{e-gc2}
  \end{equation}
for all $x,y \in M$ and for all $m \in \N^*$ satisfying $m \ge n_0$.
By \eqref{e-gc1},\eqref{e-gc2} and   \eqref{e-vd1}, there exists $C_3,C_4>0$ and $n_1>0$ such that
\begin{equation}
 p_n(x,y) \le \frac{C_3}{V(x,\sqrt{n})} \sup_{z \in B(y,A)} \exp \left(- \frac{d(x,z)^2}{C_4 n} \right) \label{e-gc3}
\end{equation}
for all $x,y \in M$ and for all $n \in \N^*$ satisfying $n \ge n_1$.
For every $z \in B(y,A)$, we have
\begin{equation}
 \label{e-gc4} d(x,y)^2 \le (d(x,z)+A)^2 \le 2 (d(x,y)^2 + A^2).
\end{equation}
By \eqref{e-gc3} and \eqref{e-gc4}, we have that $p$ satisfies $(GUE)_\infty$.

It remains to show that $p$ satisfies $(GLE)_\infty$. The proof is similar to above. As in \eqref{e-gc1}, we have the complementary inequality,
\begin{equation}
 \label{e-gc5} p_n(x,y) \ge \inf_{z \in B(y,A)} p_{k \lfloor n/k \rfloor}(x,z)
\end{equation} for all $x,y \in M$ and for all $n \in \N^*$ with $n \ge k$.
Since $p_k$ satisfies $(GLE)_\infty$, there exists $c_1,c_2,c_3,n_2>0$ such that
\begin{equation}
 p_{mk}(x,y) \ge \frac{c_1}{V(x,\sqrt{m})} \exp \left( -\frac{d(x,y)^2}{c_2 m} \right) \label{e-gc6}
\end{equation}
for all $x,y \in M$ and for all $m \in \N^*$ satisfying $m \ge n_2$ and $d(x,y) \le c_3 m$.
By \eqref{e-gc5}, \eqref{e-gc6}, there exists $c_4,c_5>0$ and $n_3>0$ such that
\begin{equation} \label{e-gc7}
 p_n(x,y) \ge \frac{c_1}{V(x,\sqrt{n})} \inf_{z \in B(y,A)}\exp \left( -\frac{d(x,y)^2}{c_4 n} \right)
\end{equation}
for all $x,y \in M$ and for all $n \in \N^*$ satisfying $n \ge n_2$ and $d(x,y) \le c_5 n$.
By interchanging $y$ and $z$ in \eqref{e-gc4} along with \eqref{e-gc7} yields $(GLE)_\infty$ for the kernel $p$.
\end{proof}
We describe two examples that does not fall under the framework given by Definition \ref{d-compat} but nevertheless the methods developed in this work still applies.
\begin{example}[Random walk with jumps supported in an annulus]\label{x-annulus} 

We consider a measured, complete, \emph{length} space $(M,d,\mu)$ satisfying $\operatorname{diam}(M)= + \infty$ and \ref{doub-loc}  . 
Let $P$ be a $\mu$-symmetric Markov operator whose kernel $p(x,y)$ satisfies the following estimate: there exists $C_1>0$ and $h>0,h_1>0,h_2>0$ such that
\begin{equation} \label{e-ann}
C_1^{-1} \frac{\one_{B(x,2h) \setminus B(x,h)} (y)}{V(x,h)} \le p(x,y) \le C_1 \frac{\one_{B(x,h_2) \setminus B(x,h_1)} (y)}{V(x,h)}
\end{equation}
for all $x \in M$ and for $\mu$-almost every $y \in M$.

In this case, it is easy to verify that the density $p_2$ is weakly $(h/5,2h_2)$-compatible with $(M,d,\mu)$.
Note that for all $x \in M$, there exists $z \in M$ such that $d(x,z)= 3h/2$. Note that by Lemma \ref{l-vloc} and \eqref{e-ann}, there exists $C_2 >0$ such that
for all $x,y \in M$ with $d(x,y) \le h/5$
\[
p_2(x,y)  \ge \int_{B(z,h/4)} p_1(x,w) p_1(y,w) \, \mu(dw) \ge \frac{C_2^{-1}}{V(x,h/5)} 
\]
and for all $x,y \in M$ with $d(x,y) \le 2h_2$ we have
\[
p_2(x,y) \le \int_{B(x,2h_2)} p_1(x,w) p_1(y,w) \, \mu(dw) \le \frac{C_2}{V(x,2h_2)}.
\]
Therefore $p_2$ is weakly $(h/5,2h_2)$ compatible with $(M,d,\mu)$.

For example, it is clear that the application to Liouville property will not be affected if we replace the operator $P$ by $P^2$.
If the underlying space satisfies volume doubling and Poincar\'{e} inequality we can use our main results to
obtain Gaussian estimates \hyperlink{ge-inf}{$(GE)_\infty$}  provided $(M,d,\mu)$ satisfies \ref{doub-inf} and \ref{poin-mms}.
To prove the above statement, we simply note by Theorem \ref{t-main0}, Lemma \ref{l-wscompat} and Lemma \ref{l-gausscompare} that $p_2$ satisfies \hyperlink{ge}{$(GE)$} and by a similar argument $p_3$ satifies \hyperlink{ge}{$(GE)$}.
\end{example}
\begin{example} \label{x-twoballs}
 We describe another example similar to Example \ref{x-annulus}. Consider $\R^n$ equipped with Euclidean distance $d$ and Lebesgue measure $\mu$. Let $e$ denote an arbitrary unit vector in $\R^n$.
 Consider the $\mu$-symmetric random walk with the kernel
 \[
  p(x,y)=  \frac{ \one_{B(x+2 e, 1) \cup B(x-2e,1)}(y)}{2V(x,1)}.
 \]
Although $p$ is not compatible with $(\R^n,d,\mu)$, similar to Example \ref{x-annulus} one can check that $(\R^n,d,\mu)$ satisfies that $p_2$ and $p_3$ are $(1/3,9)$-compatible with $(\R^n,d,\mu)$ and that
the kernel $p_k$ satisfies \hyperlink{ge-inf}{$(GE)_\infty$}.
\end{example}

\section{Comparison of Dirichlet forms}
Let $(M,d,\mu)$ be a metric measure space  with a $\mu$-symmetric Markov operator $P$ and corresponding kernel $p$.
Recall that we defined the Dirichlet form $\mathcal{E}(f,g)= \langle f , \Delta g \rangle$ for $f,g \in L^2(M)$.
We define another Dirichlet form $\mathcal{E}_*$ which is the Dirichlet form corresponding to the Markov operator $P^2$, that is
\[
 \mathcal{E}_*(f,g)= \langle f, (I - P^2)g\rangle =\norm{f}_2^2 - \norm{Pf}_2^2.
\]
for all $f , g \in L^2(M)$.
\begin{remark}
Functional inequalities involving the Dirichlet form (for instance Nash, Sobolev, log Sobolev, Poincar\'{e} inequalities) can be transferred to an inequality
concerning the Markov semigroup, which in turn sheds light on asymptotic behavior of Markov chains.
For a continuous time Markov semigroup $(P_t)_{t \ge 0}$ a crucial identity to carry out this is $ \frac{ d \norm{P_t f}_2^2}{dt} =   - 2 \mathcal{E}(P_tf, P_tf)$ (for instance \cite[Theorems 4.2.5 and 6.3.1]{BGL14})
By the above definition, we have a similar identity for discrete time Markov semigroup:
\[
\partial_k \norm{P^kf}_2^2:= \norm{P^{k+1}f}_2^2- \norm{P^kf}_2^2 = - \mathcal{E}_*(P^kf,P^kf).
\]
for all $f \in L^2(M)$. This is the main reason why we sometimes prefer $\E_*$ instead of $\E$.
\end{remark}
The above remark motivates us to compare the Dirichlet forms $\mathcal{E}$ and $\mathcal{E}_*$.
\begin{lemma} \label{l-comp-dir}
Consider a $\mu$-symmetric Markov chain on $(M,d,\mu)$ with Markov operator $P$ and Dirichlet forms $\mathcal{E}$ and $\mathcal{E}_*$ defined as above. We have the following:
 \begin{itemize}
   \item[(a)]  $\mathcal{E}_*(f,f) \le 2 \mathcal{E}(f,f)$ for all $f \in L^2(M)$.
  \item[(b)] Assume further that $P$ has a strongly $(h,h')$-compatible kernel $p$ with respect to $(M,d,\mu)$.
  Then there exists a constant $C >0$ such that $\mathcal{E}(f,f) \le C \mathcal{E}_* (f,f)$ for all $f \in L^2(M)$.
 \end{itemize}
\end{lemma}

\begin{proof}
\begin{itemize}
 \item[(a)]
 Note that
\[
 \langle Pf ,f\rangle \le  \frac{1}{2} \left( \langle Pf,Pf \rangle + \langle f,f\rangle\right) = \frac{1}{2} \left( \langle P^2f, f \rangle + \langle f,f\rangle  \right)
\]
Hence
\[
 \mathcal{E}(f,f)= \langle f,f \rangle_M - \langle Pf, f \rangle_M \ge  \langle f,f \rangle - \frac{1}{2} \left( \langle P^2f, f \rangle + \langle f,f\rangle  \right) = \frac{1}{2} \mathcal{E}_*(f,f)
\]
\item[(b)] The conclusion follows from Property (d) of Definition \ref{d-compat} by observing that
\begin{align}
 \label{e-dirform} \mathcal{E}(f,f)&= \frac{1}{2} \int_M\int_M (f(x) - f(y))^2 p(x,y)\, dx\, dy \\
  \label{e-dirform2} \mathcal{E}_*(f,f)&= \frac{1}{2} \int_M\int_M (f(x) - f(y))^2 p_2(x,y) \,dx \,dy
\end{align}
\end{itemize}
\end{proof}
\begin{remark}
 The inequality $\mathcal{E}(f,f) \le C \mathcal{E}_* (f,f)$ is not true in general.
 Consider nearest neighbor (simple) random walk on a finite bipartite graph. Let $f$ be a function on the graph that assigns +1 to one partition and -1 to other. It is easy to check that $Pf = -f$ and therefore
 $2 \norm{f}_2^2=\mathcal{E}(f,f) \le C \mathcal{E}_* (f,f)=0$ fails.
\end{remark}
\section{Markov chains killed on exiting a ball} \label{s-kill}
To obtain lower bounds on the heat kernel, we consider the corresponding Markov process killed on exiting a ball $B$ (See Chapter \ref{ch-glb}).
Moreover functional inequalities like Nash and Sobolev inequalities that we will encounter are local to balls.
Motivated by these considerations, we introduce Markov chains killed on exiting a ball and their corresponding Markov operator and kernel.
Let $(X_n)_{n \in \N}$  be a Markov chain on $(M,d,\mu)$ driven by a $\mu$-symmetric Markov operator $P$ with kernel $p_1$ with respect to $\mu$.
The corresponding Markov chain $(X^B_n)_{n \in \N}$ that is killed on exiting a ball $B$ has state space $B \cup \set{ \partial_B}$ where $\partial_B$ is
the absorbing \emph{cemetery state}. The Markov chain $(X^B_n)_{n \in \N}$ killed on exiting $B$  is defined as
\[
  X^B_n = \begin{cases}
           X_n & \mbox{if $n < \zeta$} \\
         \partial_B &\mbox{if $n \ge \zeta$}
          \end{cases}
\]
where $\zeta$ is the \emph{lifetime} of the process defined by
\[
 \zeta= \min \Sett{k}{X_k \notin B}.
\]
For the killed Markov chain, we consider  functions $f:B \cup \mathcal{\partial_B} \to \R$ with the `Dirichlet' boundary condition $f(\mathcal{\partial_B})=0$.
Therefore, we can define corresponding quantities like Markov kernel and Markov operator just by restriction to $B$.
Define the \emph{restricted kernel} $p_B: B \times B \to \mathbb{R}$, as a restriction of $p_1$ on $B \times B$.
We endow $B$ with the measure $\mu_B$ which is the restriction of $\mu$ to all Borel subsets of $B$. We denote by $L^2(B)=L^2(B, \mu_B)$.
We define the Markov operator $P_B$ with kernel $p_B$ with respect to $\mu_B$ as
\begin{equation}\label{e-defPB}
 P_Bf(x) := \int_B f(y)p_1(x,y) \,\mu(dy)= \int_B p_B(x,y) f(y) \mu_B(dy).\end{equation}
Define the corresponding Dirichlet forms
\begin{equation}\label{e-defEB} \mathcal{E}^B(f,f):= \langle f , (I-P_B)f\rangle_{L^2(B)},\hspace{8mm} \mathcal{E}_*^B(f,f):= \langle f , (I-P_B^2)f\rangle_{L^2(B)}\end{equation}
for all $f \in L^2(B)$.
Similar to \eqref{e-kernel}, we define the  kernel $p_k^B(x,y)$ iteratively as
\begin{equation}\label{e-defpB}
 p^B_{k+1}(x,y) := \left[P_B p^B_k(x,\cdot) \right](y) = \int_B p^B_k(x,z) p^B_1(y,z) \, \mu(dz)
\end{equation}
for all $k \in \N^*$ and for all $x,y \in B$. It is easy to check that the proof of Lemma \ref{l-kernel} (b),(c) applies to the kernel $p_B$.
As before,  the function $(x,y) \mapsto p^B_k(x,y)$ is well-defined for all $k \ge 2$. Further $p_B(x,\cdot) \in L^1(B)$ for all $x \in M$.
It is easy to see that
\begin{equation}
 \label{e-pbcomp} p_k^B(x,y) \le p_k(x,y)
\end{equation}
for all $x,y \in M$ and for all $k \ge 2$.

The operator $P_B$ is positivity preserving, that is $f \ge 0$ implies $P_B f\ge 0$. However unlike $P$, the operator $P_B$ is not necessarily conservative, that is
$P_B \one \neq \one$ in general. Analogous  to \eqref{e-contr}, we have that $P_B$ is a contraction on all $L^p(B)$ for all $1\le p \le +\infty$.
We also define the corresponding `Dirichlet Laplacian' $\Delta_{P_B}:= I - P_B$.

We will compare Dirichlet forms on balls with Dirichlet forms on $M$ below.
\begin{lemma}
 \label{l-dircomp-ball}
 Let $f \in L^2(B)$ and let $\tilde{f} \in L^2(M)$ denote an extension of $f$ defined by
 \begin{equation}
  \tilde{f} = \begin{cases} f & \mbox{ in $B$} \\ 0 & \mbox{ in $B^\complement$}.\end{cases} \label{e-extend}
 \end{equation}
 Then
 \begin{itemize}
  \item[(a)] $\E^B(f,f) = \E(\tilde{f},\tilde{f})$.
  \item[(b)] $\E^B_*(f,f) \ge \E_*(\tilde f,\tilde f)$.
 \end{itemize}
\end{lemma}
\begin{proof}
 For (a), observe that
 \begin{align*}
  \E^B(f,f) &= \langle f,f \rangle_{L^2(B)} - \langle P_B f,f\rangle_{L^2(B)} \\
 &= \langle \tilde f,\tilde f \rangle_{L^2(M)} - \langle P \tilde f,\tilde f \rangle_{L^2(M)} = \E(\tilde f, \tilde f).
 \end{align*}
For (b), we have
\begin{align*}
  \E^B(f,f) &= \langle f,f \rangle_{L^2(B)} - \langle P_B f,P_B f\rangle_{L^2(B)} \\
 &= \langle \tilde f,\tilde f \rangle_{L^2(M)} - \langle \one_B P \tilde f,\one_B P \tilde f \rangle_{L^2(M)} \\
 & \ge \langle \tilde f,\tilde f \rangle_{L^2(M)} - \langle P \tilde f, P \tilde f \rangle_{L^2(M)} = \E_*(\tilde f, \tilde f).
 \end{align*}
\end{proof}

We warn the reader of the following abuse of notation.
We may consider a function $f \in L^2(B)$ as a function in $L^2(M)$ using the extension given by \eqref{e-extend}.
Alternatively we may consider a function $f \in L^2(M)$ as a function in $L^2(B)$ by the restriction $f \vline_B$.

\chapter{Sobolev-type inequalities}\label{ch-sob}
J. Moser proved parabolic Harnack inequalities for second-order uniformly elliptic divergence form operators in $\R^d$ \cite{Mos64}.
This approach was successfully adapted by numerous authors.
The previous versions of Theorem \ref{t-main0} as given in \cite{Sal92,Del99,Stu96} used Moser's iterative method as
a crucial ingredient.
Along with Poincar\'{e} inequality and volume doubling, Moser's iteration relies on repeated applications of a Sobolev inequality.

We recall the difficulty arising due to Sobolev inequalities mentioned in the introduction.
The Sobolev inequalities in the previous works \cite{Sal92,Del97,Del99,Stu96} are of the form
\begin{equation} \label{e.sob}
 \norm{f}^2_{2 \delta/ (\delta-2)} \le \frac{C r^2}{V_\mu(x,r)^{2/\delta}} \left( \mathcal{E}(f,f) + r^{-2} \norm{f}_2^2 \right)
\end{equation}
for all `nice' functions $f$ supported in $B(x,r)$. However \eqref{e.sob} along with \eqref{e-2contr}
implies that  $L^2(B(x,r)) \subseteq L^{2 \delta/ (\delta-2)} (B(x,r))$ for all balls $B(x,r)$ which can happen only
 if the space is discrete. Hence for discrete time Markov chains on continuous spaces the Sobolev inequality \eqref{e.sob} fails to hold.
 In this chapter, we prove a weaker form of the above Sobolev inequality (see \eqref{e-Sob}) and study its properties.
 In the next two sections, we will use the Sobolev inequality \eqref{e-Sob} to run the Moser's iterative method and obtain
 elliptic Harnack inequality and Gaussian upper bounds.

 We adapt the approach of \cite{Sal92} to obtain a Sobolev inequality using \ref{doub-inf} and \ref{poin-inf}. The main result of this chapter is the following Sobolev inequality.
 \begin{theorem} \label{t-Sob}
Let $(M,d,\mu)$ be a quasi-$b$-geodesic metric measure space satisfying  \ref{doub-loc}, \ref{doub-inf} and Poincar\'{e} inequality at scale $h$ \ref{poin-mms}.
Suppose that a Markov operator $P$ has a kernel $p$ that is $(h,h')$-compatible with respect to $\mu$.
Let $P_B$ and $\E^B$ denote the corresponding Markov operator and Dirichlet form restricted to a ball $B \subset M$.
Then there exists $\delta > 2$ and $C_S>0$ such that for all  $r >0$, for all $x \in M$, and for all  $f \in L^2(B)$, we have
\begin{equation}
 \label{e-Sob} \norm{P_B f}_{2 \delta/(\delta-2)}^2 \le \frac{C_S r^2}{V(x,r)^{2/\delta}} \left( \E^B(f,f) + r^{-2} \norm{f}_2^2 \right)
\end{equation}
where $B=B(x,r)$.
 \end{theorem}
\begin{remark}
 Since $P_B$ is a contraction, note that \eqref{e.sob} implies \eqref{e-Sob}.
 Since we rely on the weaker Sobolev inequality \eqref{e-Sob}, our methods give an unified approach to Gaussian bounds for graphs and continuous spaces.
 However we will encounter new difficulties due to  \eqref{e-Sob}.
\end{remark}
Let $s >0$ and $f \in L^1_{\operatorname{loc}}(M,\mu)$. We define $f_s$ as
\begin{equation}
 \label{e-def-fs} f_s(x)  := f_{B(x,s)} = \frac{1}{V(x,s)} \int_{B(x,s)} f(x) \,\mu(dx).
\end{equation}
\section{Pseudo-Poincar\'{e} and Nash inequalities}
As in \cite[Lemma 2.4]{Sal92}, we need a pseudo-Poincar\'{e} inequality.
\begin{lemma}[Pseudo-Poincar\'{e} inequality] \label{l-pp}
 Under the hypotheses of Theorem \ref{t-Sob}, there exists $C_{0} >0$ and $s_{0}>0$  such that
 \begin{equation}
  \label{e-pp} \norm{f-f_s}_2^2 \le C_{0} s^2 \E(f,f)
 \end{equation}
for all $f \in L^2(M)$ and for all $s \ge s_0$.
\end{lemma}
\begin{proof}
 Let $(X,d,m)$ be a $2s$-net of $(M,d,\mu)$ as given in Definition \ref{d-net}. By Proposition \ref{p-net}(a), the collection of balls $J= \Set{ B(x,2s) }{x \in X }$ cover $M$.
 Therefore
 \begin{align}
  \nonumber \norm{f-f_s}_2^2 &\le \sum_{2B \in J} \int_{2B} \abs{f(x)-f_s(x)}^2 \mu(dx) \\
  &\le 2 \left( \sum_{2 B \in J} \int_{2B} \abs{f(x)-f_{3B}}^2 \mu(dx) + \int_{2B} \abs{f_s(x)-f_{3B}}^2 \mu(dx) \right). \label{e-pp1}
 \end{align}
For the first term, we use \ref{poin-mms}, to obtain $C_1,C_2,s_1 >0$ such that
\begin{equation}
\label{e-pp2} \int_{2B} \abs{f(x) -f_{3B}}^2 \mu(dx) \le \int_{3B} \abs{f(x)-f_{3B}}^2 \mu(dx)  \le  C_1 s^2 \int_{3C_2 B} \abs{\nabla f}^2_h(x) \mu(dx)
\end{equation}
for all $s \ge s_0$ and for all $f \in L^2(M)$.
For the second term in \eqref{e-pp1}, we use Jensen's inequality to obtain
\begin{align}
 \nonumber \int_{2B} \abs{f_s(x)-f_{3B}}^2 \mu(dx) &\le  \int_{2B} \frac{1}{V(x,s)} \int_{B(x,s)} \abs{f(y)-f_{3B}}^2 \mu(dy) \, \mu(dx)  \\
 \label{e-pp3} &\le  \left( \int_{2B} V(x,s)^{-1} \mu(dx) \right) \cdot \left( \int_{3B} \abs{f(y)-f_{3B}}^2 \mu(dy) \right)
\end{align}
for all $f \in L^2(M)$ and for all $2B \in J$. By  \eqref{e-vd1}, there exists $C_3>0$ such that
\begin{equation}
\label{e-pp4} \int_{2B} \frac{\mu(dx)}{V(x,s)} = \frac{1}{\mu(2B)}  \int_{2B} \frac{\mu(2B) \mu(dx)}{V(x,s)} \le \frac{1}{\mu(2B)} \int_{2B}  \frac{V(x,4s) \mu(dx)}{V(x,s)} \le C_3
\end{equation}
for all $s \ge s_0$ and for all $2B \in J$. By \eqref{e-pp1},\eqref{e-pp2},\eqref{e-pp3} and \eqref{e-pp4}, there exists $C_0>0$ such that
\begin{equation} \label{e-pp5}
  \norm{f-f_s}_2^2 \le C_1 (1 + C_3) s^2  \sum_{2B \in J}  \int_{3C_2 B} \abs{\nabla f}^2_h(x) \mu(dx) \le C_0 s^2 \E(f,f)
\end{equation}
for all $f \in L^2(M)$. The last inequality in \eqref{e-pp5} follows from Proposition \ref{p-net}(h), \eqref{e-dirform} along with \eqref{e-compat}.
\end{proof}
The following lemma is a consequence of doubling hypothesis.
\begin{lemma} \label{l-nd}
 Let $(M,d,\mu)$ be a  measure space satisfying  \ref{doub-loc} and \ref{doub-inf}.
 Then for all $b>0$, there exists $C_b>0$, $\delta>2$ such that
 \begin{equation}\label{e-nd}
\norm{f_s}_2^2 \le \frac{C_b}{V(x,r)}\left( \frac{r}{s}\right)^\delta \norm{f}_1^2
 \end{equation}

 for all $f \in L^1(M)$ is supported in $B=B(x,r)$ and for all $b \le s \le r$
\end{lemma}
\begin{proof}
 By H\"{o}lder inequality, we have
 \begin{equation}
  \label{e-nd1} \norm{f_s}_2^2 \le \norm{f_s}_\infty \norm{f_s}_1.
 \end{equation}
Since $f$ is supported in $B(x,r)$ and $s \le r$ we have
\begin{equation*}
 \norm{f_s}_\infty \le \norm{f}_1  \sup_{y \in B(x,r+s)} \frac{1}{V(y,s)} \le \frac{\norm{f}_1}{V(x,r)}  \sup_{y \in B(x,r+s)} \frac{V(y,3r)}{V(y,s)}.
\end{equation*}
By \eqref{e-vd1}, there exists $C_1 >0$ and $\delta >2$ such that
\begin{equation}\label{e-nd2}
  \norm{f_s}_\infty \le \frac{C_1}{V(x,r)} \left( \frac{r}{s}\right)^\delta \norm{f}_1
\end{equation}
for all $b \le s \le r$ and for all $f \in L^1$ supported in $B(x,r)$.

Further there exists $C_2>0$ such that
\begin{align}
 \norm{f_s}_1 &= \int_{B(x,r+s)} \abs{f_s(y)} \mu(dy) \le \int_{B(x,r+s)}\frac{1}{V(y,s)}\int_{B(y,s)} \abs{f(z)} \, \mu(dz) \, \mu(dy) \nonumber \\
 & \le \int_{B(x,r+s)} \abs{f(z)}\int_{B(z,s)}\frac{1}{V(y,s)}\, \mu(dy)  \, \mu(dz) \nonumber\\
 \label{e-nd3} & \le C_2 \int_{B(x,r+s)} \abs{f(z)}\int_{B(z,s)}\frac{1}{V(z,s)}\, \mu(dy)  \, \mu(dz)  = C_2 \norm{f}_1
\end{align}
for all $b \le s \le r$ and for all $f \in L^1$ supported in $B(x,r)$.
The second line follows from Fubini's theorem and \eqref{e-nd3} above follows from \eqref{e-vd2}.
The desired conclusion \eqref{e-nd} follows from \eqref{e-nd1},\eqref{e-nd2} and \eqref{e-nd3}.
\end{proof}

Next, we show a Nash inequality using the  pseudo-Poincar\'{e} inequality and doubling hypotheses by  adapting the approach of \cite[Theorem 2.1]{Sal92}.
\begin{prop}
 \label{p-nash}
Let $(M,d,\mu)$ be a quasi-$b$-geodesic metric measure space satisfying  \ref{doub-loc}, \ref{doub-inf} and Poincar\'{e} inequality at scale $h$ \ref{poin-mms}.
Suppose that a Markov operator $P$ has a kernel $p$ that is $(h,h')$-compatible with respect to $\mu$. Let $\E$ denote the Dirichlet form corresponding to $P$.
Then there exists $\delta > 2$, $C_N>0$ such that
\begin{equation}
 \label{e-nash} \norm{P f}_{2}^{2 + (4/\delta)} \le \frac{C_N r^2}{V(x,r)^{2/\delta}} \left( \E(f,f) + r^{-2} \norm{f}_2^2 \right) \norm{f}_1^{4/\delta}
\end{equation}
 for all  $r>0$, for all $x \in M$, and for all  $f \in L^2(M)$ with $f$  supported in $B(x,r)$.

\end{prop}
\begin{proof} We start with an observation that \eqref{e-nash} follows directly for small values of $r$. Let $r_0 >0$ be an arbitrary constant.
If $r \le  r_0$, by \eqref{e-compat} and \eqref{e-volc}, there exists $C_1,C_2>0$ such that for all functions $f \in L^1(M)$ supported in $B(x,r)$, we have
\begin{align}
 \nonumber \norm{Pf}_\infty &\le \norm{f}_1 \sup_{y \in B(x,r+h')} \frac{C_1}{V(y,h')} \le \norm{f}_1 \sup_{y \in B(x,r_0+h')} \frac{C_1}{V(y,h')}\\
 &\le \frac{C_1}{V(x,r_0)} \norm{f}_1 \sup_{y \in B(x,r_0+h')} \frac{V(y,2r_0+h')}{V(y,h')} \le \frac{C_2}{V(x,r)} \norm{f}_1. \label{e-ns3}
\end{align}
By H\"{o}lder inequality along with \eqref{e-ns3} and \eqref{e-contr}, we have $C_3>0$ such that
\begin{equation} \label{e-ns4}
 \norm{Pf}_2 \le \norm{Pf}_\infty^{1/2} \norm{Pf}_1^{1/2} \le \frac{C_3}{V(x,r)^{1/2}} \norm{f}_1
\end{equation}
for all function $f \in L^2(M)$ supported in $B(x,r)$ with $r \le r_0$.
By \eqref{e-ns4} and \eqref{e-contr} and by the choice $C_N \ge C_3^{4/\delta}$, it suffices to show \eqref{e-nash} for the case $r > r_0$.

Note that
 \begin{equation}
  \label{e-ns1}\norm{P f}_{2} \le \norm{Pf- (Pf)_s}_{2} + \norm{(Pf)_s}_2.
 \end{equation}
We use pseudo-Poincar\'{e} inequality (Lemma \ref{l-pp}) to bound the first term and use the $(h,h')$-compatibility of $P$ along with doubling hypotheses to bound the second term.
To obtain \eqref{e-nash}, we minimize the bound on right hand side of \eqref{e-ns1} by varying $s$.

By Lemma \ref{e-pp}, there exists $C_0 \ge 1 $ and $r_0>0$ such that
\begin{equation}
  \norm{Pf- (Pf)_s}_{2} \le C_0 s \sqrt{\E(Pf,Pf)}\label{e-ns2}
\end{equation}
for all $f \in L^2(M)$ and for all $s \ge r_0$.

By \eqref{e-nd} and \eqref{e-contr}, there exists $C_4>0$ and $\delta>2$ such that
\begin{equation}
 \label{e-ns5} \norm{(Pf)_s}_2 \le \frac{C_{4}}{V(x,r)^{1/2}} \left(\frac{r}{s} \right)^{\delta/2} \norm{f}_1
\end{equation}
for all $f \in L^2(M)$ supported in $B(x,r)$ and for all $r_0 \le s \le r$.
Combining \eqref{e-ns1},\eqref{e-ns2}, \eqref{e-ns5}, we obtain
\begin{equation}
 \label{e-ns6} \norm{Pf}_2 \le C_0 s \left( \sqrt{\E(Pf,Pf)} + r^{-1} \norm{Pf}_2 \right) + \frac{C_{4}}{V(x,r)^{1/2}} \left(\frac{r}{s} \right)^{\delta/2} \norm{f}_1
\end{equation}
for all $f \in L^2(M)$ supported in $B(x,r)$ and for all $s \ge r_0$ and for all $r \ge r_0$.
In order to minimize the right side of \eqref{e-ns6}, the choice of $s$ (up to a constant factor) is
\begin{equation} \label{e-ns7}
 s_1(f) := \left( \frac{\norm{f}^2_1 r^{\delta}}{(\E(Pf,Pf) + r^{-2} \norm{Pf}_2^2)V(x,r) }  \right)^{1/(\delta+2)}.
\end{equation}
However, we want to choose $s \ge r_0$ in \eqref{e-ns6}. We will do that by showing that $s_1(f)$ is bounded below.
For all $r \ge r_0$, by \eqref{e-2contr} we have
\begin{equation}
 \E(Pf,Pf) + r^{-2} \norm{Pf}_2^2 \le (2 + r_0^{-2}) \norm{Pf}_2^2 \le (2 + r_0^{-2}) \norm{Pf}_\infty \norm{f}_1 \label{e-ns8}
\end{equation}
for all $f \in L^2(M)$. Since $f$ is supported in $B(x,r)$,  there exists $C_5, C_6 >0$
\begin{align}
 \nonumber \norm{Pf}_\infty &\le  C_5 \norm{f}_1 \sup_{y \in B(x,r+h')} \frac{1}{V(y,h')} \le \frac{C_5}{V(x,r)} \norm{f}_1  \sup_{y \in B(x,r+h')} \frac{V(y,2r+h')}{V(y,h')} \\
\label{e-ns9} & \le \frac{C_6}{V(x,r)} r^\delta \norm{f}_1
\end{align}
for all $f \in L^2$ supported in $B(x,r)$ with $r \ge r_0$. The first line above follows from \eqref{e-compat} and the second line follows from \eqref{e-vd1} and $r \ge r_0$.
By \eqref{e-ns8} and \eqref{e-ns9}, there exists $c_1>0$ such that
\begin{equation}
\label{e-ns10}  s_1(f)= \left( \frac{\norm{f}^2_1 r^{\delta}}{(\E(Pf,Pf) + r^{-1} \norm{Pf}_2)V(x,r) }  \right)^{1/(\delta+2)} \ge c_1
\end{equation}
for all $x \in M$,  for all $r \ge r_0$ and for all $f \in L^2(M)$ supported in $B(x,r)$. By plugging in $s=(r_0/c_1) s_1(f)$ in \eqref{e-ns6}, there exists $C_N \ge C_3^{4/\delta}$ such that
\begin{equation}
 \label{e-ns11} \norm{P f}_{2}^{2 + (4/\delta)} \le \frac{C_N r^2}{V(x,r)^{2/\delta}} \left( \E(Pf,Pf) + r^{-2} \norm{Pf}_2^2 \right) \norm{f}_1^{4/\delta}
\end{equation}
for all $x \in M$,  for all $r \ge r_0$ and for all $f \in L^2(M)$ supported in $B(x,r)$.
 By \eqref{e-contr}, we have
\begin{equation} \label{e-ns12}
\sqrt{\E(Pf,Pf)} = \norm{P(I-P)^{1/2}f}_2 \le \norm{(I-P)^{1/2}f}_2 = \sqrt{\E(f,f)}
\end{equation}
for all $f \in L^2(M)$.
By \eqref{e-ns10},\eqref{e-ns11} and \eqref{e-contr}, we obtain the desired Nash inequality \eqref{e-nash}.
\end{proof}
Before we proceed, we restate the above Nash inequality for functions defined on balls.
\begin{corollary}
 \label{c-nash}
 Let $(M,d,\mu)$ be a quasi-$b$-geodesic metric measure space satisfying  \ref{doub-loc}, \ref{doub-inf} and Poincar\'{e} inequality at scale $h$ \ref{poin-mms}.
Suppose that a Markov operator $P$ has a kernel $p$ that is $(h,h')$-compatible with respect to $\mu$. Let $P_B$ and $\E^B$ denote the corresponding Markov operator and Dirichlet form restricted to a ball $B \subset M$.
Then there exists $\delta > 2$, $C_N>0$ such that
\begin{equation}
 \label{e-nashB} \norm{P_B f}_{2}^{2 + (4/\delta)} \le \frac{C_N r^2}{V(x,r)^{2/\delta}} \left( \E^B(f,f) + r^{-2} \norm{f}_2^2 \right) \norm{f}_1^{4/\delta}
\end{equation}
 for all  $r>0$, for all $x \in M$, and for all  $f \in L^2(M)$ with $f$  supported in $B(x,r)$.
\end{corollary}
\begin{proof}
 We define $\tilde f \in L^2(M)$ as in \eqref{e-extend}. Since $P \tilde f = P_B f$ on $B$, we have $\norm{P_B f}_2 \le \norm{P \tilde f}_2$.
 Combining this observation along with $\norm{f}_p = \norm{\tilde f}_p$, Lemma \ref{l-dircomp-ball}(a) and Proposition \eqref{p-nash} yields \eqref{e-nash}.
\end{proof}
\begin{remark}
 It is easy to prove Nash inequality \eqref{e-nashB} using Sobolev inequality \eqref{e-Sob} just by an application of  H\"{o}lder inequality
 \[
 \norm{P_Bf}_2 \le \norm{P_B f}_{2\delta/(\delta-2)}^{\delta/(\delta+2)} \norm{P_Bf}_1^{2/(\delta+2)} \le \norm{P_B f}_{2\delta/(\delta-2)}^{\delta/(\delta+2)} \norm{f}_1^{2/(\delta+2)}
 \]
 along with the fact that $P_B$ is a contraction on $L^1(B)$. However proving \eqref{e-Sob} using \eqref{e-nashB} is harder.
 There is a direct and elementary approach using slicing of functions developed in \cite{BCLS95}. Their approach was used by Delmotte in the setting of graphs \cite[Theorem 4.4]{Del97}
 to prove a Sobolev inequality.
 However those slicing techniques not so seem to apply directly for proving \eqref{e-Sob},
 since the (sub-Markov) operator $P_B$ does not commute with the slicing maps $f \mapsto (f-s)_+ \wedge t$.
 It is an interesting open problem to make this approach work for our Sobolev-type inequalities.
 \end{remark}
 \section{Ultracontractivity estimate on balls}
In light of the above remark, we adapt a different approach based on Hardy-Littlewood-Sobolev theory for discrete time Markov semigroups as developed in \cite[Theorems 5 and 6]{CS90}.
 Our approach is to obtain an upper bound for $\norm{P_B^k}_{1 \to \infty}$ using \eqref{e-nashB} which in turn is
used to prove the Sobolev inequality \eqref{e-Sob}.

\begin{lemma} \label{l-ultB}
  Let $(M,d,\mu)$ be a quasi-$b$-geodesic metric measure space satisfying  \ref{doub-loc}, \ref{doub-inf}.
Suppose that a Markov operator $P$ has a kernel $p$ that is $(h,h')$-compatible with respect to $\mu$. Let $P_B$ and $\E^B$ denote the corresponding Markov operator and Dirichlet form restricted to a ball $B \subset M$.
Further assume that the operators $P_B$ satisfy the Nash inequality \eqref{e-nashB} with constant $\delta>2$. There exists $C_u >0$ such that
\begin{equation} \label{e-ultB}
\norm{P_B^k}_{1 \to \infty} \le \frac{C_u (1+r^2)^{\delta/2} }{V(x,r)}  \frac{(1+r^{-2})^{k-1}}{k^{\delta/2}}
\end{equation}
for all $x \in M$, for all $r>0$ and for all $k \in \N^*$ where $B=B(x,r)$.
\end{lemma}
\begin{remark}
 If two side Gaussian estimate $(GE)$ holds for $p_k$ and  if we choose $r \asymp \sqrt{k}$, then the upper bound \eqref{e-ultB} is sharp up to a constant factor.
\end{remark}
\begin{proof}[Proof of Lemma \ref{l-ultB}]
Let $x \in M$, $r>0$ and $B=B(x,r)$.
Our first step is an upper bound for $\norm{P_B^k}_{1 \to 2}$.
Let $f \in L^1(B)$ be an arbitrary function with $\norm{f}_1=1$.
The constants in this proof do not depend on the choice of $x \in M$, $k \in \N^*$, $r >0$ or $f \in L^1(B)$.

Then by H\"{o}lder inequality,
\begin{equation*}
\norm{P_Bf}_2^2 \le \norm{P_B f}_1 \norm{P_B f}_\infty \le  \norm{f}_1 \norm{P_B f}_\infty = \norm{P_Bf}_\infty.
\end{equation*}
By \eqref{e-ns3} and \eqref{e-ns9}, there exists $C_1>0$ such that
\begin{equation}
 \norm{P_Bf}_2^2 \le \norm{P_Bf}_\infty \le \frac{C_1 (1+r^2)^{\delta/2}}{V(x,r)}. \label{e-ut1}
\end{equation}
By \eqref{e-nashB}, along with Lemma \ref{l-dircomp-ball} and Lemma \ref{l-comp-dir}(b), there exists $C_N>0$
such that
\begin{equation}
\label{e-ut2} \norm{P_B g}_{2}^{2 + (4/\delta)} \le \frac{C_N r^2}{V(x,r)^{2/\delta}} \left( \E_*^B(g,g) + r^{-2} \norm{g}_2^2 \right) \norm{g}_1^{4/\delta}
\end{equation}
 for all  $r>0$, for all $x \in M$, and for all  $g \in L^2(B)$ where $B=B(x,r)$.
 Define
 \[
  v_k := (1+r^{-2})^{-(k-1)}\norm{P_B^kf}_2^2 \]
 for all $k \in \N^*$. Substituting $g=P_B^k f$ in \eqref{e-ut2} and using the fact that
 $\norm{P_B^kf}_1 \le \norm{f}_1 =1$ and $\E_*^B(P_B^k f, P^k_Bf) = \norm{P_B^k f}_2^2 - \norm{P_B^{k+1}f}_2^2$, we obtain
 the following difference inequality for $v_k$:
 \begin{equation}
  \label{e-ut3}
  v_{k+1}^{1+(2/\delta)} \le \frac{C_N (1+r^2)}{V(x,r)^{2/\delta}} (v_k - v_{k+1})
 \end{equation}
for all $k \in \N^*$. Next, we `solve' the difference inequality given by \eqref{e-ut3}. Define
\begin{equation} \label{e-ut4}
C_2 := \max\left(C_1, ((\delta C_N)/2)^{\delta/2} 2^{(1+(\delta/2))(\delta/2)} \right).
\end{equation}
We claim that
\begin{equation}
 \label{e-ut5} v_k \le C_2 \frac{(1+r^2)^{\delta/2}}{V(x,r)} k^{-\delta/2}
\end{equation}
for all $k \in \N^*$. We prove \eqref{e-ut5} by induction. The base case $k=1$ follows from \eqref{e-ut1} and \eqref{e-ut4}. For the inductive step, assume that
\eqref{e-ut5} holds for all $k=1,2,\ldots,n$ for some $n \in \N^*$.
We will show that \eqref{e-ut5} holds for $k=n+1$.
Assume to the contrary that
\begin{equation}
  \label{e-ut6} v_{n+1} > C_2 \frac{(1+r^2)^{\delta/2}}{V(x,r)} (n+1)^{-\delta/2}
\end{equation}
By \eqref{e-ut3}, \eqref{e-ut6} and the induction hypothesis, we obtain
\begin{align}
 v_{n+1}^{1+(2/\delta)} & <  \frac{C_N C_2 (1+r^2)^{1 + (\delta/2)}}{V(x,r)^{1 + (2/\delta)}}  \left(n^{-\delta/2} - (n+1)^{-\delta/2} \right) \nonumber \\
\nonumber &< \frac{C_N C_2 (1+r^2)^{1 + (\delta/2)}}{V(x,r)^{1 + (2/\delta)}}  \frac{\delta}{2} n^{-(1+(\delta/2))} \\
\nonumber & \le \frac{C_2 (1+r^2)^{1 + (\delta/2)}}{V(x,r)^{1 + (2/\delta)}} ((\delta C_N)/2) 2^{1+(\delta  /2)} (n+1)^{-(1+(\delta/2))} \\
\label{e-ut7} & \le  \left( C_2 \frac{(1+r^2)^{\delta/2}}{V(x,r)} (n+1)^{-\delta/2} \right)^{ 1 +(2/\delta)}.
\end{align}
The second line above follows from intermediate value theorem, the third line follows from $n \ge 1$ and the last line follows from \eqref{e-ut4}. The desired
contradiction follows from \eqref{e-ut6} and \eqref{e-ut7}. Using \eqref{e-ut3}, we obtain the estimate
\begin{equation}
\label{e-ut8}  \norm{P_B^k}_{1 \to 2}^2 \le \frac{C_2 (1+r^2)^{\delta/2} }{V(x,r)}  \frac{(1+r^{-2})^{k-1}}{k^{\delta/2}}
\end{equation}
for all $x \in M$ and all $r >0$ where $B=B(x,r)$.
Since $P_B$ is self-adjoint operator in $L^2(B)$, by duality we have the bound
\[
 \norm{P_B^k}_{1 \to \infty} \le \norm{P_B^{\lfloor (k/2) \rfloor}}_{1 \to 2} \norm{P_B^{\lceil (k/2) \rceil}}_{1 \to 2}.
\]
Using the above bound along with \eqref{e-ut8} yields \eqref{e-ultB} for $k \ge 2$. The case $k=1$ follows from \eqref{e-ut1}.
\end{proof}

We are  ready to prove the  Sobolev inequality \eqref{e-Sob} using the ultracontractivity estimate \eqref{e-ultB} above.

For an operator $T$, we define the operator $(I-T)^{1/2}$  as
\[
(I-T)^{1/2} = \sum_{k=0}^\infty a_k T^k
\]
 where $a_k$ is defined by the Taylor series
$(1-x)^\alpha = \sum_{k=0}^\infty a_k x^k$
 for $x \in (-1,1)$. By a classical estimate on coefficient of Taylor series, there exists $C_a>0$ such that
 \begin{equation}\label{e-taylor}
 \frac{C_a^{-1}} {(k+1)^{1/2}}  \le a_k \le \frac{C_a} {(k+1)^{1/2}}
 \end{equation}
  for all $k \in \N_{\ge 0}$.
  \section{Local Sobolev inequality}
  We use the ultracontractivity estimate \eqref{e-ultB} to obtain Sobolev inequality \eqref{e-Sob}.
  The proof uses Riesz-Thorin and Marcinkiewicz interpolation theorems which we briefly review in Appendix \ref{a-int}.
  \begin{prop}\label{p-ult2sob}
     Let $(M,d,\mu)$ be a quasi-$b$-geodesic metric measure space satisfying  \ref{doub-loc}, \ref{doub-inf}.
Suppose that a Markov operator $P$ has a kernel $p$ that is $(h,h')$-compatible with respect to $\mu$. Let $P_B$ and $\E^B$ denote the corresponding Markov operator and Dirichlet form restricted to a ball $B \subset M$.
Assume that there exists $C_u >0$ such that
\begin{equation} \label{e-sb0}
\norm{P_B^k}_{1 \to \infty} \le \frac{C_u (1+r^2)^{\delta/2} }{V(x,r)}  \frac{(1+r^{-2})^{k-1}}{k^{\delta/2}}
\end{equation}
for all $x \in M$, for all $r>0$ and for all $k \in \N^*$ where $B=B(x,r)$. Then we have the Sobolev inequality \eqref{e-Sob}.
  \end{prop}
 \begin{proof}
 As in the proof of Nash inequality \eqref{e-nash}, we start by considering the case $r \le 1$.
 By \eqref{e-ns3}, there exists $C_2>0$ such that
 \begin{equation}\label{e-sb1}
  \norm{P_B }_{1 \to \infty} \le \frac{C_1}{V(x,r)}
 \end{equation}
for all balls $B=B(x,r)$ with $r \le 1$. Since  $P_B$ is a contraction on all $L^p(B)$, we have
\begin{equation} \label{e-sb2}
     \norm{P_B}_{2(\delta-1)/(\delta-2) \to 2(\delta-1)/(\delta-2)} \le 1.
\end{equation}
Applying Riesz-Thorin interpolation between \eqref{e-sb1} and \eqref{e-sb2} yields
\begin{equation*}
\norm{P_B }_{2 \to 2 \delta/(\delta-2)} \le \left( \frac{C_1}{V(x,r)} \right)^{1/\delta}
\end{equation*}
for all  balls $B=B(x,r)$ with $r \le 1$. By choosing $C_S \ge C_1^{2/\delta}$, we have \eqref{e-Sob} for all balls $B(x,r)$ with $r \le 1$.

Next we consider the case $r >1$. Since
\[ \E^B(f,f) + r^{-2} \norm{f}_2^2 = \norm{ \left((1+r^{-2})I - P_B\right)^{1/2}f}^2, \]
it suffices to show that there exists  $C_2>0$ such that
\begin{equation}
 \label{e-sb3} \norm{P_B \left( I - (1+r^{-2})^{-1}P_B\right)^{-1/2}}_{2 \to 2 \delta/(\delta-2)} \le C_2 \frac{(1+r^2)^{1/2}}{V(x,r)^{1/\delta}}
\end{equation}
for all balls $B=B(x,r)$ with $r>1$. To see this, note that $C_S = \max(C_1^{2/\delta}, 2 C_2^2)$ satisfies \eqref{e-Sob}.
Define
\begin{equation}
 \label{e-sb4} E(B) := \frac{(1+r^2)}{\mu(B)}, \hspace{1cm} T_B:= P_B \left( I - (1+r^{-2})^{-1}P_B\right)^{-1/2}.
\end{equation}
Let $p \in[1,\delta)$ and $q \in [\delta/(\delta-1),\infty)$ satisfy
\begin{equation}
 \label{e-sb5} p^{-1} =q^{-1} + \delta^{-1}.
\end{equation}
For all $p \in[1,\delta)$ and $q \in [\delta/(\delta-1),\infty)$ satisfying \eqref{e-sb5}, we show that the operator $T_B$ is of weak-type $(p,q)$. An application of
Marcinkiewicz interpolation then yields \eqref{e-Sob}.
Recall that $T_B = \sum_{k=1}^\infty a_{k-1} (1+r^{-2})^{-(k-1)}P_B^k$. For $N \in \N^*$, we define operators
\[
R_{B,N} := \sum_{k=1}^N a_{k-1} (1+r^{-2})^{-(k-1)}P_B^k, \hspace{1cm} S_{B,N} := T_B - R_{B,N}.
\]
By \eqref{e-sb0} and Riesz-Thorin interpolation, we obtain
\begin{equation}
 \label{e-sb6} \norm{P_B^k}_{p \to \infty} \le C_u^{1/p} E(B)^{\delta/(2p)} \frac{(1+r^{-2})^{(k-1)/p}}{k^{\delta/(2p)}}
\end{equation}
for all balls $B$, for all $k \in \N^*$ and for all $1 \le p <\infty$.
For each $p \in [1,\delta)$, there exists $C_3>0$ such that
\begin{align}
\nonumber \norm{S_{B,N}}_{p \to \infty} &\le \sum_{k=N+1}^\infty a_{k-1} (1+r^{-2})^{-(k-1)} \norm{P_B^k}_{p \to \infty} \\
\nonumber &\le C_u^{1/p} E(B)^{\delta/(2p)} C_a \sum_{k=N+1}^\infty  k^{-1/2}k^{-\delta/(2p)}   \\
& \le C_3 E(B)^{\delta/(2p)} N^{-\delta/(2q)} \label{e-sb7}
\end{align}
 for all balls $B$, where $q$ is given by \eqref{e-sb5}. In \eqref{e-sb7} $C_3$ depends only on $p,q,\delta$ but not on $B=B(x,r)$.
In the second line above we use \eqref{e-sb6} and \eqref{e-taylor} and we used \eqref{e-sb5} and $p \in[1,\delta)$ in the last line.
By the same argument as above and increasing $C_3=C_3(p)$ if necessary, we may assume that
\begin{equation}
 \label{e-sb8} \norm{T_B}_{p \to \infty} \le C_3 E(B)^{\delta/(2p)}
\end{equation}
 for all balls $B$.

Let $g \in L^p(B)$ satisfy $\norm{g}_p=1$. For $\lambda>0$, let $N_0= N_0(\lambda,B)$ denote
the smallest positive integer such that $C_3 E(B)^{\delta/(2p)}N_0^{-\delta/(2q)} \le \lambda/2$.
By union bound, for each $p \in [1,\delta)$ and $q$ given by \eqref{e-sb5}, there exists $C_4,C_5>0$ such that
\begin{align}
 \nonumber \mu_B \Sett{x  \in B }{ \abs{T_B g(x)} > \lambda } &\le   \mu_B \Sett{x  \in B }{ \abs{R_{B,N_0} g(x)} > \lambda/2 } \\
 & \hspace{1cm}+ \mu_B \Sett{x  \in B }{ \abs{S_{B,N_0} g(x)} > \lambda/2 } \nonumber \\
 & \le  \mu_B \Sett{x  \in B}{ \abs{R_{B,N_0} g(x)} > \lambda/2 } \nonumber \\
 & \le (2/\lambda)^p \norm{R_{B,N_0} g}_p^p \nonumber \\
 & \le C_u^p (2/\lambda)^p \left( \sum_{k=1}^{N_0} k^{-1/2} \right)^p \le C_4 (2C_u)^p \lambda^{-p}N_0^{p/2} \nonumber \\
 & \le C_5 E(B)^{q/2} \lambda^{-q} \label{e-sb9}
\end{align}
for all balls $B=B(x,r)$.
In the second step above we used the definition of $N_0$. The third step follows from Chebyshev inequality, the fourth step follows from \eqref{e-taylor} and $\norm{P_B^k}_{p \to p} \le 1$. The last step \eqref{e-sb9} follows from
\eqref{e-taylor}, \eqref{e-sb5}, \eqref{e-sb8} and the definition of $N_0$.
By Marcinkiewicz interpolation theorem and the estimates given by \eqref{e-sb9}, there exists $C_6 >0$ such that
\[
 \norm{T_B}_{2 \to 2\delta/(\delta-2)} \le C_6  \sqrt{E(B)}
\]
for all balls $B=B(x,r)$. This is precisely \eqref{e-sb3} which we intended to prove.
\end{proof}
We record two important consequences of Proposition \ref{p-ult2sob} first of which is the proof of Theorem \ref{t-Sob}
\begin{proof}[Proof of Theorem \ref{t-Sob}]
 Theorem \ref{t-Sob} follows from Corollary \ref{c-nash}, Lemma \ref{l-ultB} and Proposition \ref{p-ult2sob}.
\end{proof}
The next corollary shows that Sobolev inequality is necessarily true under doubling hypothesis and
 Gaussian upper bounds \ref{gue}.
\begin{corollary} \label{c-nsob}
Let $(M,d,\mu)$ be a quasi-$b$-geodesic metric measure space satisfying  \ref{doub-loc}, \ref{doub-inf}.
Suppose that a Markov operator $P$  has a kernel $p$ that is $(h,h')$ compatible with respect to $\mu$.
Further assume that iterated kernel $p_k$ that satisfies \ref{gue}.
Let $P_B$ and $\E^B$ denote the corresponding Markov operator and Dirichlet form restricted to a ball $B \subset M$.
Then the Sobolev inequality \eqref{e-Sob} holds.
\end{corollary}
\begin{proof}
By Proposition \ref{p-ult2sob} it suffices to show the ultracontractivity estimate \eqref{e-sb0} on $\norm{P_B^k}_{1\to \infty}$.
By \ref{gue}, there exists $C_1>0$ such that
\begin{equation}
 \label{e-nsob1} \norm{P_B^k}_{1\to \infty} \le \sup_{y \in B, z \in B} p_k(y,z) \le \sup_{y \in B(x,r)} \frac{C_1}{V(y,\sqrt{k})}
\end{equation}
for all balls $B=B(x,r)$ and for all $k \in \N^*$.
By \eqref{e-vd1}, there exists $\delta>2$ and $C_2>0$ such that
\begin{equation} \label{e-nsob2}
 \sup_{y \in B(x,r)} \frac{1}{V(x,\sqrt{k})} \le \frac{1}{V(x,r)}  \sup_{y \in B(x,r)} \frac{V(x,2(r \vee  \sqrt{k}) ) }{V(y,\sqrt{k})} \le \frac{1}{V(x,r)} C_2 \left( \frac{2( r \vee \sqrt{k})}{\sqrt{k} }\right)^\delta
\end{equation}
for all balls $B(x,r)$ and for all $k \in \N^*$. The desired estimate \eqref{e-sb0} follows from \eqref{e-nsob1} and \eqref{e-nsob2}.
\end{proof}
\section{Sobolev inequality implies large scale doubling property}
Next, we show that Sobolev inequality implies \ref{doub-inf} under natural hypotheses. More precisely
\begin{prop}\label{p-sobvd}
Let $(M,d,\mu)$ be a metric measure space satisfying \ref{doub-loc}.
Let $P$ be $(h,h')$ compatible Markov operator in a metric measure space $(M,d,\mu)$ satisfying Sobolev inequality \eqref{e-Sob}.
Then $(M,d,\mu)$ satisfies the large scale doubling property \ref{doub-inf}.
\end{prop}
We need the following volume comparison lemma.
\begin{lemma} \label{l-vc}
 Let $(M,d,\mu)$ be a quasi-$b$-geodesic metric measure space satisfying \ref{doub-loc} and let $h'\ge b>0$.
 Then there exists $C_{0}>0$ such that
 \begin{equation}
  \label{e-vch} V(x,r+h') \le C_{0} V(x,r)
 \end{equation}
for all $x \in M$ and for all $r \ge 3h'$.
\end{lemma}
\begin{proof}
 Let $Y$ be a maximal $h'$-separated subset of $B(x,r)$ where $x \in M$ and $r \ge 3h'$.
 The collection of balls $ \Sett{B(y,h'/2)}{y \in Y}$ are disjoint and hence
 \begin{equation}\label{e-vch1}
  V(x,r) \ge \sum_{y \in Y \cap B(x,r-h')} V(y,h'/2).
 \end{equation}
However since $B(x,r) \subseteq \cup_{y \in Y} B(y,h')$ and $r \ge 3h'$, we have
\begin{equation}\label{e-vch2}
\emptyset \neq B(x,r-2h') \subseteq \cup_{y \in Y \cap B(x,r-h')} B(y,h'),
\end{equation}

By quasi-$b$-geodesicity and $b \le h'$, there exists $C_1>0$ such that for all  $z \in B(x,r+h')$, there exists a $b$-chain  $x_0,x_1,\ldots,x_m$  $b$-chain from $x$ to $z$ such that
\begin{equation} \label{e-vch3}
x_i \in B(x,r-2h') \mbox{ and } d(x_i ,z) \le C_1 h'.
\end{equation}
Combining \eqref{e-vch2} and \eqref{e-vch3}, we obtain
\begin{equation}\label{e-vch4}
B(x,r+h') \subseteq \cup_{y \in Y \cap B(x,r-h')} B(y,(C_1+1)h').
\end{equation}
Combining \eqref{e-vch4},Lemma \ref{l-vloc} and \eqref{e-vch1}, we obtain
\begin{align*}
  V(x,r+h') &\le \sum_{y \in Y \cap B(x,r-h')} V(y,(C_1+1)h') \\
  &\le C_{h'/2,(C_1+1)h'}  \sum_{y \in Y \cap B(x,r-h')} V(y,h'/2) \le C_{h'/2,(C_1+1)h'} V(x,r).
\end{align*}
\end{proof}

\begin{proof}[Proof of Proposition \ref{p-sobvd}].
We adapt the argument of \cite[Proposition 2.1]{CG98}.
However unlike in \cite[Proposition 2.1]{CG98}, we do not consider volumes of arbitrarily small balls.

Let $x \in M$ and $r \ge 3h'$ be arbitrary. For $s>0$, define the `tent function'
\[
f_s(y)= \max( s- d(x,y) , 0).
\]
By  $(h,h')$ compatibility of $P$, we have $P_{B(x,r)} f_{3h'} \ge h' \one_{B(x,h')}$.
Therefore by applying \eqref{e-Sob}, we have
\begin{equation*}
 (h')^2 V(x,h')^{(\delta-2)/\delta} \le \frac{C_S r^2}{V(x,r)^{2/\delta}} \left( (h')^2 V(x,4h') + r^{-2} (3h')^2 V(x,3h')\right)
\end{equation*}
for all $r \ge 3h'$ and for all $x \in M$.
Combined with Lemma \ref{l-vloc}, there exists $C_1>0$ such that
\begin{equation}\label{e-sv1}
 \frac{V(x,r)}{V(x,h')} \le C_1 r^\delta
\end{equation}
for all $r \ge 3h'$ and for all $x \in M$.

Let $3h' \le s \le r$. Then by  $(h,h')$ compatibility of $P$, we have $P_{B(x,r)} f_{s} \ge (s/6) \one_{B(x,s/2)}$.
Hence by Sobolev inequality \eqref{e-Sob}, \eqref{e-compat} and Lemma \ref{l-dircomp-ball}(a), we obtain
\begin{align*}
 (s/6)^2 V(x,s/2)^{(\delta -2)/\delta} &\le \frac{C_S r^2}{V(x,r)^{2/\delta}} \left((h')^2 V(x,s+h) + r^{-2} s^{2} V(x,s)\right) \\
\end{align*}

Combined with Lemma \ref{l-vc}, there exists $C_2>0$ such that
\begin{equation} \label{e-sv2}
 V(x,s) \ge  \left( \frac{s^\delta V(x,r) }{C_2 r^\delta}\right)^{2/\delta} V(x,s/2)^{(\delta -2)/\delta}
\end{equation}
for all $x \in M$ and for all $3h' \le s \le r$. We replace $s$ by $s/2$ in \eqref{e-sv2} and iterate to obtain
\begin{equation} \label{e-sv3}
 V(x,s) \ge 4^{- \sum_{j=0}^{i-1} j(\delta -2)^j/\delta^j }\left( \frac{s^\delta V(x,r) }{C_2 r^\delta}\right)^{(2/\delta)\sum_{j=0}^{i-1} (\delta -2)^j/\delta^j} V(x,s/2^i)^{(\delta -2)^i/\delta^i}
\end{equation}
for all $3h' \le s/2^{i-1} \le s \le r$. In particular if we choose $i=\lceil \log_2(s/3h') \rceil$, we have $(3h')/2 \le s/2^i \le 3h'$. Hence by \eqref{e-sv3} and \eqref{e-sv1},we have
\begin{equation} \label{e-sv4}
 V(x,s) \ge 4^{- \sum_{j=0}^{\infty} j(\delta -2)^j/\delta^j }\left( \frac{s^\delta V(x,r) }{C_2 r^\delta}\right)^{(2/\delta)\sum_{j=0}^{i-1} (\delta -2)^j/\delta^j} \left( \frac{V(x,r)}{C_1 r^\delta} \right)^{(\delta -2)^i/\delta^i}
\end{equation}
for all $x \in M$ and for all $3h' \le s \le r$, where $i=\lceil \log_2(s/3h') \rceil$. By \eqref{e-sv4}, there exists $C_3>0$ such that
\begin{equation}\label{e-sv5}
 \frac{V(x,r)}{V(x,s)} \le C_3 (r/s)^\delta s^{(\delta-2)^i/\delta^{i-1}}
\end{equation}
for all $x \in M$ and for all $3h' \le s \le r$, where $i=\lceil \log_2(s/3h') \rceil$.
Since the map $s \mapsto \exp \left(\delta((\delta-2)/\delta)^{\lceil \log_2(s/3h') \rceil} \ln s\right)$ is bounded in $[3h',\infty)$, by \eqref{e-sv5} there exists $C_4>0$ such that
\begin{equation*}
 \frac{V(x,r)}{V(x,s)} \le C_4\left(\frac{r}{s}\right)^\delta
\end{equation*}
for all $x \in M$ and for all $3h' \le s \le r$. The above equation clearly implies \ref{doub-inf}.
\end{proof}

\chapter{Elliptic Harnack inequality}\label{ch-ehi}
In this chapter, we prove elliptic Harnack inequality for non-negative harmonic functions.
As before, we consider a metric measure space $(M,d,\mu)$ and
a Markov operator $P$ that is $(h,h')$-compatible with $(M,d,\mu)$. Recall that the operator
$\Delta := I - P$ is the Laplacian corresponding to $P$.
\section{Harmonic functions}
\begin{definition}\label{d-harmonic}
 Let $P$ be a Markov operator on $(M,d,\mu)$.
 A function $f:U \to \R$ is $P$-harmonic in $B(x,r)$ if
 \[
  \Delta f(y) =  f(y) - Pf(y) = 0
 \]
for all $y \in B(x,r)$.

Similarly, we say  $f:M \to R$ is $P$-subharmonic (resp. $P$-superharmonic) in $B(x,r)$ if
\[
 \Delta f(y) \le 0  \mbox{ (resp. $\ge 0$)}
\]
for all $y \in B(x,r)$.

We say a function $f:M \to \R$ is $P$-harmonic (resp. subharmonic, superharmonic) if $\Delta f \equiv 0$ (resp. $\Delta f \le 0$, $\Delta f \ge 0$).
\end{definition}
\begin{remark}\label{r-harmonic}\leavevmode
\begin{enumerate}[(a)]
\item Consider a Markov operator $P$ that is $(h,h')$-compatible with $(M,d,\mu)$.
By \eqref{e-compat}, $Pf(y)$ depends only on $f$ in $B(y,h')$.
Therefore the property that $f:M \to \R$ is $P$-harmonic in $B(x,r)$ depends only on the values of $f$ in $B(x,r+h')$. Hence in this case it suffices to have
$B(x,r+h') \subseteq \operatorname{Domain}(f)$.
 \item We use the term harmonic instead of $P$-harmonic if the Markov operator
$P$ is clear from the context. Same holds for superharmonic or subharmonic
functions.
\end{enumerate}

\end{remark}
The main result of the chapter is the following elliptic Harnack inequality.
\begin{theorem}[Elliptic Harnack inequality]
 \label{t-ehi}
Let $(M,d,\mu)$ be a quasi-$b$-geodesic metric measure space satisfying  \ref{doub-loc}, \ref{doub-inf} and Poincar\'{e} inequality at scale $h$ \ref{poin-mms}.
Suppose that a Markov operator $P$ has a kernel $p$ that is $(h,h')$-compatible with respect to $\mu$ for some $h >b$.
Then there exists $c>0,r_0>0,C_E>0$ such that for all $x \in M$, for all
$r \ge r_0$ and for all
non-negative  functions $u:B(x,r) \to \R_{\ge 0}$ that are $P$-harmonic in  $B(x,r)$  the following Harnack inequality holds:
\begin{equation}
 \label{e-ehi} \sup_{x \in B(x,cr)} u \le C_E \inf_{x \in B(x,cr)} u.
\end{equation}
\end{theorem}
In \eqref{e-ehi}, the sup and inf must be understood as essential sup and essential inf with respect to $\mu$.

 We follow Moser's iteration method \cite{Mos61} to prove the elliptic Harnack inequality.
 Our approach is an adaptation of Delmotte's approach except that we have to rely on a
 weaker version of Sobolev inequality and a modified version of John-Nirenberg inequality.
 Moser's iteration relies on estimating the quantities
 \begin{equation} \label{e-phi}
 \phi(u,p,B') := \left( \frac{1}{\mu(B')} \int_{B'} \abs{u}^p \, d\mu \right)^{1/p}
 \end{equation}
for different balls $B' \subset B$ and for different values of $p \in \R \setminus \set{0}$.
By Jensen's inequality, $p \mapsto \phi(u,p,B')$ is non-decreasing function. The function $\phi$ satisfies
$\lim_{p \to -\infty} \phi(u,p,cB) = \inf_{cB} u$ and $\lim_{p \to +\infty} \phi(u,p,cB) = \sup_{cB} u$ \cite[Lemma 14.1.4]{Jos13}.
To obtain \eqref{e-ehi},
Moser's iterative method relies on establishing bounds of the form
$\phi(u,p_1,B')\le  C_{p_1,p_2} \phi(u,p_2,B'')$ for different values of $p_1,p_2 \in \R \setminus \set{0}$ satisfying $p_1 < p_2$.
Sobolev inequality and Poincar\'{e} inequality are crucial ingredients to run this iterative procedure.
For a function $f$ that is defined on a ball $B$, we denote the mean integral by
\[
 f_B= \dashint_B f \,d\mu = \frac{1}{\mu(B)} \int_B f \,d\mu.
\]

We start with a local version of the above elliptic Harnack inequality.
\begin{lemma}
 \label{l-compharm}
 Let $(M,d,\mu)$ be a quasi-$b$-geodesic space satisfying \ref{doub-loc} and let $P$ be a weakly $(h,h')$-compatible Markov operator with $(M,d,\mu)$ for some $h > b$.
  There exists $C >0$ and $r_0>0$ such that
  \begin{equation} \label{e-le0}
   u(y) \le C u(z)
  \end{equation}
for all $x \in M$, for all $r \ge r_0$, for all $y,z \in B(x,r/2)$ satisfying $d(y,z) \le h'$ and
for all non-negative functions $u:B(x,r+h') \to \R$  harmonic in $B(x,r)$.
\end{lemma}
\begin{proof}
 There exists $c_1>0$ and $l \in \N^*$ such that
 \begin{equation} \label{e-le1}
 p_l(z,w) = p_l(w,z) \ge \frac{c_1 \one_{B(z,2h')}(w)}{V(w,h')}
 \end{equation}
 for all $y,w \in M$. The proof of \eqref{e-le1} is analogous to that of \eqref{e-ws0}.
 Therefore by \eqref{e-le1}, \ref{doub-loc} weak $(h,h')$-compatibility of $p_1$ and triangle inequality, there exists $c_2 >0$ such that
 \begin{equation}
  \label{e-le2}
  p_l(z,w) \ge \frac{c_1 \one_{B(z,2h')}(w)}{V(w,h')} \ge \frac{c_1 \one_{B(y,h')}(w)}{V(w,h')}  \ge c_2 p_1(w,y) = c_2 p_1(y,w)
 \end{equation}
for all $y,z,w \in M$ satisfying $d(y,z) \le h'$.

Choose $r_0$ large enough so that $r/2+ l h' \le r+h'$ for all $r \ge r_0$.
Note that for every harmonic function $u:B(x,r+h') \to \R$ in $B(x,r)$ with $r \ge r_0$  and for all $z \in B(x,r/2)$, we have
\begin{equation}\label{e-le3}
 u(z) =  P^l u(z) = \int_{B(z, l h')} p_l(z,w)  u(w) \mu(dw)
\end{equation}
By \eqref{e-le3} and \eqref{e-le2}, we obtain
\begin{equation} \label{e-le4}
 u(z)= \int_{B(z, l h')} p_l(z,w)  u(w) \mu(dw) \ge c_2 \int_{B(y,  h')} p_1(y,w)  u(w) \mu(dw)= c_2 u(y)
\end{equation}
for all non-negative harmonic functions $u$ in $B(x,r)$ for all $x \in M$, for all $z,y \in B(x,r/2)$ with $r \ge r_0$.
The choice $C=c_2^{-1}$ satisfies \eqref{e-le0}.
\end{proof}
\section{John-Nirenberg inequality}
Moser \cite{Mos61}, used John-Nirenberg inequality to obtain an estimate of the form $\phi(u,-q,B') \le C' \phi(u,q,B')$ for some $q,C'>0$.
An alternative approach is to use an abstract lemma of Bombieri and Guisti was later proposed by Moser  \cite[Section 2.2.3]{Sal02}.

John-Nirenberg inequality is an estimate on distribution of functions of bounded mean oscillation which were introduced in \cite{JN61}.
A locally integrable function $f:B \to \R$ define  is of bounded mean oscillation (BMO) if
\[
 \norm{f}_{BMO(B) } := \sup_{B' \subset B} \frac{1}{\mu(B')} \int_{B'} \abs{f - f_{B'}} \,d\mu < \infty.
\]
John-Nirenberg inequality  states that  functions of bounded mean oscillation have an exponentially decaying distribution function.

In \cite[Theorem 5.2]{ABMH11} a version of John-Nirenberg inequality is shown for spaces satisfying the doubling hypothesis \ref{doub-glob}.
However for us, the metric measure space $(M,d,\mu)$ only satisfies \ref{doub-loc} and \ref{doub-inf}.
Since we do not have doubling hypothesis on arbitrarily small balls, we introduce a modified version of BMO seminorm (BMO seminorm at scale $h$) defined as
\begin{equation}
 \norm{f}_{BMO(B(x_0,r_0)),h } = \sup_{\substack{ B(y,r) \subseteq B(x_0,r_0), r \ge h}} \frac{1}{V(y,r)} \int_{B(y,r)} \abs{f - f_{B(y,r)}} \,d\mu.
\end{equation}

Our proof is motivated by the presentation in \cite{ABMH11}. We start by recalling the Vitali covering lemma.
\begin{lemma}[Vitali covering lemma] \label{l-vit}
Let $\mathcal{F}$ be a family of balls with positive and uniformly bounded radii in a metric space $(M,d)$. Then there
exists a disjoint subfamily $\mathcal{G} \subseteq \mathcal{F}$ such that
\begin{equation*}
 \bigcup_{B \in \mathcal{F}} B \subseteq \bigcup_{B \in \mathcal G} 5 B.
\end{equation*}
In fact, every ball $B \in \mathcal F$ meets a ball $B' \in \mathcal G$ with radius at least
half that of $B$ and therefore satisfies $B \subseteq 5 B'$.
\end{lemma}
The proof of Vitali covering lemma follows from an application of Zorn's lemma. We refer the reader to \cite[Theorem 1.2]{Hei01} for a proof of Lemma \ref{l-vit}.
A crucial ingredient in the proof of John-Nirenberg inequality is the following version of
Calder\'{o}n-Zygmund decomposition lemma.
Since we replaced \ref{doub-glob} by weaker assumptions  \ref{doub-loc} and \ref{doub-inf}, we need
some other method to control the behavior of a BMO function at small length scales.
This is why we assume a local Harnack inequality (by Lemma \ref{l-compharm} the local Harnack inequality holds for harmonic functions).
\begin{lemma}[Calder\'{o}n-Zygmund decomposition lemma] \label{l-cz}
Suppose $(M,d,\mu)$ be a metric measure space satisfying  \ref{doub-loc} and \ref{doub-inf}.
 Let $f$ be a non-negative locally integrable function on $B(x_0,11 r_0)$ for some $r_0 \ge r_1 \ge h >0$. Further we assume that there exists $C_1 \ge 1$ such that
 $f$ satisfies the local Harnack inequality
 \begin{equation} \label{e-cz00}
  f(y) \le C_1 f(z)
 \end{equation}
 for all $y,z \in B(x_0,r_0+h)$ satisfying $d(y,z) \le h$.
Further, assume that
\begin{equation}\label{e-cz01}
  \lambda_0 \ge \frac{1}{V(x_0,r)} \int_{B(x_0,11r_0)} f \, d\mu
\end{equation}
Then there exists countable (possibly finite) family of disjoint balls  $\mathcal{F}=\set{B_i}$ of disjoint balls centered in $B(x_0,r)$
and satisfying $5 B_i \subseteq B(x_0,11r_0)$ for all $B_i \in \mathcal F_0$ so that
\begin{enumerate}[(i)]
 \item $f(x) \le C_1 \lambda_0$ for all $x \in B(x_0,r_0) \setminus \left( \bigcup_{B_i \in \mathcal F} 5 B_i \right)$.
 \item $\lambda_0 < \dashint_{B_i} f \,d\mu \le C_2 \lambda_0$ for all $B_i \in \mathcal F_0$.
  \item $C_2^{-1} \lambda_0 < \dashint_{5 B_i} f \,d\mu \le \lambda_0$ for all $B_i \in \mathcal F_0$.
\end{enumerate}
The family of balls $\mathcal F_0$ satisfying the above conditions are called Calder\'{o}n-Zygmund balls at level $\lambda_0$.
Moreover if  $\lambda_0 \le \lambda_1 \le \ldots \le \lambda_N$, then the family Calder\'{o}n-Zygmund balls $\mathcal F_n$
corresponding to different levels $\lambda_n$ may be chosen in such a way that every $B_i(\lambda_{n+1}) \in \mathcal{F}_{n+1}$ is contained in
some $5 B_j(\lambda_n)$ where $B_j(\lambda_n) \in \mathcal F_n$.
\end{lemma}
\begin{proof}
We denote $B(x_0,r_0)$ as $B_0$.
Define a maximal function
\[
 M_{B_0} f(x)= M_{B(x_0,r_0)} f(x) = \sup_{\substack{ B(y,r) \subset B(x_0,r_0+h):\\ y \in B(x_0,r_0), r \ge h, B(y,r) \ni x}} \dashint_{B(y,r)} f \,d\mu
\]
for all $x \in B(x_0,r)$. We define
\[
 E_\lambda = \Sett{x \in B(x_0,r_0) }{ M_{B(x_0,r_0)} f(x) > \lambda }.
\]
First consider $\lambda_N$. By definition for every $x \in E_{\lambda_N}$,
there exists a ball $B_x = B(y_x,r_x)$ satisfying $y_x \in B_0$, $x \in B_x$, $B_x \subseteq B(x_0,r_0+h)$, $r_x \ge h$ and
\begin{equation} \label{e-cz1}
 \lambda_0 \le \lambda_1 \le \ldots \le \lambda_N < \dashint_{B_x} f \, d\mu.
\end{equation}
Let $k=k_x \in \N^*$ be such that $5^{k-1} r_x \le 2 r_0 \le 5^k r_x$.
Then $B_0 \subseteq 5^k B_x \subseteq 11 B_0$. Combining this with \eqref{e-cz01}, we have
\[
 \dashint_{5^k B_x} f \, d\mu \le \frac{1}{\mu(B_0)} \int_{11 B_0} f \, d\mu \le \lambda_0 \le \lambda_N.
\]
However since $\dashint_{B_x} f \, d\mu > \lambda_N$, there exist smallest $n_x \ge 1$  such that
\begin{equation}
 \label{e-cz2} \dashint_{5^{n_x} B_x} f \, d\mu \le \lambda_N
\end{equation}
and
\begin{equation}
 \label{e-cz3} \dashint_{5^{j} B_x} f \, d\mu > \lambda_N
\end{equation}
for all $j=0,1,\ldots,n_x-1$. The balls $ 5^{n_x-1} B_x$ forms a covering of $E_{\lambda_N}$. Therefore by Vitali covering lemma (Lemma \ref{l-vit}),
we pick a family $\mathcal F_N$  of pairwise disjoint balls $B_i= 5^{n_{x_i}-1}B_{x_i}$ satisfying $E_{\lambda_N} \subseteq \bigcup_{B \in \mathcal F_N} 5B$.
We now check the construction above satisfies the desired properties. By \eqref{e-cz2}, \eqref{e-cz3} and \eqref{e-vd1}, there exists $C_h>0, \delta>0$ such that
\begin{equation*}
 \lambda_N < \dashint_{5^{n_x -1}B_x} f \, d\mu \le C_h 5^\delta \dashint_{5^{n_x}B_x} f \, d \mu \le  C_h 5^\delta \lambda_N.
\end{equation*}
Choosing $C_2 = 5^\delta C_h$, we obtain properties (ii) and (iii) of Calder\'{o}n-Zygmund decomposition.

It remains to verify (i). If $x \in B_0 \setminus \left( \bigcup_{B_i \in \mathcal F} 5 B_i \right) \subseteq B_0 \setminus E_{\lambda_N}$, we have $M_{B_0}f(x) \le \lambda_N$.
Therefore by \eqref{e-cz00}, we have
\[
 \lambda_N \ge M_{B_0}f(x) \ge \dashint_{B(x,h)} f\,d\mu \ge C_1^{-1} f(x).
\]
This give property (i). We have now constructed the desired decomposition at level $\lambda_N$. Next we consider $\lambda_{N-1}$.

Since $E_{\lambda_N} \subseteq E_{\lambda_{N-1}}$, for every $x \in E_{\lambda_N}$, we may start with exactly the same ball satisfying \eqref{e-cz1} as before.
For every $x \in E_{\lambda_{N-1}} \setminus E_{\lambda_N}$, we choose a ball $B_x=B(y_x,r_x)$ satisfying $B_x \ni x$, $y_x \in B_0$, $r_x \ge h$, $B_x \subset B(x_0,r_0+h)$ and
\begin{equation}
 \label{e-cz4} \lambda_0 \le \ldots \le \lambda_{N-1} < \dashint_{B_x} f \, d\mu.
\end{equation}
As before for each ball $B_x$, we choose the smallest integer $m_x \ge 1$ such that
\begin{equation}
 \label{e-cz5} \dashint_{5^{m_x} B_x } f \, d \mu \le \lambda_{N-1}
\end{equation}
and
\begin{equation}
 \label{e-cz6} \dashint_{5^{j} B_x} f \, d\mu > \lambda_{N-1}
\end{equation}
for $j=0,1,\ldots,m_x -1$. Note that if $x \in E_{\lambda_N}$, then $n_x \le m_x$. As before, we apply Vitali's covering lemma
to the  balls $\Sett{5^{m_x-1}B_x}{ x \in E_{\lambda_{N-1}}}$ to obtain a pairwise disjoint family of balls $\mathcal F_{N-1}$ satisfying
(i)-(iii) with $\lambda_0$ replaced by $\lambda_{N-1}$.

Let $B_i(\lambda_N) \in \mathcal F_N$. Then $B_i(\lambda_N) = 5^{n_x -1}B_x$ for some $x \in E_{\lambda_N}$. Since $n_x \le m_x$, we have
$B_i(\lambda_N) \subset 5^{m_x -1}B_x$. By Vitali's covering lemma, there exists $B_j(\lambda_{N-1}) \in \mathcal F_{N-1}$ such that
$B_i(\lambda_N) \subseteq 5^{m_x -1}B_x \subseteq 5 B_j(\lambda_{N-1})$. We continue this procedure to get decomposition at all levels
$\lambda_0 \le \ldots \le \lambda_N$.
\end{proof}
\begin{remark}
 In the above proof, we use \eqref{e-cz00} to obtain property (i) of the Calder\'{o}n-Zygmund decomposition.
 Typically property (i) is proved using  Lebesgue differentiation theorem.
 However the proof of  Lebesgue differentiation theorem requires \ref{doub-glob}. (See \cite{ABMH11} and \cite[Theorem 1.8]{Hei01})
\end{remark}
Next, we prove the John-Nirenberg inequality for spaces satisfying \ref{doub-loc} and \ref{doub-inf}.
\begin{prop}[John-Nirenberg inequality] \label{p-jn}
 Let $(M,d,\mu)$ be a metric measure space satisfying  \ref{doub-loc} and \ref{doub-inf}.
 Let $f$ be a non-negative locally integrable function on $B(x_0,11 r_0)$ for some $r_0 \ge h >0$. Further we assume that there exists $C_1 \ge 1$ such that
 $f$ satisfies the local Harnack inequality
 \begin{equation} \label{e-lochar}
  f(y) \le C_1 f(z)
 \end{equation}
 for all $y,z \in B(x_0,r_0+h)$ satisfying $d(y,z) \le h$.
Then there exists $C_2 >0$ such that
\begin{equation}
 \label{e-jn} \mu \left( \Sett{ x\in B_0 }{ \abs{f-f_{B_0}}} > \lambda  \right) \le C_2 \mu(B_0) \exp(-\lambda / (C_2 \norm{f}_{\BMO(11B_0),h}))
\end{equation}
for all $\lambda >0$. The constant $C_2$ depends only on $C_1,h$ and
 constants associated with doubling hypotheses \ref{doub-loc} and \ref{doub-inf}.
\end{prop}
\begin{proof}
 Let $B_0=B(x_0,r_0)$. Without loss of generality, we assume $f_{B_0}=0$ and $\norm{f}_{BMO(11B_0),h}=1$ . It suffices to consider $f$ such that $f_{B_0}=0$ and $\norm{f}_{BMO(11B_0),h}=1$ as we may replace
the function $f$ by $(f-f_{B_0})/\norm{f}_{BMO(11B_0),h}$.

By \eqref{e-vd1}, there exists $C_3 >0$ such that
\begin{align*}
 \frac{1}{\mu(B_0)} \int_{11 B_0} \abs{f-f_{B_0}} \, d \mu & \le C_3 \dashint_{11 B_0 } \abs{f - f_{11 B_0}} \, d\mu + C_3 \abs{f_{B_0} - f_{11 B_0}} \\
 & \le C_3 \norm{f}_{BMO(11B_0),h'} + C_3 \dashint_{B_0} \abs{f - f_{11 B_0}} \,d \mu \\
 &\le 2 C_3^2 \norm{f}_{BMO(11B_0),h} = 2 C_3^2.
\end{align*}
If $B_j$ is the Calder\'{o}n-Zygmund balls at level $C_1^{-1} \lambda$  where $\lambda \ge 2 C_1 C_3^2$, then by Lemma \ref{l-cz}
\begin{enumerate}
 \item[(i)] $\abs{f(x)} \le \lambda$ for all $x \in B_0 \setminus \cup_j 5B_j$.
 \item[(ii)]$C_1^{-1} \lambda < \dashint_{B_j} \abs{f} \, d\mu \le C_3 C_1^{-1} \lambda$ for all $j$.
  \item[(iii)]$C_1^{-1}C_3^{-1} \lambda < \dashint_{5 B_j} \abs{f} \, d\mu \le  C_1^{-1} \lambda$ for all $j$.
\end{enumerate}
By (i) and \eqref{e-vd1}, we have
\begin{equation}
 \label{e-jn1} \mu \left( \Sett{x \in B_0}{ \abs{f(x)} > \lambda} \right) \le \sum_j \mu( 5 B_j) \le C_3 \sum_{j} \mu(B_j)
\end{equation}

In order to estimate $\sum_{j} \mu(B_j)$, we consider Calder\'{o}n-Zygmund decomposition
at levels $ C_1^{-1} \lambda > C_1^{-1} \gamma \ge 2 C_3^2$ as in Lemma \ref{l-cz}.
We partition the family $\set{B_j(C_1^{-1} \lambda)}_j$ as follows:
First we collect those which are contained in $5 B_1(C_1^{-1} \gamma)$.
From the remaining balls we collect those balls which are contained in $5 B_2(C_1^{-1} \gamma)$ and so on. More precisely, we partition
the Calder\'{o}n-Zygmund balls at level $C_1^{-1}\lambda$ as
\[
 \set{B_j(C_1^{-1} \lambda)} = \bigcup_k \set{B_j(C_1^{-1} \lambda)}_{j \in J_k},
\]
where $J_k$'s are defined as
\begin{align*}
 J_1 &= \Sett{ j}{ B_j(C_1^{-1} \lambda) \subseteq 5 B_1(C_1^{-1}\gamma)} \\
  J_2 &= \Sett{ j }{ B_j(C_1^{-1} \lambda) \subseteq 5 B_2(C_1^{-1}\gamma), j \notin J_1} \\
    J_3 &= \Sett{ j }{ B_j(C_1^{-1} \lambda) \subseteq 5 B_3(C_1^{-1}\gamma), j \notin J_1 \cup J_2 }
\end{align*}
and so on. By (ii), we have
\begin{align}
 \nonumber \lambda \sum_j \mu(B_j(C_1^{-1} \lambda)) &\le  C_1 \sum_{j} \int_{B_j(C_1^{-1} \lambda)} \abs{f}\, d\mu \\
 & \le  C_1 \sum_k \sum_{j \in J_k} \int_{B_j(C_1^{-1} \lambda)} \abs{f}\,d\mu. \label{e-jn2}
\end{align}
In addition for each $k$, we have
\begin{align*}
  \sum_{j \in J_k} \int_{B_j(C_1^{-1} \lambda)} \abs{f}\,d\mu & \le \sum_{j \in J_k} \int_{B_j(C_1^{-1} \lambda)}\abs{ \left( \abs{f} + C_1^{-1} \gamma - \abs{f_{5 B_k(C_1^{-1} \gamma)} } \right)} \,d\mu \\
  & \le \sum_{j \in J_k} \int_{B_j(C_1^{-1} \lambda)} \abs{f - f_{5 B_k(C_1^{-1} \gamma)}} \, d\mu + C_1^{-1} \gamma \sum_{j \in J_k} \mu(B_j(C_1^{-1} \lambda)) \\
  & \le   \int_{5 B_k(C_1^{-1} \lambda)} \abs{f - f_{5 B_k(C_1^{-1} \gamma)}} \, d\mu + C_1^{-1} \gamma \sum_{j \in J_k} \mu(B_j(C_1^{-1} \lambda)) \\
    & \le  \mu(5 B_k(C_1^{-1} \lambda)) + C_1^{-1} \gamma \sum_{j \in J_k} \mu(B_j(C_1^{-1} \lambda)) \\
  & \le C_3 \mu( B_k(C_1^{-1} \gamma)) + C_1^{-1} \gamma \sum_{j \in J_k} \mu(B_j(C_1^{-1} \lambda)).
\end{align*}
The fourth line above follows from $\norm{f}_{\BMO(11B_0),h}=1$. We sum over $k$ and
use \eqref{e-jn2}
\begin{equation*}
\lambda \sum_j \mu(B_j(C_1^{-1} \lambda)) \le C_1 C_3 \sum_{k}  \mu( B_k(C_1^{-1} \gamma)) + \gamma  \sum_j \mu(B_j(C_1^{-1} \lambda))
\end{equation*}
for all $\lambda \ge \gamma \ge 2 C_1 C_3^2$. This implies
\begin{equation*}
(\lambda -\gamma) \sum_{j} \mu(B_j(C_1^{-1} \lambda)) \le C_1 C_3 \sum_k \mu(B_k(C_1^{-1} \gamma))
\end{equation*}
for all $\lambda \ge \gamma \ge 2 C_1 C_3^2$.

In particular if $\lambda \ge a:= 2 C_1 C_3^2$, we have
\begin{equation}
 \label{e-jn4} \sum_j \mu(B_j(C_1^{-1} (\lambda+a))) \le \frac{1}{2} \sum_{k}  \mu( B_k(C_1^{-1} \lambda)).
\end{equation}

Let $\lambda \ge a$ and let $N = \floor{\lambda/a}$. Then we apply
the Calder\'{o}n-Zygmund decomposition at levels $C_1^{-1} a < 2 C_1^{-1} a <\ldots < C_1^{-1} N a$.
By \eqref{e-jn1} and repeated application of \eqref{e-jn4}, we obtain
\begin{align}
 \nonumber \mu \left( \Sett{x \in B_0 }{ \abs{f(x)} > \lambda } \right) & \le \mu \left( \Sett{ x \in B_0}{ \abs{f(x)} > Na }\right)\\
 \nonumber & \le C_3 \sum_j \mu(B_j(C_1^{-1} N a)) \le C_3 2^{-N+1} \sum_{j} \mu(B_j(C_1^{-1}a)) \\
 & \le 2 C_3 2^{-N} \mu(11 B_0) \le 4 C_3^2 2^{-\lambda/a} \mu(B_0)\nonumber \\& \le 4 C_3^2 \exp( - (\lambda \ln 2 )/a) \mu(B_0) \label{e-jn5}
\end{align}
The case $\lambda <a$ follows easily since
\[
 \mu \left( \Sett{x \in B_0 }{ \abs{f(x)} > \lambda } \right) \le \mu(B_0) \le 4 C_3^2 \exp( - (\lambda \ln 2 )/a) \mu(B_0).
\]
The choice $C_2 = \max( 4 C_3^2 , a/\ln 2)$ satisfies \eqref{e-jn}.
\end{proof}
We have the following corollary.
\begin{corollary}
 \label{c-jn}
  Let $(M,d,\mu)$ be a metric measure space satisfying  \ref{doub-loc} and \ref{doub-inf}.
 Let $f$ be a non-negative locally integrable function on $B(x_0,11 r_0)$ for some $r_0 \ge h' >0$. Further we assume that there exists $C_1 \ge 1$ such that
 $f$ satisfies the local Harnack inequality
 \begin{equation} \label{e-localhar}
  f(y) \le C_1 f(z)
 \end{equation}
for all $y,z \in B(x_0,r_0+h')$ satisfying $d(y,z) \le h$.
Then there exists $c_0,C_0 >0$ such that
\begin{equation}
 \label{e-cjn}  \int_{B_0}  e^ {\left( c_0 f(y)/ \norm{f}_{\BMO(11 B_0), h'}  \right) } \, dy \int_{B_0}  e^{\left( -c_0 f(y)/ \norm{f}_{\BMO(11 B_0), h'}  \right)} \, dy \le C_0^2 \mu(B_0)^2
\end{equation}
where $B_0=B(x_0,r_0)$.
The constants $c_0,C_0$ depends only on $C_1,h'$ and
 constants associated with doubling hypotheses \ref{doub-loc} and \ref{doub-inf}.
\end{corollary}
\begin{proof}
There exists $C_2,C_3>0$ such that
 \begin{align*}
\lefteqn{ \int_{B_0}  \exp {\left( c_0  (f(y) - f_{B_0})/ \norm{f}_{\BMO(11 B_0), h'}  \right) } \, dy } \\
  & \le \mu(B_0)+  \sum_{k= 0}^{\infty} \mu \left( \Sett{y \in B_0 }{ k < \frac{f(y) - f_{B_0}}{\norm{f}_{\BMO(11 B_0), h'}} \le k +1  }\right) e^{c_0(k+1)} \\
  & \le \mu(B_0)\left( 1 + C_2 \sum_{k=0}^\infty e^{c_0(k+1)} e^{-k/C_2} \right) \le C_0 \mu(B_0)
 \end{align*}
In the last line above, we fix $c_0=1/(2C_2)$ where $C_2$ is the constant from Proposition \ref{p-jn}. Replacing $f$ by $-f$ in the above inequality and
multiplying those two inequalities yields \eqref{e-cjn}.
\end{proof}

\section{Discrete Calculus}
Before we dive into computations, we introduce  simplifying notations and collect basic rules that mimics calculus rules in a discrete setting.
Let $f$ be a function on $\mathbb{N} \times M$ or on $M$.
Depending on context, we may abbreviate $f(k,x)$ to $f_k(x),f_k$ or even $f$.
\begin{itemize}
\item[1.] `Gradient'
\begin{equation} \label{e-gradxy}
\nabla_{xy} f := f(y) - f(x)
\end{equation}
and the `time derivative'
\begin{equation} \label{e-time-derivative}
\partial_k f(x) := f(k+1,x) - f(k,x).
\end{equation}
\item[2.] Differentiation of product
\begin{equation} \label{e-prs-1}
\nabla_{xy}(fg)=  ( \nabla_{xy}f) g(y) + (\nabla_{xy}g) f(x).
\end{equation}
\item[3.] Differentiation of square
\begin{equation} \label{e-prs-2}
\nabla_{xy}f^2 = 2 (\nabla_{xy}f)f(x)+ (\nabla_{xy}f)^2.
\end{equation}
\item[4.] The same formulas for the `time derivatives':
\begin{equation} \label{e-prt-1}
\partial_k(fg)= (\partial_k f)g_{k+1} + (\partial_k g)f_k
\end{equation}
and
\begin{equation} \label{e-prt-2}
\partial_k(f^2)= 2(\partial_k f)f_{k} + (\partial_k f)^2.
\end{equation}
\item[5.] Let $\Delta = I -P$ denote the Laplacian corresponding to a $\mu$-symmetric Markov operator $P$ with kernel $p_1$. Then
\[
\Delta f(x) :=(I-P)f(x) = \int_M p_1(x,y) \nabla_{yx}f \,dy.
\]
\item[6.] Integration by parts: If $f,g \in L^2(M,\mu)$, then
\begin{equation} \label{e-int-p}
\int_M \Delta f(x) g(x) dx = \frac{1}{2} \int_M \int_M (\nabla_{xy} f)  (\nabla_{xy} g) p_1(x,y) \,dy \,dx.
\end{equation}
\item[7.] Consider a $\mu$-symmetric Markov operator with kernel $p_1$. We define $\abs{\nabla f}$ corresponding to the Markov operator $P$ as
\begin{equation}\label{e-gradp}
\abs{\nabla_P f}^2(x) := \int_M (\nabla_{xy}f)^2 p_1(x,y) \,dy.
\end{equation}
\end{itemize}
We caution the reader to be aware of different uses of the symbol $\nabla$ in \eqref{e-gradh}, \eqref{e-gradxy} and \eqref{e-gradp} with slight change in subscript.
The subscript could be a positive real number, a pair of points or a Markov operator.
We hope the different notations of $\nabla$ would be clear from the context.
\section{Logarithm of a harmonic function}
If $u$ is a positive harmonic function, then we show that $\log u$ has bounded BMO seminorm.
This combined with John-Nirenberg inequality yields  $\phi(u,-q,c_1 B) \le C' \phi(u,q,c_1 B)$ for some $q,C'>0$ and $c_1 \in (0,1)$.
\begin{lemma}
 \label{l-minplu}
 Let $(M,d,\mu)$ be a quasi-$b$-geodesic metric measure space satisfying  \ref{doub-loc}, \ref{doub-inf} and Poincar\'{e} inequality at scale $h$ \ref{poin-mms}.
Suppose that a Markov operator $P$ has a kernel $p$ that is $(h,h')$-compatible with respect to $\mu$ for some $h>b$.
Let $u$ be a positive $P$-harmonic function on $B=B(x,r)$.
Let $\eta$ be a non-negative function on $B$ satisfying $\supp(\eta) \subset B(x,(r/2)-h')$.
There exists $C_0>0$ and $r_0 >2 h'$ satisfies
\begin{equation} \label{e-mp}
 \int_{B/2} \int_{B/2} \left( \ln \frac{u(y)}{u(z)} \right)^2 \eta(z)^2 p_1(y,z) \,dy \,dz \le C_0 \int_{B/2}\int_{B/2} \left(\nabla_{yz} \eta \right)^2  p_1(y,z) \,dy \,dz
\end{equation}
for all balls $B$, for all functions $u$, $\eta$ satisfying the above requirements.
\end{lemma}
\begin{proof}
Define $\psi := \eta^2/u$.  By product rule \eqref{e-prs-1}
\begin{equation}\label{e-mp1}
 \nabla_{yz} \psi = \nabla_{yz}(1/u) \eta(z)^2 + (1/u(y)) \nabla_{yz}( u^2).
\end{equation}
Using integration by parts \eqref{e-int-p} along with $\supp(\eta) \subset B(x,(r/2)-h')$, we deduce
\begin{equation} \label{e-mp2}
 \int_{B/2} \int_{B/2}  p_1(y,z) (\nabla_{yz} \psi) (\nabla_{yz} u)   \, dy \,dz =0.
\end{equation}
Combining \eqref{e-mp1}, \eqref{e-mp2}, we have
\begin{align} \label{e-mp3}
\lefteqn{-\int_{B/2} \int_{B/2} p_1(y,z) (\nabla_{yz}u) \left( \nabla_{yz} \frac{1}{u} \right) \eta(z)^2 \,dy \,dz} \nonumber \\
&\le \int_{B/2} \int_{B/2} p_1(y,z) \abs{\nabla_{yz}u} \abs{\nabla_{yz} \eta^2} \frac{1}{u(y)} \,dy \,dz.
\end{align}
By Lemma \ref{l-compharm}, $u$ satisfies the local Harnack inequality on $B/2$ for large enough balls $B$. Hence there exists $c_1,C_1>0$ and $r_0 >2 h'$ such that
\begin{align}
 \label{e-mp4}- (\nabla_{yz}u) \left( \nabla_{yz} \frac{1}{u} \right) = \frac{ (u(y)-u(z))^2}{u(y)u(z)} &\ge c_1 \left( \ln \frac{u(y)}{u(z)} \right)^2 \\
 \label{e-mp5} \abs{\nabla_{yz}u}/ u(y) &\le C_1 \abs{ \ln \frac{u(y)}{u(z)}}
\end{align}
for all positive $P$-harmonic functions $u$ on $B=B(x,r)$, for all $y,z \in B/2$ with $d(y,z) \le h'$ and $r > r_0$.
Combining \eqref{e-mp3}, \eqref{e-mp4} and \eqref{e-mp5}, we obtain
\begin{align}
 \label{e-mp6}
\lefteqn{\int_{B/2} \int_{B/2} p_1(y,z) \left( \ln \frac{u(y)}{u(z)} \right)^2 \eta(z)^2  \, dy \, dz} \nonumber \\
&\le \frac{C_1}{c_1} \int_{B/2}\int_{B/2}  p_1(y,z)\abs{\nabla_{yz} \eta}(\eta(y)+\eta(z)) \abs{ \ln \frac{u(y)}{u(z)}} \,dy \,dz
\end{align}
Since $p_1(y,z)=p_1(z,y)$ for $\mu \times \mu$-almost every $(y,z) \in M\times M$, we have
\begin{align}
 \label{e-mp7}
 \nonumber \lefteqn{ \int_{B/2}\int_{B/2}  p_1(y,z)\abs{\nabla_{yz} \eta} \eta(y) \abs{ \ln \frac{u(y)}{u(z)}} \,dy \,dz }\\
 &= \int_{B/2}\int_{B/2}  p_1(y,z)\abs{\nabla_{yz} \eta} \eta(z) \abs{ \ln \frac{u(y)}{u(z)}} \,dy \,dz
\end{align}
By \eqref{e-mp6} and \eqref{e-mp7}
\begin{align}
  \label{e-mp8}
\lefteqn{\int_{B/2} \int_{B/2} p_1(y,z) \left( \ln \frac{u(y)}{u(z)} \right)^2 \eta(z)^2  \, dy \, dz}  \nonumber \\
&\le \frac{2C_1}{c_1} \int_{B/2}\int_{B/2}  p_1(y,z)\abs{\nabla_{yz} \eta}\eta(z) \abs{ \ln \frac{u(y)}{u(z)}} \,dy \,dz
\end{align}
By H\"{o}lder inequality
\begin{align} \label{e-mp9}
\nonumber \lefteqn{ \left(\int_{B/2} \int_{B/2} p_1(y,z) \abs{ \nabla_{yz} \eta } \eta(z) \abs{ \ln \frac{u(y)}{u(z)} } \,dy \,dz \right)^2 }\\
& \le &  \int_{B/2} \int_{B/2} p_1(y,z)  \abs{ \nabla_{yz} \eta}^2 \, dy \, dz  \cdot  \int_{B/2} \int_{B/2} p_1(y,z) \left( \ln \frac{u(x)}{u(y)} \right)^2 \eta(z)^2   \, dy \,dz.
\end{align}
Combining \eqref{e-mp8} and \eqref{e-mp9}, we obtain \eqref{e-mp} with $C_0 = 4 C_1^2/c_1^2$.
\end{proof}
In the next proposition, we show that logarithm of a harmonic function has bounded mean oscillation. Then using John-Nirenberg inequality
we prove a weak form of elliptic Harnack inequality.
\begin{prop}
 \label{p-negpos}
 Under the assumptions of Theorem \ref{t-ehi}, there exists $q >0$, $c_0 \in (0,1)$ and $C_0,r_0 >0$  such that
 \begin{equation}
  \label{e-negpos} \phi(u,-q,c_0 B) \le C_0 \phi(u,q,c_0 B)
 \end{equation}
 for all $P$-harmonic functions $u$  on $B=B(x,r)$ with $r \ge r_0$ and for all $x \in M$.
\end{prop}
\begin{proof}
Let $c_1 \in (0,1)$ (its value will be determined later in the proof).
Let $B=B(x,r)$ and let $B_1=B(x_1,r_1) \subseteq c_1 B$ with $r_1 \ge h'$.
For any positive harmonic function $u$ on $B$, by \ref{poin-inf} there exists $C_1,C_2,C_3>1$ such that
\begin{align} \label{e-np1}
 \int_{B_1} \abs{\ln u(y) - (\ln u)_{B_1}}^2 \, dy &\le C_1 r_1^2 \int_{C_2 B_1} \abs{\nabla(\ln u)}^2_h(y) \,dy \nonumber \\
 & \le C_3 r_1^2 \int_{(C_2+1)B_1} \int_{(C_2+1)B_1} p_1(y,z) \left( \ln \frac{u(y)}{u(z)} \right)^2 \, dy \,dz
\end{align}
We used \ref{poin-mms} in the first line and \eqref{e-compat} and $r \ge h'$
We choose $c_1=1/(3 (C_2+2))$, so that $(C_2+1)B_1 \subseteq B/3$ for all $B_1 \subseteq c_1B$.
We define $\eta$ as
\[
 \eta(y) = \max\left( 1 , \min\left(0, \frac{ (C_2+2)r_1 - d(y,x_1)}{r_1} \right) \right).
\]
Note that for large enough $r$, we have $\supp \eta \subseteq (C_2+2)B_1 \subseteq (C_2+2)c_1 B \subseteq B(x, (r/2)-h')$.
Since  $\eta \equiv 1$ on $(C_2+1)B_1$, there exists $C_4,C_5 >0$
\begin{align} \label{e-np2}
\lefteqn{\int_{(C_2+1)B_1} \int_{(C_2+1)B_1} p_1(y,z) \left( \ln \frac{u(y)}{u(z)} \right)^2 \, dy \,dz}\nonumber\\ & \le
\int_{(C_2+2)B_1} \int_{(C_2+2)B_1} p_1(y,z) \left( \ln \frac{u(y)}{u(z)} \right)^2 \eta(z)^2 \, dy \,dz \nonumber \\
&\le C_4 \int_{B/2}\int_{B/2} p_1(y,z) \left(\nabla_{yz} \eta\right)^2 \, dy\,dz  \le C_5 r_1^{-2}  \mu(B_1)
\end{align}
In the last line above we used Lemma \ref{l-minplu}, \eqref{e-compat}, definition of $\eta$, triangle inequality and \eqref{e-vd1}.
By H\"{o}lder inequality
\begin{equation}\label{e-np2h}
 \left(  \int_{B_1} \abs{\ln u(y) - (\ln u)_{B_1}} \, dy \right)^2 \le \mu(B_1) \int_{B_1} \abs{\ln u(y) - (\ln u)_{B_1}}^2 \, dy
\end{equation}

Combining \eqref{e-np1}, \eqref{e-np2} and \eqref{e-np2h} we obtain
\begin{equation}
 \label{e-np3} \norm{ \ln u}_{\BMO(c_1B),h'} \le ( C_3 C_5)^{1/2}
\end{equation}
for all positive harmonic functions $u$ on $B=B(x,r)$ and for all $r$ sufficiently large.
By Lemma \ref{l-compharm}, \eqref{e-np3} and Corollary \ref{c-jn}, there exists $q>0, C_6>0$ such that
\[
 \phi(u,q,(c_1/11)B)^q \phi(u,-q,(c_1/11)B)^{-q} \le C_6^2
\]
for all sufficiently large balls $B$ and for all positive $P$-harmonic functions $u$ on $B$. This immediately yields \eqref{e-negpos}.
\end{proof}
\section{Mean value inequality for subharmonic functions}\label{s-mvi-el}
For the rest of the chapter, we will rely on \ref{doub-inf}, \ref{doub-loc} and the Sobolev inequality \eqref{e-Sob} to prove Theorem \ref{t-ehi}.
We obtain various inequalities on subharmonic functions. The following elementary property of subharmonic and superharmonic functions is useful.
\begin{lemma}\label{l-sub}
Let $P$ be a Markov operator.
 \begin{enumerate}[(a)]
  \item  If $f$ is a non-negative function that is $P$-subharmonic in $B(x,r)$, then $f^p$ is $P$-subharmonic in $B(x,r)$ for all $p \in [1,\infty)$.
  \item  If $f$ is a positive function that is $P$-superharmonic in $B(x,r)$, then $f^p$ is $P$-subharmonic in $B(x,r)$ for all $p <0$.
 \end{enumerate}
\begin{proof}
 If $y \in B(x,r)$, then by Jensen's inequality and the fact that $f$ is $P$-subharmonic in $B(x,r)$
 \[
  f^p(y) \le (P f(y))^p \le (P f^p)(y).
 \]
This proves (a). We again use Jensen's inequality, $f$ is $P$-superharmonic in $B(x,r)$ and $p<0$ to obtain
\[
 f^p(y) \le (Pf(y))^p \le (P f^p) (y)
\]
\end{proof}
\end{lemma}
Moser's iteration relies on repeated application of the following Lemma.
\begin{lemma} \label{l-elel}
Let $(M,d,\mu)$ be a quasi-$b$-geodesic metric measure space satisfying  \ref{doub-loc} and \ref{doub-inf}.
Suppose that a Markov operator $P$ has a kernel $p$ that is $(h,h')$-compatible with respect to $\mu$ for some $h >b$.
Further assume that $P$  satisfies the Sobolev inequality \eqref{e-Sob}.
There exists $C_0>0$ such that
\begin{equation} \label{e-elel}
\phi(u,2(1+2/\delta),B(x,(1-\sigma)r-h') \le C_0 \sigma^{-\delta/(\delta+2)} \phi(u,2,B(x,r+h'))
\end{equation}
for all $x \in M$, for all $r \ge 3h'$, for all $\sigma \in (0,1/2)$ and
for all functions $u$ that are non-negative and $P$-subharmonic on $B(x,r)$.
\end{lemma}
\begin{proof}
 Define
 \begin{equation}\label{e-ee0}
  \psi(y) := \max \left( 0, \min\left(1,  \frac{r - d(x,y)}{\sigma r }\right)\right).
 \end{equation}
Note that $\psi \equiv 0$ in $B(x,r)^\complement$ and $\psi \equiv 1$ in $B(x,(1-\sigma)r)$.
Since $\Delta u \le 0$ in $B(x,r)$  and $u \ge 0$, we have
\begin{align}
 \nonumber 0 &\le - \int_{B(x,r)} \psi^2(y) u(y) \Delta u(y) \, dy \\
 \nonumber & =  -\frac{1}{2} \int_{B(x,r+h')} \int_{B(x,r+h')} p_1(y,z)  \left( \nabla_{yz}(\psi^2u) \right) (\nabla_{yz} u ) \, dy \, dz \\
 \nonumber &= - \frac{1}{2} \int_{B(x,r+h')} \int_{B(x,r+h')} p_1(y,z) \psi^2(y)  \left( \nabla_{yz} u\right)^2 \, dy \, dz   \\
 \label{e-ee1} & \hspace{5mm} -\frac{1}{2} \int_{B(x,r+h')} \int_{B(x,r+h')} p_1(y,z) u(z) \left( \nabla_{yz} \psi^2 \right) (\nabla_{yz} u ) \, dy \, dz.
\end{align}
The above steps follows from integration by parts \eqref{e-int-p} and product rule \eqref{e-prs-1}. We use the inequality $ab \le a^2/4 + b^2$ to obtain
\begin{align}
 \nonumber \abs{ u(z) \left( \nabla_{yz} \psi^2 \right) (\nabla_{yz} u )} & = \abs{ (\psi(y) +\psi(z)) u(z) ( \nabla_{yz} \psi)  (\nabla_{yz}  u)} \\
 \label{e-ee2} & \le \frac{1}{4} (\psi^2(y)+\psi^2(z)) \left( \nabla_{yz} u\right)^2 + 2 u^2(z) \left( \nabla_{yz} \psi \right)^2.
\end{align}
Since $p_1(y,z) =p_1(z,y)$ for $\mu\times\mu$-almost every $(y,z)$, we have
\begin{equation}
 \label{e-ee3} \int_{B_1}  \int_{B_1} p_1(y,z) \psi^2(y)  \left( \nabla_{yz} u\right)^2 \, dy \,dz =  \int_{B_1} \int_{B_1} p_1(y,z) \psi^2(z)  \left( \nabla_{yz} u\right)^2 \, dy \,dz
\end{equation}
where $B_1:= B(x,r+h')$.
Combining \eqref{e-ee1}, \eqref{e-ee2} and \eqref{e-ee3}
\begin{equation}\label{e-ee4}
 \int_{B_1} \int_{B_1} p_1(y,z) \psi^2(y)  \left( \nabla_{yz} u\right)^2 \, dy \,dz \le 4 \int_{B_1} \int_{B_1} p_1(y,z) u^2(z)  \left( \nabla_{yz} \psi \right)^2 \, dy \,dz.
\end{equation}
The inequality $ (a+b)^2 \le 2(a^2 +b^2)$ along with product rule \eqref{e-prs-1} implies
\begin{align}\label{e-ee5}
 \nonumber  \int_{B_1} \int_{B_1} p_1(y,z) \left( \nabla_{yz}(\psi u) \right)^2 \, dy \,dz   & \le 2  \int_{B_1} \int_{B_1} p_1(y,z) \psi^2(y) \left( \nabla_{yz}u \right)^2 \, dy \,dz\\
& \hspace{2mm} + 2 \int_{B_1} \int_{B_1} p_1(y,z) u^2(z) \left( \nabla_{yz} \psi \right)^2 \, dy \,dz.
\end{align}
Combining \eqref{e-ee4} and \eqref{e-ee5}, we obtain
\begin{equation} \label{e-ee6}
 \int_{B_1} \int_{B_1} p_1(y,z) \left( \nabla_{yz}(\psi u) \right)^2 \, dy \,dz \le 10 \int_{B_1} \int_{B_1} p_1(y,z) u^2(z) \left( \nabla_{yz}\psi  \right)^2 \, dy \,dz.
\end{equation}
By \eqref{e-ee0} and \eqref{e-compat}, there exists $C_1>0$ such that
\[
  \left( \nabla_{yz} \psi\right)^2 p_1(y,z) \le (h')^2 \sigma^{-2} r^{-2} p_1(y,z)
\]
for all $y \in M$ and for $\mu$-almost every $z \in M$. Combined with \eqref{e-ee6}, we have
\begin{equation} \label{e-ee7} 
  \int_{B_1} \int_{B_1} p_1(y,z) \left( \nabla_{yz}(\psi u) \right)^2 \, dy \,dz \le 10 (h')^2 \sigma^{-2} r^{-2}   \int_{B_1} u^2(z) \, dz .
\end{equation}
We define
\[
 u_1:=P_{B(x, (1-\sigma)r)} u, \hspace{1cm} u_2:= P_{B_1}(\psi u).
\]
Since $\psi \equiv 1$ in $B(x,(1-\sigma)r)$, by \eqref{e-compat} we have
\begin{equation}\label{e-ee8}
 u_2(y)=u_1(y)=P u (y) = u(y) - \Delta u(y) \ge u(y)
\end{equation}
for all $y \in B(x,(1-\sigma)r -h' )$.
By \eqref{e-ee8} along with H\"{o}lder inequality, we have
\begin{align}\label{e-ee9}
\nonumber \lefteqn{\int_{B(x,(1-\sigma)r-h')} u^{2(1+(2/\delta))} \, d\mu \le  \int_{B(x,(1-\sigma)r)} u_1^{2(1+(2/\delta))} \, d\mu} \\
& \le \nonumber \left( \int_{B(x,(1-\sigma)r)} u_1^2 \, d\mu \right)^{2/\delta} \left( \int_{B(x,(1-\sigma)r)} u_1^{(2\delta)/(\delta-2)} \, d\mu \right)^{(\delta-2)/\delta} \\
&\le  \left( \int_{B(x,(1-\sigma)r)} u^2 \, d\mu \right)^{2/\delta} \left( \int_{B(x,(1-\sigma)r)} u_2^{(2\delta)/(\delta-2)} \, d\mu \right)^{(\delta-2)/\delta}
\end{align}
In \eqref{e-ee9}, we used that $P_{B(x,(1-\sigma)r)}$ is a contraction in $L^2$ and that $u_2 \ge u_1$ in $B(x,(1-\sigma)r)$. By Sobolev inequality \eqref{e-Sob}, Lemma \ref{l-dircomp-ball}(a)
and integration by parts \eqref{e-int-p}
\begin{align}
 \label{e-ee10}
 \nonumber\lefteqn{\left( \int_{B_1} u_2^{(2 \delta)/(\delta-2)} \, d\mu \right)^{(\delta - 2)/\delta} } \\
 \nonumber &\le C_S \frac{(r+h')^2}{2 V(x,r+h')^{2/\delta}} \int_{B_1} \int_{B_1} p_1(y,z) \left( \nabla_{yz} (\psi u) \right)^2 \, dy\, dz  \\
 & \hspace{6mm}+ C_S \frac{1}{V(x,r+h')^{2/\delta}} \int_{B_1} (\psi u)^2 \, d\mu
\end{align}
By using \eqref{e-ee9},  \eqref{e-ee10},  \eqref{e-ee7}, $\psi \le 1 $, $r \ge 3 h'$ and \eqref{e-vd1}, there exists $C_2>0$ such that
\begin{equation*}
 \dashint_{B(x,(1-\sigma)r-h')}  u^{2 (1+(2/\delta))} \, d\mu \le {C_2 \sigma^{-2} }\left( \dashint_{B(x,r+h')} u^2 \, d\mu \right)^{1+(2/\delta)}.
\end{equation*}
This immediately yields \eqref{e-elel}.
\end{proof}
We modify the proof of the Lemma \ref{l-elel} to obtain a reverse Poincar\'{e} inequality for all $P$-harmonic functions (not necessarily non-negative).
The below reverse Poincar\'{e} inequality  and its proof is essentially same as \eqref{e-ee7}. 
\begin{lemma}[Reverse Poincar\'{e} inequality]\label{l-rp}
Let $(M,d,\mu)$ be a quasi-$b$-geodesic metric measure space satisfying  \ref{doub-loc} and \ref{doub-inf}.
Suppose that a Markov operator $P$ has a kernel $p$ that is weakly $(h,h')$-compatible with respect to $\mu$ for some $h >b$.
For all $\Omega > 1$, there exists $C = C(\Omega)$ such that for all $P$-harmonic functions $u$, for all $x \in M$ and for all $r >  3 h'/(\Omega -1)$
\begin{equation}\label{e-rpoin}
 \int_{B(x,r)} \abs{ \nabla_P u} ^2 \, d\mu \le C r^{-2} \int_{B(x,\Omega r)} u^2 \, d\mu .
\end{equation}
In particular, there exists $C_R =C(2)$ such that 
 such that for all $P$-harmonic functions $u$, for all $x \in M$ and for all $r >  3 h'$
\begin{equation}\label{e-rpoin1}
 \int_{B(x,r)} \abs{ \nabla_P u} ^2 \, d\mu \le C_R r^{-2} \int_{B(x,2 r)} u^2 \, d\mu.
\end{equation}
\end{lemma}
\begin{proof} We repeat the steps in the proof of Lemma \ref{l-elel}.
Define 
\begin{equation} \label{e-rp0}
 \psi(y) := \max \left( 0 , \min \left( 1 , \frac{\Omega r - h' - d(x,y)}{ (\Omega -1)r - 2 h'}\right) \right).
\end{equation}
Note that $\psi \equiv 0$ in $B(x,\Omega r - h')^\complement$ and $\psi \equiv 1$ in $B(x,r+h')$. Since $\Delta u = (I-P)u =0$,  for all $r > 3 h'/(\Omega -1)$ and for all $x \in M$  we have
\begin{align}
 \nonumber 0 &= - \int_{M} \psi^2(y) u(y) \Delta u(y) \, dy = - \int_{B(x,\Omega r -h')} \psi^2(y) u(y) \Delta u(y) \, dy\\
 \nonumber & =  -\frac{1}{2} \int_{B(x,\Omega r)} \int_{B(x,\Omega r)} p_1(y,z)  \left( \nabla_{yz}(\psi^2u) \right) (\nabla_{yz} u ) \, dy \, dz \\
 \nonumber &= - \frac{1}{2} \int_{B(x,\Omega r)} \int_{B(x,\Omega r)} p_1(y,z) \psi^2(y)  \left( \nabla_{yz} u\right)^2 \, dy \, dz   \\
 \label{e-rp1} & \hspace{5mm} -\frac{1}{2} \int_{B(x,\Omega r)} \int_{B(x,\Omega r)} p_1(y,z) u(z) \left( \nabla_{yz} \psi^2 \right) (\nabla_{yz} u ) \, dy \, dz.
\end{align}
The above steps follows from integration by parts \eqref{e-int-p} and product rule \eqref{e-prs-1}. We use the inequality $ab \le a^2/4 + b^2$ to obtain
\begin{align}
 \nonumber \abs{ u(z) \left( \nabla_{yz} \psi^2 \right) (\nabla_{yz} u )} & = \abs{ (\psi(y) +\psi(z)) u(z) ( \nabla_{yz} \psi)  (\nabla_{yz}  u)} \\
 \label{e-rp2} & \le \frac{1}{4} (\psi^2(y)+\psi^2(z)) \left( \nabla_{yz} u\right)^2 + 2 u^2(z) \left( \nabla_{yz} \psi \right)^2.
\end{align}
Since $p_1(y,z) =p_1(z,y)$ for $\mu\times\mu$-almost every $(y,z)$, we have
\begin{equation}
 \label{e-rp3} \int_{B_1}  \int_{B_1} p_1(y,z) \psi^2(y)  \left( \nabla_{yz} u\right)^2 \, dy \,dz =  \int_{B_1} \int_{B_1} p_1(y,z) \psi^2(z)  \left( \nabla_{yz} u\right)^2 \, dy \,dz
\end{equation}
where $B_1:= B(x,\Omega r)$.
Combining \eqref{e-rp1}, \eqref{e-rp2} and \eqref{e-rp3}
\begin{equation}\label{e-rp4}
 \int_{B_1} \int_{B_1} p_1(y,z) \psi^2(y)  \left( \nabla_{yz} u\right)^2 \, dy \,dz \le 4 \int_{B_1} \int_{B_1} p_1(y,z) u^2(z)  \left( \nabla_{yz} \psi \right)^2 \, dy \,dz.
\end{equation}
The inequality $ (a+b)^2 \le 2(a^2 +b^2)$ along with product rule \eqref{e-prs-1} implies
\begin{align}\label{e-rp5}
 \nonumber  \int_{B_1} \int_{B_1} p_1(y,z) \left( \nabla_{yz}(\psi u) \right)^2 \, dy \,dz   & \le 2  \int_{B_1} \int_{B_1} p_1(y,z) \psi^2(y) \left( \nabla_{yz}u \right)^2 \, dy \,dz\\
& \hspace{2mm} + 2 \int_{B_1} \int_{B_1} p_1(y,z) u^2(z) \left( \nabla_{yz} \psi \right)^2 \, dy \,dz.
\end{align}
Combining \eqref{e-rp4} and \eqref{e-rp5}, we obtain
\begin{equation} \label{e-rp6}
 \int_{B_1} \int_{B_1} p_1(y,z) \left( \nabla_{yz}(\psi u) \right)^2 \, dy \,dz \le 10 \int_{B_1} \int_{B_1} p_1(y,z) u^2(z) \left( \nabla_{yz}\psi  \right)^2 \, dy \,dz.
\end{equation}
By \eqref{e-rp0} and \eqref{e-compat}, there exists $C_1>0$ such that
\[
  \left( \nabla_{yz} \psi\right)^2 p_1(y,z) \le (3h')^2 (\Omega -1)^{-2} r^{-2} p_1(y,z)
\]
for all $y \in M$, for $\mu$-almost every $z \in M$ and for all $r > 3h'/(\Omega -1)$. Combined with \eqref{e-ee6}, we have
\begin{equation} \label{e-rp7} 
  \int_{B_1} \int_{B_1} p_1(y,z) \left( \nabla_{yz}(\psi u) \right)^2 \, dy \,dz \le  (3 h')^2 (\Omega -1)^{-2} r^{-2}  \int_{B_1} u^2(z) \, dz .
\end{equation}
for all $P$-harmonic functions $u$, for all $r > 3h'/(\Omega -1)$ and for all $x \in M$. Since $\psi \equiv 1$ in $B(x,r+h')$
the desired inequality \eqref{e-rpoin} follows from \eqref{e-rp7}.
\end{proof}

The next lemma is a $L^2$-mean value inequality for positive $P$-subharmonic functions.
\begin{lemma}
\label{l-2toinf}
Let $(M,d,\mu)$ be a quasi-$b$-geodesic metric measure space satisfying  \ref{doub-loc} and \ref{doub-inf}.
Suppose that a Markov operator $P$ has a kernel $p$ that is $(h,h')$-compatible with respect to $\mu$ for some $h >b$.
Further assume that $P$ satisfies the Sobolev inequality \eqref{e-Sob}.
There exists $C_1>0$ and $r_1>0$ such that
\begin{equation} \label{e-2toinf}
\phi(u,\infty,B(x,r/6)) \le C \phi(u,2,B(x,r+h'))
\end{equation}
for all $x \in M$, for all $r \ge r_1$ and
for all functions $u$ that are non-negative and $P$-subharmonic on $B(x,r)$.
\end{lemma}
\begin{proof}
 Define a sequence of radii iteratively by $r(1)=r+h'$,
 \[
  r(i+1) = (r(i)-h')\left(1- \frac{1}{ 3^{i+1}}\right)- h'
 \]
for $i=1,2,\ldots, \lceil \log r \rceil$. By the above definition, there exists $r_0>0$ such that
\begin{equation} \label{e-2i1}
  r( \lceil \log r \rceil +2 ) - h'   \ge  r \left(1 -  \sum_{j=1}^i 3^{-(i+1)}  \right) - 4 h' ( \log r + 3 )  \ge r/2 \ge 3 h'
\end{equation}
for all $r \ge r_0$. We define the balls $B_i= B(x,r(i))$ for $i \in \N^*$ and exponents $p_i = (1 + 2/\delta)^i$ for $i \in \N_{\ge 0}$.
By Lemma \ref{l-sub} $u^{p_i}$ is $P$-subharmonic for all $i \in \N_{\ge 0}$.
By applying Lemma \ref{l-elel} to the function $u^{p_{i-1}}$ that is $P$-subharmonic in $B_i$, we obtain
\begin{equation}
\label{e-2i2} \phi(u,2p_{i},B_{i+1}) \le C_0^{1/p_{i-1}} 3^{-(i+1)/p_{i}} \phi(u, 2p_{i-1} , B_{i})
\end{equation}
for $i=1,2,\ldots,\lceil \log r \rceil$ and $r \ge r_0$. Combining   the estimates in \eqref{e-2i2}, there exists $C_2>0$ such that
\begin{equation}\label{e-2i3}
\phi(u,2p_{\lceil \log r \rceil},B_{\lceil \log r \rceil+1}) \le  C_2 \phi(u, 2, B(x,r+h'))
\end{equation}
for all $x \in M$, for all $r \ge r_0$ and
for all non-negative subharmonic $u$ in $B(x,r)$.
There exists $C_3,C_4>0$ such that
\begin{align} \label{e-2i4}
 \sup_{B(x,r/2)} u^{2 p_{\lceil \log r \rceil}} &\le   \sup_{B(x,r/2)} P (u^{2 p_{\lceil \log r \rceil}}) \nonumber \\
 & \le   \sup_{y \in B_{\lceil \log r \rceil + 1}} \frac{C_3}{V(y,h')} \int_{B_{\lceil \log r \rceil}+1} u^{2 p_{\lceil \log r \rceil}}  \, d\mu  \nonumber\\
 & \le  \frac{C_3 }{\mu(B_{\lceil \log r \rceil +1})} \left( \sup_{y \in B(x,r)} \frac{V(y,2r)}{V(y,h')} \right) \int_{B_{\lceil \log r \rceil+1}} u^{2 p_{\lceil \log r \rceil}} \, d\mu \nonumber \\
 & \le C_4 r^\delta \dashint_{B_{\lceil \log r \rceil+1}} u^{2 p_{\lceil \log r \rceil}} \, d\mu
\end{align}
The first line above follows from Lemma \ref{l-sub}, the second line follows from \eqref{e-compat} and \eqref{e-2i1}, the third line follows from \eqref{e-2i1}
and the last line from \eqref{e-vd1} and \eqref{e-2i1}.
Combining \eqref{e-2i3} and \eqref{e-2i4}, we obtain \eqref{e-2toinf}.
\end{proof}
The next lemma is analogous to Lemma \ref{l-elel} and will be used for an iteration procedure.
\begin{lemma} \label{l-it2}
Let $(M,d,\mu)$ be a quasi-$b$-geodesic metric measure space satisfying  \ref{doub-loc} and \ref{doub-inf}.
Suppose that a Markov operator $P$ has a kernel $p$ that is $(h,h')$-compatible with $(M,d,\mu)$ for some $h >b$.
Further assume that $P$ satisfies the Sobolev inequality \eqref{e-Sob}.
There exists $C_0>0,r_0>0$ such that
\begin{align}\label{e-deh}
\lefteqn{\int_{B(x,r/2)} \int_{B(x,r/2)} \psi(y)^2 \abs{\nabla_{yz}(u^p)} p_1(y,z) \,dy \, dz } \nonumber \\
& \le  C_0 \left( \frac{2p}{2p-1} \right)^2 \int_{B(x,r/2)} \int_{B(x,r/2)} u(y)^{2p} \abs{ \nabla_{yz} \psi } p_1(y,z) \,dy \,dz
\end{align}
for all $x \in M$, for all $r \ge r_0$, for all $p \in (0,1] \setminus \set{1/2}$, for all positive functions $u$ that are $P$-harmonic on $B(x,r)$
and for all $\psi \ge 0$ with $\supp(\psi) \subseteq B(x,r/2-h')$.
\end{lemma}
\begin{proof}
Let $\eta := u^{2p-1} \psi$, where $\psi \ge 0$ satisfies $\supp(\psi)\subseteq B(x,r/2-h')$ and $u >0$ is $P$-harmonic in $B(x,r)$. By  product rule \eqref{e-prs-1}
\[
\nabla_{yz} \eta= \left( \nabla_{yz} (u^{2p-1}) \right) \psi(y)^2 + u(z)^{2p-1} \left( \nabla_{yz} \psi^2 \right).
\]
By integration by parts  \eqref{e-int-p}, we obtain
\begin{align}
\label{e-deh1} \lefteqn{ \int_B \int_B p_1(y,z) ( \nabla_{yz} u) \left( \nabla_{yz} (u^{2p-1}) \right) \psi(y)^2 \,dy \,dz }\\
\nonumber & = -\int_B \int_B p_1(y,z) \left( \nabla_{yz}  u\right)u(z)^{2p-1}  \left( \nabla_{yz} (\psi^2)\right) \,dy \,dz
\end{align}
where $B:=B(x,r/2)$.
There exists $C_1>0$ such that
\begin{align} \label{e-deh2}
{\abs{2p -1}} \left( \nabla_{yz}(u^p) \right)^2 &\le {p^2} (\nabla_{yz} u)(\nabla_{yz}(u^{2p-1})) \\
\label{e-deh3}
 \abs{\nabla_{yz} u} u(z)^{p-1} &\le C_1 p^{-1} \abs{\nabla_{yz} (u^p)}.
\end{align}
for all $p \in (0,1]$, for all $y,z \in M$ with $d(y,z) \le h'$ and for all positive $u$.
The estimate   \eqref{e-deh2} is elementary and is a version of  Stroock-Varopoulos inequality. The proof of \eqref{e-deh2} is essentially contained in \cite[Lemma 2.4]{MOS13}.
The estimate \eqref{e-deh3} follows from mean value theorem and the local Harnack inequality given by Lemma \ref{l-compharm}.
Combining \eqref{e-deh1}, \eqref{e-deh2} and \eqref{e-deh3}, we have
\begin{align} \label{e-deh4}
\nonumber \lefteqn{ C_1^{-1}\frac{\abs{2p-1}}{p}\int_B \int_B p_1(y,z) \psi(y)^2 \abs{ \nabla_{yz}(u^p)}^2 \,dy \,dz } \\
& \le  \int_B \int_B p_1(y,z) u(z)^p \abs{\nabla_{yz} \psi } \abs{\psi(y)+\psi(z)} \abs{ \nabla_{yz}(u^p)} \,dy \,dz \nonumber \\
& \le \left( \int_B \int_B p_1(y,z) u(y)^{2p} \abs{ \nabla_{yz}\psi }^2 \,dy\,dz \right)^{1/2} \nonumber \\
& \hspace{4mm}\times  \left( \int_B \int_B p_1(y,z) 2(  \psi(y)^2+\psi(z)^2) \abs{ \nabla_{yz}(u^p)}^2 \, dy \,dz \right)^{1/2}.
\end{align}
We use  Cauchy-Schwarz inequality and $(a+b)^2 \le 2(a^2+b^2)$ in the last step.
By the $\mu\times\mu$-almost everywhere symmetry of $p_1$, we have
\begin{equation}\label{e-deh5}
 \int_B \int_B p_1(y,z)  \psi(z)^2 \abs{ \nabla_{yz}(u^p)}^2 \,dy \,dz  = \int_B \int_B p_1(y,z)  \psi(y)^2 \abs{ \nabla_{yz}(u^p)}^2 \,dy \,dz.
\end{equation}
Combining \eqref{e-deh4} and \eqref{e-deh5} yields \eqref{e-deh}.
\end{proof}
We do another iteration procedure between the exponents $q$ and $2$ using Lemma \ref{l-it2}.
\begin{lemma}\label{l-qto2}
 Let $(M,d,\mu)$ be a quasi-$b$-geodesic metric measure space satisfying  \ref{doub-loc} and \ref{doub-inf}.
Suppose that a Markov operator $P$ has a kernel $p$ that is $(h,h')$-compatible  to $(M,d,\mu)$ for some $h >b$.
Further assume that $P$ Sobolev inequality \eqref{e-Sob}.
For any fixed $q>0$, there exists $C_1>0,c_1 \in (0,1/2)$ and $r_1>0$ such that
\begin{equation} \label{e-qto2}
\phi(u,2,B(x,c_1r)) \le C_1 \phi(u,q,B(x,r/2))
\end{equation}
for all $x \in M$, for all $r \ge r_1$ and
for all functions $u$ that are non-negative and $P$-subharmonic on $B(x,r)$.
\end{lemma}
\begin{proof}
 If $q \ge 2$, then \eqref{e-qto2} follows from Jensen's inequality. Hence it suffices to consider $q \in (0,2)$.

 Define $\theta:=\delta/(\delta-2)$.
  We slightly decrease $q$ if necessary so that $q \theta^k  \neq 1/2$ for all $i \in \N$.
Define $k \in \N^*$ as the integer that satisfies $q \theta^{k-1} < 2\le q \theta^k$.
Define $c_1:=4^{-k}$ and  iteratively define
\[
 s_{i}:=2 s_{i-1} + 2h'
\]
for $i=1,\ldots,k$, where  $s_0:=c_1 r$. Fix $r_0>0$ such that
$s_k \le r/2-h'$ for all $r \ge r_0$ where $k$ and $s_k$ are defined as above.

Define $q_i:= q \theta^i/2$, $B_i=B(x,s_{k-i})$ for $i=0,1,\ldots,k$.
Define the functions
\[
 \psi_i(y) = \max \left( 0, \min\left(1, \frac{2s_{k-i-1}+h' - d(x,y)}{s_{k-i-1}} \right) \right)
\]
for $i=0,1,\ldots,k-1$. Note that $\psi_i \equiv 1$ in $B(x,s_{k-i-1}+h')$ and $\psi \equiv 0$ in $B(x,s_{k-i}-h')^\complement$.

By Sobolev inequality \eqref{e-Sob} there exists $C_2>0$ such that
\begin{align} \label{e-it1}
\nonumber \left( \int_{B_i} \left( P_{B_{i}} ( \psi_i u^{q_i})(y) \right) ^{2 \theta} \,dy \right)^{1/\theta}
 & \le \frac{C_2 s_{k-i}^2}{\mu( B_{i})^{2/\delta}} \int_{B_{i}} \int_{B_{i}} p_1(y,z) \abs{ \nabla_{yz} (\psi_i u^{q_i})}^2 \,dy \,dz \\
& \hspace{3mm}+ \frac{C_2}{\mu( B_{i})^{2/\delta}} \int_{B_{i}}  \psi_i(y)^2 u(y)^{2q_i} \,dy
\end{align}
for all $i=0,1,\ldots,k-1$. By \eqref{e-compat} and Lemma \ref{l-compharm} there exists $C_3>0$ such that
\begin{equation*}
 P_{B_{i}} ( \psi u^{q_i})(y) = \int_{B(y,h')} u^{q_i}(z) p_1(y,z) \,dz \ge C_3^{-q_i}u^{q_i}(y)
\end{equation*}
for all $y \in B_{i+1}$.
Therefore
\begin{equation}
 \label{e-it2} \left( \int_{B_{i+1}}   u(y)^{2q_{i+1}}  \,dy \right)^{1/\theta}   \le   \left(C_3^{q_i} \int_{B_i} \left( P_{B_{i+1}} ( \psi u^{q_i})(y) \right) ^{2 \theta} \,dy \right)^{1/\theta}
\end{equation}
for $x \in B_{i+1}$.
There exists $C_4,C_5,C_6>0$ such that
\begin{align}\label{e-it3}
\nonumber \lefteqn{\int_{B_{i}} \int_{B_{i}} p_1(y,z) \abs{ \nabla_{yz} (\psi u^{q_i}) }^2 \,dy \,dz} \\
 \nonumber& \le  2 \int_{B_{i}} \int_{B_{i}} p_1(y,z) \psi(y)^2 \abs{ \nabla_{yz} ( u^{q_i})}^2 \,dy \,dz   + 2 \int_{B_{i}} \int_{B_{i}}p_1(y,z) \abs{ \nabla_{yz} \psi}^2 u(z)^{2 q_i} \, dy \,dz \\
 \nonumber& \le  C_4 \left[ \left( \frac{2q_i}{2q_i-1} \right)^2 +1  \right] \int_{B_{i}} \int_{B_{i}} p_1(y,z) \abs{ \nabla_{yz} \psi_i }^2 u(y)^{2 q_i} \,dy \,dz  \\
\nonumber & \le \frac{C_5}{s_{k-i-1}^2} \left[ \left( \frac{2q_i}{2q_i-1} \right)^2 +1  \right] \int_{B_{i}}  u(z)^{2 q_i} \,dz \\
& \le \frac{C_6}{s_{k-i-1}^2} \int_{B_{i}}  u(z)^{2 q_i} \,dz.
\end{align}
In the first step above, we used product rule \eqref{e-prs-1} and the inequality $(a+b)^2 \le 2(a^2+b^2)$.
In the second step  we use Lemma \ref{l-it2} and in the third step we use \eqref{e-compat}.
In the last step, we simply bound $2q_i/ \abs{2q_i-1}$ by $\max_{ 0 \le i \le k} 2p_i/ \abs{2p_i-1}< \infty$.

Combining \eqref{e-it1}, \eqref{e-it2}, \eqref{e-it3} along with $s_{k-i}/s_{k-i-1} \le 4^k$  yields
\begin{equation*}
 \left( \int_{B_{i+1}}   u(y)^{2q_{i+1}}   \,dy \right)^{1/\theta}   \le   \frac{C_7}{ \mu( B_{i})^{2/\delta} }  \int_{B_{i}}  u(y)^{2q_i} \,dy
\end{equation*}
for some $C_7>0$. Combined with $r \ge r_0$ and \eqref{e-vd1}, we deduce
\begin{equation}\label{e-it4}
 \phi(u,2q_{i+1},B_{i+1}) \le C_8  \phi(u,2q_{i},B_{i})
\end{equation}
for $i=0,1,\ldots,k-1$, for all $x \in M$, for all $r \ge r_0$ and for all $P$-harmonic $u>0$.
The estimates \eqref{e-it4} along with Jensen's inequality implies \eqref{e-qto2} with $C_1= C_8^{k}$ and $c_1=4^{-k}$.
\end{proof}
We are now ready to prove elliptic Harnack inequality.
\begin{proof}[Proof of Theorem \ref{t-ehi}]
It suffices to consider the case $u>0$ because we can replace $u \ge 0$ by $u+\epsilon$ and let $\epsilon \downarrow 0$.

Note that we have Sobolev inequality \eqref{e-Sob} by Theorem \ref{t-Sob}.
 There exists $r_0 >0$ $C_i>0,c_i \in (0,1)$ for $1 \le i \le 5$ such that for all $x \in M$ and for all $r \ge r_0$ and for all positive functions $u$ that are $P$-harmonic on $B:=B(x,r)$
 \begin{align*}
  \phi(u,\infty,c_1 B) &\le C_1 \phi(u, c_2, B) \\
& \le  C_2 \phi(u, q , c_3 B) \\
& \le C_3 \phi(u,-q,c_4,B) \\
 &\le C_4 \phi(u, -\infty , c_5 B).
 \end{align*}
The first line above follows from Lemma \ref{l-2toinf}, the second line above follows from Lemma \ref{l-qto2} and the third line follows from Proposition \ref{p-negpos}.
The last line follows from applying Lemma \ref{l-2toinf} to the function $u^{-q/2}$ which is subharmonic by Lemma \ref{l-sub}(b).
Choosing $c= \min(c_1,c_5)$ yields the elliptic Harnack inequality.
\end{proof}
The constant $c \in (0,1)$ in \eqref{e-ehi} is flexible.
More precisely, we can slightly improve the conclusion of Theorem \ref{t-ehi} for $b$-geodesic spaces by an easy chaining.
\begin{corollary}[Elliptic Harnack inequality]
 \label{c-ehi}
Let $(M,d,\mu)$ be a  $b$-geodesic space satisfying  \ref{doub-loc}, \ref{doub-inf} and Poincar\'{e} inequality \ref{poin-mms} at scale $h$.
Suppose that a Markov operator $P$ has a kernel $p$ that is $(h,h')$-compatible with $(M,d,\mu)$ for some $h >b$.
Then for all $c \in (0,1)$, there exists $r_{0}>0,C_{E}>0$ such that for all $x \in M$, for all
$r \ge r_0$ and for all
non-negative  functions $u:B(x,r) \to \R_{\ge 0}$ that are $P$-harmonic in  $B(x,r)$  the following Harnack inequality holds:
\begin{equation}
 \label{e-cehi} \sup_{x \in B(x,cr)} u \le C_E \inf_{x \in B(x,cr)} u.
\end{equation}
\end{corollary}
The above corollary is a consequence of Theorem \ref{t-ehi} applied repeatedly to a sequence of points in an approximate geodesic.
We do not use the above corollary. The proof of Corollary \ref{c-ehi} is left to the reader.
\section{Applications of elliptic Harnack inequality}
We present two immediate and well-known applications of elliptic Harnack inequality.
\begin{prop}[Liouville property]\label{p-liov}
 Assume that $(M,d,\mu)$ is a quasi-$b$-geodesic metric measure space satisfying  \ref{doub-loc}, \ref{doub-inf} and Poincar\'{e} inequality \ref{poin-mms} at scale $h$ .
Suppose that a Markov operator $P$ has a kernel $p$ that is $(h,h')$-compatible with $(M,d,\mu)$ for some $h >b$.
Then all non-negative $P$-harmonic functions are constant (strong Liouville property). Therefore all bounded harmonic functions are constant (weak Liouville property).
\end{prop}
\begin{proof}
 Let $u$ be a non-negative harmonic function. Then $v=u - \inf u$ is a non-negative harmonic function with $\inf v=0$.
 By elliptic Harnack inequality, there exists $c \in (0,1)$ and $C >1$ such that $\sup_{B(x,cr)} v \le C \inf_{B(x,cr)} v$ for all large enough $r$.
 Letting $r \to \infty$, we have $\sup_{M} v \le 0$ which implies $v \equiv 0$. This proves strong Liouville property.
 The weak Liouville property follows from the observation that for any bounded harmonic function $h$, the function $h - \inf h$ is a non-negative harmonic function.
\end{proof}
The following H\"{o}lder regularity-type estimate is a direct consequence of elliptic Harnack inequality. Our argument is an adaptation of Moser's argument \cite[Section 5]{Mos61} which uses an oscillation inequality.
\begin{prop}\label{p-holder}
 There exists $c \in (0,1)$, $\alpha >0$ , $C >0$ and $r_1> 0$ such that
 \begin{equation} \label{e-holder}
 \abs{u(y)-u(z)} \le C \left( \frac{\max(d(y,z),1)}{r}\right)^\alpha \sup_{B(x,r)} u
 \end{equation}
 for all $y,z \in B(x,cr)$,
for all $x \in M$, for all $r \ge r_1$ and
 for all non-negative functions $u:M \to \R$ that is $P$-harmonic on $B(x,r)$ with $B(x,r) \neq M$.
\end{prop}
\begin{proof}
Let $c,r_0,C_E$ be constants from from Theorem \ref{t-ehi}.
We optionally decrease the $c$  so that $c \le 1/4$.
Let $B=B(x,r)$ be an arbitrary ball with $r \ge r_0$, $B(x,r) \neq M$ and $y,z \in B(x,cr)$.
Define a sequence of balls by
\[
 s_i := c^{-1+i} r_1, \hspace{1cm} B_i:=B(y,s_i)
\]
for $i \in \N^*$, where $s_1:= \max( r_0, d(y,z)+h')$. Note that $B(y,h') \cup B(z,h') \subseteq B_1$.
Choose $r_1:=2 \max(h',r_0)$ so that $B_1 \subseteq B(x,r)$ for all $r \ge r_1$ and for all $y,z \in B(x,cr)$.
Let $r \ge r_1$ and let $k:=\max  \Sett{ i \in \N^*}{ B_1 \subset B }$.
Since $B(x,r) \neq M$ there exists $C_1>0$ such that
\begin{equation}\label{e-ho1}
 k \le C_1 \log(r/s_1) +1.
\end{equation}
Denote by $M_i:=\sup_{B_i} u$ and $m_i:= \inf_{B_i}u$ for $i=1,\ldots,k$,  where $u$ is an arbitrary non-negative function $u:M \to \R$ that is $P$-harmonic on $B(x,r)$.
By elliptic Harnack inequality of Theorem \ref{t-ehi}, we have
\begin{align}
 \label{e-ho2} M_i - m_{i-1} = \sup_{B_{i-1}} (M_i-u) &\le C_E \inf_{B_{i-1}} (M_i-u) = C_E(M_i-M_{i-1}), \\
 \label{e-ho3}  M_{i-1} - m_{i} = \sup_{B_{i-1}} (u-m_i) &\le C_E \inf_{B_{i-1}} (u-m_i) = C_E(m_{i-1}-m_{i})
\end{align}
for $i=2,3,\ldots,k$. By adding \eqref{e-ho2} and \eqref{e-ho3}, we obtain
\begin{equation}\label{e-ho4}
 M_{i-1}-m_{i-1} \le \frac{C_E-1}{C_E+1} (M_i-m_i)
\end{equation}
for $i=2,3,\ldots,k$.
Combining \eqref{e-ho4} along with \eqref{e-ho1}, we obtain
\begin{equation}
 \label{e-ho5} M_1-m_1 \le \left( \frac{C_E-1}{C_E+1}\right)^{k-1} (M_k-m_k) \le \left( \frac{C_E-1}{C_E+1}\right)^{C_1 \log(r/s_1)} \sup_{B(x,r)}u.
\end{equation}
Since $u$ is $P$-harmonic in $B(x,r)$, we have
\[
 \abs{u(y)-u(z)} =  \abs{P(y)-P(z)} \le \sup_{B(y,h')\cup B(z,h')} u -\inf_{B(y,h')\cup B(z,h')} u \le M_1 -m_1.
\]
The above inequality along with \eqref{e-ho5} implies \eqref{e-holder}.
\end{proof}
Note that above result does not give H\"{o}lder continuity for harmonic functions which is in contrast to \cite[Section 5]{Mos61}.
However we will see that Proposition \ref{p-holder} is useful.  In particular, we use Proposition \ref{p-holder} to prove Gaussian lower bounds in
Chapter \ref{ch-glb}.

\chapter{Gaussian upper bounds}\label{ch-gub}
The goal of this chapter is to prove the following Gaussian upper bounds using Sobolev inequality.
The results of this chapter rely only on \ref{doub-loc}, \ref{doub-inf} and the Sobolev inequality \eqref{e-Sob}.
We do not assume the Poincar\'{e} inequality \ref{poin-mms} to show Gaussian upper bounds.
More precisely, we show
\begin{prop} \label{p-gue}
 Let $(M,d,\mu)$ be a quasi-$b$-geodesic metric measure space satisfying \ref{doub-loc} and \ref{doub-inf}.
Suppose that a Markov operator $P$ has a kernel $p$ that is $(h,h')$-compatible with $(M,d,\mu)$ for some $h >b$.
Further assume that $P$ satisfies the Sobolev inequality \eqref{e-Sob}. There exists $C >0$ such that
\begin{equation} \label{e-gue}
 p_n(x,y) \le \frac{C}{V(x,\sqrt{n})} \exp\left( \frac{-d(x,y)^2}{C n}\right)
\end{equation}
for all $x \in M$ and for all $n \in \N_{\ge 2}$.
\end{prop}
The first step is to obtain the following on-diagonal upper bound.
\begin{prop} \label{p-diag}
 Let $(M,d,\mu)$ be a quasi-$b$-geodesic metric measure space satisfying \ref{doub-loc} and \ref{doub-inf}.
Suppose that a Markov operator $P$ has a kernel $p$ that is $(h,h')$-compatible with $(M,d,\mu)$ for some $h >b$.
Further assume that $P$ satisfies the Sobolev inequality \eqref{e-Sob}. There exists $C_0 >0$ such that
\begin{equation} \label{e-diag}
 p_n(x,x) \le \frac{C}{V(x,\sqrt{n})}
\end{equation}
for all $x \in M$ and for all $n \in \N_{\ge 2}$.
\end{prop}
 A crucial ingredient in the proof of Proposition \ref{p-diag} is a $L^1$ to $L^\infty$ mean value inequality for the solutions of a  heat equation.
 We again rely on Moser's iterative method and the calculations are similar but more involved than those encountered in Section \ref{s-mvi-el} for harmonic functions.
 The lazy walk defined in Example \ref{x-lazy} will play an important role in this chapter. Recall that for a Markov operator $P$, the corresponding
 `lazy' versions of Markov operator and Laplacian are given by
 \begin{equation} \label{e-lazy}
  P_L = (I+P)/2 \mbox{ and } \Delta_L = \Delta/2= (I-P)/2.
 \end{equation}
 For $a,b \in \N$, we denote the integer intervals by
 \[
  \nint{a}{b} := \Sett{i \in \N}{ a \le i \le b}.
 \]
The following definition is analogous to  Definition \ref{d-harmonic}. Caloric functions are solutions to heat equation.
\begin{definition}\label{d-caloric}
 Let $P$ be a Markov operator  on $(M,d,\mu)$ and let $a,b \in \N$.
 A function $u:\N \times M \to \R$ is $P$-caloric (respectively $P_L$-caloric) in $\nint{a}{b} \times B(x,r)$ if
 \[
 \partial_k u(y) + \Delta u_k(y) =0 \hspace{1cm}(\mbox{respectively }  \partial_k u(y) + \Delta_L u_k(y)=0)
 \]
for all $k \in \nint{a}{b}$ and for all $y \in B(x,r)$.

Similarly, we say a function $u:\N \times R \to R$ is $P$-subcaloric (resp. $P$-supercaloric) in $\nint{a}{b} \times B(x,r)$ if
\[
 \partial_k u(y) + \Delta u_k(y) \le 0 \hspace{1cm}(\mbox{respectively } \ge 0)
\]
for all $k \in \nint{a}{b}$ and for all $y \in B(x,r)$. Analogously, we define $P_L$-subcaloric and $P_L$-supercaloric functions simply by replacing $\Delta$ with $\Delta_L$ in the equation above.
\end{definition}
\begin{remark}\label{r-caloric}\leavevmode
 \begin{enumerate}[(a)]
  \item We can restate the above definitions using $\partial_k u + \Delta u_k = u_{k+1} - P u_k$ and $\partial_k u + \Delta_L u_k = u_{k+1} - P_L u_k$.
  \item Consider a Markov operator $P$ that is $(h,h')$-compatible with $(M,d,\mu)$.
  Similar to Remark \ref{r-harmonic}(a), the property that $u:\N \times M \to \R$ is $P$-caloric (or $P_L$-caloric) in $\nint{a}{b} \times B(x,r)$ depends only
  on the value of $u$ in $\nint{a}{b+1} \times B(x,r+h')$. Therefore it suffices if the function $u$ has a domain that satisfies $\nint{a}{b+1} \times B(x,r+h') \subseteq \operatorname{Domain}(u)$.
 \end{enumerate}
\end{remark}

Although our eventual goal is to prove parabolic Harnack inequality for $P$-caloric functions, the Moser's iteration procedure
is applied to  $P_L$-caloric functions. The laziness is introduced to handle certain technical difficulties that arise due to discreteness of time.
Another method to avoid these technical difficulties that arise due to discreteness of time
is to carry out Moser's iteration method for solutions of the continuous time heat equation $\frac{\partial u}{\partial t} + \Delta u=0$
(See \cite[Section 2]{Del99} for this method on graphs).

In continuous time case the product rule of differentiation implies $\frac{\partial(u^2)}{\partial t}= 2 u \frac{\partial u}{\partial u}$; however for discrete time the analogous formula is
$\partial_k(u^2) = 2 u_k \partial u_k + \left( \partial_k u \right)^2$. The `error term'  $ \left( \partial_k u \right)^2$ due to discreteness of time  is a source of difficulty  in the proofs
of Caccioppoli inequality and an integral maximum principle for $P$-caloric and $P$-subcaloric functions.
However as we shall see, this `error term' can be handled using a Cauchy-Schwarz inequality for $P_L$-caloric and $P_L$-subcaloric functions (See Remark \ref{r-error}).
As a result, we will primarily be concerned with $P_L$-caloric and $P_L$-subcaloric functions for now. The assumption (d) in Definition \ref{d-compat} will allow to compare the random
walks driven by $P$ and $P_L$.

As mentioned in the beginning of Chapter \ref{ch-sob}, we rely on a version of Sobolev inequality that is weaker than the ones assumed in previous works.
This causes new difficulties for Moser's iteration method which relied on a Sobolev inequality.
The difficulty is even more significant in the parabolic case compared to that of the elliptic case in Chapter \ref{ch-ehi}.
This is because the difference between the strong \eqref{e.sob} and weak \eqref{e-Sob} formulations of Sobolev inequalities is not significant for harmonic functions.
To see why this might be true, note that if a function $u$ is $P$-harmonic in $B=B(x,r)$ then $P_B u = u$ in $B(x,r-h)$ and therefore the weaker formulation \eqref{e-Sob} yields an estimate that is close to that of \eqref{e.sob}.
However the same cannot be said about $P$-caloric functions.

The following lemma and its proof is analogous to that of Lemma \ref{l-sub}.
\begin{lemma}\label{l-subcal}
Let $P$ be a Markov operator.
 Assume that the function $u: \N \times M \to \R_{\ge 0}$   is $P$-subcaloric
 in $\nint{a}{b} \times B(x,r)$ for some $x \in M$, $r >0$  and $a,b \in \N$.
Then $u^p$ is a $P$-subcaloric  in $\nint{a}{b} \times B(x,r)$ for all $p \ge 1$.
\end{lemma}
\begin{proof}
Note that
\[
u_{k+1}^p(y) =(\partial_k u + u_k)^p(y) \le  (-\Delta u_k + u_k)^p = \left( P u_k(y) \right)^p \le  P (u_k^p) (y) \]
for all $(k,y) \in \nint{a}{b} \times B(x,r)$. The first inequality above follows
from the fact that $u\ge0$ is $P$-subcaloric
 in $\nint{a}{b} \times B(x,r)$ and the second follows from Jensen's inequality.
\end{proof}
For a function $f: \N \times M \to \R$ and a Markov operator $P$ on $M$, we denote the function $Pf:\N \times M \to \R$
\[
 P f (k,x) :=  (P f (k,\cdot) ) (x) = (Pf_k)(x)
\]
for all $k \in \N$ and for all $x \in M$. We require the following property of subcaloric functions.
\begin{lemma}\label{l-solres}
Let $(M,d,\mu)$ be a metric measure space and let $P$ be a Markov operator that is $(h,h')$-compatible  to $(M,d,\mu)$.
If $u:\N \times M \to \R$ is  $P_L$-subcaloric in $\nint{a}{b} \times B(x,r)$, then $P u$ is  $P_L$-subcaloric  in $\nint{a}{b} \times B(x,r-h')$)
for all $x \in M$ and for all $r > h'$.
\end{lemma}
\begin{proof}
If $(k,y) \in \nint{a}{b} \times B(x,r-h')$ and $u:\N \times M \to \R$ is  $P_L$-subcaloric in $\nint{a}{b} \times B(x,r)$ , then
\[
[(Pu)_{k+1} - P_L (Pu)_k](y)= P \left( u_{k+1} - P_L u_k \right)(y) \le 0.
\]
In the above equality, we used that $P$ and $P_L$ commute. The inequality follows from \eqref{e-compat} and the fact that any Markov operator is positivity preserving.
\end{proof}
\section{Mean value inequality for subcaloric functions}
We will prove the following mean value inequality in a weak form. The inequality bounds from above a weak version of $L^\infty$ norm on a space-time cylinder by a weak version of $L^1$ norm.
Our version of the mean value inequality in Lemma \ref{l-mvi} is weaker than the one known for graphs \cite[Theorem 4.1]{CG98} mainly because we rely on a weaker Sobolev-type inequality \eqref{e-Sob}.
Although the mean value inequality is weaker, we will obtain on-diagonal upper bounds using Lemma \ref{l-mvi}. Using an integral maximum principle argument, we will obtain Gaussian upper bounds in Chapter \ref{ch-gub}.
\begin{lemma} \label{l-mvi}
 Under the assumptions of Proposition \ref{p-diag}, there exists constants $C_1>0,n_1>0$ such that
 \begin{equation} \label{e-mvi}
  \inf_{ k \in \nint{0}{n} }\sup_{y  \in B(x,\sqrt{n}/2)} P^{2 \lceil \log \sqrt{n} \rceil +2} u_k (y) \le \frac{C_1}{V(x,\sqrt{n})} \sup_{k \in \nint{0}{n}} \int_{B(x,\sqrt{n}+h')} u_k \, d\mu
 \end{equation}
for all $n \in \N^*$ satisfying $n > n_1$, for all $x \in M$,  for all non-negative functions $u:\N \times M \to \R$ that is $P_L$-subcaloric in $\nint{0}{n}\times B(x,\sqrt{n})$.
\end{lemma}
The proof of Lemma \ref{l-mvi} relies on Moser's iteration procedure.
Couhlon and Grigor'yan \cite[Section 4]{CG98} obtained a similar (stronger) mean value inequality in the graph setting using an iteration procedure.
However they relied on a Faber-Krahn inequality that is equivalent to the Sobolev inequality \eqref{e.sob} and therefore does not hold for
discrete time Markov chains on continuous spaces.

In this section, we carry out Moser's iteration procedure for subcaloric functions relying on the weaker\footnote{`weaker' compared to Sobolev inequalities in \cite{Sal92,Sal95,Stu96,Del97,Del99,HS01}.}
Sobolev inequality \eqref{e-Sob}.
To prove the elementary iterative step of iteration, we need the following discrete Caccioppoli inequality. The proof is an adaptation  \cite[Proposition 4.5]{CG98}.
The next two Lemmas together may be regarded as the parabolic version of Lemma \ref{l-elel}.
\begin{lemma}[Caccioppoli inequality] \label{l-caccio}
Under the assumptions on Proposition \ref{p-diag}, we have
\begin{equation} \label{e-caccio}
\int_M \partial_k (u^2) \psi^2 \, d \mu + \frac{1}{8} \E(u_k \psi, u_k \psi) \le \frac{17}{8} \int_M \int_M \abs{\nabla_{yz} \psi}^2 u_k^2(y) p_1(y,z) \, dy \,dz
\end{equation}
for all $x \in M$, for all $r>0$, for all non-negative functions $\psi:M \to \R_{\ge 0}$ satisfying $\supp(\psi) \subseteq B(x,r)$, for all $a,b \in \N$, for all $k \in \nint{a}{b}$ and
for all non-negative functions $u : \N \times M \to \R_{\ge 0}$ such that  $u$ is $P_L$-subcaloric  in $ \nint{a}{b} \times B(x,r)$.
\end{lemma}
\begin{proof}
Fix $x \in M$, $r >h'$ and define $B:=B(x,r+h')$. Let $u : \N \times M \to \R_{\ge 0}$ be such that  $u$ is $P_L$-subcaloric  in $ \nint{a}{b} \times B(x,r)$.
We start with the elementary inequality
\begin{equation}\label{e-dc1}
 \partial_k(u^2)(y) \le - u_k(y) \Delta u_k (y) +\frac{1}{4} \left( \Delta u_k(y) \right)^2
\end{equation}
for all $(k,y) \in \nint{a}{b} \times B(x,r)$, as we now show.
Since $u$ is $P_L$-subcaloric  in $ \nint{a}{b} \times B(x,r)$, we have $u_{k+1}(y) \le P_L u_k(y)$
for all $(k,y) \in \nint{a}{b} \times B(x,r)$.
Combined with the fact that $u$ is non-negative, we have $u^2_{k+1}(y) \le \left( P_L u_k(y) \right)^2$ for all $(k,y) \in \nint{a}{b} \times B(x,r)$ which can be rearranged into \eqref{e-dc1}.

Let $(k,y) \in \nint{a}{b} \times B(x,r)$. Recall that $B=B(x,r+h')$. Using \eqref{e-dc1}, integration by parts \eqref{e-int-p} and $\supp(\psi) \subseteq B(x,r)$, we have
 \begin{align}
  \int_B \psi^2 \partial_k (u^2)  \, d\mu &\le -\frac{1}{2} \int_B \int_B \left(\nabla_{yz} u_k\right) \left(\nabla_{yz}(u_k \psi^2)\right) p_1(y,z) \, dy \, dz \nonumber \\
  &\hspace{4mm} + \frac{1}{4} \int_B (\Delta u_k(y))^2 \psi^2(y) \,dy. \label{e-dc2}
 \end{align}
The second term in \eqref{e-dc2} can be handled using Cauchy-Schwarz inequality as
\begin{align}
 ( \Delta u_k(y) )^2  &=  \left( -\int_M (\nabla_{yz}u_k) p_1(y,z)\,dz \right)^2 \nonumber \\
&\le \left( \int_M  p_1(y,z) \,dz \right)\left( \int_M (\nabla_{yz}u_k)^2 p_1(y,z) \,dz \right) \nonumber \\
& =  \int_M (\nabla_{yz}u_k)^2 p_1(y,z) \,dz. \label{e-dc3}
\end{align}
For the first term in \eqref{e-dc2}, we use product rule \eqref{e-prs-1}
\begin{equation} \label{e-dc4}
 \nabla_{yz}(u_k \psi^2) = u_k(z) \nabla_{yz} \psi^2  + \psi^2(y) \nabla_{yz} u_k.
\end{equation}
Combining \eqref{e-dc2}, \eqref{e-dc3} and \eqref{e-dc4}, we have
\begin{align}
\nonumber \lefteqn{\int_B \psi^2(y)  \partial_k (u^2) (y) \,dy + \frac{1}{4} \int_B \int_B \left( \nabla_{yz} u_k \right)^2 \psi^2(y) p_1(y,z) \,dy \,dz} \\
 \label{e-dc5} & \le  -\frac{1}{2} \int_B \int_B u_k(z) \left( \nabla_{yz} \psi^2 \right) \left( \nabla_{yz} u_k \right) p_1(y,z) \, dy \,dz.
\end{align}
The right side of \eqref{e-dc5} can be bounded using $ t_1 t_2 \le t_1^2/8 +2 t_2^2$ as
\begin{align}
\nonumber \abs{-u_k(z) \left( \nabla_{yz} \psi^2 \right) \left( \nabla_{yz} u_k \right)} & \le  u_k(z) \psi(y) \abs{\left( \nabla_{yz} \psi \right) \left( \nabla_{yz} u_k \right)} \\
\nonumber  & \hspace{4mm} + u_k(z) \psi(z)\abs{\left( \nabla_{yz} \psi \right) \left( \nabla_{yz} u_k \right)} \\
\nonumber  & \le  \frac{1}{8} ( \psi^2(y) + \psi^2(z) ) \left( \nabla_{yz} u_k \right)^2 \\
\label{e-dc6}  & \hspace{4mm} + 4 u_k^2(z) \abs{\nabla_{yz} \psi}^2.
\end{align}
Using $p_1(y,z)=p_1(z,y)$ for $\mu\times\mu$-almost every $(y,z)$, we obtain
\begin{equation}
 \label{e-dc7} \int_B \int_B \psi^2(y) \left(\nabla_{yz}u_k \right)^2p_1(y,z)  \,dy \,dz =  \int_B \int_B \psi^2(z) \left(\nabla_{yz}u_k\right)^2p_1(y,z)  \, dy \,dz.
\end{equation}
Combining \eqref{e-dc5},\eqref{e-dc6} and \eqref{e-dc7}, we deduce
\begin{align}
\nonumber \lefteqn{\int_B\psi^2(y) \partial_k(u^2) (y) \, dy + \frac{1}{8} \int_B \int_B \left( \nabla_{yz} u_k \right)^2 \psi^2(y) p_1(y,z) \, dy \,dz} \\
 \label{e-dc8} & \le  2 \int_B \int_B u_k^2(z) \left( \nabla_{yz} \psi \right)^2 p_1(y,z) \, dy \,dz.
\end{align}
Since  $\supp(\psi) \subseteq B(x_0,r-h')$, using  integration by parts \eqref{e-int-p} we have
\begin{equation}
 \label{e-dc9} \mathcal{E}(\psi u_k, \psi u_k) = \frac{1}{2} \int_B \int_B \abs{ \nabla_{yz} (u_k \psi) }^2  p_1(y,z) \, dy \,dz.
\end{equation}
Using product rule \eqref{e-prs-1} and the inequality $(t_1+t_2)^2 \le 2(t_1^2 + t_2^2)$, we obtain
\begin{align}
\nonumber \abs{ \nabla_{yz} (u_k \psi) }^2 &= \abs{ \psi(y) (\nabla_{yz} u_k ) + u_k(z)  (\nabla_{yz} \psi)}^2 \\
\label{e-dc10} & \le  2 \left(\psi^2(y) (\nabla_{yz} u_k )^2 + u_k^2(z)  (\nabla_{yz} \psi)^2 \right).
\end{align}
Combining \eqref{e-dc8}, \eqref{e-dc9}, \eqref{e-dc10} and  $\mu\times\mu$-almost everywhere symmetry of $p_1$ yields
\eqref{e-caccio}.
\end{proof}
\begin{remark}\label{r-error}
Recall the product rule of differentiation
$\partial_k(u^2) = 2 u_k \partial u_k + \left( \partial_k u \right)^2$  gives rise to the `error term'  $ \left( \partial_k u \right)^2$ which occurs due to discreteness of time.
This error term occurs in \eqref{e-dc2} and is controlled using Cauchy-Schwarz inequality in \eqref{e-dc3}. However the estimate given by \eqref{e-dc3} is sufficient to prove Caccioppoli inequality
only in the presence of some laziness. A similar difficulty arises in the proof of discrete integral maximum principle and is the reason behind considering the operator $P_L$ as opposed to $P$ in this section.
\end{remark}

Next, we prove the elementary iterative step of  Moser's iteration in parabolic setting. The proof relies on Caccioppoli inequality \eqref{e-caccio} and
Sobolev inequality \eqref{e-Sob}.
Let $\mu_c$ denote the counting measure on $\N$ and let $(M,d,\mu)$ be a metric measure space. We denote the product measure  on $\N \times M$ by $\tilde \mu:=\mu_c \times \mu$.
Similar to \eqref{e-phi}, we define
\begin{equation}\label{e-tphi}
 \tphi(u,p,Q):= \left( \frac{1}{\tmu(Q)} \int_{Q} u^p \,d\tmu \right)^{1/p}
\end{equation}
for all $p>0$, for all $Q \subset \N \times M$ and for all functions $u:\N \times M \to \R_{\ge 0}$.
\begin{lemma} \label{l-pel}
 Under the assumptions of Proposition \ref{p-diag}, for all $K_1 \ge 1$, there exists $C_1>0,r_1>0$ (depending on $K_1$) such that
 \begin{align}
 \lefteqn{\hspace{-0.8cm} \tphi(Pu,2 + (4/\delta), \nint{ \lceil (1 -\sigma^2) a_0 + \sigma^2 a_1 \rceil}{a_1} \times B(x,(1-\sigma)r-h') ) } \nonumber \\
 &\le C_1 \sigma^{-1} \tphi(u, 2, \nint{a_0}{a_1} \times B(x,r+h')) \label{e-pel}
 \end{align}
for all $\sigma \in (0,1/2)$, for all $x \in M$, for all $r \ge r_1$, for all $a_0,a_1 \in \N$ satisfying $K_1^{-1} r^2 \le a_2-a_1 \le K_1 r^2$ and for all non-negative functions
$u: \N \times M \to \R_{\ge 0}$ such that $u$ is $P_L$-subcaloric in $\nint{a_0}{a_1} \times B(x,r)$.
\end{lemma}
\begin{proof}
Let $x \in M$, $\sigma \in (0,1/2)$ and let $r>r_1\ge 4h'$, where $r_1$ will be determined later.
Let $u$ be a non-negative function that  is $P_L$-subcaloric in $\nint{a_0}{a_1} \times B(x,r)$.

We start by defining appropriate cut-off functions in space and time.
 Define $B:=B(x,r+h')$ and $\psi:M \to \R_{\ge 0}$ as
 \[
  \psi_\sigma(y):= \max \left( 0, \min\left(1, \frac{r - d(x,y)}{\sigma r}\right)\right).
 \]
Note that $\supp(\psi_\sigma) \subseteq B(x,r)$ and $\psi \equiv 1$ on $B(x,(1-\sigma) r)$.
Define $a_\sigma:=\lceil (1 -\sigma^2) a_0 + \sigma^2 a_1 \rceil$ and $\chi:\N \to \R$ as
\[
\chi_\sigma(k) = \begin{cases}
1 &\text{if $k \ge a_\sigma$}\\
0 &\text{if $k \le a_0$}\\
\frac{k-a_0}{a_\sigma-a_0} &\text{otherwise}.
\end{cases}
\]
Since $u$ is non-negative and $P_L$-subcaloric in $\nint{a_0}{a_1}\times B(x,r)$, by Caccioppoli inequality (Lemma \ref{l-caccio}) and
product rule \eqref{e-prt-1}, we obtain
\begin{align}
 \nonumber\lefteqn{\int_B \left( \partial_k(\chi_\sigma u)^2 \right) \psi_\sigma^2  \,d\mu + \frac{\chi_\sigma^2(k+1)}{8} \mathcal{E}^B ( \psi_\sigma u_k, \psi_\sigma u_k)} \\
\label{e-pel1} & \le  \frac{17}{8} \chi_\sigma^2(k+1) \int_B \int_B  \abs{\nabla_{yz} \psi_\sigma}^2 u_k^2(y) p_1(y,z) \,dy \,dz + \partial_k \chi_\sigma^2 \int_B u_k^2 \psi_\sigma^2 \,d\mu
\end{align}
for all $k \in [a,b)$.
Since $p_1$ is $(h,h')$-compatible with $(M,d,\mu)$, we have
\begin{equation} \label{e-pel2}
\abs{\nabla_{yz} \psi}^2 p_1(y,z)  \le \frac{(h')^2}{(\sigma r)^2}p_1(y,z).
\end{equation}
We use product rule \eqref{e-prt-1}, triangle inequality, $\chi_\sigma \le 1$ and $a_\sigma-a_0 \ge \sigma^2(a_1 -a_0) \ge \sigma^2 K_1^{-1} r^2$ to deduce
\begin{equation} \label{e-pel3}
\abs{ \partial_k \chi_\sigma^2 } \le (\chi_\sigma (k+1) + \chi_\sigma(k))\abs{\partial_k \chi_\sigma} \le 2 \abs{\partial_k \chi_\sigma}\le \frac{2}{(a_\sigma-a_0)} \le \frac{2K_1}{ \sigma^2 r^2}
\end{equation}
Combining \eqref{e-pel1}, \eqref{e-pel2} and \eqref{e-pel3}, there exists $C_2>0$ such that
\begin{equation}
\label{e-pel4} \int_B \psi_\sigma^2 \left( \partial_k(\chi_\sigma u)^2 \right) \,d\mu + \frac{\chi_\sigma^2(k+1)}{8} \mathcal{E}^B ( \psi_\sigma u_k, \psi_\sigma u_k) \le  \frac{C_2}{\sigma^2 r^2} \int_B   u_k^2 \,d\mu
\end{equation}
for all $k \in \nint{a_0}{a_1}$. In \eqref{e-pel4}, $C_2$ depends only on $K_1$ and $h'$.

Adding  \eqref{e-pel4}, from $k=a_0$ to $k \in \nint{a_0}{a_1}$, yields
\begin{align}
 \label{e-pel5} \sup_{k \in [a_\sigma,a_1]} \int_B \psi_\sigma^2  u_k^2 \,d\mu & \le  \frac{C_2}{\sigma^2 r^2} \sum_{k=a_0}^{a_1} \int_B   u_k^2 \,d\mu \\
 \label{e-pel6} \sum_{k=a_\sigma}^{a_1} \mathcal{E}(\psi_\sigma u_k, \psi_\sigma u_k) & \le \frac{8 C_2}{\sigma^2 r^2} \sum_{k=a_0}^{a_1} \int_B   u_k^2\,d\mu.
\end{align}
Define $w_k := P_B(\psi_\sigma u_k)$. Since $\psi \equiv 1$ on $B(x,(1-\sigma) r)$, by \eqref{e-compat} $w_k = P_B (\psi_\sigma u_k) = P u_k$ on $B(x,(1-\sigma)r - h')$.
Combined with H\"{o}lder inequality, we have
 \begin{equation} \label{e-pel7}
 \int_{B(x,(1-\sigma)r-h')} \left( P u_k \right)^{2(1+2/\delta)} \, d\mu \le \left( \int_B w_k^2 \,d\mu  \right)^{2/\delta} \left( \int_B w_k^{2\delta/(\delta-2)} \, d\mu\right)^{(\delta -2 )/\delta}.
\end{equation}
Since $P_B$ is a contraction in $L^2(B)$, we have
\begin{equation}
 \label{e-pel8}   \int_B w_k^2 \, d\mu    \le   \int_B \psi_\sigma^2 u_k^2 \,d\mu.
\end{equation}
By  Sobolev inequality \eqref{e-Sob}, Lemma \ref{l-dircomp-ball}(a) and \eqref{e-compat}, we obtain
\begin{equation}\label{e-pel9}
\left( \int_{B} w_k^{2 \delta/(\delta-2)} \,d\mu \right)^{(\delta-2)/\delta}
\le  \frac{C_S r^2}{V(x_0,r)^{2/\delta}}  \left( \mathcal{E} (\psi_\sigma u_k , \psi_\sigma u_k)  + r^{-2} \int_B \psi_\sigma^2 u_k^2\, d\mu \right)
\end{equation}
 By \eqref{e-pel5}, \eqref{e-pel6},\eqref{e-pel7},\eqref{e-pel8}, \eqref{e-pel9} and $a_1 - a_0 \le K_1 r^2$, there exists $C_3 >0$ such that
\begin{equation}\label{e-pel10}
\sum_{k=a_\sigma}^{a_1} \int_{ B(x,(1-\sigma)r-h')}\left(P u_k \right)^{2(1+2/\delta)} \,d\mu \le  \frac{C_4  r^2}{V(x,r)^{2/\delta}} \left( (r \sigma)^{-2} \sum_{k=a_0}^{a_1} \int_B u_k^2 \,d\mu\right)^{1+2/\delta}.
\end{equation}
We choose $r_1 \ge 4h'$ so that
$a_\sigma \le a_{1/2} \le (a_0+a_1)/2$ for all $a_0,a_1 \in \N$ so that $a_1-a_0 \ge K_1^{-1} r_1^2$.
Since $r \ge 4h'$ and $\sigma < 1/2$, we have $(1-\sigma)r -h' \ge (r/2)-h' \ge r/4$. Hence by \eqref{e-vd1}, $K_1^{-1}r^2 \le a_1 - a_0 \le K_1 r^2$ along with \eqref{e-pel10}, we have \eqref{e-pel}.
\end{proof}
\begin{proof}[Proof of Lemma \ref{l-mvi}]
 We carry out Moser's iteration in two stages. In the first stage of the iteration, we obtain a $L^1$ to $L^2$ mean value inequality and in the second stage we show a $L^2$ to $L^\infty$
 mean value inequality. Combining the two stages yields the desired $L^1$ to $L^\infty$ mean value inequality. The proof relies on
 repeated application of the elementary iterative step given by Lemma \ref{l-pel}.

 Let $r_1(0):=\sqrt{n}+h'$, $a_1(0):=0$, $N:= \lceil \log \sqrt{n} \rceil$ and $\theta:=1+(2/\delta)$.
 For the first stage of iteration, we iteratively define the quantities
\begin{align*}
 r_1(i+1) &:= (r_1(i) -h')\left(1- \frac{4^{-1}}{ 3^{N+1-i}} \right) -h'\\
 a_1(i+1) &:= \left\lceil\left( 1 - \frac{4^{-2} }{9^{N_r+1-i} } \right) a_1(i) +  \frac{4^{-2}}{9^{N_r+1-i} } n \right\rceil
\end{align*}
 for $i=0,1,\ldots,N$. We define a non-increasing sequence of space-time cylinders
 \[
  Q_i(i)= \nint{a_1(i)}{n} \times B(x,r_i), \hspace{5mm}\mbox{for $i=0,1,\ldots,N+1$}.
 \]
The following estimates are straightforward from definitions of $r_1$ and $a_1$: There exists $n_0>0$ such that for all $n \ge n_0$, we have
\begin{align}
 \nonumber r_1(N+1) &\ge \sqrt{n} \left(1 - 4^{-1} \sum_{j=1}^{N+1} 3^{-j} \right) -2 ( \log \sqrt{ n} +3 + h') \\
 &\ge (7/8) \sqrt{n} -2 ( \log \sqrt{ n} +3 + h') \ge (6/7) \sqrt{n},\label{e-mv1} \\
 n- a_1(N+1) & \ge n \left(1 - 4^{-2} \sum_{j=1}^{N+1} 9^{-j} \right) - 2 (N+1) \nonumber \\
 \label{e-mv2} & \ge  (31/32)n - 2 ( \log \sqrt{n} +2) \ge (15/16)n.
\end{align}

Let $u:\N \times M \to \R_{\ge 0}$ be an arbitrary non-negative function that is $P_L$-subcaloric in $\nint{0}{n}\times B(x,\sqrt{n})$ where $n \ge n_1$.
By Lemma \ref{l-solres} $P^i u$ is $P_L$-subcaloric in $\nint{0}{n}\times B(x,\sqrt{n} - ih')$
and therefore $P_L$-subcaloric in $\nint{a_1(i)}{n}\times B(x,r_1(i)-h')$ for all $i=0,1,\ldots,N+1$.
Hence by applying Lemma \ref{l-pel} for the function $P^i u$ which is  $P_L$-subcaloric on $\nint{a_1(i)}{n}\times B(x,r_1(i)-h')$ with $\sigma=4^{-1} 3^{-(N+1-i)}$,  we have $C_2>0$ such that
\begin{equation} \label{e-mv3}
 \tphi(P^{i+1} u, 2 \theta, Q_{i+1}) \le C_2 3^{N+1-i} \tphi(P^{i} u, 2 , Q_{i})
\end{equation}
for all $i=0,1,\ldots,N$. We may choose $K_1=8$ in the application of Lemma \ref{l-pel} above due to \eqref{e-mv1} and \eqref{e-mv2}.

By H\"{o}lder inequality along with \eqref{e-mv1}, \eqref{e-mv2} and \eqref{e-vd1}, there exists $C_3>0$ such that
\begin{equation}\label{e-mv4}
 \tphi(P^{i+1} u , 2 ,Q_1(i+1))\le C_3 \tphi(P^{i+1} u , 1,Q_1(i+1))^\alpha \tphi(P^{i+1} u , 2 \theta ,Q_1(i+1))^\beta
\end{equation}
for all $i=1,2,\ldots,N$, where $\alpha=1-\beta=2/(\delta+4)$. By \eqref{e-compat}, $u \ge 0$, \eqref{e-mv1},\eqref{e-mv2} and  \eqref{e-vd1}, there exists $C_4 >0$ such that
\begin{equation}\label{e-mv5}
 \tphi(P^{i} u , 1,Q_1(i)) \le C_4 \tphi( u , 1,Q_1(0))
\end{equation}
for all $i=0,1,\ldots,N+1$. Combining \eqref{e-mv3}, \eqref{e-mv4}, \eqref{e-mv5}, there exists $C_5>0$ such that
\begin{equation}
 \label{e-mv6}  \tphi(P^{i+1} u , 2 ,Q_1(i+1))\le C_5 3^{\beta(N+1-i)} \tphi( u , 1,Q_1(0))^\alpha \tphi(P^{i} u , 2  ,Q_1(i))^\beta
\end{equation}
for $i=1,\ldots,N$. By iterating \eqref{e-mv6}, we obtain
\begin{equation}\label{e-mv7}
 \tphi(P^{N+1} u , 2 ,Q_1(N+1))\le C_5^{\sum_{i=0}^\infty \beta^i} 3^{ \sum_{i=1}^ \infty i \beta^i} \tphi( u , 1,Q_1(0))^{(1 - \beta^N)}\tphi( P u , 2,Q_1(1))^{ \beta^N}.
\end{equation}
Since $u \ge 0$, by H\"{o}lder inequality, \eqref{e-compat} and \eqref{e-vd1}, there exists $C_6,C_7>0$ such that
\begin{align}
\nonumber  \int_{B(x,r_1(1))} (P u_i)^2 \, d\mu &\le   \left( \sup_{B(x,r_1(1))} Pu_i \right)\int_{B(x,r_1(1))} P u_i \,d\mu \\
 \nonumber &\le \left( \int_{B(x,\sqrt{n}+h')} u_i \, d\mu \right)^2 \sup_{y \in B(x,\sqrt{n})} \frac{C_6}{V(y,h')} \\
  &\le \frac{C_7 n^{\delta/2}}{V(x,\sqrt{n})} \left( \sup_{i \in \nint{0}{n}} \int_{B(x,\sqrt{n}+h')} u_i \, d\mu \right)^2 \label{e-mv8}
\end{align}
for all $i \in \nint{0}{n}$. Combining \eqref{e-mv7}, \eqref{e-mv8} along with \eqref{e-vd1} yields
\begin{equation}\label{e-mv9}
 \tphi(P^{N+1} u , 2 ,Q_1(N+1))\le \frac{C_8}{V(x,\sqrt{n})} \sup_{k \in \nint{0}{n}} \int_{B(x,\sqrt{n}+h')} u_k \, d\mu
\end{equation}
for some $C_8 >0$. The inequality \eqref{e-mv9} is a $L^1$ to $L^2$ mean value inequality and this concludes the first part of iteration.

For the second part, we define $v=P^{N+1}u$, $a_2(0)=a_1(N+1)$ and $r_2(0)= r_1(N+1)$.
As before, we iteratively define
\begin{align*}
 r_2(i+1) &:= (r_2(i) -h')\left(1- \frac{4^{-1}}{ 3^{i+1}} \right) -h',\\
 a_2(i+1) &:= \left\lceil\left( 1 - \frac{4^{-2} }{9^{i+1} } \right) a_2(i) +  \frac{4^{-2}}{9^{i+1} } n \right\rceil
\end{align*}
for $i=1,2,\ldots,N+1$.  As before, define a non-increasing sequence of space-time cylinders
by $Q_2(i):= \nint{a_2(i)}{n} \times B(x, r_2(i))$ for $i=0,1,\ldots,n$. Note that
$Q_2(0)=Q_1(N+1)$.

Similar to \eqref{e-mv1} and \eqref{e-mv2}, there exists $n_1 \ge n_0$ such that for
all $n \ge n_1$,
\begin{align}
\label{e-mvi1} r_2(i) \ge r_2(N+1) &\ge \sqrt{n}/2 \\
\label{e-mvi2} n-a_2(i) \ge n - a_2(N+1) & \ge n/2
\end{align}
for all $i=0,1,\ldots,N+1$.
By Jensen's inequality, we have
\begin{equation}\label{e-mvi3}
(P^{i+1} v)^{\theta^{i+1}} \le \left( P \left[(P^i v)^{\theta^{i}} \right] \right)^\theta
\end{equation}
for all $i \in \N$.
By Lemma \ref{l-solres} and Lemma \ref{l-subcal}, the function $(P^i v)^{\theta^i}$ is $P_L$-subcaloric in $\nint{a_2(i)}{n} \times B(x,r_2(i)-h)$ for all $i=0,1,\ldots,N+1$.
Therefore by Lemma \ref{l-pel} for the function  $(P^i v)^{\theta^i}$ and \eqref{e-mvi3}, there exists $C_9>0$ such that
\begin{equation}\label{e-mvi4}
 \tphi( P^{i+1}v  , 2\theta^{i+1}, Q_2(i+1)) \le C_9^{\theta^{-i}} 3^{(i+1)\theta^{-i}} \tphi(P^{i} v, 2 \theta^i ,Q_2(i))
 \end{equation}
for $i=0,1,\ldots,N-1$. Iterating the inequalities \eqref{e-mvi4}, there exists $C_{10}>0$ such that
\begin{equation}\label{e-mvi5}
 \tphi( P^{N}v  , 2\theta^{N}, Q_2(N))\le C_{10} \tphi( v, 2  ,Q_2(0))=C_{10} \tphi( v, 2  ,Q_1(N+1)).
\end{equation}
There exists $C_{11},C_{12},C_{13}>0$ such that, for all $k \in N$
\begin{align} \label{e-mvi6}
 \sup_{y \in B(x,r_2(N+1))} P^{N+1}v_k (y)  &\le C_{11} \dashint_{B(y,h')} P^N v_k \, d\mu \nonumber \\
 &\le C_{11} \left(\dashint_{B(y,h')} (P^N v_k)^{2 \theta^N} \, d\mu \right)^{1/(2\theta^N)} \nonumber \\
 & \le C_{12} n^{\delta/(4 \theta^N)} \left(\dashint_{B(x,r_2(N))} (P^N v_k)^{2 \theta^N} \, d\mu \right)^{1/(2\theta^N)} \nonumber \\
 & \le C_{13} \left(\dashint_{B(x,r_2(N))} (P^N v_k)^{2 \theta^N} \, d\mu \right)^{1/(2\theta^N)}
\end{align}
The first line above follows from \eqref{e-compat}, the second line follows from Jensen's inequality, the third line follows from  \eqref{e-vd1}
and the last line follows from the fact that $n \mapsto n^{\delta/(4 \theta^{\log n})}$ is bounded in $[2,\infty)$.
By \eqref{e-mvi5}, \eqref{e-mvi6} and $v=P^{N+1}u$, we have a $L^2$ to $L^\infty$ mean value inequality
\begin{equation}\label{e-mvi7}
 \inf_{k \in \nint{0}{n}} \sup_{B(x,r_2(N+1))} P^{2N+2}u \le  C_{10} C_{13}  \tphi( P^{N+1}u, 2  ,Q_1(N+1)).
\end{equation}
Combining \eqref{e-mv9} and \eqref{e-mvi7}, we have the desired inequality \eqref{e-mvi}.
\end{proof}
\section{On-diagonal upper bound}
The following lemma provides a useful example of $P_L$-caloric function.
\begin{lemma}\label{l-solheat}
Let $(M,d,\mu)$ be a metric measure space. Let $P$ be Markov operator equipped with  kernel $(p_k)_{k \in \N}$ that is
$(h,h')$-compatible with $(M,d,\mu)$.
 Define for all $k \in \N$, the function $h_k: M \times M \to \R$ by
 \begin{equation}\label{e-defh}
  h_k(x,y):=\left( P_L^k p_2(x,\cdot) \right)(y) = 2^{-n} \sum_{i=0}^n \binom{n}{i} p_{i+2}(x,y)
 \end{equation}
where $P_L=(I+P)/2$ as before.
Then for all $x \in M$, the function
\[
 (k,y) \mapsto h_k(x,y)
\]
is $P_L$-caloric in $\N \times M$.
\end{lemma}
\begin{proof}
The second equality in \eqref{e-defh} is a consequence of binomial theorem and Lemma \ref{l-kernel}(c).
Note that
 \[
  P_L(h_k(x,\cdot))(y)= P_L(P_L^kp_2(x,\cdot))(y)= P_L^{k+1}\left( p_2(x,\cdot)\right) (y)=h_{k+1}(x,y).
 \]
Therefore $(k,y)\mapsto h_k(x,y)$ is $P_L$-caloric in $\N \times M$ for all $x \in M$.
\end{proof}
We are ready to prove Proposition \ref{p-diag} using the mean value inequality \eqref{e-mvi}.
\begin{proof}[Proof of Proposition \ref{p-diag}]
 Let $h_k(x,y)$ be defined as \eqref{e-defh}. Choose $n_1 \in \N$ such that
 \begin{equation}\label{e-du1}
  2 \lceil \log \sqrt{n} \rceil +4 \le n
 \end{equation}
 for all $n \ge n_1$.
By Lemma \ref{l-mvi}, Lemma \ref{l-solheat} and $\int_M h_k(x,y)\,dy=1$, there exists $n_2 \ge n_1$ and $C_1>0$ such that the $P_L$-caloric function $(k,y)\mapsto h_k(x,y)$ satisfies
the mean value inequality
\begin{equation}\label{e-du2}
 \inf_{k \in \nint{0}{n} } P^{ 2 \lceil \log \sqrt{n} \rceil +2} h_k(x,x)  \le \inf_{k \in \nint{0}{n} } \sup_{y \in B(x,\sqrt{n}/2)} P^{ 2 \lceil \log \sqrt{n} \rceil +2} h_k(x,y) \le \frac{C_1}{V(x,\sqrt{n})}
\end{equation}
for all $x \in M$ and for all $n \in \N$ satisfying $n \ge n_2$.

By \eqref{e-da}, we have $p_2(x,\cdot)- \alpha p_1(x,\cdot) \ge 0$ $\mu$-almost everywhere for each $x \in M$.
By \eqref{e-pcomp} of Lemma \ref{l-con-d} and Lemma \ref{l-mker}, we have
\begin{equation}\label{e-du4}
 p_{k}(x,x) \le \alpha^{-1} p_{2 \lceil k/2 \rceil} (x,x) \le \alpha^{-1} p_{2 \lceil k/2 \rceil} (x,x) \ge \alpha^{-1} p_{2n}(x,x)
\end{equation}
for all $x \in M$ and for all $2 \le k \le 2n$.
By \eqref{e-du4} and \eqref{e-du1},
\begin{equation}\label{e-du5}
  P^{ 2 \lceil \log \sqrt{n} \rceil +2} h_k(x,x) \ge \alpha^{-1} p_{2n}(x,x)
\end{equation}
for all $x \in M$, for all $k \in \nint{0}{n}$ and for all $n \ge n_2$.
Combining \eqref{e-du5}, \eqref{e-pcomp}, \eqref{e-defh} and \eqref{e-du2}, there exists $C_2>0$ such that
\begin{equation}\label{e-du6}
 p_{n}(x,x) \le \frac{C_2}{V(x,\sqrt{n})}
\end{equation}
for all $n \ge 2n_2$. Since $P$ is a contraction in $L^\infty$ by \eqref{e-compat}, Lemma \ref{l-kernel}(c) and \eqref{e-vd1}, there exists $C_3,C_4>0$ and $\delta > 2$ such that
\begin{equation}\label{e-du7}
 p_{n}(x,x) \le \frac{C_3}{V(x,h')} \le \frac{C_4 n^{\delta/2}}{V(x,\sqrt{n})}
\end{equation}
for all $x \in M$ and for all $ n \in \N$ with $n \ge 2$.
Combining \eqref{e-du6} and \eqref{e-du7} gives the diagonal bound \eqref{e-diag}.
\end{proof}
\section{Discrete integral maximum principle}\label{s-dimp}
We use Discrete integral maximum principle and diagonal upper bound to obtain Gaussian upper bounds.
This approach is detailed in \cite{CGZ05} for graphs. A crucial assumption in \cite{CGZ05} is the laziness assumption
for the corresponding Markov chain $(X_n)_{n \in N}$ given by $\inf_{x \in M} \PP_x(X_1=x) >0$.
As explained in \cite[Section 3]{CGZ05} this laziness assumption is not too restrictive for graphs because under natural conditions the iterated operator $P^2$ corresponds to a lazy Markov chain.
However this fails to be true for continuous spaces.

Since the laziness assumption is unavoidable for discrete integral maximum principle, we consider the Markov operator $P_L=(I+P)/2$ instead of $P$.
Using discrete integral maximum principle corresponding to $P_L$ and diagonal estimate on $p_k$, we obtain off-diagonal estimates on $h_k$ defined in \eqref{e-defh}.
We rely on careful comparison between off-diagonal estimates of $h_k$  and the Markov kernel $p_k$.
The comparison arguments are new but elementary and involves Stirling's approximation. Our comparison arguments rely crucially on the compatibility assumption \eqref{e-da}.
Similar comparison arguments for off-diagonal estimates was carried out in \cite[Section 3.2]{Del99} to compare Markov chains on graphs with its corresponding continuous time version.

The main technical tool to prove Gaussian upper bounds is the following discrete integral maximum principle.
The statement below and its proof is adapted from \cite[Proposition 2.1]{CGZ05}.
\begin{prop}[Discrete integral maximum principle]\label{p-imp}
Suppose that $P$ is a Markov operator  that is $(h,h')$-compatible with a metric measure space $(M,d,\mu)$.
Let $f$ be a strictly positive continuous function on $\nint{0}{n} \times M$ such that,
\begin{equation}
\label{e-condn-f} \partial_k f(x) + \frac{ \abs{\nabla_P f_{k+1}}^2}{4 f_{k+1}}(x) \le 0.
\end{equation}
for all
$x \in M$ and $k \in \nint{0}{n-1}$ where $\abs{\nabla_P f}$ is as defined in \eqref{e-gradp}.
Let $u:\N \times M$  bounded function   that is $P_L$-caloric on $\nint{0}{n-1} \times M$ satisfying $\supp(u_0)\subset B(w,R)$ for some $w \in M, R \in (0,\infty)$.
Then the function
\[
k \mapsto J_k = J_k(u):= \int_M u_k^2 f_k \,d\mu
\]
is non-increasing in $\nint{0}{n}$.
\end{prop}
\begin{proof}
Since $\supp(u_0) \subseteq B(w,R)$, by \eqref{e-compat} $\supp(u_k) \subseteq B(w,R+kh')$.
Therefore by continuity of $f_k$ and boundedness of $u$ all the integrals $J_k$ are finite.
By product rule \eqref{e-prt-1}, \eqref{e-prt-2} and $\partial_k u = -\Delta u_k /2$, we have for all $k \in \nint{0}{n-1}$
\begin{align}
\nonumber \partial_k J(u) &=\int_M \partial_k(u^2f) \, d\mu \\
 &= 2 \int_M u_k \partial_k u f_{k+1} \,d\mu + \int_M \left( \partial_k u\right)^2 f_{k+1}\,d\mu + \int_M u_k^2\partial_k f\, d\mu \nonumber \\
  &= -\int_M u_k f_{k+1} \Delta u_k \,d\mu + \int_M \left( \partial_k u\right)^2 f_{k+1}\,d\mu + \int_M u_k^2\partial_k f\, d\mu.
\label{e-imp1}
\end{align}
Using integration by parts \eqref{e-int-p} and product rule \eqref{e-prs-1}, the first term in \eqref{e-imp1} is
\begin{align}
\label{e-imp2} \nonumber \lefteqn{-\int_M u_k f_{k+1} \Delta u_k \,d\mu} \\
\nonumber &=  - \frac{1}{2} \int_M \int_M (\nabla_{xy} u_k) \nabla_{xy}(u_k f_{k+1}) p_1(x,y)\, dy \,dx \\
\nonumber &=  -\frac{1}{2}  \int_M \int_M  \left[(\nabla_{xy} u_k)^2  f_{k+1}(x) + (\nabla_{xy} u_k) u_k(y) (\nabla_{xy}f_{k+1}) \right] p_1(x,y) \,dy \,dx \\
 &=  -\frac{1}{2}  \int_M \int_M  \left[(\nabla_{xy} u_k)^2  f_{k+1}(x) + (\nabla_{xy} u_k) u_k(x) (\nabla_{xy}f_{k+1}) \right] p_1(x,y) \,dy \,dx
\end{align}
In order to get the last equation we switch $x$ and $y$ and use the fact that $p_1(x,y)=p_1(y,x)$ for $\mu \times \mu$-almost every $(x,y)$.
To handle the second term in \eqref{e-imp1}, we use $\partial_k u = -\Delta u_k/2$ \eqref{e-dc3} to obtain
\begin{equation}
\label{e-imp3} \int_M (\partial_k u)^2 f_{k+1} \,d\mu \le  \frac{1}{4} \int_M \int_M (\nabla_{xy}u_k)^2 f_{k+1}(x) p_1(x,y) \,dy \,dx
\end{equation}
for all $k \in \nint{0}{n-1}$.
Substituting \eqref{e-imp2} and \eqref{e-imp3} in \eqref{e-imp1}, we deduce
\begin{align*}
\partial_k J(u) &\le -\frac{1}{4} \int_M \int_M (\nabla_{xy} u_k)^2 f_{k+1}(x) p_1(x,y) \,dy \,dx + \int_M u_k^2(x) \partial_k f(x) \,dx\\
& - \frac{1}{2}  \int_M \int_M (\nabla_{xy} u_k) u_k(x) (\nabla_{xy}f_{k+1})p_1(x,y) \,dy \,dx \\
& =  - \frac{1}{4} \int_M \int_M \left( \nabla_{xy} u_k \sqrt{f_{k+1}(x)} + \frac{u_k(x)}{  \sqrt{f_{k+1}(x)}} \nabla_{xy}f_{k+1} \right)^2 p_1(x,y) \,dy \,dx \\
&  \hspace{4mm}+ \int_M  u_k^2(x)\left( \frac{\abs{\nabla_P f_{k+1}}^2(x)}{4 f_{k+1}(x)}  + \partial_k f(x) \right) \,dx.
\end{align*}
The given condition \eqref{e-condn-f} ensures that $\partial_k J \le 0$, that is $J_{k+1} \le J_k$ for all $k \in \nint{0}{n-1}$.
\end{proof}
The following lemma essentially follow from \cite[Proposition 2.5]{CGZ05}. We repeat the proof for completeness.
Lemma \ref{l-weightfn} provides a weight function $f$ that will be used in the application of discrete integral maximum principle.
\begin{lemma}\label{l-weightfn}
 Let $(M,d,\mu)$ be a  metric measure space  and
let $P$ be a Markov operator  that is $(h,h')$-compatible with $(M,d,\mu)$.
Let $\sigma: M \to \mathbb{R}$ be a 1-Lipschitz function such that $\inf \sigma \ge h'$. There exists a positive number $D_1$ such that
for all $D \ge D_1$, the weight function
\begin{equation}
\label{defn-f} f_k(x)=f_k^D(x):= \exp \left( - \frac{\sigma^2(x) }{D(n+1-k)} \right)
\end{equation}
satisfies \begin{equation*}
 \partial_k f(x) + \frac{ \abs{\nabla_P f_{k+1}}^2}{4 f_{k+1}}(x) \le 0.
\end{equation*} for all $x\in M$, for all $n \in \N^*$ and $k \in \nint{0}{n-1}$.
\end{lemma}
\begin{proof}
Note that
\begin{align}
-\partial_k f(x)&= \left( \exp \left(\frac{\sigma^2(x)}{D(n+1-k)(n-k)} \right) -1\right) f_{k+1}(x) \nonumber \\
\label{e-cdf1}& \ge  \left( \exp\left( \frac{\sigma^2(x)}{2 D(n-k)^2} \right)-1\right) f_{k+1}(x)
\end{align}
and
\begin{align*}
\abs{\nabla_P f_{k+1}(x)}^2 &= \int_M p_1(x,y) \left( \exp \left( - \frac{\sigma^2(y)}{D(n-k)} \right) - \exp \left( - \frac{-\sigma^2(x)}{ D(n-k)} \right) \right)^2  \,dy \\
& =  f_{k+1}^2(x) \int_M p_1(x,y) \left( \exp \left(  \frac{\sigma^2(x)-\sigma^2(y)}{D(n-k)} \right) - 1 \right) ^2  \,dy
\end{align*}
for all $k \in \nint{0}{n-1}$.
By the Lipschitz condition and the hypothesis $\sigma(x) \ge 1$, we have
\[
\abs{\sigma^2(x) - \sigma^2(y)} = \abs{\sigma(x) -\sigma(y)} \abs{\sigma(x) + \sigma(y)}  \le 2 h' \sigma(x) + (h')^2 \le 3h' \sigma(x)
\]
for all $x,y \in M$ such that $d(x,y) \le h'$. Next we use the following elementary inequality: if $\abs{t} \le s$, then
\[
\abs{e^t -1 } \le e^s -1.
\]
Combining together the previous lines and \eqref{e-compat}, we obtain
\begin{equation}
\label{e-cdf2} \abs{ \nabla_P f_{k+1}(x) }^2 \le f_{k+1}^2(x) \left( \exp \left( \frac{3h' \sigma(x)}{D(n-k)} \right) - 1 \right)^2 .
\end{equation}
Next let us use another elementary fact:  there exists $B>0$ such that, for all $t >0$,
\[
(e^t-1)^2\le 4(e^{Bt^2} - 1).
\]
Setting $t=3h'\sigma(x)/(D(n-k))$, we obtain that
\[
\frac{1}{4} \left( \exp \left( \frac{3h'\sigma(x)}{D(n-k)} \right) -1 \right)^2  \le \exp \left( \frac{B (3h')^2 \sigma^2(x)}{ D^2 (n-k)^2} \right) -1.
\]
Hence, if $D \ge D_1: = 2 B \left[3h'\right]^2$, then the right hand side of the above inequality is bounded from above by
\[
\exp \left( \frac{\sigma^2(x)}{2 D(n-k)^2}\right) -1 .
\]
Combining with \eqref{e-cdf1} and \eqref{e-cdf2}, we obtain
\begin{align*}
\frac{ \abs{\nabla f_{k+1}(x)}^2}{ 4 f_{k+1}(x)} &\le \frac{f_{k+1}(x)}{4} \left( \exp\left( \frac{3h'\sigma(x)}{D(n-k)} \right) -1 \right)^2 \\
& \le f_{k+1}(x) \left( \exp \left( \frac{ \sigma^2(x) }{2 D (n-k)^2} \right) - 1 \right) \le - \partial_k f(x)
\end{align*}
for all $x \in M$ and for all $k \in \nint{0}{n-1}$.
\end{proof}
Next, we need the following estimate on $h_k$ defined in \eqref{e-defh}. The proof uses the diagonal estimate in Proposition \ref{p-diag}.
\begin{lemma} \label{l-hdiag}
Under the assumptions of Proposition \ref{p-gue},
there exists $C_0 > 0$ such that
\begin{equation} \label{e-diagh}
\int_{M} h_n^2(x,y) dy \le \frac{C_0}{V(x,\sqrt{n+2})}
\end{equation}
 for all $n \in \mathbb{N}$ and for all $x \in M$  where $h$ is as defined in \eqref{e-defh}.
\end{lemma}
\begin{proof}
By \eqref{e-defh} of  Lemma \ref{l-solheat}, Lemma \ref{l-kernel}(c) and Vandermonde's convolution formula, we have
\begin{align}\label{e-hd1}
 \int_M (h_n(x,y))^2 \, dy & = 4^{-n} \int_M \left( \sum_{i=0}^n \binom{n}{i} p_{i+2}(x,y) \right)^2 \,dy \nonumber \\
 &= 4^{-n}  \sum_{i=0}^{2n} \binom{2n}{i} p_{i+4}(x,x)
\end{align}
for all $x \in M$. By Proposition \ref{p-diag}, there exists $C_1>0$ such that
\[
 p_k(x,x) \le \frac{C_1}{V(x,\sqrt{k})}
\]
for all $k \ge 2$ and for all $x \in M$. Combined with \eqref{e-hd1} and \eqref{e-vd1}, we obtain $C_2>0, \delta>2$ such that
\begin{align}
  \int_M (h_n(x,y))^2 \, dy &\le 4^{-n}  \sum_{i=0}^{2n} \binom{2n}{i}  \frac{C_1}{V(x,\sqrt{i+4})} \nonumber \\
  &\le \frac{C_2}{V(x,\sqrt{2n+4})} 4^{-n}  \sum_{i=0}^{2n} \binom{2n}{i} \left( \frac{2n+4}{i+4}\right)^{\delta/2} \label{e-hd2}
\end{align}
for all $n \in \N$ and all $x \in M$.
By the above inequality, we have
\begin{align}
 \nonumber 4^{-n}  \sum_{i=0}^{2n} \binom{2n}{i} \left( \frac{2n+4}{i+4}\right)^{\delta/2} &\le   4^{-n}  \sum_{i=0}^{2n} \binom{2n}{i} \left( \frac{2n+4}{i+4}\right)^{\kappa} \\
 & \le 4^{2\kappa} \kappa! \sum_{i=0}^{2n} \binom{2n + \kappa}{i+\kappa}  2^{-(2n +\kappa)} \le 4^{2\kappa} \kappa! \label{e-hd3}
\end{align}
where $\kappa := \lceil \delta/2 \rceil \in \N^*$.
Combining \eqref{e-hd2}, \eqref{e-hd3} along with \eqref{e-vd1} implies \eqref{e-diagh}.
\end{proof}
Our next result involves repeated application of the discrete integral maximum principle.
\begin{lemma}\label{l-ED}
  Let $(M,d,\mu)$ be a quasi-$b$-geodesic metric measure space satisfying \ref{doub-loc} and \ref{doub-inf}.
Suppose that a Markov operator $P$ has a kernel $p$ that is $(h,h')$-compatible with $(M,d,\mu)$ for some $h >b$.
Further assume that $P$ satisfies the Sobolev inequality \eqref{e-Sob}. Define
\begin{equation}\label{e-defEd}
 E_D(k,x):= \int_M h_k^2(x,z) \exp \left(\frac{d_1^2(x,z)}{Dk}\right) \,dz
\end{equation}
for all $k \in \N^*$ and $x \in M$,
where  $d_1(x,z):= \max(d(x,z),h')$ and $h_k$ is defined by \eqref{e-defh}.
There exists $C,D>0$ such that
\begin{equation}\label{e-ED}
 E_D(k,x) \le \frac{C}{V(x,\sqrt{k})}
\end{equation}
for all $x \in M$ and for all $k \in \N^*$.
\end{lemma}
\begin{proof}
Let $x \in M$ be an arbitrary point.
The constants below do not depend on the choice of $x$.
Define
\[
 I(R,k)=I(R,k,x):= \int_{B(x,R)^\complement} h_k^2(x,z) \,dz
\]
for $R>0$ and $k \in \N$. We start by estimating $I(R,k)$ using iteration. The iterative step
is contained in the following estimate: There exists $D_1>0$ such that
\begin{equation}\label{e-ed1}
 I(R,n) \le \exp((h')^2/D_1)\left( I(r,k)+ \exp \left(-\frac{(R-r)^2}{2D_1(n-k)} \right)\int_M h_k^2(x,z) \,dz \right)
\end{equation}
for all $R,r$ satisfying $R>r>0$, for all $n \in \N^*$ and for all $k \in \nint{0}{n-1}$.

To prove \eqref{e-ed1}, we define
\[
 \sigma_R(z) := \max(R-d(x,z),0) +h'.
\]
Note that $\sigma_R$ is $1$-Lipschitz with $\inf \sigma_R \ge h'$.
Define
\[
 f_k(z):= \exp \left( - \frac{\sigma_R^2(z) }{D_1(n+1-k)} \right)
\]
for all $z \in M$ and all $k \in \nint{0}{n}$, where $D_1$ is the constant from Lemma \ref{l-weightfn}.
Since $f_k \ge \exp(-(h')^2/D_1)$ in $B(x,R)^\complement$, we have
\begin{equation}\label{e-ed2}
 I(R,n) = \int_{B(x,R)^\complement} h_n^2(x,z) \, dz \le \exp((h')^2/D_1) \int_M h_n^2(x,z) f_n(z) \, dz.
 \end{equation}
By Lemma \ref{l-weightfn} and Proposition \ref{p-imp}, we have
\begin{equation}\label{e-ed3}
 \int_M h_n^2(x,z) f_n(z) \, dz \le \int_M h_k^2(x,z) f_k(z) \, dz
\end{equation}
for all $k \in \nint{0}{n}$.
Since $\sigma_R \ge R-r$ in $B(x,r)$ and $f_k \le 1$, we have
\begin{align}\label{e-ed4}
\int_M h_k^2(x,z) f_k(z) \,dz
&=\int_{B(z,r)^\complement} h_k^2(x,z) f_k(z)\, dz + \int_{B(x,r)} h_k^2(x,z) f_k(z) \,dz   \nonumber\\
& \le   I(r,k) +  \exp \left( - \frac{(R-r)^2}{D_1(n+1-k)} \right) \int_{B(x,r)} h_k^2(x,z)\, dz \nonumber\\
& \le I(r,k)+ \exp \left( - \frac{(R-r)^2}{2D_1(n-k)} \right) \int_{M} h_k^2(x,z) \,dz
\end{align}
for all $k \in \nint{0}{n-1}$ and for all $ R>r>0$.
Combining \eqref{e-ed2}, \eqref{e-ed3} and \eqref{e-ed4}, we obtain \eqref{e-ed1}.
Now by Lemma \ref{l-hdiag} and \eqref{e-ed1}, there exists $C_1>1$ such that
\begin{equation}\label{e-ed5}
 I(R,n) \le \exp((h')^2/D_1)\left( I(r,k)+ \exp \left(-\frac{(R-r)^2}{2D_1(n-k)} \right) \frac{C_1}{V(x,\sqrt{k})}\right)
\end{equation}
for all $n \in \N^*$, for all $k \in \nint{0}{n-1}$ and for all $ R>r>0$.

Next, we show that there exists $C_2,C_3>0$ such that
\begin{equation}\label{e-ed6}
 I(R,k) \le \frac{C_2}{V(x,\sqrt{k})} \exp\left( - \frac{R^2}{C_3 k} \right)
\end{equation}
for all $ R > 10h'$ and for all $ k \in \N^*$.
By \eqref{e-compat} and \eqref{e-defh}, we have $I(R,k)=0$ if $R > (k+2)h'$.
Hence it suffices to consider the case $(k+2) h' \ge R$.

Given any finite decreasing sequence $\{ R_j \}_{j=1}^{j_0}$ of real numbers and any finite strictly  decreasing sequence
$\{ k_j \}_{j=1}^{j_0}$ such that $R_1= R, k_1=k$ and $I(R_{j_0},k_{j_0})=0$, we can iterate \eqref{e-ed2} and obtain
\begin{equation}
\label{e-ed7} I(R,k) \le \sum_{j=1}^{j_0-1} \frac{ C_1 \exp(j(h')^2/D_1)}{V(x,\sqrt{k_{j+1}})} \exp \left( - \frac{ (R_j- R_{j+1} )^2 }{2 D_1 (k_j - k_{j+1})} \right).
\end{equation}
Let $R > 10h'$ and define
\[
R_j:= R/2 + R/(j+1),   t_j := k/2^{j-1}, k_j:= \lceil t_j \rceil
\]
so that $R_1=R$ and $k_1=k$.
Let $j_0 = \min \Sett{j}{ R_j > h'(k_j+2)}$ (note that $j_0 > 1$ since $(k+2) h' \ge R$). By construction one has $I(R_{j_0},k_{j_0})=0$.
Also, for all $j < j_0$ we have $k_j > R_j-1> R/2 - 1$. Since $R>10h'$, we have
\begin{equation*}
t_j- t_{j+1}=t_j/2 \ge (k_j - 1)/2 \ge \frac{1}{2} \left( \frac{R_j}{h'}  - 3\right) \ge \frac{1}{2} \left( \frac{R}{2h'}  - 3\right) >1
\end{equation*}
which means $k_j > k_{j+1}$ for all $j \in \nint{1}{j_0-1}$.
Therefore
\begin{equation}\label{e-ed8}
 k_j - k_{j+1} \le k/2^{j-1}- k/2^{j} + 1= k/2^{j} + 1 \le k/2^{j-1}
\end{equation}
for all $j \in \nint{1}{j_0-1}$.
Using \eqref{e-ed8} and the identity
\[
(R_j - R_{j+1})^2  = \frac{R^2}{(j+1)^2(j+2)^2},
\]
we obtain
\begin{equation}\label{e-ed9}
\frac{(R_j-R_{j+1})^2}{2D_1 (k_j-k_{j+1})} \ge \frac{R^2}{C_3 k} (j+1),
\end{equation}
where
\[
C_3 := \max_{j\ge 1} \frac{D_1 (j+1)^3(j+1)^2}{2^{j-2}} \in (0,\infty)
\]
Therefore by \eqref{e-ed7} and \eqref{e-ed9}, we have
\begin{equation}\label{e-ed10}
I(R,k) \le \sum_{j=1}^{j_0-1} \frac{C_1}{V(x,\sqrt{t_{j+1}})} \exp \left( \frac{j(h')^2}{D_1}- \frac{R^2}{C_3 k}(j+1) \right)
\end{equation}
By \eqref{e-ed8} and \eqref{e-vd1}, there exists $C_4>1$ such that
\[
\frac{V(x,\sqrt{t_j})}{V(x,\sqrt{t_{j+1}})} \le C_4
\]
for all $j \in \nint{1}{j_0-1}$.
Therefore
\[
\frac{V(x,\sqrt{t_1})}{V(x,\sqrt{t_{j+1}})} = \frac{V(x,\sqrt{t_1})}{V(x,\sqrt{t_{2}})} \frac{V(x,\sqrt{t_2})}{V(x,\sqrt{t_{3}})} \ldots \frac{V(x,\sqrt{t_j})}{V(x,\sqrt{t_{j+1}})} \le C_4^j
\]
Thus setting $L:= \log(C_1 C_4)$, we obtain
\[
\frac{C_1}{V(x,\sqrt{t_{j+1}})} \le \frac{1}{V(x,\sqrt{k})} \exp(jL),
\]
for all $j \in \nint{1}{j_0-1}$.
Therefore by \eqref{e-ed10}, we have
\begin{equation} \label{e-ed11}
I(R,k) \le \frac{1}{V(x,\sqrt{k})} \exp\left( -\frac{R^2}{C_3 k} \right) \sum_{j=1}^{j_0-1} \exp \left( -j \left( \frac{R^2}{C_3 k} - L - \frac{(h')^2}{D_1}\right)\right).
\end{equation}
for all $R >10h'$ and for all $k \in \N^*$ satisfying $ R \le (k+2)h'$.
We consider two cases.\\
Case 1: Let
\[
\frac{R^2}{C_3 k} - L - \frac{(h')^2}{D_1} \ge \log 2
\]
In this case, by \eqref{e-ed11} we have
\[
I(R,k) \le \frac{1}{ V(x, \sqrt{k} )} \exp \left( -\frac{R^2}{C_3 k} \right) \sum_{j=1}^{j_0 -1} 2^{-j} \le  \frac{1}{ V(x, \sqrt{k} )} \exp \left( - \frac{R^2}{C_3 k} \right)
\]
Case 2: Let
\[
\frac{R^2}{C_3 k} - L - \frac{(h')^2}{D_1} <\log 2
\]
In this case we estimate $I(R,k)$ differently as
\begin{align*}
I(R,k) &\le \int_M h_k^2(x,z) \,dz \le  \frac{C_1}{V(x,\sqrt{k})} \\
&\le \frac{2 C_1}{V(x,\sqrt{k})} \exp \left( L+\frac{(h')^2}{D_1} -  \frac{R^2}{C_3 k} \right) = \frac{C_6 }{ V(x,\sqrt{k})} \exp \left(- \frac{R^2}{C_3 k} \right).
\end{align*}
Combining the two cases we have \eqref{e-ed6}.

Finally, we are ready to prove \eqref{e-ED}. Define for $j \in \mathbb{N}$,
\[
\mathcal{A}_j^R :=
\begin{cases}
\{ z \in M: d_1(x,z) \le R \} , & j=0 \\
\{ z \in M: 2^{j-1}R <d_1(x,z) \le 2^j R \}, & j \ge 1,
\end{cases}
\]
and
\begin{equation}\label{e-ed12}
 E_D(k,x) = \sum_{j=0}^\infty \int_{ \mathcal{A}^R_j} h_k^2(z) \exp \left( \frac{d_1^2(x,z)}{Dk} \right) \,dz.
\end{equation}
For all $D >0$ and for all $R \ge h'$ the first term admits the estimate
\begin{equation}
\label{e-ed13} \int_{ \mathcal{A}_0^R} h_k^2(x,z) \exp \left( \frac{d_1^2(x,z)}{Dk} \right)\, dz \le  \frac{C_1}{V(x,\sqrt{k})} \exp \left( \frac{R^2}{Dk} \right).
\end{equation}
Now for the remaining terms we have
\begin{equation} \label{e-ed14}
 \int_{\mathcal{A}_j^R} h_k^2(x,z) \exp \left( \frac{d_1^2(x,z)}{Dk} \right)\,dz \le \exp \left( \frac{4^j R^2}{Dk} \right) I(2^{j-1}R,k)
\end{equation}
for all $R > 10 h'$ and $j \in \N^*$. By \eqref{e-ed6}
\[
I(2^{j-1}R,k) \le \frac{C_2}{V(x,\sqrt{k})} \exp \left(- \frac{4^{j-1} R^2}{C_3 k} \right) .
\]
Combining with \eqref{e-ed14}
\begin{align}
\nonumber \int_{\mathcal{A}_j^R} h_k^2(x,z) \exp \left( \frac{d_1^2(x,z)}{Dk} \right) dz & \le  \exp \left( \frac{4^j R^2}{Dk} \right) \frac{C_2}{V(x,\sqrt{k})} \exp \left(-\frac{4^{j-1} R^2}{C_3 k} \right) \\
\label{e-ed16} & \le  \frac{C_2}{V(x,\sqrt{k})}\exp \left(- \frac{4^{j-1} R^2}{D k} \right)
\end{align}
for all $j \in \N^*$,
provided $D \ge 5 C_3$ and $R >10 h'$.
Define
\begin{equation}
\label{e-ed17} D := \max \left( 5 C_3 , \frac{ (11 h')^2 }{\log 2} \right).
\end{equation}
Then by \eqref{e-ed12}, \eqref{e-ed13} and \eqref{e-ed16} we obtain, for all $R > 10h'$
\begin{equation}
\label{e-ed18} E_{D} (k,x)  \le  \frac{C_1}{V(x, \sqrt{k})} \exp \left( \frac{ R^2}{D k }\right) + \frac{C_2} { V(x, \sqrt{k})} \sum_{j=1}^\infty \exp \left( - \frac{4^{j-1}R^2}{D k }\right).
\end{equation}
Given $k \in \mathbb{N}^*$ choose $R$ so that $R^2/(D k) = \log 2$ which by \eqref{e-ed17} satisfies $R > 10h'$.
Therefore by \eqref{e-ed18}, we conclude
\[
E_{D} (k,x) \le \frac{2 C_1}{V(x,\sqrt{k})} + \frac{C_2}{V(x,\sqrt{k})} \sum_{j=1}^\infty 2^{-4^{j-1}} \le \frac{2C_1+ C_2}{V(x,\sqrt{k})}
\]
which is the desired estimate \eqref{e-ED}.
\end{proof}
We use Lemma \ref{l-ED} to prove a Gaussian upper bound for $h_k$.
\begin{lemma}\label{l-hkbd} Under the assumptions of Proposition \ref{p-gue},
there exists positive reals $C_0,D_0$ such that
 \begin{equation}\label{e-hk}
  h_{2k}(x,y) \le \frac{C_0}{  V(x,\sqrt{k}) } \exp \left( - \frac{d^2(x,y)}{D_0 k }\right)
 \end{equation}
for all $x,y \in M$ and for all $k \in \N^*$.
\end{lemma}
\begin{proof}
 By triangle inequality and the inequality $(a+b)^2 \le 2(a^2+b^2)$, we have
 \begin{equation}\label{e-hk1}
  d_1(x,y)^2 \le 2(d_1(x,z)^2 + d_1(y,z)^2)
 \end{equation}
for all $x,y,z \in M$, where $d_1(x,y):= \max(d(x,y),h')$ as before.
By \eqref{e-pcomp}, \eqref{e-hk1} and Cauchy-Schwarz inequality we have
\begin{align}\label{e-hk2}
\nonumber {h_{2k}(x,y)}
&= \sum_{i=0}^{2k} \binom{2k}{i} \left( \frac{1}{2} \right)^{2k} p_{i+2} (x,y) \\
\nonumber & \le  \alpha^{-2}\sum_{i=0}^{2k} \binom{2k}{i} \left( \frac{1}{2} \right)^{2k} p_{i+4} (x,y)
=  \alpha^{-2} \int_M h_k(x,z) h_k(y,z) \,dz\\
\nonumber & \le  \alpha^{-2} \int_M h_k(x,z) h_k(z,y) e^{d_1(x,z)^2/2Dk}  e^{d_1(z,y)^2/2Dk} e^{-d_1(x,y)^2/4Dk} \,dz\\
\nonumber & \le  \alpha^{-2} \sqrt{E_{D} (k,x) E_{D} (k,y)} e^{-d_1^2(x,y)/4Dk} \\
& \le  \alpha^{-2}\sqrt{E_{D} (k,x) E_{D} (k,y)} e^{-d(x,y)^2/4Dk}
\end{align}
for all $x,y \in M$, for all $k \in \N^*$  and for all $D>0$, where $\alpha>0$ is from \eqref{e-da}.
The equality in the second line above follows from a calculation analogous to \eqref{e-hd1}.

The bound \eqref{e-hk2} and Lemma \ref{l-ED} implies that there exists $C_1,D_1>0$ such that
 \begin{equation}\label{e-hk3}
  h_{2k}(x,y) \le \frac{C_1}{ \left( V(x,\sqrt{k}) V(y,\sqrt{k}) \right)^{1/2}} \exp \left( - \frac{d^2(x,y)}{D_1 k }\right)
 \end{equation}
for all $x,y \in M$ and for all $k \in \N^*$. However by \eqref{e-vd1}, there exists $C_2,C_3,C_4>0,\delta>0$ such that
\begin{align}\label{e-hk4}
 \frac{V(x,\sqrt{k})}{V(y,\sqrt{k})} &\le  \frac{V(y,\sqrt{k}+d(x,y))}{V(y,\sqrt{k})}  \le C_2 \left( 1+ \frac{d(x,y)}{\sqrt{k}} \right)^{\delta}  \nonumber \\
 &\le C_3 \left( 1+ \frac{d^2(x,y)}{k} \right)^{\delta/2} \le C_4 \exp \left(  \frac{d^2(x,y)}{2 D_1 k }\right)
\end{align}
for all $x,y \in M$. Combining \eqref{e-hk3} and \eqref{e-hk4} yields the desired Gaussian upper bound \eqref{e-hk}.
\end{proof}

\section{Comparison with lazy random walks}
We want to convert the Gaussian bounds on $h_k$ given by Lemma \ref{l-hkbd} to Gaussian bounds on $p_k$.
To accomplish this we need the following elementary polynomial identities.
\begin{lemma}\label{l-pol}
For all $\beta >0$ and for all $n \in \N^*$, we have the following polynomial identities
 \begin{align}
\label{e-pol1} z^n &= \sum_{ k \in \nint{1}{n}, k \text{ odd} } \binom{n}{k}\beta^{n-k} (z- \beta)^{k-1} z\\
\nonumber &\hspace{4mm} +  \sum_{ k \in \nint{1}{ n-1}, k \text{ odd} } \binom{n-1}{k} \beta^{n-1-k} (z- \beta)^{k-1} (z^2 - 2 \beta z), \\
\label{e-pol2}\left( \frac{1+z}{2}\right)^n &= \frac{1}{2^n} + \sum_{ k \in \nint{1}{n}, k \text{ odd} } \binom{n}{k} \left( \frac{1+\beta}{2} \right)^{n-k} \left( \frac{1}{2} \right)^k (z - \beta)^{k-1}z \\
\nonumber &\hspace{4mm}+ \sum_{k \in \nint{1}{n-1}, k \text{ odd}} s_{n,k} \left( \frac{1+\beta}{2} \right)^{n-1-k} \left( \frac{1}{2} \right)^{k+1} (z-\beta)^{k-1} (z^2 - 2 \beta z)
\end{align}
where $(z -\beta)^0 =1$ and
\[
s_{n,k} = (1+ \beta)^{-(n-1-k)}  \sum_{i=k+1}^n \binom{n}{i}  \binom{i-1}{k}  \beta^{i-1-k} \ge \binom{n-1}{k}.
\]
\end{lemma}

\begin{proof}
Note that
\begin{equation}\label{e-pl1}
z^n= z \left( \frac{z^n - (2 \beta - z)^n }{2(z-\beta)} \right) + (z^2 - 2\beta z) \left( \frac{z^{n-1} - (2 \beta - z)^{n-1} }{2(z-\beta)} \right)
\end{equation}
for all $z \neq \beta$.
To obtain \eqref{e-pol1}, we expand $z^n, z^{n-1}, (2 \beta - z)^n ,  (2 \beta - z)^{n-1}$ in \eqref{e-pl1} using binomial expansion and the substitution
\[
z= \beta+ (z - \beta) \mbox{ and } 2 \beta -z = \beta - (z - \beta).
\]

To show \eqref{e-pol2}, we use binomial expansion on $(1+z)^n$ and then use \eqref{e-pol1} to obtain
\begin{align}\label{e-pl2}
(1+z)^n &=  1+ \sum_{i=1}^n \binom{n}{i} z^i\nonumber \\
&=1+  \sum_{i=1}^n  \sum_{k \in \nint{1}{i}, k \text{ odd} } \binom{n}{i} \binom{i}{k} \beta^{i-k} (z- \beta)^{k-1} z\nonumber\\
& \hspace{4mm} +  \sum_{i=1}^n \sum_{ k \in \nint{1}{i-1}, k \text{ odd} } \binom{n}{i} \binom{i-1}{k} \beta^{i-1-k} (z- \beta)^{k-1} (z^2 - 2 \beta z).
\end{align}
The coefficient of $(z- \beta)^{k-1}z$ in \eqref{e-pl2} is
\[
\sum_{i=k}^n  \binom{n}{i}\binom{i}{k} \beta^{i-k} = \binom{n}{k} \sum_{i=k}^n   \binom{n-k}{i-k} \beta^{i-k}
=  \binom{n}{k}  ( 1+ \beta)^{n-k}.
\]
Similarly,  the coefficient of $(z- \beta)^{k-1}(z^2 - 2 \beta z)$ in \eqref{e-pl2} is
\begin{align*}
\sum_{i=k+1}^n \binom{n}{i}  \binom{i-1}{k}  \beta^{i-1-k} &= \binom{n-1}{k} \sum_{i=k+1}^n \frac{n}{i} \binom{n-1-k}{i-1-k} \beta^{i-1-k} \\
&\ge  \binom{n-1}{k} \sum_{i=k+1}^n \binom{n-1-k}{i-1-k} \beta^{i-1-k} \\
&= \binom{n-1}{k} (1+ \beta)^{n-1-k}.
\end{align*}
This gives \eqref{e-pol2} with $s_{n,k}   \ge \binom{n-1}{k} $.
\end{proof}
We are now prepared to prove Gaussian upper bounds for $p_k$.
\begin{proof}[Proof of Proposition \ref{p-gue}]
 By Lemma \ref{l-con-d} there exists $\beta >0$ such that
 $u_k,v_k : M \times M \to \R$ satisfy
 \begin{align}
u_{k}(x,y) &:= \left[(P - \beta I)^{k}  p_2(x,\cdot) \right](y) \ge 0, \label{e-gu1} \\
v_{k}(x,y) &:= \left[(P - \beta I)^{k} (p_3(x,.) - 2 \beta p_2(x,\cdot)) \right](y) \ge 0 \label{e-gu2}
\end{align}
for all $x,y \in M$ and for all even non-negative integers $k$. For instance $\beta=\alpha/2$ where $\alpha$ is given by \eqref{e-da} would satisfy the above requirements.

Using Lemma \ref{l-kernel}(c) and \eqref{e-pol1} of Lemma \ref{l-pol}, we have
\begin{align}
\label{e-gu3} p_{n+1}(x,y) &= \left[P^n p_1(x,\cdot) \right](y) \nonumber  \\
& =\sum_{ k \in \nint{1}{n}, k \text{ odd} } \binom{n}{k} \beta^{n-k} u_{k-1}(x,y)\nonumber \\
& \hspace{4mm}+  \sum_{ k \in \nint{1}{n-1}, k \text{ odd} } \binom{n-1}{k} \beta^{n-1-k} v_{k-1}(x,y)
\end{align}
for all $n \in \N^*$ and for all $x,y \in M$. By \eqref{e-gu1}, \eqref{e-gu2}, Lemma \ref{l-con-d} and Lemma \ref{l-pol}, we have
\begin{align}\label{e-gu4}
\nonumber h_{2n}(x,y) &= \left[\left((I+P)/2 \right)^{2n} p_2(x,\cdot) \right](y) \\
 & \ge \alpha  \sum_{ k \in \nint{1}{2n}, k \text{ odd} } \binom{ 2n}{k} \left( \frac{1+\beta}{2} \right)^{2n-k} \left( \frac{1}{2} \right)^k u_{k-1}(x,y) \nonumber \\
 & \hspace{4mm} + \alpha \sum_{1 \le k \le 2n-1, k \text{ odd}} s_{2n,k} \left( \frac{1+\beta}{2} \right)^{2n-1-k} \left( \frac{1}{2} \right)^{k+1} v_{k-1}(x,y)
\end{align}
for all $x,y \in M$.
Define the ratio of coefficients in \eqref{e-gu3} and \eqref{e-gu4} as
\begin{equation}\label{e-gu5}
a_{k,n} = \frac{ \binom{ 2n}{k} \left( \frac{1+\beta}{2} \right)^{2n-k} \left( \frac{1}{2} \right)^k}{ \binom{n}{k} \beta^{n-k} } \mbox{ and }
b_{l,n} = \frac{  \binom{2n-1}{l} \left( \frac{1+\beta}{2} \right)^{2n-1-l} \left( \frac{1}{2} \right)^{l+1} } { \binom{n-1}{l} \beta^{n-1-l} }
\end{equation}
for each $k \in \nint{1}{n}$ and for each $l \in \nint{1}{l-1}$. If $k \in \nint{1}{n-1}$, then
\[
\frac{a_{k+1,n}}{a_{k,n}}= \frac{ \beta}{1 + \beta} \frac{2n-k}{n-k}.
\]
Therefore $a_{k+1,n} \ge a_{k,n}$ if and only if $k \ge n (1 - \beta)$. Thus $a_{k,n}$ reaches minimum for  $k = \lceil n( 1 -\beta) \rceil$.
By Stirling's approximation there exists constant $C_1> 0$ such that for all $r \in \mathbb{N}^*$,
\[
C_1^{-1} r^{r+(1/2)} e^{-r} \le r! \le C_1 r^{r+(1/2)} e^{-r} .
\]
We use the Stirling's approximation to estimate $a_{k,n}$ at $k = n ( 1 - \beta) + \epsilon$ where $ \epsilon = \lceil n( 1 -\beta) \rceil-n ( 1 - \beta) \in [0,1)$.
There exists $c_1>0$ such that
\begin{align*}
\min_{k \in \nint{1}{n}}a_{k,n} & \ge  a_{\lceil n( 1 -\beta) \rceil,n} \nonumber \\
\nonumber & \ge
\frac{C_1^{-4}  (2n)^{2n+(1/2)}  e^{-2n} (\beta n - \epsilon )^{\beta n + (1/2) - \epsilon}
e^{-\beta n + \epsilon} \left( 1+\beta \right)^{(1+\beta)n - \epsilon} }{ 2^{2n} n^{n+(1/2)} e^{-n} (n(1+\beta) - \epsilon)^{n(1+\beta) +(1/2)- \epsilon}e^{-n(1+\beta)+\epsilon} \beta^{\beta n- \epsilon} }\\
& \ge  c_1
\end{align*}
for all $n \in \N^*$ satisfying $n \ge 2/\beta$.
Therefore there exists $c_2>0$ such that
\begin{equation} \label{e-gu6}
a_{k,n} \ge c_2
\end{equation}
for all $n \in \mathbb{N}^*$ and for all $ k \in \nint{1}{k}$.
Similarly,
\begin{equation} \label{e-gu7}
b_{l,n} = \frac{1}{2} a_{l+1,n} \ge \frac{1}{2} c_2
\end{equation}
for all  $n \in \mathbb{N}^*$ and for all $l \in \nint{1}{n-1}$.
Combining \eqref{e-gu1}, \eqref{e-gu2}, \eqref{e-gu3}, \eqref{e-gu4}, \eqref{e-gu5}, \eqref{e-gu6} and\eqref{e-gu7}, there exists $c_3 >0$ such that
\begin{equation}\label{e-gu8}
h_{2n}(x,y) \ge c_3 p_{n+1}(x,y)
\end{equation}
for all $n \in \N^*$, and for all $x,y \in M$. Combining \eqref{e-gu8} along with Lemma \ref{l-hkbd} yields the Gaussian upper bound \eqref{e-gue}.
\end{proof}
We have shown the following equivalence
\begin{theorem} \label{t-main2}
 Let $(M,d,\mu)$ be a quasi-$b$-geodesic metric measure space satisfying \ref{doub-loc}.
Suppose that a Markov operator $P$ has a kernel $p$ that is $(h,h')$-compatible with $(M,d,\mu)$ for some $h >b$. Then the following are equivalent:
\begin{enumerate}[(i)]
 \item Sobolev inequality \eqref{e-Sob}.
 \item Large scale volume doubling property \ref{doub-inf} and Gaussian upper bounds \ref{gue}.
\end{enumerate}
\end{theorem}
\begin{proof}
By Corollary \ref{c-nsob}, (ii) implies (i).

Next, we assume the Sobolev inequality \eqref{e-Sob}.
 By Proposition \ref{p-sobvd} we have \ref{doub-inf}. In addition, by Proposition \ref{p-gue} we have \ref{gue}.
 This proves (i) implies (ii).
\end{proof}

\chapter{Gaussian lower bounds}\label{ch-glb}
In this chapter, we use elliptic Harnack inequality and Gaussian upper bounds to establish Gaussian lower bounds.
The proofs in this chapter is adapted from \cite{HS01}.
 In \cite{HS01}, Hebisch and Saloff-Coste provide an alternate approach to prove parabolic Harnack inequality using
 elliptic Harnack inequality and Gaussian upper bounds.
 This method avoids  relying on the full strength of Moser's iteration method in parabolic setting.

 Although \cite{HS01} concerns diffusions on strictly local Dirichlet spaces, we will see that
 their methods can be extended to discrete time Markov chains on quasi-geodesic spaces.
 This extension was alluded to in \cite{HS01}  where the authors say ``This route to the parabolic Harnack inequality seems especially valuable
 in the setting of analysis on graphs which is not covered by the present strictly local Dirichlet space framework.
In fact, the results above originated from our desire to overcome some of the difficulties that appear in the case of graphs. This will be developed elsewhere.''

The main result of this chapter is the following Gaussian lower bound.
\begin{prop}\label{p-gle}
 Let $(M,d,\mu)$ be a quasi-$b$-geodesic metric measure space satisfying  \ref{doub-loc}, \ref{doub-inf}, $\operatorname{diam}(M)= \infty$ and Poincar\'{e} inequality at scale $h$ \ref{poin-mms}.
Suppose that a Markov operator $P$ has a kernel $p$ that is $(h,h')$-compatible with respect to $\mu$ for some $h >b$. Then the corresponding kernel $p_k$ satisfies Gaussian lower bounds \ref{gle}.
\end{prop}
Note that under the assumptions of Proposition \ref{p-gle}, we have Gaussian upper bounds \ref{gue}.
This is a direct consequence of Theorem \ref{t-Sob} and   Proposition \ref{p-gue}.

We focus on the case $\operatorname{diam}(M)=\infty$ just for simplicity.
In fact, we expect these methods to work when $\operatorname{diam}(M) < \infty$.
However when the space has finite diameter, it is important to find optimal constants (or close to optimal) for various functional inequalities.
To compute these optimal constants, one has to exploit the specific structure of the Markov chain under consideration. We plan to address the finite diameter case in a sequel.

 \section{On-diagonal lower bounds}
The first step is to obtain lower bounds on $p_k(x,x)$. It is well-known that Gaussian upper bounds implies a matching diagonal lower bounds.
We repeat the proof for convenience.
\begin{lemma} \label{l-dlb}
Under the assumption of Proposition \ref{p-gle}, there exists $c_0>0$ such that
\[
p_n(x,x) \ge  \frac{c_0}{V(x,\sqrt{n})}
 \]
 for all $x \in M$ and for all $n \in \N$ satisfying $n \ge 2$.
\end{lemma}

\begin{proof}
By Lemma \ref{l-con-d} it suffices to prove the inequality for even $n$,
since there exists $\alpha>0$ such that \[p_{2k+1}(x,x) \ge \alpha p_{2k}(x,x)\] for all $x \in M$ and for all $k \in \N^*$.

Let $n \in \mathbb{N}^*$ be even.
By Cauchy-Schwarz inequality and Lemma \ref{l-kernel}(c), we have
\begin{align}
\nonumber p_n(x,x) &=  \int_M p_{n/2}^2(x,y)\,dy \ge \int_{B(x, \sqrt{T})}  p_{n/2}^2(x,y) \,dy  \\
\nonumber & \ge  \frac{1}{V(x,\sqrt{T})} \left( \int_{B(x, \sqrt{T})}  p_{n/2}(x,y) dy \right) ^2 \\
\label{e-dlb1} & =  \frac{1}{V(x,\sqrt{T})} \left( 1 - \int_{B(x, \sqrt{T})^\complement}  p_{n/2}(x,y) dy \right) ^2 
\end{align}
for all $T>0$ and for all $n \in 2\N^*$.

By Theorem \ref{t-Sob} and   Proposition \ref{p-gue} we have \ref{gue}.
By \ref{gue}, there exists $C_1,C_2>0$
\begin{equation}
\label{e-dlb2} p_{k}(x,y) \le \frac{C_1}{V(x,\sqrt{k})} \exp \left( - \frac{d^2(x,y)}{C_2 k} \right)
\end{equation}
for all $x,y \in M$ and $ k \in \mathbb{N}^*$.
There exists $C_3>1$ such that for all $A>\max(1,(8 C_2 \delta)^2)$, we have
\begin{align}\label{e-dlb3}
\nonumber \int_{B(x, \sqrt{Ak} )^\complement} p_{k}(x,y) \,dy &= \sum_{i=1}^ \infty \int_{ 2^{i-1} \sqrt{Ak} < d(x,y) \le 2^i \sqrt{Ak} } p_k(x,y) \,dy \\
\nonumber & \le  C_1 \sum_{i=1}^ \infty \frac{ V(x, 2^i \sqrt{Ak})} { V(x,\sqrt{k})} \exp \left( - \frac{4^iT}{ 4 C_2 k} \right) \\
\nonumber & \le  C_3 \sum_{i=1}^ \infty  \exp \left( \delta \log \left( {2^i \sqrt{A}} \right)- \frac{4^i A}{ 4 C_2 } \right) \\
\nonumber & \le  C_3 \sum_{i=1}^ \infty   \exp \left( \delta {2^i \sqrt{A}}-  \frac{4^i A}{ 4 C_2 }  \right)  \\
& \le   C_3 \sum_{i=1}^ \infty  \exp \left( - \frac{4^i A}{ 8 C_2 } \right)  \le   C_4  \exp \left(  - \frac{A}{2C_2} \right)
\end{align}
for all $k \in \N^*$ and for all $x \in M$.
We used  \eqref{e-dlb2} in the second line above and \eqref{e-vd1} in the third line.
By \eqref{e-dlb3}, there exists $A_1>1$ such that
\begin{equation}\label{e-dlb4}
\int_{B(x, \sqrt{A_1 k} )^\complement} p_{k}(x,y) \, dy <  1/2
\end{equation}
for all $k \in \N^*$ and for all $x \in M$.
We choose $T= A_1(n/2)$ in \eqref{e-dlb1} and use \eqref{e-dlb4} and \eqref{e-vd1}, to obtain
\[
 p_n(x,x) \ge \frac{1}{2 V(x, (A_1 n/2)^{1/2})} \ge \frac{c_1}{V(x,\sqrt{n})}
\]
for all $n \in 2\N^*$ and for all $x \in M$.
\end{proof}
The following lemma is a discrete time analog of \cite[Lemma 3.7]{HS01}, where we transfer the
on-diagonal lower bound given by Lemma \ref{l-dlb} to on-diagonal lower bound for the `Dirichlet kernel' $p^B_k$ on a  ball $B$ defined in \eqref{e-defpB}.
\begin{lemma} \label{l-balldlb}
Under the assumptions of Proposition \ref{p-gle},
there exists $c > 0$ and $A>\max(1,h')$ such that
\[
p^{B(x,r)}_n(x,x) \ge \frac{c}{V(x,\sqrt{n})}
\]
 for all $x \in M$, for all $n \in \mathbb{N}^*$ with $n \ge 2$ and for all $r \ge A \sqrt{n}$
\end{lemma}
\begin{proof}
We abbreviate $B(x,r)$ by $B$. We denote the exit time from ball $B$ by
\[
 \tau:= \min \Sett{k}{X_k \notin B}
\]
where $(X_k)_{k \in \N}$ is the Markov chain driven by the kernel $p_k$.

By strong Markov property, the Dirichlet kernel $p_n^B$ can be expressed in terms of $p_k$ as
\begin{equation} \label{e-blb1}
p^B_n(x,x) = p_n(x,x)  - \EE_x \left[ p_{n-\tau}(X_\tau,x) \one_{\nint{1}{n-1}}(\tau)\right]
\end{equation}
for all $n \ge 2$ and for all $x \in M$,
where $\EE_x$ denotes that $X_0=x$.
 If we choose $A>h'$, by \eqref{e-compat},we can rewrite \eqref{e-blb1} as
\begin{equation} \label{e-blb2}
p^B_n(x,x)= p_n(x,x)  - \EE_x \left[ p_{n-\tau}(X_\tau,x) \one_{\nint{2}{n-2}}(\tau)\right]
\end{equation}
for all $n \ge 2$ and for all $x \in M$ with $B=B(x,r)$ satisfying $r > h'$.
For the first term in \eqref{e-blb2}, by Lemma \ref{l-dlb}, there exists $c_1>0$ such that
\begin{equation}\label{e-blb2a}
p_n(x,x) \ge \frac{c_1}{V(x,\sqrt{n})}
\end{equation}
for all $x \in M$ and for all $n\ge 2$.

We use Gaussian upper bound \ref{gue} to estimate the second term in \eqref{e-blb2}. There exists $C_1,C_2,C_3,C_4,\delta>0$ and such that
\begin{align}\label{e-blb3}
\nonumber\EE_x \left[ p^B_{n-\tau}(X_\tau,y) \one_{\nint{1}{n-1}}(\tau)\right] & \le  \sup_{l \in \nint{2}{n-2}} \sup_{y \notin B(x,r)} \frac{C_1}{V(x,\sqrt{l})} e^{- d(x,y)^2/(C_2 l) } \\
& \le  \sup_{l \in \nint{2}{n-2}}\frac{C_1}{V(x,\sqrt{l})} e^{- (A^2 n)/(C_2 l) } \nonumber \\
\nonumber & \le \frac{C_3}{V(x,\sqrt{n})} \sup_{l \in \nint{2}{n-2}} (n/l)^{\delta/2} e^{- (A^2 n)/(C_2 l) } \nonumber \\
 & \le  \frac{C_4}{A^\delta V(x,\sqrt{n})}
\end{align}
for all $x \in M$, for all $n \ge 2$, for all $A > h'$ and for all $B=B(x,r)$ with $r \ge A \sqrt{n} > h'$.
In the first line above we used \eqref{e-dlb2}, in the second line we used $d(x,y) \ge r \ge A \sqrt{n}$ and in the third line we used \eqref{e-vd1}.

Clearly we can choose $A>h'$ large enough such that $C_4/A^\delta < c_1/2$.
Therefore by \eqref{e-blb2},\eqref{e-blb2a} and \eqref{e-blb3}, we obtain the desired bound.
\end{proof}
 \section{Spectrum of the Dirichlet Laplacian on balls}
Our next result is a bound on the spectrum of $P_B$ or alternatively on the Dirichlet Laplacian $\Delta_{P_B}$.
The following Proposition is a discrete time analog of \cite[Theorem 2.5]{HS01}.
However unlike \cite{HS01}, we cannot apply the stronger Sobolev inequality \eqref{e.sob}.
\begin{prop}\label{p-spec}
 Let $(M,d,\mu)$ be a quasi-$b$-geodesic metric measure space satisfying  \ref{doub-loc}, \ref{doub-inf} and Poincar\'{e} inequality at scale $h$ \ref{poin-mms}.
Suppose that a Markov operator $P$ has a kernel $p$ that is $(h,h')$-compatible with respect to $\mu$ for some $h >b$. Then there exists positive reals $a,\epsilon_0$ such that
 \begin{equation}\label{e-spec}
 \norm{P_{B(x,r)}}_{2 \to 2} := \sup_{f \in L^2(B(x,r)), \norm{f}_2=1} \norm{P_{B(x,r)}f}_2 \le 1- \frac{a}{r^2}
 \end{equation}
 for all $x \in M$ and for all $r \in \R$ satisfying $r \ge h'$ and $r \le \epsilon_0\operatorname{diam}(M)$.
\end{prop}
\begin{proof}
We abbreviate the ball $B(x,r)$ by $B$.
Note that $P_B$ is a contraction in $L^2(B)$, that is $\norm{P_B}_{2\to 2} \le 1$.
Since $P_B$ is a bounded, self-adjoint operator in $L^2(B)$, by \cite[Proposition 2.13]{Con90} we have
\begin{equation}\label{e-spc1}
  \norm{P_{B}}_{2 \to 2} = \sup_{f \in L^2(B), f \not\equiv 0} \frac{ \abs{\langle f, P_B f \rangle_{B}} }{\norm{f}_2^2}
\end{equation}
where $\langle \cdot,\cdot \rangle_B$ denotes the inner product in $L^2(B)$.
Therefore it suffices to show that there exists positive reals $a,\epsilon_0$ such that
\begin{equation} \label{e-spc2}
 -\left(1- \frac{a}{r^2} \right) \le \frac{\langle f, P_B f \rangle_B }{\norm{f}_2^2} \le 1- \frac{a}{r^2}
\end{equation}
for all $f \in L^2(B)$ and for all $B=B(x,r)$ with $r \ge h'$ and $r \le \epsilon_0 \operatorname{diam}(M)$.

We prove \eqref{e-spc1} in two steps. We start with the proof of upper bound in \eqref{e-spc2}.
With slight abuse of notation, we consider $L^2(B) \subseteq L^2(M)$ using the map given by \eqref{e-extend}.
By this identification, a function $f \in L^2(M)$ with $\supp(f) \subseteq B$ can be considered to be in $L^2(B)$.

By Lemma \ref{l-dircomp-ball}(a), we can rewrite the upper bound in \eqref{e-spc2} as
\begin{equation}\label{e-spc3}
 \frac{\E(f,f) }{\norm{f}_2^2}=  \frac{\E^B(f,f) }{\norm{f}_2^2}= \frac{\norm{f}_2^2  - \langle f, P_B f \rangle_B }{\norm{f}_2^2} \ge \frac{a}{r^2}.
\end{equation}
Since $\E(\abs{f},\abs{f}) \le \E (f,f)$, in order to show \eqref{e-spc3} it suffices to consider the case $f \ge 0$.

By \eqref{e-ns11} and \eqref{e-ns12} of Proposition \ref{p-nash} along with Lemma \ref{l-comp-dir}(b), there exists
 $C_N>0$ such that
 \begin{equation}
 \label{e-spc4} \norm{P f}_{2}^{2 + (4/\delta)} \le \frac{C_N r^2}{V(x,r)^{2/\delta}} \left( \E_*(f,f) + r^{-2} \norm{Pf}_2^2 \right) \norm{f}_1^{4/\delta}
\end{equation}
for all $x \in M$, for all $r > 0$ and for all functions $f \in L^2(M)$ supported in $B(x,r)$.
By \eqref{e-spc4}, we have
 \begin{equation}
 \label{e-spc5} \norm{P f}_{2}^{2} \left(  \frac{\norm{Pf}_2^{4/\delta}}{\norm{f}_1^{4/\delta}} - \frac{C_N }{V(x,Kr)^{2/\delta}} \right) \le \frac{C_N (Kr)^2}{V(x,Kr)^{2/\delta}} \left( \norm{f}_2^2 - \norm{Pf}_2^2\right)
\end{equation}
for all $x \in M$, for all $r > 0$, for all $K>1$ and for all functions $f \in L^2(M)$ supported in $B(x,r)$.
If $f \ge 0$ , we have \[ \norm{Pf}_1= \langle Pf , \one \rangle = \langle f ,P \one \rangle=\langle f , \one \rangle=\norm{f}_1. \]
Hence by H\"{o}lder inequality, \eqref{e-compat} and \eqref{e-vd1}, there exists $C_1>0$ such that
\begin{equation}\label{e-spc5b}
 \norm{f}_1 = \norm{Pf}_1 \le  \left( V(x,r+h')\right)^{1/2} \norm{P f}_2 \le C_1 V(x,r)^{1/2} \norm{Pf}_2
\end{equation}
for all $f \ge 0$ with $f \in L^2(M)$ and $\supp(f) \subseteq B(x,r)$ and $r \ge h'$.
Combining \eqref{e-spc2} and \eqref{e-spc5b}, we have
\begin{equation}
 \label{e-spc6} \norm{P f}_{2}^{2} \left( C_1^{-4/\delta} - \frac{C_N V(x,r)^{2/\delta} }{V(x,Kr)^{2/\delta}} \right) \le C_N (Kr)^2 \left( \norm{f}_2^2 - \norm{Pf}_2^2\right)
\end{equation}
for all $K>1$, for all $r \ge h'$, for all $x \in M$ and for all $f \in L^2(M)$ with $\supp(f) \subseteq B(x,r)$ and $f \ge 0$.
By Lemma \ref{l-rvd}, there exists $K>1$ such that
\begin{equation} \label{e-spc7}
  \frac{C_N V(x,r)^{2/\delta} }{V(x,Kr)^{2/\delta}} < \frac{1}{2} C_1^{-4/\delta}
\end{equation}
for all $x \in M$, for all $r \ge h'$ and all $r \le \operatorname{diam}(M) / K$.
Combining \eqref{e-spc6} and \eqref{e-spc7}, there exists $\epsilon_0=K^{-1}>0$ ,$C_2>0$ such that
\begin{equation}\label{e-spc8}
 \norm{Pf}_2^2 \le C_2 r^{2}  \left( \norm{f}_2^2 - \norm{Pf}_2^2\right)
\end{equation}
 for all $x \in M$, for all $f \in L^2(M)$ with $\supp(f) \subseteq B(x,r)$ and $f \ge 0$, where $r$ satisfies $r \ge h'$ and $r \le \epsilon_0 \operatorname{diam}(M)$.
By Lemma \ref{l-comp-dir}(a) and \eqref{e-spc8}, there exists $a>0$ such that
\begin{equation}\label{e-spc9}
 \frac{\E(f,f)}{\norm{f}_2^2} \ge \frac{\E(\abs{f},\abs{f})}{\norm{f}_2^2} \ge  \frac{\E_*(\abs{f},\abs{f})}{2 \norm{\abs{f}}_2^2}
 =\frac{1}{2}\left( 1 -  \frac{\norm{\left(P\abs{f}\right)}_2^2}{\norm{f}_2^2}\right)  \ge \frac{1}{2(1+C_2r^2)} \ge \frac{a}{r^2}
\end{equation}
 for all $x \in M$, for all $f \in L^2(M)$ with $\supp(f) \subseteq B(x,r)$, where $r$ satisfies $r \ge h'$ and $r \le \epsilon_0 \operatorname{diam}(M)$.
Therefore by \eqref{e-spc3} and \eqref{e-spc9}, there exists $\epsilon_0>0$ and $a>0$ such that
\begin{equation} \label{e-spc10}
 \frac{\langle f, P_B f \rangle_B }{\norm{f}_2^2} \le 1- \frac{a}{r^2}
\end{equation}
for all $f \in L^2(B)$ and for all $B=B(x,r)$ with $r \ge h'$ and $r \le \epsilon_0 \operatorname{diam}(M)$.
By integration by parts \eqref{e-int-p} and symmetry of $p_1$ we have
\begin{align}\label{e-spc11}
\nonumber{\mathcal{E}(f, f)+ \mathcal{E}(\abs{f}, \abs{f}) }
&= \frac{1}{2} \int_M \int_M p_1(x,y)  \left[ (\nabla_{xy}f)^2 + ( \nabla_{xy} \abs{f} ) ^2 \right] \,dy \,dx \\
& \le  \int_M \int_M p_1(x,y) ( f(x)^2 + f(y)^2) \, dy \, dx = 2 \norm{f}_2^2
\end{align}
for all $f \in L^2(M)$.
Combining \eqref{e-spc9} and \eqref{e-spc11}, there exists $a,\epsilon_0>0$ such that
\begin{equation}\label{e-spc12}
  \frac{\E(f,f)}{\norm{f}_2^2} \le 2 - \frac{\E(\abs{f},\abs{f})}{\norm{f}_2^2} \le 2 - \frac{a}{r^2}
\end{equation}
 for all $x \in M$, for all $f \in L^2(M)$ with $\supp(f) \subseteq B(x,r)$, where $r$ satisfies $r \ge h'$ and $r \le \epsilon_0 \operatorname{diam}(M)$.
Therefore by \eqref{e-spc12} and Lemma \ref{l-dircomp-ball}(a), there exists $\epsilon_0>0$ and $a>0$ such that
\begin{equation} \label{e-spc13}
 \frac{\langle f, P_B f \rangle_B }{\norm{f}_2^2} \ge - \left(1- \frac{a}{r^2}\right)
\end{equation}
for all $f \in L^2(B)$ and for all $B=B(x,r)$ with $r \ge h'$ and $r \le \epsilon_0 \operatorname{diam}(M)$.
Combining \eqref{e-spc10} and \eqref{e-spc13} yields \eqref{e-spc2}, which along with \eqref{e-spc1} implies \eqref{e-spec}.
\end{proof}
\begin{remark}\label{r-fin}\leavevmode
\begin{enumerate}[(a)]
 \item A simple consequence of Proposition \ref{p-spec} is that there exists $a,\epsilon_0 >0$ such that
\[
 \spec(P_B) \subseteq \left[ - \left(1 - a r^{-2} \right), 1 - ar^{-2} \right], \hspace{4mm}  \spec(\Delta_{P_B}) \subseteq \left[ a r^{-2}, 2 - ar^{-2} \right]
\]
 for all $x \in M$ and for all $r$ satisfying $r \ge h'$ and $r \le \epsilon_0\operatorname{diam}(M)$.
\item If $\operatorname{diam}(M)= \infty$, then for all balls $B=B(x,r)$ with $r \in (0,\infty)$, we have  \[\norm{P_B}_{2 \to 2} < 1.\] The case $r \ge h'$ is clear from Proposition \ref{p-spec}. The case $r < h'$ follows
from $\norm{P_B}_{2 \to 2} \le \norm{P_{B(x,h')}}_{2 \to 2}$.
\item Note that if $\operatorname{diam}(M) < \infty$, then the conclusion Proposition \ref{p-spec} is vacuously true as one can choose $\epsilon_0 = h'/(2 \operatorname{diam}(M))$.
However if $h' \ll \operatorname{diam}(M)$ and if we have good control of the constants in various functional inequalities,
we  can prove useful estimates which in turn yields applications to estimates on mixing times. We will extend the techniques developed here to finite diameter spaces elsewhere.
\item  Note that the condition $r \le \epsilon_0 \operatorname{diam}(M)$ is necessary. Too see this consider the case when $\operatorname{diam}(M) < \infty$ and $B(x,r)=M$.
It is clear that \eqref{e-spec} fails to be true because $P_{B(x,r)} \one = \one$.
\end{enumerate}
 \end{remark}
 \section{Near diagonal lower bound}
As in \cite[Proposition 3.5]{HS01}, the following near diagonal estimate is an important step
in obtaining Gaussian lower bounds.
\begin{prop}[Near diagonal lower bound]\label{p-ndlb}
Under the same assumptions as in Proposition \ref{p-gle}, there exists  positive reals $\epsilon_1, c_1$  such that $p_k$ satisfies the lower bound
\begin{equation}
\label{e-ndlb} \inf_{y \in B(x, \epsilon_1 \sqrt{k})} p_k(x,y) \ge \frac{c_1} { V(x,\sqrt{k})}
\end{equation}
for all $x \in M$ and for all $k \in \N^*$ satisfying $k \ge 2$.
\end{prop}
From the above near diagonal lower bound, we will see that the  Gaussian lower bound follows by a well-established `chaining argument'.

The idea behind the proof of Proposition \ref{p-ndlb} is to convert the elliptic H\"{o}lder-like regularity estimate (Proposition \ref{p-holder}) into a
parabolic H\"{o}lder-like regularity estimate for the function $(k,y) \mapsto p^B_k(x,y)$ as follows:
\begin{lemma} \label{l-crl}
Under the assumptions of Proposition \ref{p-gle},
for all $\sigma > 0$ and all $A \ge 1$, there exists three positive reals $C_{\sigma,A}$, $\epsilon_0 \le A$ and $N_0 \ge 2$ such that
\begin{equation}\label{e-crl}
\abs{p^B_k(x,y)- p^B_k(x,x)} \le \left[ \sigma+ C_{\sigma,A} \left( \frac{d(x,y) \vee 1}{ \sqrt{k}} \right)^\alpha \right] \frac{1}{ V(x,\sqrt{k})}
\end{equation}
for all $x \in M$, $k \in \mathbb{N}^*$ with $k \ge N_0$ and for all $y \in  B(x,\epsilon_0 \sqrt{k})$,
where $B=B(x,A\sqrt{k})$ and $\alpha$ is the exponent in \eqref{e-holder}.
\end{lemma}
The proof of Lemma \ref{l-crl} is long and involves many technical estimates.
We will need some upper bounds on $p_k^B(y,z)$ and its `time derivative'
\[
 \partial_k p^B(y,z) := p_{k+1}(y,z) - p_k(y,z)
\]
for all $y,z \in B$.
\begin{lemma}\label{l-tech}
Under the assumptions of Proposition \ref{p-gle}, the following estimates hold:
\begin{enumerate}[(i)]
\item There exists $C_1, D_1 >0 $ such that
\begin{equation} \label{e-th1}
 p^{B(x,A\sqrt{k})}_j(y,z) \le \frac{C_1}{V(y, \sqrt{j})} \exp \left( - \frac{d(y,z)^2}{D_1 j} \right)
\end{equation}
for all $x \in M$, for all $k \in \N^*$, for all $j \ge 2$, for all $A \ge 1$ and for all $y,z \in B(x,A\sqrt{k})$.
\item There exists $C_2, \delta > 0$ such that
\begin{equation} \label{e-th2}
 \abs{ \partial_k p^{B(x,A\sqrt{k})}(y,z)} \le \frac{C_2 A^\delta}{k V(x, \sqrt{k})}
\end{equation}
for all $x \in M$, for all $k \in \N_{\ge 2}$, for all $A \ge 1$ and for all $y,z \in B(x,A\sqrt{k})$.
\item For all $A > 1 \vee h'$, there exists $\epsilon, a_1 >0 $, such that for all $\theta \in (0,1)$, there exists $C_\theta$ such that,
\begin{equation} \label{e-th3}
 p^{B(x,A\sqrt{k})}_j(y,z) \le \frac{C_\theta A^\delta}{V(x, \sqrt{k})} \left( 1 - \frac{a_1}{A^2 k}\right)^j
\end{equation}
for all $x \in M$, for all $k \in \N^*$, for all $j \in \N$ satisfying $j \ge \max(2,\theta k)$ and for all $y,z \in B(x,A\sqrt{k})$.
\end{enumerate}
\end{lemma}
\begin{proof}
 The first inequality \eqref{e-th1} follows from Proposition \ref{p-gue} and the inequality $p^{B(x,A\sqrt{k})}_j \le p_j$ for all $j \ge 2$.

 For $k \ge 20$, we decompose $k=k_1+k_2+k_3+k_4$ such that $k_1,k_3 \in 2 \N^*$, $k_i \in \N^*$ and $k_i \ge k/5$ for $i=1,2,3,4$.
Note that, we require $k_1,k_3$ to be even. We abbreviate $B(x,A\sqrt{k})$ by $B$.
By Cauchy-Schwarz inequality and Lemma \ref{l-dircomp-ball} and Lemma \ref{l-comp-dir}(b) there exists $C_4>0$ such that
\begin{align}
\label{e-th4}{\abs{ \partial_k p^B(y,z) } }
\nonumber &= \abs{ \langle (I-P_B) p^B_{k_1+ k_2}(y,\cdot),p^B_{k_3+ k_4}(z,\cdot)  \rangle_B} \\
\nonumber &= \abs{ \langle (I-P_B)^{1/2} p^B_{k_1+ k_2}(y,\cdot), (I-P_B)^{1/2} p^B_{k_3+ k_4}(z,\cdot)  \rangle_B} \\
\nonumber &\le \left[ \mathcal{E}^B (p_{k_1+k_2}^B(y,\cdot), p_{k_1+k_2}^B(y,\cdot)),\mathcal{E}^B (p_{k_3+k_4}^B(z,\cdot),p_{k_3+k_4}^B(z,\cdot))\right]^{1/2} \\
 &\le C_4 \left[ \mathcal{E}^B_* (p_{k_1+k_2}^B(y,\cdot), p_{k_1+k_2}^B(y,\cdot)),\mathcal{E}^B_* (p_{k_3+k_4}^B(z,\cdot),p_{k_3+k_4}^B(z,\cdot))\right]^{1/2}.
\end{align}
Since $k_1$ is even and $k_1 \ge k/5$, by spectral decomposition and Proposition \ref{p-gue} there exists $C_5,C_6,\delta>0$ such that
\begin{align}\label{e-th5}
\nonumber  \mathcal{E}^B_* (p_{k_1+k_2}^B(y,\cdot), p_{k_1+k_2}^B(y,\cdot)) &=\norm{(I- P_B^2)^{1/2} P_B^{k_1} p_{k_2}^B(y,.) }_2^2 \\
\nonumber &\le \norm{(I- P_B^2)^{1/2} P_B^{k_1}}_{2 \to 2}^2 \norm{p_{k_2}(y,.) }_2^2  \\
\nonumber & \le \left( \sup_{ \lambda \in [0,1] } (1 - \lambda)^{1/2} \lambda^{k/10} \right)^2  \sup_{y \in B}p_{2k_2}(y,y) \\
\nonumber & \le  C_5 k^{-1} \sup_{y \in B(x,A\sqrt{k})} \frac{1}{V(y,\sqrt{k})}  \\
\nonumber &\le \frac{C_5}{k V(x,\sqrt{k})} \sup_{y \in B(x,A\sqrt{k})} \frac{V(y,(A+1)\sqrt{k})}{V(y,\sqrt{k})}\\
& \le \frac{C_6}{k V(x,\sqrt{k})}
\end{align}
for all $k \ge 20$, for all $x \in M$, for all $A >1$ and for all $y \in B=B(x,A\sqrt{k})$.
In the last line above we used \eqref{e-vd1}. By \eqref{e-th5} and \eqref{e-th4}, we obtain
the  desired bound \eqref{e-th2} for $k \ge 20$.

If $2 \le k \le 20$  , we  use \eqref{e-th1} and triangle inequality $\abs{\partial_k p^B} \le p^B_{k+1}+ p^B_k$
 to obtain \eqref{e-th2}.

For the proof of \eqref{e-th3}, we use Proposition \ref{p-spec}.
As before we denote $B(x, A\sqrt{k})$ by $B$.

We first consider the case where $j \in \N^*$ is even.
By Proposition \ref{p-spec}, for each $A \ge (1 \vee h')$, there exists $a>0,\epsilon>0$ such that
\begin{align}
\nonumber \sup_{y,z \in B}p^B_j(y,z) &= \sup_{x \in B} \norm{p^B_{j/2}(y,\cdot)}_2^2=\norm{P_B^{j/2}}_{2 \to \infty} ^2 \\
\nonumber &\le \norm{P_B^{(j/2) - j_1}}_{2 \to 2} ^2 \norm{P_B^{j_1}}_{2 \to \infty} ^2 \\
\label{e-th6} & \le  \left( 1 - \frac{a}{A^2 k} \right)^{j - 2j_1} \sup_{y \in B} p_{2j_1} (y,y)
\end{align}
for all $x \in M$, for all $1 \le j_1 \le (j/2)$, for all $k \in \N^*$.
We choose $j_1 := \lceil \theta k/4 \rceil$  in \eqref{e-th6} and use \eqref{e-vd1} to obtain positive reals $\delta >0$ and $C_7= C_7(\theta)$
\begin{equation}\label{e-th7}
  \sup_{y \in B(x,A\sqrt{k})} p_{2j_1} (y,y) \le \frac{C_{7} A^\delta}{ V(x,\sqrt{k})}
\end{equation}
for all $x \in M$, for all $\theta \in (0,1)$, for all $A \ge 1$ and for all $k \in \N^*$ where $j_1 = \lceil (\theta k/4) \rceil$.
For all $\theta \in (0,1)$, there exists $C_8 = C_8(\theta)>0$ such that
\begin{equation} \label{e-th8}
 \left( 1 - \frac{a}{A^2 k} \right)^{ - 2\lceil (\theta k/4)\rceil}  \le (1-a)^{-2} \left( 1 - \frac{a}{ k} \right)^{ -  (\theta k/2)} \le C_8
\end{equation}
for all $k \in \N^*$, for all $A \ge 1 \vee h'$.
Combining \eqref{e-th6}, \eqref{e-th7} and \eqref{e-th8}, we obtain the bound \eqref{e-th3} for all even $j \ge 2 N^*$.

For all odd $j \in N^*$ satisfying $j \ge 3$, we use the even case and the bound $\sup_{y, z \in B} p_j^B(y,z) \le \sup_{y,z \in B} p_{j-1}^B(y,z)$ to obtain \eqref{e-th3}.
\end{proof}
\begin{remark}
The constants $C_1,C_1,C_2,C_\theta$ and $a_1$ in Lemma \ref{l-tech} do not depend on $A$, $x$ and $k$.
\end{remark}

\begin{proof}[Proof of Lemma \ref{l-crl}]
One of the consequences of Proposition \ref{p-spec} as noted in Remark \ref{r-fin}(b) is that $\norm{P_{B(x,r)}}_{2 \to 2} <1$ for all $x \in M$ and for all $r \in (0,\infty)$.
Therefore $\Delta_{P_B} : L^2(B) \to L^2(B)$ is invertible with inverse
\begin{equation} \label{e-cl1}
 \Delta_{P_B}^{-1} = (I - P_B)^{-1} = \sum_{j=0}^\infty P_B^j.
\end{equation}
Further the inverse $\Delta_{P_B}^{-1}$ is bounded with $\norm{\Delta_{P_B}}_{2 \to 2} \le \left(1 - \norm{P_B}_{2 \to 2} \right)^{-1}$. Motivated by this remark, we define
`Green's function on a ball'
\begin{equation}\label{e-cl2}
 G^B(y,\cdot):= \sum_{i=1}^\infty p^B_i(y,\cdot)
\end{equation}
for all balls $B$ with $\norm{P_B}_{2 \to 2} <1$ and for all $y \in B$.
By \eqref{e-cl1} and \eqref{e-cl2}
\begin{align}\label{e-cl3}
 p_k^B(y,z) &= \left[ \Delta_{P_B}^{-1} \Delta_{P_B} p_k^B(y,\cdot)\right] (z) = \left[ \Delta_{P_B}^{-1} \partial_k p^B(y,\cdot)\right] (z) \nonumber \\
 &= \partial_k p^B(y,z) + \int_B G^B(z,w)\partial_k p^B(y,w) \, dw
\end{align}
for all $x \in M$, for all $A \ge 1 \vee h'$,
for all $y,z \in B= B(x,A\sqrt{k})$ and for all $k \ge 2$.
By \eqref{e-cl3} and triangle inequality, we obtain
\begin{align} \label{e-cl4}
\nonumber \abs{p_k^B(x,y) - p_k^B(x,x)} &\le \abs{ \partial_k p^B(x,y)} + \abs{\partial_k p^B(x,x)} \\
 &\hspace{4mm}+ \int_B \abs{G^B(x,w) - G^B(y,w)} \abs{\partial_k p^B(x,w)} \,dw
\end{align}
for all $x \in M$, for all $A \ge 1$, and for all $y \in B=B(x,A\sqrt{k})$.
We write the right side in \eqref{e-cl4} by splitting it into four parts as
\begin{align*}
 K&= \abs{ \partial_k p^B(x,y)} + \abs{\partial_k p^B(x,x)} \\ I_1 + I_2 + J &= \int_B \abs{G^B(x,w) - G^B(y,w)} \abs{\partial_k p^B(x,w)} \,dw.
\end{align*}
where $I_1,I_2,J$ are terms corresponding to the integration over the sets
\[
 W_1 = \Sett{w \in B}{d(x,w) \le \eta \sqrt{k}}, \hspace{2mm} W_2 = \Sett{w \in B}{d(y,w) \le \eta \sqrt{k}}
\]
for $I_1,I_2$ and
\[
 W=  \Sett{w \in B}{d(x,w) > \eta \sqrt{k} \mbox{ and }d(y,w) > \eta \sqrt{k} }
\]
for $J$, where $\eta >0$ will be chosen later.

As before, we will abbreviate $B(x,A \sqrt{k})$ by $B$.
By Lemma \ref{l-tech}(b), there exists $C_2 >0$, $\delta> 0$ such that
\begin{equation}\label{e-cl5}
 K \le 2 \sup_{y,z \in B} \abs{\partial_k p^B(y,z)} \le \frac{2 C_2 A^\delta}{k V(x,\sqrt{k})} \le \frac{\tau}{V(x,\sqrt{k})}
\end{equation}
for all $\tau >0$, for all $ A \ge 1$, for all $x \in M$ and for all $k \in \N^*$ satisfying $k \ge 2$ and $k \ge (2 C_2 A^\delta) / \tau$.

Next, we bound $I_1$ and $I_2$. We treat $I_2$ in detail but the same estimate applies to $I_1$.
By  Lemma \ref{l-tech}(b), we have
\begin{align}\label{e-cl6}
\nonumber I_2  &\le  \left( \sup_{z_1,z_2 \in B} \partial_k p^B(z_1,z_2) \right) \int_{W_2} (G^B(x,w) + G^B(y,w)) \, dw \\
& \le  \frac{ C_2 A^\delta}{k V(x,\sqrt{k})}\int_{W_2} (G^B(x,w) + G^B(y,w)) \, dw
\end{align}
for all $k \ge 2$.
By Lemma \ref{l-tech}(c), there exists $C_\theta>0$, $a_1>0$ such that
\begin{align}\label{e-cl7}
\nonumber \int_{W_2} G^B(z,w) \, dw &\le \int_{W_2} \sum_{j=1}^{\lfloor \theta k \rfloor} p_j^B(z,w) \, dw + \int_{W_2} \sum_{j=\lfloor \theta k \rfloor+1}^\infty p_j^B(z,w)\,dw \\
\nonumber & \le  \theta k +  \sum_{j=\lfloor \theta k \rfloor+1}^\infty \int_{W_2} p_j^B(z,w)\,dw \\
& \le \theta k + \frac{C_\theta A^\delta}{V(x,\sqrt{k})} \left(1 - \frac{a_1}{A^2 k}\right)^{\theta k} \frac{A^2 k}{a_1} \mu(W_2) \nonumber \\
& \le  \theta k + \frac{C_\theta  A^{\delta+2} k }{a_1 V(x,\sqrt{k})} V(y,\eta\sqrt{k})
\end{align}
for all $x \in M$, for all $A \ge 1\vee h'$, for all $\theta \in (0,1)$, for all $k \in \N^*$ with $k \ge 2/\theta$ and for all $y,z \in B=B(x,A\sqrt{k})$.
For all $y \in B(x,A\sqrt{k})$, by Lemma \ref{l-rvd} and \eqref{e-vd1} there exists $C_3>1,\gamma>0$ such that
\begin{equation}\label{e-cl8}
 \frac{ V(y,\eta \sqrt{k})}{V(x,\sqrt{k})} \le \frac{V(y,\eta \sqrt{k})}{V(y,\sqrt{k})} \frac{V(x,2A\sqrt{k})}{V(x,\sqrt{k})}  \le C_3 A^\delta \eta^\gamma
\end{equation}
for all $x \in M$, for all $A \ge 1$, for all $y \in B(x,A\sqrt{k})$, for all $\eta \in (0,1)$ and for all $k \in \N^*$ with $k \ge (b/\eta)^2$.

For all $\tau > 0$, we choose $\theta \in (0,1)$ and $\eta \in (0,1)$ such that
\begin{equation}\label{e-cl8a}
 \theta \le \frac{\tau}{4 C_2 A^\delta}, \hspace{3mm} \frac{2 C_\theta C_3 A^{2\delta +2}  \eta^\gamma}{a_1} \le \frac{\tau}{4 C_2 A^\delta}.
\end{equation}
Given the above choice of $\theta, \eta$, for all $\tau > 0$, for all $A \ge 1 \vee h'$, by \eqref{e-cl6}, \eqref{e-cl7}, \eqref{e-cl8} there exists $N_1 \ge 2$ such that
\begin{equation} \label{e-cl9}
 \max(I_1,I_2) \le \frac{\tau}{V(x,\sqrt{k})}
\end{equation}
for all $x \in M$ and for all $k \in \N^*$ with $k \ge N_1$.
By \eqref{e-cl5} and \eqref{e-cl9}, for all $\sigma >0,A \ge 1 \vee h'$ there exists $N_2 \ge 2$ and $\eta \in (0,1)$ such that
\begin{equation}\label{e-cl10}
 K+I_1+I_2 \le \frac{\sigma}{V(x,\sqrt{k})}
\end{equation}
for all $x \in M$ and for all $k \in \N^*$ with $k \ge N_2$.

It remains to handle $J$. For the rest of the proof, we fix the choice of $\eta \in (0,1)$ from \eqref{e-cl8a}. Since $p_1^B(x,\cdot)$ is only defined up to $\mu$-almost everywhere, so is $G^B(x,\cdot)$.
However since $p_j^B(x,\cdot)$ is a genuine function for all $j \ge 2$, by \eqref{e-compat} we can redefine $G^B$ in \eqref{e-cl2} as
\begin{equation}\label{e-cl11}
 G^B(y,z) = \sum_{j =2}^\infty p^B_j(y,z)
\end{equation}
for all $y,z \in B$ with $d(y,z) > h'$. In other words $G^B(y,\cdot)$ can be defined as a genuine function in $B \setminus B(y,h')$ with $G^B(y,z)= G^B(z,y)$ for all $y,z \in B$ with
$d(y,z) > h'$. Further the function
\[z \mapsto G^B(w,z)=G^B(z,w)\]
is $P$-harmonic in $B(y,d(y,z)-3h')$, whenever $y \in B$, $B(y,d(x,w)- 2h') \subseteq B$ and $d(y,w) > 3h'$.
Therefore for all $x \in M$, for all $A \ge 1$ and for all $k \in \N^*$ with $k \ge (6 h'/\eta)^2$, for all $w \in B(x, A\sqrt{k}) \setminus B(x,\eta \sqrt{k})$,
the function $z \mapsto G^{B(x,A\sqrt{k})}(z,w)$ is $P$-harmonic in $B(x,\eta\sqrt{k}/2)$. By the H\"{o}lder-type regularity estimate for harmonic functions (Proposition \ref{p-holder}), there exists
$C_4 >0,N_3\ge 2 \vee (6h'/\eta)^2, \alpha>0,\epsilon_0 \in (0,\eta/2)$ such that
\begin{equation}\label{e-cl12}
 \abs{G^B(x,w)-G^B(y,w)} \le C_4 \left( \frac{d(x,y) \vee 1}{\eta \sqrt{k}} \right)^\alpha \sup_{z \in B(x,\eta \sqrt{k}/2)} G^B(z,w)
\end{equation}
for all $x \in M$, for all $y \in B(x,\epsilon_0 \sqrt{k})$, for all $A \ge 1$, $w \in B(x, A\sqrt{k}) \setminus B(x,\eta \sqrt{k})$  and for all $k \in \N^*$ with $k \ge  N_3$.

Following \eqref{e-cl12}, we need to estimate $\sup_{w \in B \setminus B(x,\eta \sqrt{k}), z \in B(x,\eta \sqrt{k}/2)} G^B(z,w)$. For all $z,w \in B$ such that $d(z,w)>h'$, we have
\[ G^B(z,w)  =  \sum_{j=2}^k p^B_j(z, w) + \sum_{j=k+1}^\infty p^B_m(z, w).\]
For the first term, by Lemma \ref{l-tech}(a) and \eqref{e-vd1} there exists $C_1,D_1, C_5,C_6 >1$ and $\delta >0$ such that
\begin{align}\label{e-cl13}
 \nonumber \sum_{j=2}^k p^B_j(z, w) &\le \sum_{j=2}^k \frac{C_1}{V(z,\sqrt{j})} \exp\left( - \frac{d(y,z)^2}{D_1 j} \right) \frac{V(z,2\sqrt{k})}{V(x,\sqrt{k})} \\
 &\le \frac{C_5}{ V(x,\sqrt{k})} \sum_{j=2}^k \left( \frac{k}{j}\right)^{\delta/2} \exp\left( - \frac{\eta^2 k }{4 D_1 j} \right) \nonumber \\
 &\le \frac{C_6 k}{V(x,\sqrt{k})}
\end{align}
for all $z \in B(x,\sqrt{k})$, for all $w \in B$ such that $d(z,w)> \eta \sqrt{k}/2 \ge h'$. To obtain \eqref{e-cl13} above, we used that the function $t \mapsto t^{\delta/2}\exp(- \eta^2 t/(4D_1))$ is bounded in $(0,\infty)$.

Next, we bound $p_j$ for large values of $j$. By Lemma \ref{l-tech}(c) there exists $C_7>0$ such that
\begin{equation}\label{e-cl14}
  \sum_{j=k+1}^\infty p^B(z,w) \le C_7 \frac{A^\delta}{V(x,\sqrt{k})} \sum_{j=k+1}^\infty \left( 1- \frac{a_1}{A^2 k}\right)^j \le \frac{C_7 A^{\delta+2}k}{a_1 V(x,\sqrt{k})}
\end{equation}
for all $k \in \N^*$,  $A \ge 1$, for all $x \in M$   and for all $z,w \in B=B(x,A\sqrt{k})$.

Combining \eqref{e-cl12}, \eqref{e-cl13}, \eqref{e-cl14} along with Lemma \ref{l-tech}(b) and \eqref{e-vd1}, for each $A \ge 1$ and any choice of $\eta \in (0,1)$,
there exists $C_{8} \ge 1$, $N_4 \ge 2$, $\epsilon_0 \in (0,1)$ (depending on $A,\eta$) and $\alpha>0$ such that
\begin{equation}\label{e-cl15}
 \abs{G^B(x,w)-G^B(y,w)} \le C_8  \left( \frac{d(x,y) \vee 1}{ \sqrt{k}} \right)^\alpha  \frac{1}{V(x,\sqrt{k})}
\end{equation}
for all $x \in M$, for all $y \in B(x,\epsilon_0 \sqrt{k})$ and for all $k \in \N^*$ satisfying $k \ge N_4$, where $B=B(x,A\sqrt{k})$ and $\alpha$ is as in \eqref{e-holder}.
Combining \eqref{e-cl10} and \eqref{e-cl15}, we obtain the desired estimate  \eqref{e-crl}.
\end{proof}
Now, we are ready to prove the near diagonal lower bound using Lemma \ref{l-crl} and Lemma \ref{l-balldlb}.
\begin{proof}[Proof of Proposition \ref{p-ndlb}]
 By Lemma \ref{l-balldlb}, there exists $A \ge 1 \vee h'$ and $c >0$ such that
 \begin{equation}\label{e-nb1}
  p_k^{B(x,A\sqrt{k})} (x,x) \ge \frac{c}{ V(x,\sqrt{k})}
 \end{equation}
for all $x \in M$ and for all $k \in \N^*$ with $k \ge2$.
By Lemma \ref{l-crl}, there exists $C_1>1, N_1 \ge 2$, $\epsilon \in (0,1) ,\alpha>0$ such that
\begin{equation}\label{e-nb2}
 \abs{p^B_k(x,y)- p^B_k(x,x)} \le \left[ \frac{c}{3} + C_1 \left( \frac{d(x,y) \vee 1}{ \sqrt{k}} \right)^\alpha \right] \frac{1}{ V(x,\sqrt{k})}
\end{equation}
for all $x \in M$, for all $k \in \N^*$ with $k \ge N_0$, for all $y \in B(x,\epsilon \sqrt{k})$ where $B=B(x,A\sqrt{k})$.
Next, we choose $\epsilon_1 \in (0,\epsilon)$ and $N_1 \ge N_0$ such that for all $k \ge N_1$, we have
\[
C_1 \left( \frac{ \epsilon_1 \sqrt{k} \vee 1}{ \sqrt{k}} \right)^\alpha \le C_1 \max(\epsilon^\alpha , N_0^{-\alpha/2}) \le \frac{c}{3}.
\]
By the above choice of $\epsilon_1,N_1$ along with \eqref{e-nb1},\eqref{e-nb2} and the triangle inequality, we have
\begin{equation*}
 \inf_{y \in B(x,\epsilon_1 \sqrt{k})} p_k^{B(x,A\sqrt{k})} (x,y) \ge \frac{c}{3 V(x,\sqrt{k})}
\end{equation*}
for all $x \in M$ and for all $k \in \N^*$ with $k \ge N_1$.
Since $p_k^B \le p_k$, the above equation yields the desired near diagonal lower bound \eqref{e-ndlb} for all $k\ge N_1$.

If $k \in \nint{2}{N_1}$, then we reduce $\epsilon$ if necessary so that $\epsilon \le h/\sqrt{N_1}$. Hence $d(x,y) \le \epsilon \sqrt{k}$ and $k \le N_1$ implies $d(x,y) \le h$.
Therefore by \eqref{e-pcomp} of Lemma \ref{l-con-d} and \eqref{e-compat}, we obtain \eqref{e-ndlb} for all $k \in \nint{2}{N_1}$.
\end{proof}
\section{Off-diagonal lower bounds}
The near diagonal lower bound of Proposition \ref{p-ndlb} can be easily upgraded to full Gaussian lower bounds \ref{gle} by a well-known chaining argument
(See \cite[Theorem 5.1]{HS93}, \cite[Theorem 3.8]{Del99}).
For general quasi-geodesic spaces, we rely on the chain lemma (Lemma \ref{l-chain}). We now prove the main result of this chapter, i.e. Gaussian lower bound.
\begin{proof}[Proof of Proposition \ref{p-gle}]
 By Lemma \ref{l-chain} there exists $C_1>1$ such that for all $b_1 \ge b$ and for all $x,y \in M$, there exists a $b_1$-chain $x=x_0,x_1,\ldots,x_m=y$  with
 \begin{equation}\label{e-gl1}
  m \le \left\lceil\frac{C_1 d(x,y)}{b_1} \right\rceil.
 \end{equation}
By Proposition \ref{p-ndlb}, there exists $\epsilon>0,c_1>0$ such that
\begin{equation}\label{e-gl2}
 \inf_{y \in B(x,\epsilon \sqrt{k})} p_k(x,y) \ge \frac{c_1}{V(x,\sqrt{k})}
\end{equation}
for all $x \in M$ and for all $k \ge 2$.
If
\begin{equation}\label{e-gl3}
 s:= \frac{C_1 \epsilon^2 k}{C_2 d(x,y)} \ge b,
\end{equation}
then there exists a $s$-chain $x=x_0,x_1,\ldots,x_m=y$ between $x$ and $y$ with
\begin{equation}\label{e-gl4}
 m:=\left\lceil\frac{C_2 d(x,y)^2 }{\epsilon^2 k} \right\rceil.
\end{equation}
However \eqref{e-gl3} holds whenever $d(x,y) \le c_3k$ and $c_3 \le C_1 \epsilon^2/C_2 b$.
If $C_2 \ge 1$ and $d(x,y) \ge \epsilon \sqrt{k}$, we have
\begin{equation}\label{e-gl5}
 m:=\left\lceil\frac{C_2 d(x,y)^2 }{\epsilon^2 k} \right\rceil \le \frac{2 C_2 d(x,y)^2 }{\epsilon^2 k} .
\end{equation}
If $d(x,y) \le c_3 k$ and $c_3 \le \epsilon/\sqrt{(2C_2)}$, we have
\begin{equation}\label{e-gl6}
 \frac{k}{m} \ge \frac{\epsilon^2 k^2}{C_2 d(x,y)^2} \ge \frac{\epsilon^2}{C_2 c_3^2} \ge 2.
\end{equation}
We fix $c_3= \min \left( \epsilon/\sqrt{(2C_2)}, C_1 \epsilon^2/C_2 b \right)$, so that
\eqref{e-gl3},\eqref{e-gl4} and \eqref{e-gl6} are satisfied. We will fix $C_2 \ge 1$ later.

We will require
\begin{equation}\label{e-gl7}
 d(x_i,x_{i+1}) \le s= \frac{C_1 \epsilon^2 k}{C_2 d(x,y)}  \le\frac{\epsilon}{3} \sqrt{\frac{k}{2m}}  \le \frac{\epsilon}{3} \sqrt{\left\lfloor \frac{k}{m} \right\rfloor}
\end{equation}
for all $i=0,1,\ldots,m-1$ and for all $k \ge m$.
We fix $C_2:=36 C_1^2 \ge 1$, so that by \eqref{e-gl5} we deduce
\begin{equation} \label{e-gl8}
 s=\frac{C_1 \epsilon^2 k}{C_2 d(x,y)}  \le \frac{\epsilon}{3} \left( \frac{\epsilon^2 k^2}{4 C_2 d(x,y)^2} \right)^{1/2} \le \frac{\epsilon}{3} \sqrt{\frac{k}{2m}} \le \frac{\epsilon}{3} \sqrt{\left\lfloor \frac{k}{m} \right\rfloor}
\end{equation}
for all $x,y \in M$ and $k \in \N^*$ such that $d(x,y) \ge \epsilon \sqrt{k}$ and $k/m \ge 2$, where $s,m$ is as defined in \eqref{e-gl3} and \eqref{e-gl4}.
Define $k_0,\ldots,k_{m-1}$ such that
\[
 k_i := \left\lfloor \frac{k}{m} \right\rfloor \mbox{ or } \left\lfloor \frac{k}{m} \right\rfloor+1
\]
satisfying $\sum_{i=0}^{m-1} k_i=k$.
Consider the $s$-chain $x=x_0,\ldots,x_m=y$ between $x$ and $y$ where $s,m$ are given by \eqref{e-gl3},\eqref{e-gl4}.
By \eqref{e-gl8} and definition of $k_i$, for all $w_i \in B(x_i, (\epsilon/3)\sqrt{\lfloor k/m\rfloor})$, for $i=0,1,\ldots,m-1$ we have
\[
 d(w_i,w_{i+1}) \le \epsilon\sqrt{\lfloor k/m\rfloor}  \le \epsilon\sqrt{k_i}.
\]
Therefore by \eqref{e-gl2}, \eqref{e-gl6} and \eqref{e-vd2}, there exists $c_4,c_5 \in (0,1)$ such that for all for $i=0,1,\ldots,m-1$, $w_i \in B(x_i, (\epsilon/3)\sqrt{\lfloor k/m\rfloor})$, we have
\begin{equation}\label{e-gl9}
 p_{k_i}(w_i,w_{i+1}) \ge \frac{c_1}{V(w_i,\sqrt{k_i})} \ge  \frac{c_4}{V(w_i,\sqrt{\lfloor k/m\rfloor})} \ge \frac{c_5}{V(x_i,\sqrt{\lfloor k/m\rfloor})}
\end{equation}
for all $x,y \in M$, $k \ge 2$ satisfying $d(x,y) \ge \epsilon \sqrt{k}$ and $d(x,y) \le c_3k$.

Define $B_i=B(x_i,(\epsilon/3) \sqrt{\lfloor k/m \rfloor})$.
By Chapman-Kolmogorov equation and \eqref{e-gl9}, for all $x,y \in M$, $k \ge 2$ satisfying $d(x,y) \ge \epsilon \sqrt{k}$ and $d(x,y) \le c_3k$, we obtain
\begin{align}\label{e-gl10}
\lefteqn{p_k(x,y)}\nonumber \\
\nonumber &= \int_{M} \ldots \int_{M} p(x_0,w_1)p(w_1,w_2)\ldots p(w_{m-2},w_{m-1}) p(w_{m-1},y) \, dw_1 \ldots \,dw_{m-1}\\
& \ge \int_{B_{m-1}} \ldots \int_{B_1} p(x_0,w_1)p(w_1,w_2)\ldots p(w_{m-2},w_{m-1}) p(w_{m-1},y) \, dw_1 \ldots \,dw_{m-1} \nonumber\\
& \ge \frac{c_5^{m-1}}{V(x,\sqrt{k})} \prod_{i=1}^{m-1} \frac{V(x_i,(\epsilon/3) \sqrt{\lfloor k/m \rfloor})}{V(x_i,\sqrt{\lfloor k/m \rfloor})}
\end{align}
By \eqref{e-vd1}, \eqref{e-gl5}, \eqref{e-gl6} and \eqref{e-gl10}, there exists $c_6,c_7\in(0,1)$ such that
\begin{align} \label{e-gl11}
 p_k(x,y) &\ge \frac{c_6^m}{V(x,\sqrt{k})} \ge \exp \left( \frac{2 C_2 d(x,y)^2 \log c_6}{\epsilon^2 k} \right) \frac{1}{V(x,\sqrt{k})} \nonumber \\
 &\ge  \frac{1}{V(x,\sqrt{k})} \exp\left( -\frac{d(x,y)^2}{c_7 k}\right)
\end{align}
for all $x,y \in M$, $k \ge 2$ satisfying $d(x,y) \ge \epsilon \sqrt{k}$ and $d(x,y) \le c_3k$.
This yields \ref{gle} for the case $d(x,y) \ge \epsilon \sqrt{k}$.

The case $d(x,y) \le \epsilon \sqrt{k}$ follows from \eqref{e-gl2}. This completes the proof of \ref{gle}.
\end{proof}

\chapter{Parabolic Harnack inequality} \label{ch-phi}
In this chapter, we use the two sided Gaussian estimates on the heat kernel to prove parabolic Harnack inequality.
Moreover, we show the necessity of Poincar\'{e} inequality and large scale volume doubling using parabolic Harnack inequality.

Based on ideas of Nash \cite{Nas58}, Fabes and Stroock \cite{FS86} gave a proof of parabolic Harnack inequality using Gaussian bounds on the heat kernel for uniformly
elliptic operators on $\R^n$. This idea of using Gaussian estimates on the heat kernel to prove parabolic Harnack inequality was extended in various settings \cite{SS91,Sal90,Del99,BGK12}.
Delmotte \cite{Del99} introduced a discrete version  of balayage formula to prove parabolic Harnack inequality on graphs.
We use a direct adaptation of Delmotte's method to prove parabolic Harnack inequality.

Recall that we defined caloric function as solutions to the discrete time heat equation $\partial_k u +\Delta u_k = 0$ in Definition \ref{d-caloric}.
We introduce the parabolic Harnack inequality for non-negative caloric functions.
\begin{definition}\label{d-phi}
 Let $(M,d,\mu)$ be a metric measure space and let $P$ be a Markov operator on  $(M,d,\mu)$.
  Let $0<\zeta<1$ and $0 < \theta_1 < \theta_2 < \theta_3 < \theta_4$.
 We that a $\mu$-symmetric Markov operator $P$ (or equivalently its heat kernel $p_k$) on  $(M,d,\mu)$ satisfies the discrete-time parabolic Harnack inequality
  \[
   \hypertarget{phi}{H(\zeta,\theta_1,\theta_2,\theta_3,\theta_4)}
  \]
 if there exists positive reals $C,R$ such that for all $x \in M,r \in \R,a \in \N$ with $r>R$ and every non-negative $P$-caloric function $u: \N \times M \to \R_{\ge 0}$ on
 \[
  Q= \nint{a}{a+ \lfloor \theta_4r^2 \rfloor } \times B(x,r) ,
 \]
we have
\[
 \sup_{Q_\ominus} u \le C \inf_{Q_\oplus} u,
\]
where
\begin{align*}
 Q_\ominus &:= \nint{a+\lceil \theta_1 r^2 \rceil}{a+\lfloor \theta_2 r^2 \rfloor} \times B(x,\zeta r), \\
 Q_\oplus &:= \nint{a+ \lceil \theta_3 r^2 \rceil}{a+\lfloor \theta_4 r^2 \rfloor} \times B(x,\zeta r).
 \end{align*}
 \end{definition}
\begin{remark}\label{r-phi}\leavevmode
\begin{enumerate}[(i)]
 \item The exact values of the constants $\zeta \in (0,1)$ and $\theta_1,\theta_2,\theta_2,\theta_4$ are unimportant.
 For example, for graphs and length spaces if the parabolic Harnack inequality is satisfied for one set
 of constants, then it is satisfied for every other set of constants. The argument in  \cite[Proposition 5.2(iv)]{BGK12} can be adapted for graphs and length spaces in the above discrete-time setting.
 \item It suffices to consider the case $a=0$ in the definition above by simply by shifting the function in the time component.
 \item Analogous to Remark \ref{r-caloric}(b), if $P$ is $(h,h')$-compatible with $(M,d,\mu)$ we may only require the function $u$ to be defined on a smaller domain.
\end{enumerate}
\end{remark}
\section{Gaussian estimates implies parabolic Harnack inequality}
In this section, we prove the following parabolic Harnack inequality using two sided Gaussian bounds.
\begin{prop}\label{p-phi}
 Let $(M,d,\mu)$ be a quasi-$b$-geodesic metric measure space satisfying \ref{doub-loc}.
 Suppose that a Markov operator $P$ has a kernel $p_k$ that is weakly $(h,h')$-compatible
 with respect to $\mu$ for some $h>b$. Moreover, suppose that $p_k$ satisfies two sided Gaussian estimate $(GE)$. Then there exists $\eta \in (0,1)$ such that
 $P$ satisfies the parabolic Harnack inequality \hyperlink{phi}{$H(\eta/2,\eta^2/2,\eta^2,2\eta^2,4\eta^2)$}.
\end{prop}
First we start by verifying that Gaussian lower bound implies large scale volume doubling property.
\begin{lemma}\label{l-gevd}
 Let $(M,d,\mu)$ be a quasi-$b$-geodesic metric measure space satisfying \ref{doub-loc}.
 Suppose that a Markov operator $P$ has a kernel $p_k$ that  satisfies \ref{gle}.
 Then $(M,d,\mu)$ satisfies \ref{doub-inf}.
\end{lemma}
\begin{proof}
By \ref{gle} there exists $c_1,c_2,c_3>0$ such that
\begin{equation*}
 p_n(x,y) \ge \frac{c_1}{V(x,\sqrt{n})} \exp \left( - d(x,y)^2/c_2n \right)
\end{equation*}
for all $x,y \in M$ satisfying $d(x,y)\le c_3 n$ and for all $n \in \mathbb{N}^*$.
Therefore there exists $N_1\ge1$ such that $4 \sqrt{n} \le c_3 n$ for all $n \ge N_1$.
By the Gaussian lower bound above
\begin{equation*}
 1= \int_M p_n(x,y)\,dy \ge \int_{B(x,4 \sqrt{n})} p_n(x,y) \,dy \ge \frac{V(x,4 \sqrt{n}) }{V(x,\sqrt{n})}c_1  \exp( -4/c_2)
\end{equation*}
for all $x \in M$ and for all $n \ge N_1$.
Therefore there exists $R := N_1^2$ such that for all $x \in M$ and for all $r \ge R$, we have
\[
V(x,r) \ge V(x,\lfloor r \rfloor) \ge  c_1 \exp(-4/c_2) V(x,4 \lfloor r \rfloor) \ge   c_1 \exp(-4/c_2) V(x,2r).
\]
\end{proof}
We show the following near diagonal lower bounds as a consequence of two sided Gaussian bound \hyperlink{ge}{$(GE)$}.
\begin{lemma}\label{l-ndlb}
 Under the assumptions of Proposition \ref{p-phi}, there exists $c_1>0$, $\eta \in (0,1)$ and $R_0>0$ such that for all $x \in M$, for all $r \ge R_0$, for all
 $y ,z \in B(x,\eta r)$, for all $k \in \N^*$ satisfying $ (\eta r)^2 \le k \le (2 \eta r)^2$, we have
 \begin{equation}\label{e-nlb}
  p_k^{B(x,r)} (y,z) \ge \frac{c_1}{V(x,\sqrt{k})}.
 \end{equation}
\end{lemma}
\begin{proof}
We abbreviate $B(x,r)$ by $B$. We denote the exit time from ball $B$ by
\[
 \tau:= \min \Sett{n}{X_n \notin B}
\]
where $(X_n)_{n \in \N}$ is the Markov chain on $M$ corresponding to the kernel $p_k$.

By strong Markov property and $\mu$-symmetry, the Dirichlet kernel $p_k^B$ can be expressed in terms of $p_k$ as
\begin{equation} \label{e-nl1}
p^B_k(y,z)= p_k(y,z)  - \EE_y \left[ p_{k-\tau}(z,X_\tau) \one_{\nint{1}{k-1}}(\tau)\right]
\end{equation}
for all $n \ge 2$ and for all $x \in M$,
where $\EE_y$ denotes that the Markov chain starts at $X_0=y$.
We choose $R_0 > (1-\eta)^{-1} h'$, so that  by \eqref{e-compat}
\[
 \EE_y \left[ p_{k-\tau}(z,X_\tau) \one_{\nint{1}{k-1}}(\tau)\right]  =  \EE_y \left[ p_{k-\tau}(z,X_\tau) \one_{\nint{2}{k-2}}(\tau)\right]
\]
for all $y,z \in B(x,\eta r)$, for all $k \ge 2$, for all $x \in M$ and for all $r \in \R$ with $r \ge R_0$.
Combining this with \eqref{e-nl1} and $X_\tau \notin B$, we have
\begin{equation}\label{e-nl2}
 p^B_k(y,z) \ge p_k(y,z)  - \sup_{l \in \nint{2}{k}} \sup_{w \notin B(x,r)} p_{l}(z,w)
\end{equation}
for all $y,z \in B(x,\eta r)$, for all $k \ge 2$, for all $x \in M$ and for all $r \in \R$ with $r \ge (1-\eta)^{-1} h'$.

 Note that by Lemma \ref{l-gevd} we have \ref{doub-inf}.
Therefore by \ref{gle}, \eqref{e-vd1} and $k \ge (\eta r)^2$, there exists $c_2,c_3 >0$ and $R_1>0$ such that
\begin{equation}\label{e-nl3}
 p_k(y,z) \ge \frac{c_2}{V(y,\sqrt{k})} \exp \left(- \frac{ (2 \eta r)^2}{c_2 (\eta r)^2} \right) \ge \frac{c_3}{V(x,\sqrt{k})}
\end{equation}
for all $x \in M$, for all $r \ge R_1$, for all $\eta \in (0,1)$, for all $y,z \in B(x,\eta r)$ and for all $k \in \N^*$ satisfying $(\eta r)^2 \le k$.

For the second term in \eqref{e-nl2} by \ref{gue}, there exists $C_1>0$ such that
\[
 p_l(z,w) \le \frac{C_1}{V(z,\sqrt{l})} \exp \left( -\frac{d(z,w)^2 }{C_1 l} \right) \le \frac{C_1}{V(z,\sqrt{l})} \exp \left( -\frac{ (1- \eta)^2 r^2 }{C_1 l} \right)
\]
for all $l \in \N^*$ with $l \ge 2$, for all $x \in M$, for all $r >0$, for all $\eta \in (0,1)$, for all $z \in B(x,\eta r)$ and for all $w \notin B(x,r)$.
Combined this with \eqref{e-vd1} and $k \le (2 \eta r)^2$, there exists $C_2,C_3,C_4,\delta> 0$ such that for all $\eta \in (0,1/2)$, for all $x \in M$, for all $z \in B(x,\eta r)$, for all
$k \in \N^*$ satisfying $(\eta r)^2 \le k \le (2 \eta r)^2$, for all $l \in \nint{2}{k}$ and for all $w \notin B(x,r)$, we have
\begin{align}\label{e-nl4}
\nonumber p_l(z,w) &\le  \frac{C_2}{V(z,\sqrt{k})} \left( \frac{k }{l} \right)^{\delta} \exp \left( -\frac{ (1- \eta)^2 r^2 }{C_1 l} \right) \\
& \le \frac{C_3 \eta^{2\delta} }{V(x,\sqrt{k})} \left( \frac{ r^2 }{l} \right)^{\delta} \exp \left( -\frac{  r^2 }{4 C_1 l} \right) \nonumber \\
& < \frac{C_4 \eta^{2 \delta}}{V(x,\sqrt{k})}.
\end{align}
The second line above follows from $\eta < 1/2$ and \eqref{e-vd1} and the last line follows from the fact the function $t \mapsto t^\delta \exp(-t/4C_1)$ is bounded in $(0,\infty)$.
Combining \eqref{e-nl2}, \eqref{e-nl3} and \eqref{e-nl4}, there exists $c_1>0$ and $R_0 >0$ such that $p_k^B$ satisfies \eqref{e-nlb}.
\end{proof}
The following lemma provides  a discrete time version of Balayage decomposition for the heat equation.
\begin{lemma}
 \label{l-balay}
 Let $(M,d,\mu)$ be a quasi-$b$-geodesic metric measure space satisfying \ref{doub-loc}.
 Suppose that a Markov operator $P$ has a kernel $p_k$ that is weakly $(h,h')$-compatible
 with respect to $\mu$ for some $h>b$.
  Then for all $x \in M$, for all $r >h'$, for all $r_1$ such that $0<r_1 < r_1 +h' < r$, for all $a,b \in \N$,
  for all non-negative function $u: N \times M \to \R_{\ge 0}$ that is $P$-caloric in $\nint{a}{b} \times B(x,r)$,  there exists a non-negative function $v:N \times M \to \R$ (depending on $u$)
  such that $\supp(v) \subseteq \nint{a+1}{b} \times\left( B(x,r_1+h') \setminus B(x,r_1) \right)$ and for all $y \in B(x,r_1)$ and for all $k \in \nint{a}{b+1}$, we have
 \begin{equation}\label{e-balay}
 u(k,y) = \int_{B(x,r_1+h')} p_{k-a}^B (y,z) u(a,z) \,dz + \sum_{l=a+1}^{k-1} \int_{B(x,r_1+h')} p_{k-l}^B(y,w) v(l,w) \, dw,
 \end{equation}
 where $B=B(x,r)$.
\end{lemma}
\begin{proof}
Denote by $B_1= B(x,r_1+h')$ and $B=B(x,r)$.
Define
\[
v_1(k,y)= u(k,y) - \int_{B(x,r_1 +h')} p_{k-a}^B (y,z) u(a,z) \,dz
\]
for all $(k,y) \in \nint{a+1}{b+1} \times B(x,r_1+h')$. Note that
\[(k,y) \mapsto \int_{B_1 } p_{k-a}^B (y,z) u(a,z) \,dz\] is $P$-caloric in $\nint{a+1}{b} \times B(x,r_1)$.
Since $u \ge 0$, by \eqref{e-compat} we have $v_1(a+1,y) =0$ for all $y \in B(x,r_1)$ and
by maximum principle $v_1 \ge 0$ in $\nint{a+1}{b+1} \times B(x,r_1)$.

Next, we construct $v: \N \times M \to \R$ iteratively. We assume that $\supp(v) \subseteq \nint{a+1}{b} \times\left( B(x,r_1+h') \setminus B(x,r_1) \right)$.
Define $v(a+1,y)=v_1(a+1,y)$ for all $y \in B(x,r_1+h') \setminus B(x,r)$.

Since $v_1$ is a difference of two $P$-caloric functions, we have $v_1$ is $P$-caloric in $ \nint{a+1}{b}\times B(x,r_1)$.
We repeat this construction iteratively by defining
\begin{equation} \label{e-blg1}
 v_{i+1}(k,y)= v_i(k,y) - \int_{B(x,r_1+h')} p_{k-a-i}^B(y,z) v_i(a+i,z) \,dz
\end{equation}
for all $(k,y) \in \nint{a+i+1}{b+1}\times B(x,r_1+h')$
and
\[
 v(a+i+1,w)= v_{i+1}(a+i+1,w)
\]
for all $w \in B(x,r_1+h') \setminus B(x,r_1)$ and $i =0 ,1,\ldots,b-a-1$.
By the same argument as above, $v_i$ is non-negative and caloric in $ \nint{a+i}{b} \times B(x,r_1)$ for all $i=0,1,\ldots,b-a+1$.
Further
\begin{equation}
 \label{e-blg2}u_i(a+i,z)= 0
\end{equation}
for all $z$ in $B(x,r_1)$ and $i=1,2,\ldots,b-a$.
Combining \eqref{e-blg1},\eqref{e-blg2} and gives \eqref{e-balay}.
\end{proof}
We are now ready to prove the parabolic Harnack inequality.
\begin{proof}[Proof of Proposition \ref{p-phi}]
Let $\eta \in (0,1)$ be as given by Lemma \ref{l-ndlb}.
 Note that for all $r > 12 h'/\eta$, we have
$\eta r -h'  > 2 \eta r/3 > \eta /2$. Moreover for all $r > 12 h'/\eta$, for all $y \in B(x,\eta r/2)$ and for all $z \in B(x,\eta r) \setminus B(x,\eta r -h')$ we have $d(y,z) > 2 h'$.
Let  $R_1 := 1+ \max(R_0, 12 h'/\eta, 10/\eta)$ where $R_0$ is the constant from Lemma \ref{l-ndlb}. By the above remarks,
\eqref{e-compat} and Lemma \ref{l-balay},  for all $x \in M$, for all $r \ge R_1$,
for all non-negative function $u$ that is $P$-caloric in $\nint{0}{\lfloor 4 \eta^2 r^2 \rfloor }\times B$ where $B=B(x,r)$, there exists a non-negative function $v$
supported in $B(x,\eta r) \setminus B(x,\eta r - h')$ such that
 \begin{equation}\label{e-ph1}
 u(k,y) = \int_{B(x,\eta r)} p_{k}^B (y,z) u(a,z) \,dz + \sum_{l=1}^{k-2} \int_{B(x,\eta r)} p_{k-l}^B(y,w) v(l,w) \, dw
 \end{equation}
for all $(k,y) \in \nint{1}{\lfloor 4 \eta^2 r^2 \rfloor +1} \times B(x,\eta r/2)$.

For some fixed  $x \in M$ and $r > R_1$, we define
\begin{equation}\label{e-ph1a}
 Q_\ominus := \nint{\lceil \eta^2 r^2/2 \rceil}{\lfloor \eta^2 r^2 \rfloor } \times B(x,\eta r/2), Q_\oplus := \nint{\lceil 2 \eta^2 r^2 \rceil}{\lfloor 4 \eta^2 r^2\rfloor} \times B(x,\eta r/2)
\end{equation}
and $Q:= \nint{0}{\lfloor 4 \eta^2 r^2\rfloor} \times B(x,\eta r)$.

By Lemma \ref{l-gevd} we have \ref{doub-inf}.
Therefore by Lemma \ref{l-ndlb} and \eqref{e-vd1} there exists $c_1,c_2>0$ such that for all $x \in M$, for all $r \ge R_1$, for all $y \in B(x,\eta r/2)$, for all $z \in B(x,\eta r)$, we have
\begin{equation}
 \label{e-ph2} \inf_{(k,y) \in Q_\oplus} p_k^B(y,z) \ge \inf_{ k \in \nint{\lceil 2 \eta^2 r^2 \rceil}{\lfloor 4 \eta^2 r^2\rfloor}}  \frac{c_1}{V(x,\sqrt{k})} \ge \frac{c_1}{V(x,2 \eta r)} \ge \frac{c_2}{V(x,\eta r)} .
\end{equation}
Similarly by Lemma \ref{l-ndlb} for all $x \in M$, for all $r \ge R_1$, for all $y \in B(x,\eta r/2)$,
for all $z \in B(x,\eta r) \setminus B(x,\eta r- h')$, for all $l \in \nint{1}{\lfloor\eta^2 r^2\rfloor-2}$ we have
\begin{equation}
 \label{e-ph3} \inf_{(k,y) \in Q_\oplus} p_{k-l}^B(y,z) \ge \inf_{ k \in \nint{\lceil 2 \eta^2 r^2 \rceil}{\lfloor 4 \eta^2 r^2\rfloor}}  \frac{c_1}{V(x,\sqrt{(k-l)})} \ge \frac{c_2}{V(x,\eta r)} .
\end{equation}

For upper bounds in $Q_\oplus$ we simply use \ref{gue} as follows. By \ref{gue} and \eqref{e-vd2},
there exists $C_1,C_2>0$ such that for all $x \in M$, for all $r \ge R_1$, for all $y \in B(x,\eta r/2)$, for all $z \in B(x,\eta r)$ we have
\begin{equation}
 \label{e-ph4} \sup_{(k,y) \in Q_\ominus} p_k^B(y,z) \le  \sup_{(k,y) \in Q_\ominus} p_k (y,z) \le\sup_{(k,y) \in Q_\ominus}  \frac{C_1}{V(y,\sqrt{k})} \le \frac{C_2}{V(x,\eta r)}.
\end{equation}
Similarly by  \ref{gue} and \eqref{e-vd1},
there exists $C_3,C_4,C_5,\delta>0$ such that for all $x \in M$, for all $r \ge R_1$, for all $y \in B(x,\eta r/2)$,
for all $z \in B(x,\eta r) \setminus B(x,\eta r- h')$, for all $(k,y) \in Q_\ominus$ and for all $ l \in \nint{1}{k-2}$ we have
\begin{align}
 \label{e-ph5}  p_{k-l}^B(y,z) & \le   p_{k-l} (y,z)\nonumber \le \frac{C_3}{V(y,\sqrt{(k-l)})} \exp\left(- \frac{d(y,z)^2}{C_3 (k-l)} \right) \\
 & \le \frac{C_4}{V(y,\eta r)} \left( \frac{\eta^2 r^2}{(k-l)}\right)^{\delta/2}\exp\left(- \frac{\eta^2 r^2 }{36 C_3 (k-l)} \right) \nonumber \\
 & \le \frac{C_5}{V(x,\eta r)}.
\end{align}
The last line follows from the fact that the function $t \mapsto t^{\delta/2} \exp(- t/(36 C_3))$ is bounded in $(0,\infty)$ along with \eqref{e-vd2}.

Combining the inequalities \eqref{e-ph2},\eqref{e-ph3},\eqref{e-ph4} and \eqref{e-ph5} along with the balayage formula \eqref{e-ph1}
for all $x \in M$, for all $r \ge R_1$,
for all non-negative function $u$ that is $P$-caloric in $\nint{0}{\lfloor 4 \eta^2 r^2 \rfloor }\times B(x,r)$, we have
\begin{equation*}
 \sup_{(k,y) \in Q_\ominus} u(k,y) \le c_2^{-1} \max(C_2,C_5) \inf_{(k,y) \in Q_\oplus} u(k,y)
\end{equation*}
where $Q_\ominus, Q_\oplus$ are as defined in \eqref{e-ph1a}. Note that by Remark \ref{r-phi}(ii), we have the desired Harnack inequality.
\end{proof}
\section{Necessity of Poincar\'{e} inequality and large scale volume doubling}
In the previous sections, we have obtain two-sided Gaussian bounds on the heat kernel and parabolic Harnack inequality assuming
large scale volume doubling and a Poincar\'{e} inequality. Now we show that large scale volume doubling and Poincar\'{e} inequality are necessary to have
 two-sided Gaussian bounds on the heat kernel and parabolic Harnack inequality.
The was first proved by Saloff-Coste in \cite[Theorem 3.1]{Sal92} using an argument due to Kusuoka and Stroock \cite{KS87}.
Delmotte \cite{Del99} followed the same strategy in discrete-time setting for random walk on graphs.
The following is an adaptation of the argument in \cite{Sal92,Del99}.
\begin{prop}\label{p-phi-vd}
 Let $(M,d,\mu)$ be a quasi-$b$-geodesic metric measure space satisfying \ref{doub-loc}.
 Suppose that a Markov operator $P$ has a kernel $p_k$ that is weakly $(h,h')$-compatible
 with respect to $\mu$ for some $h>b$ and there exists $\eta \in (0,1)$ such that $P$ satisfies the parabolic Harnack inequality \hyperlink{phi}{$H(\eta/2, \eta^2/2,\eta^2,2\eta^2,4\eta^2)$}.
 Then $(M,d,\mu)$ satisfies \ref{doub-inf} and \hyperlink{poin-mms}{$(P)_{h'}$}.
\end{prop}
\begin{proof}
 Let $x \in M$ and $r >0$. Define $u=u_{x,r}$ as
 \[
u(l,y) = \begin{cases} 1 &\mbox{if } l \in \nint{0}{ \lfloor \eta^2 r^2 \rfloor -  1 }\\
\one_{B(x,r)}(y) &\mbox{if }  l = \lfloor \eta^2 r^2 \rfloor\\
\int_{B(x,r_2(k))} p^{B(x,r)}_{l-\lfloor \eta^2 r^2 \rfloor}(y,w) \,dw& \mbox{if } l > \lfloor \eta^2 r^2 \rfloor. \end{cases}
\]
Note that $u$ is non-negative and $u$-caloric in $\N \times B(x,r)$. For $k \in \N^*$, we choose $r$ such that $\eta^2 r^2= k$.
By applying \hyperlink{phi}{$H(\eta/2, \eta^2/2,\eta^2,2\eta^2,4\eta^2)$} to the function $u$, there exists $C_H, N_1 >1$ such that
\begin{equation}\label{e-nes1}
 1 = u(k,x) \le C_H u(2k,x) = C_H \int_{B(x,\sqrt{k}/\eta)} p_k^{B(x,\sqrt{k}/\eta)}(x,z) \, dz
\end{equation}
for all $x \in M$, for all $k \in \N^*$ with $k \ge N_1$. Squaring \eqref{e-nes1} and applying Cauchy-Schwarz inequality we obtain
\begin{align*}
  1  &\le C_H^2 \left( \int_{B(x,\sqrt{k}/\eta)} p_k^{B(x,\sqrt{k}/\eta)}(x,z) \, dz \right)^2 \nonumber \\
 &\le C_H^2 V(x,\sqrt{k}/\eta) \int_{B(x,\sqrt{k}/\eta)}\left( p_k^{B(x,\sqrt{k}/\eta)}(x,z)\right)^2  \, dz \nonumber \\
 &=  C_H^2 V(x,\sqrt{k}/\eta)  p_{2k}^{B(x,\sqrt{k}/\eta)}(x,x)
\end{align*}
for all $x \in M$ and for all $k \in \N^*$ satisfying $k \ge N_1$.
Therefore
\begin{equation}\label{e-nes2}
 p_{2k}^{B(x,\sqrt{k}/\eta)}(x,x) \ge   C_H^{-2} \frac{1}{V(x,\sqrt{k}/\eta) }
\end{equation}
for all $x \in M$ and for all $k \in \N^*$ satisfying $k \ge N_1$.

Next we apply  \hyperlink{phi}{$H(\eta/2, \eta^2/2,\eta^2,2\eta^2,4\eta^2)$} to the non-negative, $P$-caloric function $(l,y)\mapsto p_{l+2}(x,y)$
on $\nint{0,\lfloor 4 \eta^2 r^2 \rfloor} \times B(x,r)$ where $r$ is chosen such that $\eta^2 r^2 = k-2 \ge k/2$. Then there exists $N_2 \ge \max(4, N_1)$ such that
for all $k \ge N_1$, we have
\begin{equation} \label{e-nes3}
 p_k(x,x) \le C_H p_{2k}(x,y)
\end{equation}
for all $x \in M$, for all $k \ge N_2$ and for all $y \in B(x, \sqrt{k}/2)$.
Integrating \eqref{e-nes3} over $y \in B(x, \sqrt{(k/2)})$, we obtain
\begin{equation}\label{e-nes4}
  p_k(x,x) \le C_H \frac{1}{V(x,\sqrt{(k/2)})}
\end{equation}
for all $x \in M$, for all $k \ge N_2$. Iterating \eqref{e-nes3} with $y=x$, we obtain
\begin{equation}\label{e-nes5}
 p_{2k}(x,x) \le C_H^l p_{2^{l+1} k}(x,x)
\end{equation}
for all $x \in M$, for all $k \ge N_2$ and for all $l \in \N$.
Combining \eqref{e-nes2}, \eqref{e-nes4}, \eqref{e-nes5} along with $p^B_{2k} \le p_{2k}$, we obtain
\begin{equation*}
 \frac{1}{V(x,\eta^{-1} \sqrt{k})} \le \frac{C_H^{l+3}}{V(x, 2^{l/2}\sqrt{k})}
\end{equation*}
for all $l \in \N$, for all $x \in M$ and for all $k \in \N$ satisfying $k \ge N_2$.
Next we choose $l$ such that $2^{l/2} > 4 \eta^{-1}$ so that there exists $C_1>1$ such that
\begin{equation*}
{V(x,4\eta^{-1} \sqrt{k})} \le C_1{V(x,\eta^{-1} \sqrt{k})}
\end{equation*}
 for all $x \in M$ and for all $k \in \N$ satisfying $k \ge N_2$.
 Therefore there exists $R_1>0$ such that for all $r >R_1$ and for all $x \in M$ we have
 \[
  V(x,r) \ge V(x, \eta^{-1} \sqrt{ \lfloor \eta^2 r^2 \rfloor}) \ge C_1^{-1}V(x, 4 \eta^{-1} \sqrt{ \lfloor \eta^2 r^2 \rfloor}) \ge C_1^{-1} V(x,2r).
 \]
This completes the proof of \ref{doub-inf}.

It remains to prove the Poincar\'{e} inequality $(P)_{h'}$.
We start by showing a near diagonal lower bound for the `Dirichlet kernel' $p^{B(x,\eta^{-1}\sqrt k)}$.

By \hyperlink{phi}{$H(\eta/2,\eta^2/2,\eta^2,2\eta^2,4\eta^2)$} applied to the function
$(l,y) \mapsto p^{B(x,\eta^{-1}\sqrt{k})}_{k+l}(x,y)$ that is $P$-caloric on $\nint{0}{4k}\times B(x,\eta^{-1}\sqrt{k})$, we have
\begin{equation}\label{e-nes6}
 p_{2k}^{B(x,\eta^{-1}\sqrt{k})} (x,x) \le C_H \inf_{y \in B(x,\sqrt{k}/2)} p_{3k}^{B(x,\eta^{-1}\sqrt{k})}(x,y)
\end{equation}
for all $k \ge N_2$ and for all $x \in M$.
Similarly by \hyperlink{phi}{$H(\eta/2,\eta^2/2,\eta^2,2\eta^2,4\eta^2)$} applied to the function
$(l,z) \mapsto p^{B(x,\eta^{-1}\sqrt{k})}_{2k+l}(z,y)$ that is $P$-caloric on $\nint{0}{4k}\times B(x,\eta^{-1}\sqrt{k})$, we have
\begin{equation}\label{e-nes7}
 p_{3k}^{B(x,\eta^{-1}\sqrt{k})} (x,y) \le C_H \inf_{z \in B(x,\sqrt{k}/2)} p_{4k}^{B(x,\eta^{-1}\sqrt{k})}(z,y)
\end{equation}
for all $k \ge N_2$, for all $x \in M$ and  for all $y \in B(x,\sqrt{k}/2)$.
Combining \eqref{e-nes2}, \eqref{e-nes6}, \eqref{e-nes7} and \eqref{e-vd1} there exists $c_1>0$ such that
for all $x \in M$ and for all $k \in \N^*$ satisfying $k \ge N_2$, we have
\begin{equation}\label{e-nes8}
 \inf_{y,z \in B(x,\sqrt{k}/2)} p^{B(x,\eta^{-1}\sqrt{k})}_{4k}(y,z) \ge \frac{c_1}{V(x,\sqrt{k})}.
\end{equation}
For a ball $B=B(x,\eta^{-1} \sqrt{k})$, we define a Markov operator
\[
 Q_B f (y) := P_B f(y) + \left(1 - \int_B p^B_1(y,z) \, dz \right) f(y)
\]
for all $y \in B$ and for all functions $f$ on $B$. Note that unlike $P_B$, the operator $Q_B$ is conservative, that is
\[
 Q_B \one_B = \one_B.
\]
For the rest of the this proof we abbreviate $B(x,\eta^{-1}\sqrt k)$ by $B$. By \eqref{e-nes8} for all $B=B(x,\eta^{-1}\sqrt{k})$ satisfying $k \ge N_2$ and for all $y \in B(x,\sqrt{k}/2)$, we have
\begin{align}\label{e-nes9}
 \nonumber Q_B^{4k} \left[ f- Q_B^{4k}f(y)\right](y) & \ge P_B^{4k} \left[ f- Q_B^{4k}f(y)\right](y) \\
 & \ge \int_{B(x,\sqrt{k}/2)}\left(f(z) -Q_B^{4k}f(y) \right)^2 p_B^{4k}(y,z) \, dz \nonumber \\
&\ge \frac{c_1}{V(x,\sqrt{k})}\int_{B(x,\sqrt{k}/2)}\left(f(z) -Q_B^{4k}f(y) \right)^2  \, dz \nonumber \\
& \ge \frac{c_1}{V(x,\sqrt{k})}\int_{B(x,\sqrt{k}/2)}\left(f(z) -f_{B(x,\sqrt{k}/2)} \right)^2  \, dz.
\end{align}
The first line above follows from $Q_B g \ge P_B g$ for all $g \ge 0$, the third line above follows from \eqref{e-nes8} and the last line above follows from the fact that
mean minimizes square error \eqref{e-mean}. By \eqref{e-nes9} along with \eqref{e-vd1} there exists $c_2>0$ such that for all $x \in M$ and for all $k \in \N^*$, we have
\begin{align}
 \label{e-nes10} \nonumber \int_{B} Q_B^{4k} \left[ f - Q_B^{4k}f(y) \right](y) \, dy &\ge \int_{B(x,\sqrt{k}/2)} Q_B^{4k} \left[ f - Q_B^{4k}f(y) \right](y) \, dy \\
 &\ge c_2 \int_{B(x,\sqrt{k}/2)}\left(f(z) -f_{B(x,\sqrt{k}/2)} \right)^2  \, dz
\end{align}
where $B=B(x,\eta^{-1}\sqrt{k})$.

By linearity of the operator $Q_B^{4k}$, we have
\begin{equation*}
Q_B^{4k} \left[ f - Q_B^{4k}f(y) \right]^2(y) = \left(Q_B^{4k} f^2 \right) (y) -  \left( Q_B^{4k}f(y)\right)^2
 \end{equation*}
Therefore by the symmetry of the operator $Q_B$ and $Q_B \one_B = \one_B$, we have
\begin{align}\label{e-nes11}
 \int_B Q_B^{4k}\left[ f - Q_B^{4k}f(y) \right]^2(y) \, dy&= \langle \one_B , Q_B^{4k} f^2 \rangle_{L^2(B)} - \norm{Q_B^{4k}f}_{L^2(B)}^2 \nonumber \\
 &= \langle Q_B^{4k}\one_B ,  f^2 \rangle_{L^2(B)} - \norm{Q_B^{4k}f}_{L^2(B)}^2 \nonumber \\
 &=  \norm{f}_{L^2(B)}^2  -  \norm{Q_B^{4k}f}_{L^2(B)}^2 \nonumber \\
 &= \sum_{l=0}^{4k-1} \left( \norm{Q_B^l f}_{L^2(B)}^2- \norm{Q_B^{l+1} f}_{L^2(B)}^2 \right).
\end{align}
The identity $\norm{f}_{L^2(B)}^2 - \norm{Q_B f}_{L^2(B)}^{2} = \norm{ (I-Q_B^2)^{1/2}}_{L^2(B)}^2$ along with the fact that $Q_B$ is a contraction in $L^2$ yields
\begin{equation}
 \norm{Q_B^l  f}_{L^2(B)}^2 - \norm{Q_B^{l+1}  f}_{L^2(B)}^2 =\norm{Q_B^l(I-Q_B^2)^{1/2} f}_{L^2(B)}^2 \le  \norm{ f}_{L^2(B)}^2 - \norm{Q_B  f}_{L^2(B)}^2 \label{e-nes12}
\end{equation}
for all $l \in \N$. Combining \eqref{e-nes11} and \eqref{e-nes12}, we obtain
\begin{equation}
\label{e-nes13} \int_B Q_B^{4k} \left[ f - Q_B^{4k} f(y) \right]^2(y)  \,dy \le 4 k \left( \norm{f}_{L^2(B)}^2 - \norm{Q_B f}_{L^2(B)}^2 \right) .
\end{equation}
Using the inequality $a^2 -b^2 \le 2a(a-b)$, we have
\begin{align}\label{e-nes14}
\nonumber  \norm{f}_{L^2(B)}^2 - \norm{Q_B f}_{L^2(B)}^2
\nonumber &=\int_B (f(y))^2 - (Q_Bf(y))^2 \,dy \\
\nonumber & \le 2 \int_B f(y) \left( f(y) - Q_B f(y)\right) \,dy \\
\nonumber &= \int_B \int_B \left(f(y)- f(z) \right)^2 p_B(y,z) \,dy \,dz \\
&\hspace{4mm}- 2 \int_B \left( 1 - \int_B p^B(y,z) \, dz \right) (f(y))^2\,dy \nonumber\\
& \le  \int_B \int_B \left(f(y)-f(z) \right)^2 p_B(y,z) \,dy \,dz.
\end{align}
Combining \eqref{e-nes10}, \eqref{e-nes13} and \eqref{e-nes14}, for all $x \in M$, for all $k \in \N^*$ with $k \ge N_2$ and for all $f \in L^2(M)$, we have
\begin{equation*}
 \int_{B(x,\sqrt{k}/2)}\left(f(z) -f_{B(x,\sqrt{k}/2)} \right)^2  \, dz \le 4 c_2^{-1} k \int_B\int_B (f(y) - f(z))^2 p(y,z) \, dy \,dz
\end{equation*}
where $B=B(x,\eta^{-1}\sqrt{k})$.
Therefore there exists $R>0$,$C_1,C_2>0$ such that for all $x \in M$, for all $r > R$ and for all $f \in L^2(M)$, we have
\begin{align}
\nonumber \label{e-nes15}\int_{B(x,r)} \left(f(z) -f_{B(x,r)} \right)^2  \, dz & \le  \int_{B(x,r)} \left(f(z) -f_{B(x, \sqrt{ \lceil 4r^2\rceil}/2)} \right)^2\\
\nonumber &\le \int_{B(x, \sqrt{ \lceil 4r^2\rceil}/2)}   \left(f(z) -f_{B(x, \sqrt{ \lceil 4r^2\rceil}/2)} \right)^2  \, dz \\
 & \le  C_1 R^2 \int_{B(x,C_2r)} \int_{B(x,C_2r)} \left(f(y)-f(z)\right)^2 p(y,z) \, dy \,dz.
\end{align}
By \eqref{e-compat} and \eqref{e-nes15}, we have the desired Poincar\'{e} inequality $(P)_{h'}$.
\end{proof}
We now have all the ingredients to prove our main result in a slightly weaker form.
\begin{prop} \label{p-main1}
 Let $(M,d,\mu)$ be a quasi-$b$-geodesic metric measure space satisfying \ref{doub-loc} and $\operatorname{diam}(M)=+\infty$.
Suppose that a Markov operator $P$ has a kernel $p$ that is $(h,h')$-compatible with $(M,d,\mu)$ with either $h=h'>b$ or $h' > h \ge 5b$ . Then the following are equivalent:
\begin{enumerate}[(i)]
 \item Parabolic Harnack inequality: there exists $\eta \in (0,1)$ such that $P$ satisfies \hyperlink{phi}{$H(\eta/2,\eta^2/2,\eta^2,2\eta^2,4\eta^2)$}.
 \item Gaussian bounds on the heat kernel: the heat kernel $p_k$ satisfies \hyperlink{ge}{$(GE)$}.
 \item The conjunction of large scale volume doubling property \ref{doub-inf} and Poincar\'{e} inequality \ref{poin-mms}.
\end{enumerate}
\end{prop}
\begin{proof}
 The implication ``(iii) implies (ii)'' follows from Theorem \ref{t-Sob}, Proposition \ref{p-gue} and Proposition \ref{p-gle}.
 (ii) implies (i) follows from  Proposition \ref{p-phi}.
 (i) implies (iii) follows from Proposition \ref{p-phi-vd} and Corollary \ref{c-allscales}.
\end{proof}
Next, we answer the question raised in Remark \ref{r-pscale}.
\begin{prop} \label{p-pscale}
  Let $(M,d,\mu)$ be a quasi-$b$-geodesic metric measure space satisfying \ref{doub-loc}, \ref{doub-inf}, \hyperlink{poin-mms}{$(P)_{h'}$} for some $h'>b$ and $\operatorname{diam}(M)=+\infty$.
  Then $(M,d,\mu)$ satisfies \ref{poin-mms} for all $h > b$.
\end{prop}
\begin{proof}
By Lemma \ref{l-largesc} it suffices to consider the case $ b < h <h'$.
 Consider the Markov chain with density
 \[
  p(x,y) = \frac{\one_{B(x,h)}(y)}{Q(x) Q(y) \sqrt{V(x,h)V(y,h)}}
 \]
 that is symmetric with respect to the measure $\mu'(dx)= Q(x) \mu(dx)$, where
 \[
  Q(x)=  \int_{M} \frac{\one_{B(x,h)}(y)}{\sqrt{V(x,h)V(y,h)}} \, \mu(dy).
 \]
By \ref{doub-loc}, there exists $C_1>0$ such that
\begin{equation} \label{e-ps1}
 C_1^{-1} \le  Q(x) \le C_1
\end{equation}
for all $x \in M$. Therefore the space $(M,d,\mu')$ satisfies \ref{doub-loc}, \ref{doub-inf}, \hyperlink{poin-mms}{$(P)_{h'}$} for some $h'>b$.
Moreover by \eqref{e-ps1}, $p$ is weakly $(h,h)$-compatible with $(M,d,\mu')$.
By the same  argument as Lemma \ref{l-wscompat}, there exists $l \in N^*$ such that $p_l$ is $(h',lh)$ compatible with  $(M,d,\mu')$. Therefore by Proposition \ref{p-main1} and Lemma \ref{l-gausscompare}
the kernel $p_k$ satisfies \hyperlink{ge}{$(GE)$}.  The Poincar\'{e} inequality \ref{poin-mms} for $(M,d,\mu')$ then follows from Propositions \ref{p-phi} and \ref{p-phi-vd}.
An easy comparison argument using \eqref{e-ps1} gives \ref{poin-mms} for $(M,d,\mu)$.
\end{proof}
The following is the main result of our work.
\begin{theorem}\label{t-main1}
  Let $(M,d,\mu)$ be a quasi-$b$-geodesic metric measure space satisfying \ref{doub-loc} and $\operatorname{diam}(M)=+\infty$.
Suppose that a Markov operator $P$ has a kernel $p$ that is $(h,h')$-compatible 
with $(M,d,\mu)$, where $h' \ge h>b$. Then the following are equivalent:
\begin{enumerate}[(i)]
 \item Parabolic Harnack inequality: there exists $\eta \in (0,1)$ such that $P$ satisfies \hyperlink{phi}{$H(\eta/2,\eta^2/2,\eta^2,2\eta^2,4\eta^2)$}.
 \item Gaussian bounds on the heat kernel: the heat kernel $p_k$ satisfies \hyperlink{ge}{$(GE)$}.
 \item The conjunction of large scale volume doubling property \ref{doub-inf} and Poincar\'{e} inequality \ref{poin-mms}.
\end{enumerate}
\end{theorem}
\begin{proof}
Combining Propositions \ref{p-main1}, \ref{p-phi-vd} and \ref{p-pscale} yields the desired result.
\end{proof}
As announced in the introduction, we will show  Theorem \ref{t-main0} and Theorem \ref{t-Del} are covered by our results.
 Theorem \ref{t-Del} is clearly a special case of Theorem \ref{t-main1}. So it remains to verify Theorem \ref{t-main0}.
\begin{proof}[Proof of Theorem \ref{t-main0}]
We need only to check the implication (c) implies (b) as the other implications follow as in Theorem \ref{t-main1}.
Although $p_1$ is only weakly $(h,h')$-compatible to $(M,d,\mu)$, by Lemma \ref{l-wscompat}, Theorem \ref{t-main1} and  Lemma \ref{l-gausscompare}, we have that $p_k$ satisfies  \hyperlink{ge}{$(GE)$}.
 \end{proof}

\chapter{Applications} \label{ch-apply}
Perhaps the most important application of the characterization of parabolic Harnack inequality and Gaussian bounds on the heat kernel is the stability
under quasi-isometries.
\begin{theorem}
  Let $(M_i,d_i,\mu_i)$  be a quasi-$b_i$-geodesic metric measure spaces satisfying \ref{doub-loc} and $\operatorname{diam}(M_i)=+\infty$, for $i=1,2$. Moreover we assume that
  $(M_1,d_1,\mu_1$ and $(M_2,d_2,\mu_2)$ are quasi-isometric metric measure spaces.
Suppose that a Markov operator $P_i$ has a kernel  that is $(h_i,h_i')$-compatible with $(M_i,d_i,\mu_i)$ with  $h_i' \ge h_i > b_i$ for $i=1,2$. Then
\begin{enumerate}[(i)]
 \item The kernel corresponding to $P_1$ satisfies \hyperlink{ge}{$(GE)$} if and only if the kernel corresponding to $P_2$ satisfies \hyperlink{ge}{$(GE)$}.
 \item The operator $P_1$ satisfies the Harnack inequality \hyperlink{phi}{$H(\eta/2,\eta^2/2,\eta^2,2\eta^2,4\eta^2)$} for some $\eta \in (0,1)$ if and only if
 $P_2$ satisfies \hyperlink{phi}{$H(\zeta/2,\zeta^2/2,\zeta^2,2\zeta^2,4\zeta^2)$} for some $\zeta \in (0,1)$.
\end{enumerate}
\end{theorem}
\begin{proof}
 The is a direct consequence of Theorem \ref{t-main1} along with stability of \ref{doub-inf} given by Proposition \ref{p-qivd}, stability of \ref{poin-mms} given by
 Proposition \ref{p-robustmms}, Proposition \ref{p-pscale} and Lemma \ref{l-largesc}.
\end{proof}
As mentioned in the introduction, it is a long standing open problem to prove such a stability result for elliptic Harnack inequality. 
A partial result in this direction is obtained by Bass. In \cite{Bas13}, Bass proves  stability of elliptic Harnack inequality  for weighted graphs under bounded perturbation of the conductances. However the weighted graphs were assumed to 
be transient and they satisfy certain regularity hypotheses. 
In \cite{Bar05}, Barlow introduced the \emph{dumbbell condition} that is stable under bounded perturbation of weights of a weighted graph and asks
if the dumbbell condition is equivalent to elliptic Harnack inequality. 

Recall that we proved an elliptic H\"{o}lder regularity estimate for $P$-harmonic functions
in Proposition \ref{p-holder} and we used the regularity in the proof of Gaussian lower bounds (Lemma \ref{l-crl}).
There is an analogous parabolic  H\"{o}lder regularity estimate which follows from parabolic Harnack inequality.
The proof is similar, for example the proof given in \cite[Theorem 5.4.7]{Sal02} can be adapted for the present setting.
Such parabolic H\"{o}lder continuity estimates were first obtained by Nash \cite{Nas58}.
\begin{prop}
  Let $(M,d,\mu)$ be a quasi-$b$-geodesic metric measure space satisfying \ref{doub-loc} and $\operatorname{diam}(M)=+\infty$.
Suppose that a Markov operator $P$ has a kernel $p$ that is weakly $(h,h')$-compatible with $(M,d,\mu)$ and satisfies parabolic Harnack inequality
\hyperlink{phi}{$H(\eta/2,\eta^2/2,\eta^2,2\eta^2,4\eta^2)$} for some $\eta \in (0,1)$. Then there exists $C>0$, $R>0$ and $\alpha>0$ such that for all $x \in M$,
for all $r > R$ and for any non-negative function $u:\N \times M \to \R$ that is
$P$-caloric in $\nint{0}{\lfloor 4 \eta^2 r^2 \rfloor} \times B(x,r) = Q$, we have the regularity estimate
\begin{align*}
\sup_{(k_1,x_1), (k_2,x_2) \in  \nint{\lceil 2 \eta^2 r^2\rceil}{\lfloor 4 \eta^2 r^2 \rfloor} \times B(x,r)} \frac{ \abs{ u(k_1,x_1) - u(k_2,x_2)}} {\left( \max(1, \abs{k_1-k_2} + d(x_1,x_2)\right)^\alpha}
& \le \frac{C} {r^{\alpha}} \sup_{Q} u.
\end{align*}
\end{prop}
Note that we do not obtain continuity, because we do not have H\'{o}lder continuity estimate at arbitrarily small distances.
Another application of elliptic Harnack inequality is Liouville property for harmonic functions that was
shown in Proposition \ref{p-liov}.

Next, we turn attention to application of two sides Gaussian estimates \hyperlink{ge}{$(GE)$}.
Of course, the estimates given by \hyperlink{ge}{$(GE)$} has enough information to determine whether or not the the random walk is transient.
The estimate given by \cite[Proposition 4.3]{Del99} can be easily generalized to metric measure spaces in which case we obtain
\begin{prop} \label{p-trans}
 Let $(M,d,\mu)$ be a quasi-$b$-geodesic metric measure space satisfying \ref{doub-loc} and $\operatorname{diam}(M)=+\infty$. Consider a $\mu$-symmetric Markov operator $P$ that is
 $(h,h')$-compatible with $(M,d,\mu)$ for some $h >b$ and
 whose kernel $p_k$ satisfies \hyperlink{ge}{$(GE)$}. Then the random walk corresponding to $P$ is transient if and only if
 \begin{equation} \label{e-crit}
\sum_{n=1}^\infty \frac{n}{V(x,n)} < +\infty
 \end{equation}
for some $x \in M$.
\end{prop}
It is easy to see that the convergence of the series in \eqref{e-crit} does not depend on the choice of $x \in M$. Unless the space is discrete, we do not have a `Green's function' as the Green operator
$\Delta^{-1} = \sum_{i=0}^\infty P^i$ does not have a kernel as there is `delta mass' singularity at the starting point. However, we may consider the  off-diagonal part of the Green operator given by the
``Green's function''
 $G(x,y) = \sum_{i=1}^\infty p_i(x,y)$. The estimate given by \cite[Proposition 4.3]{Del99} can be again generalized as follows.
 \begin{prop}
Under the assumptions of Proposition \ref{p-trans}, there exists $C>0$ such that
 \begin{equation} \label{e-gf}
C \sum_{n= \lceil d(x,y) \rceil }^\infty \frac{n}{V(x,n)}  \le G(x,y) := \sum_{i=1}^\infty p_i(x,y) \le C \sum_{n= \lceil d(x,y) \rceil }^\infty \frac{n}{V(x,n)}
 \end{equation}
for some $x \in M$ and for all $y \in M$ with $d(x,y) > h'$.
\end{prop}
As noted in \cite[Theorem 9.1]{HS93}, the Gaussian estimate is sufficient to prove law of iterated logarithm in a weak form.
The proof in \cite{HS93} can be generalized for metric measure spaces.
\begin{prop}
 Under the assumptions of Proposition \ref{p-trans}, there exist $C>0$ such that for all starting points $X_0 \in M$
 \[
  C^{-1} \le \limsup \frac{d(X_0,X_n)}{ \left( n \log \log n\right)^{1/2}}\ \le C
 \]
almost surely, where $(X_k)_{k \in N}$ is the Markov chain corresponding to $P$.
\end{prop}
We refer the reader to \cite[Section 9]{HS93} for other probabilistic applications in similar spirit.

We sketch a possible application to mixing times of Markov chains that will be developed elsewhere.
If the space has finite diameter the techniques developed here can be used to prove upper and lower bounds on \emph{mixing times}.
In this case $\mu$ is a finite measure on $M$ and can be normalized if necessary to be the stationary probability measure.
Roughly speaking, in this case for $(h,h)$-compatible Markov operator on a space with diameter $D$, it takes $(D/h)^2$ steps of the Markov chain to get close to the stationary distribution $\mu$.
The Poincar\'{e} inequality and Gaussian upper bounds can be used to obtain upper bounds on mixing time as outlined in \cite[Lemma 2.1 and Remark 1 after Lemma 2.2]{DS94}.
For lower bounds on the mixing time one would need Gaussian lower bounds. We plan to address these questions in a sequel
and obtain results complementary to those in \cite{LM10}. We refer the reader to \cite{DLM11,DLM12} for other recent works in this direction.

\section{Harmonic functions with polynomial volume growth}
In \cite{CM97}, Colding and Minicozzi proved that the space of harmonic functions with polynomial volume growth with fixed rate on a manifold satisfying volume doubling and Poincar\'{e} inequality is finite dimensional.
As a corollary, they prove a conjecture of S. T. Yau on manifolds that asserts the above property for Riemannian manifolds with non-negative Ricci curvature.
A recent surprising application of this result is an alternate proof of Gromov's theorem on groups of polynomial volume growth due to Kleiner \cite{Kle08}. This new proof avoids the solution to Hilbert's fifth problem
(Montgomery-Zippin-Yamabe structure theory). To precisely state a theorem we need the following definition.
\begin{definition} \label{d-poly}
For a metric measure space $(M,d,\mu)$ and a $\mu$-symmetric Markov operator $P$ on $M$, we define the space of $P$-harmonic functions  with growth rate $d$ as the vector space $\mathcal{H}_d(M,P)$ consisting
of all $P$-harmonic functions $u$ such that there exists $C >0, p \in M$ (depending on $u$) such that
$\abs{u(x)} \le C (1 + d(x,p)^\gamma)$ for all $x \in M$.
\end{definition}
We have the following theorem that extends the result of Colding and Minicozzi to random walks on metric measure spaces.
\begin{theorem}\label{t-cm}
Let $(M,d,\mu)$ be a quasi-geodesic metric measure spaces satisfying  $\operatorname{diam}(M)=+\infty$, volume doubling hypotheses \ref{doub-loc}, \ref{doub-inf} and Poincar\'{e} inequality \ref{poin-mms}.
Let $P$ be a Markov operator that is $(h,h')$-compatible with $(M,d,\mu)$.
Then the space of $P$-harmonic functions $\mathcal{H}_d(M,P)$ with  a fixed growth rate $d$ is finite dimensional for any $d \ge 0$.
\end{theorem}
The proof of Colding and Minicozzi's theorem in \cite{CM97} relies on three ingredients: volume doubling hypotheses \ref{doub-glob}, a Poincar\'{e} inequality \ref{e-opoin} and a reverse Poincar\'{e} inequality for harmonic functions.
We have all the three ingredients as we showed the reverse Poincar\'{e} inequality in Lemma \ref{l-rp}.
A caveat is that we have to rely on weaker versions of all the three ingredients but nevertheless we will see that Theorem \ref{t-cm} can be proved using the techniques introduced of \cite{CM97}.
T. Delmotte adapted an alternate approach due to P. Li \cite{Li97} to prove a similar statement for random walks on graphs satisfying doubling and Poincar\'{e} inequality \cite{Del98}.

The next proposition below is  a slightly weaker version of \cite[Proposition 2.5]{CM97}. 
\begin{prop}\label{p-on}
Let $(M,d,\mu)$ be a quasi-geodesic metric measure spaces satisfying  $\operatorname{diam}(M)=+\infty$, volume doubling hypotheses \ref{doub-loc}, \ref{doub-inf} and Poincar\'{e} inequality \ref{poin-mms} 
and let $P$ be a Markov operator that is $(h,h')$-compatible with $(M,d,\mu)$.
There exists $\epsilon \in (0,1)$ such that for all $p \in M$, for all $k \ge 1$ satisfying $r \ge k/\epsilon$ and for all functions $f_1,f_2,\ldots,f_n \in L^\infty_{\operatorname{loc}}(M)$ satisfying 
\begin{equation}
 \label{e-cnd2} \int_{B(p,r)} f_i^2 \, d\mu = V(p,r)
\end{equation}
for all $i=1,2,\ldots,n$;
 \begin{equation}
 \label{e-cnd1} \abs{\int_{B(p,r)} f_i f_j \, d\mu} < \frac{ V(p,r)}{2}
\end{equation}
for all $1 \le i < j \le n$; and
\begin{equation} \label{e-cnd3} 
       \int_{B(p,2r)} \left( f_i^2 + (2r)^2 \abs{\nabla_P f_i}^2 \right)\,d\mu \le k^2 V(p,r)
      \end{equation}
for all $i=1,2,\ldots,n$, we have $n \le \mathcal{N}$, where $\mathcal{N}$ depends  on $k$ but does not depend on $r \ge k/\epsilon$ or $p \in M$.
\end{prop}
\begin{proof}
 By Lemma \ref{l-doub-prop} there exists $C_D>0$ such that
 \begin{equation}\label{e-on1}
 V(x,2r_1 ) \le V(x,r_1) 
 \end{equation}
for all $r_1 \ge 1$ and for all $x \in M$. Moreover if we set $\delta := \log_2 C_D$, we have
\begin{equation}\label{e-on2}
 \frac{V(x,r_2)}{V(x,r_1)} \le C_D \left( \frac{r_2}{r_1}\right)^\delta
\end{equation}
for all $x \in M$ and for all $1 \le r_1 \le r_2$. By Lemma \ref{l-pflex} and \eqref{e-compat}, there exists constants $C_A > 0$ and 
$A \ge 1$ such that for all $x \in M$, for all $r_1 \ge 1$, for all functions $f \in L^\infty_{\operatorname{loc}}(M)$, we have
\begin{equation}\label{e-on3}
 \int_{B(x,r_1)} \abs{f -f_{B(x,r_1)}}^2 \,d\mu  \le C_A r_1^2 \int_{B(x,Ar_1)} \abs{\nabla_P f}^2 \, d\mu.
\end{equation}
Let $p \in M$, $r >0$, $k \ge 1$ and $k \le \epsilon r$. Define
\begin{equation}\label{e-on4}
 r_0:= \frac{\epsilon r}{k} \ge 1,
\end{equation}
where  $\epsilon \in (0,1)$ will be determined later. 

Let $x_1,x_2,\ldots,x_\nu$ be a $2r_0$-net of $B(p,r)$. We set
\begin{equation}\label{e-on5}
 \epsilon:= \min \left( \frac{1}{2 A}, \frac{1}{20 C_A^{1/2} C_D^{1/2} (4A+1)^{\delta/2}} \right).
\end{equation}
Since $\epsilon \le 1/2$ and  $r > r_0 \ge 1$, by \eqref{e-on1}, \eqref{e-on2}, we have
\begin{equation}\label{e-on6}
 \frac{1}{C_D} \le \frac{V(x_j,r)}{V(x_j,2r)} \le  \frac{V(x_j,r)}{V(p,r)} \le \frac{V(x_j,r_0)}{V(p,r)}  \frac{V(x_j,r)}{V(x_j,r_0)} \le C_D \left( \frac{k}{\epsilon}\right)^\delta \frac{V(x_j,r_0)}{V(p,r)}
\end{equation}
for all $j=1,2,\ldots,\nu$ and for all $r \ge 1$.
Since $r_0 \le r$ by Proposition \ref{p-net}(a), \eqref{e-on6} and \eqref{e-on1}, we have
\begin{equation} \label{e-on7}
  \sum_{j=1}^\nu V(x_j,r_0) \le V(p,r+r_0) \le V(p,2r) \le C_D V(p,r).
\end{equation}
By \eqref{e-on6} and \eqref{e-on7}, we have
\begin{equation}
 \label{e-on8} \nu \le C_D^3 (k/\epsilon)^\delta.
\end{equation}

Next we bound the overlap of the balls $\left( B(x_j,2 Ar_0) \right)_{1 \le j \le \nu}$. Define $\eta(y)$ as the cardinality of the set $\Sett{j}{ y \in B(x_j,2 A r_0)}$.
If $y \in \cap_{m=1}^{\eta(y)} B(x_{j_m}, 2 A r_0)$, then $B(y, (2 A+1) r_0)$ contains the disjoint balls $B(x_{j_m},r_0)$ and hence
\begin{equation}
 \label{e-on9} \sum_{m=1}^{\eta(y)} V(x_{j_m},r_0) \le V(y,(2A+1) r_0) 
\end{equation}
However by \eqref{e-on1}, \eqref{e-on2}, for all $y \in M$, for all $j_m$ such that $y \in B(x_{j_m},2Ar_0)$,  we have
\begin{equation}
 \label{e-no1} V(y,(2A+1)r_0) \le V(x_{j_m}, (4 A +1) r_0) \le C_D (4 A +1)^{\delta} V(x_{j_m},r_0).
\end{equation}
By \eqref{e-on9} and \eqref{e-no1}, we have 
\begin{equation} \label{e-no2}
\bar{C}:= \sup_{y \in M} \eta(y) \le C_D (4 A +1)^\delta.
\end{equation}
By Proposition \ref{p-net} the balls $B(x_j,2r_0)$ covers $B(p,r)$. We now partition $B(p,r)$ into $\nu$ disjoint subsets $S_1,S_2,\ldots,S_\nu$, where 
$B(x_j,r_0) \cap B(p,r) \subset S_j \subset B(x_j,2r_0)$.
Let $P= \Sett{x_j}{ 1 \le j \le \nu}$ denote the finite set of points in $B(p,r)$. For any function $f \in L^\infty_{\operatorname{loc}}(M)$, we set
\begin{equation}
 \label{e-no3} \mathcal{A}_{i,j} := \dashint_{B(x_j,2 r_0)} f_i \, d\mu = \frac{1}{V(x_j,2r_0)} \int_{B(x_j,2r_0)} f_i \,d\mu.
\end{equation}
By Cauchy-Schwarz inequality, \eqref{e-on6}, \eqref{e-cnd3}, we have
\begin{align}\label{e-no4}
\abs{ \mathcal{A}_{i,j}}^2 & \le \dashint_{B(x_j,2 r_0)} f_i^2 \, d\mu \le \frac{1}{V(x_j,r_0)} \int_{B(x_j,2 r_0)} f_i^2 \, d\mu \nonumber\\
&\le C_D^2 (k/\epsilon)^\delta \frac{1}{V(p,r)} \int_{B(x_j,2 r_0)} f_i^2 \, d\mu \le  C_D^2 (k/\epsilon)^\delta \frac{1}{V(p,r)} \int_{B(p,2 r)} f_i^2 \, d\mu  \nonumber \\
&\le C_D^2 k^2 (k/\epsilon)^\delta 
\end{align}
for all $i=1,\ldots,n$ and for all $j=1,\ldots,\nu$.

Let $\Lambda:= \Sett{\frac{s}{10}}{ s \in \Z, \abs{s} \le 10 C_D k (k/\epsilon)^{\delta/2}}$. Next, we define a map $f_i \mapsto \mathcal{M}(f_i)$, 
where $\mathcal{M}(f_i): P \to \Lambda$ is a function from a finite set $P$ to
another finite set $\Lambda$. With a slight abuse of notation, we intepret the function $\mathcal{M}(f_i)$ as a piecewise constant function on $B(p,r)$ that takes the value $\mathcal{M}(f_i)(x_j)$ on $S_j$, where $j=1,\ldots,\nu$.
For all $i=1,\dots,n$ and for all $j=1,\ldots,\nu$, we define $\mathcal{M}(f_i)(x_j) \in \Lambda$ as any closest point of $\Lambda$ to $\mathcal{A}_{i,j}$. 
By definition of $\Lambda$ and \eqref{e-no4}, for all $i,j$ we have
\begin{equation}
 \label{e-no5} \abs{\mathcal{A}_{i,j} - \mathcal{M}(f_i)(x_j)}^2 \le \frac{1}{400}.
\end{equation}
Combining the Poincar\'{e} inequality \eqref{e-on3}, \eqref{e-no5} and $S_j \subset B(x_i, 2 A r_0)$, we obtain
\begin{align*} 
 \nonumber \lefteqn{\int_{S_j} \abs{f_i - \mathcal{M}(f_i)(x_j)}^2 \,d\mu}\\
 &\le 2 \int_{B(x_j,2r_0)} \abs{f_i - \mathcal{A}_{i,j}}^2 \,d\mu + 2 \int_{S_j} \abs{\mathcal{A}_{i,j} - \mathcal{M}(f_i)(x_j)}^2 \,d\mu \nonumber \\
 & \le 8 r_0^2 C_A \int_{B(x_j, 2 A r_0)} \abs{\nabla_P f_i}^2 \, d\mu + \frac{\mu(S_j)}{200} 
\end{align*}
for all $i=1,\ldots,n$ and for all $j=1,\ldots,\nu$. Hence by \eqref{e-no2}, \eqref{e-cnd3} and \eqref{e-on5}, we have
\begin{align} \label{e-no6}
\nonumber\int_{B(p,r)} \abs{f_i-\mathcal{M}(f_i)}^2 \,d\mu &\le 8 r_0^2 C_A \bar{C}\int_{B(p,2r)} \abs{\nabla_P f_i}^2 \, d\mu + \sum_{j=1}^\nu \frac{\mu(S_j)}{200} \\
 &\le 2 r_0^2 C_A \bar{C} k^2 r^{-2} V(p,r) + \frac{V(p,r)}{200} \le \frac{V(p,r)}{100}.
\end{align}
By the triangle inequality along with \eqref{e-no6}, we obtain
\begin{align} \label{e-no7}
\lefteqn{ \left( \int_{B(p,r)} \abs{f_i - f_j}^2 \,d\mu \right)^{1/2} - \left( \int_{B(p,r)}  \abs{\mathcal{M}(f_i)- \mathcal{M}(f_j)}^2 \,d\mu \right)^{1/2} }\\
\nonumber&\le  \left( \int_{B(p,r)} \abs{f_i - \mathcal{M}(f_i)}^2 \,d\mu \right)^{1/2} +  \left( \int_{B(p,r)} \abs{f_j - \mathcal{M}(f_j)}^2 \,d\mu \right)^{1/2} \le\frac{\sqrt{V(p,r)}}{5}
\end{align}
for all $i \neq j$. By \eqref{e-cnd1} and \eqref{e-cnd2}, we have for $i\neq j$
\begin{equation}
 \label{e-no8} \left( \int_{B(p,r)} \abs{f_i - f_j}^2 \,d\mu \right)^{1/2} > \sqrt{V(p,r)}.
\end{equation}
Combining \eqref{e-no7} and \eqref{e-no8}, for all $i \neq j$ we obtain
\begin{equation*}
 \left( \int_{B(p,r)}  \abs{\mathcal{M}(f_i)- \mathcal{M}(f_j)}^2 \,d\mu \right)^{1/2} >0.
\end{equation*}
Hence the map $\mathcal{M}$ is injective. Therefore by \eqref{e-on8}
\begin{equation*}
 n \le  \abs{\Lambda}^\abs{P} =\abs{\Lambda}^\nu \le \mathcal{N}:= \left( 20 C_D k (k/\epsilon)^{\delta/2} + 1 \right)^{C_D^3 (k/\epsilon)^\delta}.
\end{equation*}
Note that the value of $\mathcal{N}$ does not depend on the value of $p \in M$ or $r$ but only on $k$ and the constants associated with doubling properties and Poincar\'{e} inequality.
\end{proof}
Next, we recall the a result due to Colding and Minicozzi  \cite[Proposition 4.16]{CM97}. We omit the proof as it is identical to that of \cite[Proposition 4.16]{CM97}.
\begin{prop} \label{p-4.16}
Consider a metric measure space $(M,d,\mu)$ satisfying the hypotheses \ref{doub-loc}, \ref{doub-inf} and $\operatorname{diam}(M)=+\infty$. Let $P$ be a Markov operator that is $(h,h')$-compatible with $(M,d,\mu)$ for some $0 < h \le h'$.
 Suppose that $u_1,u_2,\ldots,u_{2k} \in \mathcal{H}_d(M,P)$ are linearly independent.
 There exists $\delta> 0$, $p \in M$ such that for all $d >0$, $\Omega > 1$ and $m_0>0$, there exists $m \ge m_0$, $l \ge \frac{k}{2} \Omega^{-4 d - \delta}$, and functions $v_1,\ldots,v_l$ in the linear span of $u_i$ such that
 \begin{equation} \label{e-tc1}
  2 \Omega^{4d+2\delta} V(p,\Omega^m) = 2 \Omega^{4d + 2 \delta} \int_{B(p,\Omega^m)} v_i^2 \, d\mu  \ge \int_{B(p,\Omega^{m+1})} v_i^2 \, d\mu
 \end{equation}
and
\begin{equation}\label{e-tc2}
\int_{B(p,\Omega^m)} v_i v_j \, d\mu = \delta_{i,j} V(p,\Omega^m).
\end{equation}
\end{prop}
In Proposition \ref{p-4.16}, we may choose $\delta$ as the constant in \eqref{e-on2}.
We are now ready to prove Theorem \ref{t-cm}.
\begin{proof}[Proof of Theorem \ref{t-cm}]
Fix $\Omega > \max(4, 3h')$, $d >0$ and $p \in M$.
Let $C_R$ be as given by Lemma \ref{l-rp} and set $k^2= (8 C_R+2) \Omega^{4 d +2\delta}$. Let  $\epsilon \in (0,1)$ be given by Proposition \ref{p-on}. We choose $m_0 \in \N^*$ such that $\Omega^{m_0} \ge k/\epsilon$.
Let $\operatorname{dim} \mathcal{H}_d(M,P) \ge \mathcal{N}_0 := 4 \Omega^{4d +2 \delta} \mathcal{N}$ where $\mathcal{N}$ is given by Proposition \ref{p-on} where $k$ is as defined above.

Suppose that $u_1,u_2,\ldots,u_{\mathcal{N}_0} \in \mathcal{H}_d(M,P)$ be linearly independent. 
Then  by Proposition \ref{p-4.16} and reverse Poincar\'{e} inequality (Lemma \ref{l-rp}) there exists $C_R>0$ 
and $m > m_0$ such  that for all $f \in L^\infty_{\operatorname{loc}}(M, \mu)$, we have
harmonic functions $v_1,v_2,\ldots,v_l$ satisfying 
\begin{equation}\label{e-cm1}
 l \ge \frac{1}{4} \mathcal{N}_0 \Omega^{-4d-2\delta} = \mathcal{N},
\end{equation}
\begin{equation}\label{e-cm2}
 \int_{B(p,\Omega^m)} v_iv_j \, d\mu = V(p,\Omega^m) \delta_{i,j},
\end{equation}
\begin{equation}\label{e-cm3}
 \int_{B(p,\Omega^{m+1})} v_i^2 \, d\mu \le 2 C_R \Omega^{4d + 2\delta} V(p,\Omega^m),
\end{equation}
and
\begin{equation}\label{e-cm4}
 \int_{B(p,2 \Omega^m)} \abs{\nabla_P v_i}^2 \, d\mu \le C_R \Omega^{-2m} \int_{B(p,4\Omega^m)}  v_i^2 \,d\mu \le  2 \Omega^{4d + 2\delta- 2m} V(p,\Omega^m).
\end{equation}
Note that \eqref{e-cm2}, \eqref{e-cm3}, \eqref{e-cm4} and $\Omega >4$ implies that $v_1,v_2,\ldots,v_l$ satisfy \eqref{e-cnd2}, \eqref{e-cnd2}
\begin{equation} \label{e-cm5}
  \int_{B(p,2 \Omega^m)} v_i^2 + (2\Omega^m)^2 \abs{\nabla_P v_i}^2 \, d\mu \le  (8 C_R + 2) \Omega^{4d +2\delta} V(p,r)
\end{equation}
for all $i =1,\ldots,l$. 
Note that \eqref{e-cm1}, \eqref{e-cm2}, \eqref{e-cm5} along with Proposition \ref{p-on} implies the desired contradiction. Therefore $\operatorname{dim} \mathcal{H}_d(M,P) < \mathcal{N}_0 < \infty$.
\end{proof}
\begin{remark}
 Similar to \cite{Kle08}, we can replace the volume doubling hypotheses \ref{doub-loc}, \ref{doub-inf} of Theorem \ref{t-cm} by a weakly polynomial growth assumption on the volume growth.
\end{remark}

\section{Directions for future work}
We end with a few directions for future work.
One of the features of our work is that it provides an unified approach to Gaussian estimates for discrete time Markov chains on both discrete and continuous spaces.
Recently, there has been considerable interest in analysis and probability on fractals and fractal-like manifolds and graphs.
For many natural family of fractals the heat kernel satisfies sub-Gaussian estimates of the form
\[
  p_t(x,y) \asymp \frac{C_1} {V(x, t^{1/\beta})} \exp \left( -C_2  \left( \frac{d(x,y)^\beta}{t} \right)^{1/(\beta-1)} \right)
\]
for all $t >0$ and for all $x,y \in M$ and $\beta > 1$ is a parameter (See \cite[Theorem 1.5(e)]{BP88} for an early example).
Here $\asymp$ means that both inequalities $\le$ and $\ge$ hold with different values of constants $C_1,C_2$.
Similar to the characterizations of Gaussian estimates
in \cite{Gri91,Sal92,Stu96,Del99,HK00} there exists various characterizations for sub-Gaussian estimates both in the setting of diffusions on local Dirichlet spaces \cite{BBK06} and for discrete time Markov chains on graphs
\cite{BB04,BCK05,GT01,GT02}. As in the case of Gaussian estimates, it is desirable to obtain characterizations of sub-Gaussian estimates that are stable under quasi-isometries.
This was achieved using a condition called cutoff-Sobolev inequality first introduced by
Barlow and Bass \cite{BB04} (See also \cite{BBK06}). Our work naturally raises an analogous question for sub-Gaussian estimates on Markov chains.
\begin{problem}\nonumber
Characterize sub-Gaussian estimates for discrete time Markov chains on quasi-geodesic  metric measure spaces using geometric conditions that are stable with respect to quasi-isometries.
\end{problem}
Another direction for future work is to clarify the applications to mixing times in the finite diameter case as mentioned in Remark \ref{r-fin}(b).

As mentioned in the introduction, we state the problem concerning the stability of the elliptic Harnack inequality.
\begin{problem}\nonumber
Is elliptic Harnack inequality stable under quasi-isometries?
If so, characterize the elliptic-Harnack inequality by geometric properties that are stable under quasi-isometries.
\end{problem}
We refer the reader to \cite{Bas13, Bar05} for partial progress and conjectures aimed at solving the above problem.
\appendix
\chapter{Interpolation Theorems}\label{a-int}
In this appendix, we state Riesz-Thorin and Marcinkiewicz interpolations theorems and refer the reader to the literature for a proof.

Let $T: (X, \norm{\cdot}_X) \to (Y,\norm{\cdot}_Y)$ be a linear operator between  normed linear spaces. We denote the \emph{operator norm} by
\[
 \norm{T}_{X \to Y} = \sup_{x  \in X, x \neq \mathbf{0}} \frac{ \norm{T x}_Y} {\norm{x}_X} = \sup_{x  \in X, \norm{x} = 1} \norm{T x}_Y.
\]
If $\norm{T}_{X \to Y}<\infty$, we say the operator $T$ is \emph{bounded}. It is well known that $T$ is bounded if and only if $T$ is continuous.
We abbreviate $ \norm{T}_{L^p \to L^q}$ as $\norm{T}_{p \to q}$.
\begin{theorem}[Riesz-Thorin interpolation theorem] \label{t-rt}
 Assume that $(X, \Sigma, \mu)$ is a $\sigma$-finite measure space. Suppose $1 \le p_0, p_1 \le \infty$, $1 \le q_0,q_1 \le \infty$.
 Let $T: L^{p_0} + L^{p_1}\to   L^{q_0} + L^{q_1}$, be a linear operator such that $T:L^{p_0} \to   L^{q_0}$ and $L^{p_1} \to   L^{q_1}$
 are bounded. Then
 \[
  \norm{T}_{p_\theta \to q_\theta} \le \norm{T}_{p_0 \to q_0}^{1-\theta} \norm{T}_{p_1 \to q_1}^\theta
 \]
for all $\theta \in (0,1)$ where $1/p_\theta := (1-\theta)/p_0 + \theta/p_1$ and
$1/q_\theta := (1-\theta)/q_0 + \theta/q_1$.
\end{theorem}
We refer the reader to  \cite{Ste56} for a proof of  Stein's interpolation theorem  which in turn implies Theorem \ref{t-rt}.

Consider a $\sigma$-finite measure space $(X,\Sigma,\mu)$. The distribution function of $f$ is defined by
\[
 \lambda_f(t) = \mu \Sett{x \in X}{ \abs{f(x)} > t}
\]

We denote \emph{weak $L^p$ space} by $L^{p,w}$
For a measurable function $f$ and $1 \le p <\infty$, we define its $L^{p,w}$ norm by
\[
 \norm{f}_{p,w}= \left( \sup_{t >0} t^p \lambda_f(t) \right)^{1/p}.
\]
We say a measurable function $f \in L^{p,w}$ if  $\norm{f}_{p,w}<\infty$.
Note that $L^{p,w}$ is not a true norm, since $\norm{\cdot}_{p,w}$ does not satisfy triangle inequality.
If $f \in L^p$, then $\norm{f}_p \le \norm{f}_{p,w}$. Therefore $L^p \subset L^{p,w}$. It is easy to check that
$L^p \neq L^{p,w}$ in general.
\begin{theorem}[Marcinkiewicz interpolation theorem] \label{t-marcin}
 Let $1 \le p_0 \le q_0 <\infty$, $1 \le p_1 \le q_1 <\infty$  with $q_0 \neq q_1$.
 Let $T$ be a linear operator from $L^p_1 + L^p_2$ to the space of measurable functions. If $T$ satisfies
 \[
  \norm{T f}_{q_i,w} \le B_i \norm{f}_{p_i} \mbox{for all $f \in L^{p_i}$}, \hspace{1cm} i=0,1
 \]
then
\[
 \norm{T}_{ p_\theta \to q_\theta } \le C_{p_0,p_1,q_0,q_1,\theta} B_0^{1-\theta} B_1^\theta
\]
for all $\theta \in (0,1)$,
where $1/p_\theta := (1-\theta)/p_0 + \theta/p_1$,
$1/q_\theta := (1-\theta)/q_0 + \theta/q_1$ and  $ C_{p_0,p_1,q_0,q_1,\theta} < \infty$ depends only on
$p_0,p_1,q_0,q_1,\theta$.
\end{theorem}
We refer the reader to \cite[Theorem 2.58]{AF03} for a proof of Theorem \ref{t-marcin}.
\chapter{Examples} \label{a-example}
Here we collect various examples discussed earlier and supplement them with more examples, comments and pictures.
 \begin{example}[Euclidean space with radial weights] \label{x-nongauss}
 We expand upon the ball walk described in Example \ref{x-ballwalk} for specific metric measure spaces.
 If a weighted Riemannian manifold satisfies two-sided
 Gaussian estimates for its canonical diffusion, one might na\"{\i}vely expect the same to hold for the ball walk. However this is not true in general because the measure $\mu'$ of Example \ref{x-ballwalk}
 is not necessarily comparable to $\mu$. The measure $\mu'$ might fail to satisfy either \ref{doub-inf} or \ref{poin-inf}. Recall the example $(M,d,\mu)= (\R^n,\norm{\cdot}_2, \mu_\alpha)$  from Example \ref{x-grisal}, where
 $\mu_\alpha = (1 + \abs{x})^{\alpha/2} \, dx$.

 Note that if $\alpha >0$ there is a drift away from the origin and if $\alpha <0$ there is a drift towards the origin.
 If $\mu=\mu_\alpha$ one can verify that $\mu'=\mu'_\alpha$ is comparable to $\mu_{2\alpha}$.  Therefore the ballwalk accentuates the drift towards or away from the origin (See Figure \ref{f-rad}).
\begin{figure}[h]
  \centering
  \includegraphics[scale=0.25]{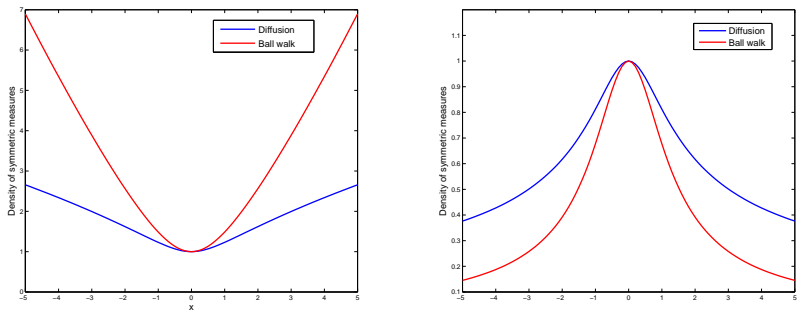}
  \caption{Density of $\mu_\alpha$ and $\mu'_\alpha$ for the case $\alpha = 0.6$ (left) and $\alpha = -0.6$ (right) in $\R$. Here  $\mu'_\alpha$ is normalized to have density $1$ at origin.}
  \label{f-rad}
\end{figure}
 In light of the above observation along with Table \ref{table:example}, Theorem \ref{t-main0} and Proposition \ref{p-rmtfae}, if $n \ge 2$ and $ \alpha \in (-n, -n/2]$, then the canonical diffusion on $(\R^n,\norm{\cdot}_2, \mu_\alpha)$ satisfies Gaussian estimates
 but the ball walk fails to satisfy Gaussian estimates because $(\R^n,\norm{\cdot}_2, \mu'_\alpha)$ does not satisfy \ref{doub-inf}. 
 In the case $n=1$ and $\alpha \in [1/2, 1)$,  the canonical diffusion on $ (\R, \norm{\cdot}_2, \mu_\alpha)$
 satisfies Gaussian estimates
 but the ball walk fails to satisfy Gaussian estimates because $(\R,\norm{\cdot}_2, \mu'_\alpha)$ does not satisfy \ref{poin-inf}.
 
 Even when both diffusion and ball walk satisfy Gaussian estimates for the transition kernels with respect to the invariant measures, the long term behavior might be different. For example, 
 if $n \ge 3$ and if $ 2- n <\alpha \le (2-n)/2$ then both diffusion and ball walk on $(\R^n,d,\mu_\alpha)$ satisfy Gaussian estimates for the transition kernels with respect to the invariant measures.
 However in this case the ball walk is recurrent but the diffusion is transient.
 
 For $n \in \N^*$, the the ball walk on $(\R^n,d,\mu_\alpha)$ is  positive recurrent chain  $\alpha < -n/2$, null recurrent for $-n/2 \le \alpha \le (2-n)/2$ and transient if $\alpha  >(2-n)/2$  (See \cite[Proposition 10.1.1]{MT09}
 and Proposition \ref{p-trans}).
  However the canonical diffusion on $(\R^n,d,\mu_\alpha)$ is  positive recurrent chain  $\alpha < -n$, null recurrent for $-n \le \alpha \le 2-n$ and transient if $\alpha  >2-n$.
\end{example}

\begin{example}[Complexes] \label{x-complex}
 Consider the Euclidean 2-complex  in $\R^3$ formed by the hyperplanes $H_{i,n} = \Sett{(x_1,x_2,x_3)}{ x_i= n}$
 where $i=1,2,3$ and $n \in \Z$. The metric is described by the intrinsic metric and the measure is the two dimensional surface measure. Dirichlet forms on such Riemannian complexes have
 been studied in \cite{PS08}. This example satisfies both Poincar\'{e} inequality \ref{poin-mms} for all $h>0$ and Volume doubling \ref{doub-inf}. The geometry of the balls depend on the center (See Figure \ref{f-complex}).

 The above example can also be viewed as a \emph{Cayley complex} \cite[p. 77]{Hat02} corresponding to the presentation
 \[ \Z^3=\left\langle a_1,a_2,a_3 \hspace{2mm}\vline \hspace{2mm}a_1a_2a_1^{-1} a_2^{-1}, a_2a_3a_2^{-1} a_3^{-1}, a_3a_1a_3^{-1} a_1^{-1}\right\rangle. \]
 More generally, consider a finitely generated and finitely presented group $G= \langle S \hspace{2mm} \vline \hspace{2mm} R \rangle$.
Note that the  $1$-skeleton of the Cayley complex  is the Cayley graph of $\langle S \hspace{2mm} \vline \hspace{2mm} R \rangle$ and  the $2$-cells (faces) are in bijection with $G \times R$.
We equipp each $2$-cell with the the usual Euclidean metric on the regular $n$-gon with edges of length 1 and we endow the space with the measure obtained by equiping each two-cell with its Lebesgue measure.
It is easy to verify that the Cayley complex is quasi-isometric to the Cayley graph of $G$ with the quasi-isometry given by the natural embedding of the Cayley graph in the Cayley complex.
By the stability of \ref{doub-inf} and \ref{poin-mms} under quasi-isometries, we have that ball walk on Cayley complexes of nilpotent groups (such groups are finitely generated and finitely presented)
satisfy two sided Gaussian bounds.
\begin{figure}[h]
  \centering
  \includegraphics[scale=0.5]{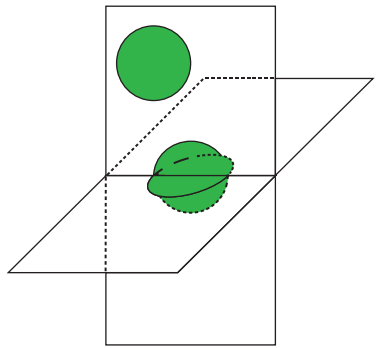}
  \caption{Two balls (in intrinsic metric) with same radius but different centers in the Cayley complex of $\Z^3$.}
  \label{f-complex}
\end{figure}

In  Figure \ref{f-grid} we consider the $1$-complex $\Sett{(x,y) \in \R^2}{ x \in \Z \mbox{ or } y \in \Z}$ equipped with intrinsic metric. Note that the geometry of the balls
depend both on the location of the center and the radius.
\begin{figure}[ht]
  \centering
  \includegraphics[width=3in]{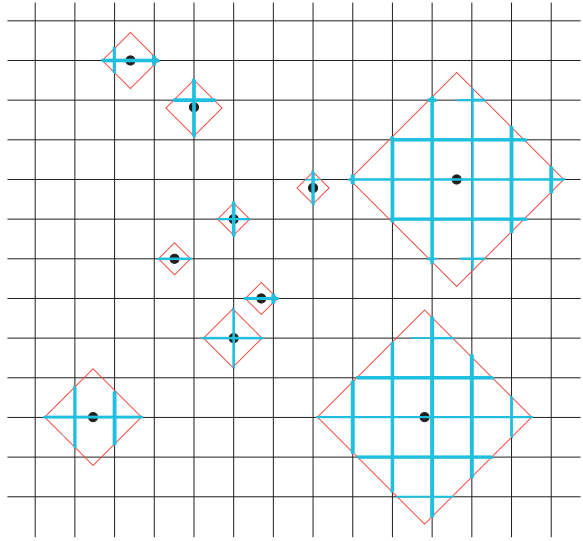}
  \caption{The blue lines shows the balls (in intrinsic metric) of various sizes while the red squares correspond to the $L^1$ balls.}
  \label{f-grid}
\end{figure}
\end{example}
Next, we consider an example from \cite[Example 3.14]{GS05}.
\begin{example}[Model manifolds/Surfaces of revolution]
 Given a  smooth function $\psi:(0,\infty) \to (0,\infty)$, denote by $M_\psi$ a \emph{model manifold}.
 Here by model manifold, we mean $\R^N$ equipped with the Riemannian metric in polar coordinates $(r,\theta) \in (0, +\infty) \times \mathbb{S}^{N-1}$ by
 \[
  \operatorname{d}s^2= \operatorname{d} r^2 + \psi(r)^2\operatorname{d}\theta^2,
 \]
where $\operatorname{d}\theta^2$ is the standard metric on $\mathbb{S}^{N-1}$ and $\psi$ is a smooth positive function on $(0,+\infty)$. The necessary and
sufficent conditions under which $\operatorname{d}s^2$ can be smoothly extended to a metric on the entire space $\R^N$ is given by
\[
 \psi(0) =0, \hspace{3mm} \psi'(0)=1, \hspace{3mm}\mbox{and } \psi''(0)=0
\]
(see \cite[equation (4.12)]{GS05}). Therefore we may choose $\psi(r)= r^\alpha$ where $\alpha \in \R$ for all $r \ge 1$ and extend it smoothly satisfying the above conditions.

It is known that that $M_\psi$ with $\psi(r)= r^\alpha$ for $r \ge 1$ satisfies parabolic Harnack inequality if and only if $-1/(N-1) < \alpha \le 1$. The Riemannian measure (in polar coordinates) is given by
$\operatorname{d}\mu = \psi(r)^{N-1}\, \operatorname{d}r \,\operatorname{d}\theta$.
One can check that the  reversible measure $\mu'$ for the ball walk satisfies
$d\mu' = V(x,1) \, d\mu \approx \psi(r)^{2(N-1)}\, dr \,d\theta$ for the case $\alpha < 0$ and $d\mu' = V(x,1) \, d\mu \approx d\mu$ for the case $\alpha  \ge 0$ (see \cite[p. 856]{GS05}).
Therefore the ball walk on the model manifold $M_\psi$ with $\psi(r)= r^\alpha$ for $r \ge 1$ satisfies the parabolic Harnack inequality and two-sided Gaussian estimates
if and only if $ -1/2(N-1) < \alpha  \le 1$. Using Lemma \ref{l-rvd}, it is easy to verify  that $M_\psi$ with $\psi(r)= r^\alpha$ for $r \ge 1$ equipped with the measure $\mu'$ fails to satisfy \ref{doub-inf}
if $\alpha \le  -1/2(N-1)$. Therefore for the case $-1/(N-1) <\alpha \le  -1/2(N-1)$, the model manifold $M_\psi$ defined above satisfies two sided Gaussian estimates for diffusion but fails to satisfy
two sided Gaussian estimates for the ball walk.

These model manifolds can also be considered as  surfaces of revolution formed by the graph of the function $\psi$ (see \cite[Section 5.1]{GS02}). 
Similar to Example \ref{x-nongauss}, the ball walk and diffusions may exhibit different behaviors in terms of null recurrence, positive recurrence and transience depending on $\alpha$ and $N$.
\end{example}
\begin{example}[Bodies of revolution]
 Another related class of examples given in \cite[Section 5.2]{GS02} are bodies of revolution. Let $f:[0,\infty) \to [0,\infty)$ be a concave function with $f(0)=0$. 
 Then the body of revolution in $\R^{n}$ (with $n>1$) defined by 
 \[
  M:= \Sett{(u,t) \in \R^n}{ u \in \R^{n-1}, t \ge 0, \norm{u}_2 \le f(t)}
 \]
where $\norm{\cdot}_2$ above denotes the Euclidean norm in $\R^{n-1}$ (See figure \ref{f-revolve}).
\begin{figure}[ht]
  \centering
  \includegraphics[width=3in]{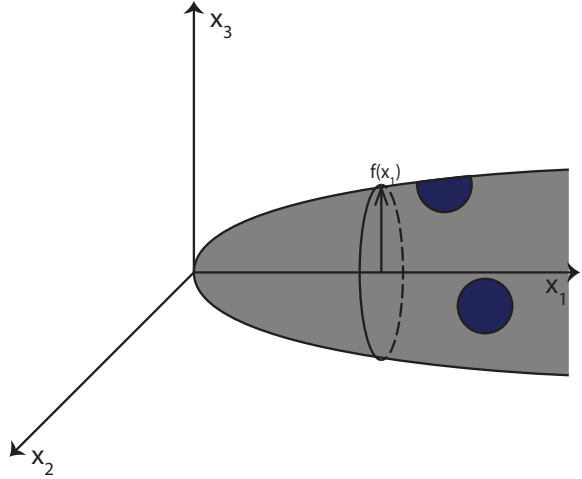}
  \caption{The body of revolution corresponding to $f$ in $\R^3$. This figure shows two balls of the same radius.}
  \label{f-revolve}
\end{figure}
Note that since $f$ is concave, $M$ is a convex subset of $\R^n$.
By the results of \cite{LY86}, we have that $M$ satisfies two-sided Gaussian estimates for the heat kernel corresponding to the Neumann Laplacian. 
Hence by Theorem \ref{t-grisal} the Neumann Laplacian satisfies Poincar\'{e} inequality in $M$ and satisfies volume doubling.
 If we set $f(x)=x^\alpha$ for some $\alpha \in (0,1)$ and by Proposition \ref{p-rmtfae} and Theorem \ref{t-main0} the ball walk on $M$ satisfies two-sided Gaussian bounds. 
 Hence by Proposition \ref{p-trans} the corresponding ball walk on $M \subset \R^n$ with $f(x)=x^\alpha$ is transient if and only if $\alpha(n-1) >1$.
\end{example}
\backmatter
\bibliographystyle{amsalpha}

\printindex

\end{document}